\documentclass[10pt,leqno]{amsart}
\usepackage{amsmath}
\usepackage{amssymb}
\usepackage{amsfonts}
\usepackage{amsthm}
\usepackage{mathrsfs}
\usepackage{dsfont}
\usepackage{mathtools}
\usepackage{bm}
\usepackage{tikz-cd}
\usepackage{todonotes}

\usepackage[backref=page]{hyperref}
\usepackage{color}

\definecolor{darkgreen}{rgb}{0.5,0.25,0}
\definecolor{darkblue}{rgb}{0,0,1}
\definecolor{answerblue}{rgb}{0,0,0.75}

\hypersetup{colorlinks,breaklinks,
linkcolor=darkblue,urlcolor=darkblue,
anchorcolor=darkblue,citecolor=darkblue}

\newcommand*{\mailto}[1]{\href{mailto:#1}{\nolinkurl{#1}}}

\newcommand{\ep}{\varepsilon}
\newcommand{\eps}{\varepsilon}
\newcommand{\pd}{\partial}
\newcommand{\LL}{\langle}
\newcommand{\RR}{\rangle}

\newcommand{\Ex}{\mathbb{E}}

\renewcommand{\d}{\mathrm{d}}

\newcommand{\X}{\mathfrak{X}}

\newcommand{\sgn}{\mathrm{sgn}}

\newcommand{\loc}{\mathrm{loc}}

\newcommand{\doublehookrightarrow}{
    \lhook\joinrel\relbar\mspace{-12mu}\hookrightarrow
}

\newcommand{\one}[1]{\mathds{1}_{\{#1\}}}
\newcommand{\weak}{\rightharpoonup}

\newcommand{\R}{\mathbb{R}}
\newcommand{\T}{{\mathbb{S}^1}}
\newcommand{\N}{\mathbb{N}}
\newcommand{\Z}{\mathbb{Z}}
\newcommand{\prob}{\mathbb{P}}

\newcommand{\abs}[1]{\left | #1 \right |}
\newcommand{\norm}[1]{\left\| #1 \right\|}
\newcommand{\normb}[1]{\bigl\| #1 \bigr\|}
\newcommand{\bk}[1]{ \left(  #1 \right)}

\newcommand{\cS}{\mathcal{S}}
\newcommand{\cL}{\mathcal{L}}
\newcommand{\cD}{\mathcal{D}}
\newcommand{\bD}{\Bbb{D}}

\newcommand{\toeps}{\xrightarrow{\eps\downarrow 0}}
\newcommand{\totau}{\xrightarrow{\tau\downarrow 0}}
\newcommand{\todelta}{\xrightarrow{\delta\downarrow 0}}

\newcommand{\ton}{\xrightarrow{n\uparrow \infty}}
\newcommand{\toj}{\xrightarrow{j\uparrow \infty}}
\newcommand{\tok}{\xrightarrow{k\uparrow \infty}}

\newcommand{\toepsweak}{\xrightharpoonup{\eps\downarrow 0}}
\newcommand{\tonweak}{\xrightharpoonup{n\uparrow \infty}}

\newcommand{\tojweak}{\xrightharpoonup{j\uparrow \infty}}
\newcommand{\toell}{\xrightarrow{\ell\uparrow \infty}}

\makeatletter
\@namedef{subjclassname@2020}{%
  \textup{2020} Mathematics Subject Classification}
\makeatother

\theoremstyle{theorem}
\newtheorem{thm}{Theorem}[section]
\newtheorem{prop}[thm]{Proposition}
\newtheorem{lem}[thm]{Lemma}

\theoremstyle{definition}
\newtheorem{defin}[thm]{Definition}
\theoremstyle{theorem}
\newtheorem{rem}[thm]{Remark}

\numberwithin{equation}{section}

\usepackage{xcolor}

\allowdisplaybreaks

\title[stochastic Camassa-Holm equation]
{Global existence of dissipative solutions 
to the Camassa--Holm equation with transport noise}

\author[Galimberti]{L. Galimberti}
\address[Luca Galimberti]{Department of Mathematical Sciences\\
NTNU Norwegian University of Science and Technology\\
NO-7491 Trondheim\\ Norway}
\email{\mailto{luca.galimberti@kcl.ac.uk}}

\author[Holden]{H. Holden}
\address[Helge Holden]{Department of Mathematical Sciences\\
NTNU Norwegian University of Science and Technology\\
NO-7491 Trondheim\\ Norway}
\email{\mailto{helge.holden@ntnu.no}}
\urladdr{\url{https://www.ntnu.edu/employees/holden}}

\author[Karlsen]{K. H. Karlsen}
\address[Kenneth H. Karlsen]{Department of Mathematics\\
University of Oslo\\
NO-0316 Oslo\\ Norway}
\email{\mailto{kennethk@math.uio.no}}

\author[Pang]{P.H.C. Pang}
\address[Peter H.C. Pang]{Department of Mathematics\\
University of Oslo\\
NO-0316 Oslo\\ Norway}
\email{\mailto{ptr@math.uio.no}}

\dedicatory{Dedicated to Gui-Qiang Chen 
on the occasion of his sixtieth birthday}

\subjclass[2020]{Primary: 35R60, 35G25; Secondary: 35A01, 35D30}

\keywords{Shallow water equation; Camassa--Holm equation; 
stochastic perturbation; transport noise; existence; viscous approximation} 

\date{\today}

\begin{document}

\begin{abstract}
We consider a nonlinear stochastic partial 
differential equation (SPDE) that takes the form 
of the Camassa--Holm equation perturbed by 
a convective, position-dependent, noise term. 
We establish the first global-in-time existence result 
for dissipative weak martingale solutions to this SPDE, 
with general finite-energy initial data. The solution is 
obtained as the limit of classical 
solutions to parabolic SPDEs. 
The proof combines model-specific 
statistical estimates with stochastic 
propagation of compactness techniques, 
along with the systematic use of tightness 
and a.s.~representations of random variables on 
specific quasi-Polish spaces. 
The spatial dependence of the noise function 
makes more difficult the analysis of a priori estimates 
and various renormalisations, giving rise to nonlinear 
terms induced by the martingale part of 
the equation and the second-order 
Stratonovich--It\^{o} correction term.
\end{abstract}

\maketitle

\setcounter{secnumdepth}{2}
\setcounter{tocdepth}{2}
{\small \tableofcontents}

\section{Introduction}\label{sec:intro}

\subsection{Background and main result}
We are interested in global weak solutions of the 
initial-value problem for the stochastic 
parabolic-elliptic system 
\begin{equation}\label{eq:SCH-tmp}
	\begin{aligned}
		& 0=\d u+ \bigl[ u\, \pd_x u 
		+ \pd_x P\bigr] \, \d t 
		+\sigma \pd_x u \circ \d W,
		\\ & -\pd_{xx}^2 P+P = u^2 
		+\frac{1}{2} \abs{\pd_x u}^2,
		\quad \text{for $(t,x)\in (0,T)\times \T$}, 
	\end{aligned}
\end{equation}
where $\T=\R/(2\pi\Z)$ is the 1D torus (circle), 
$T$ is a positive final time, $\sigma=\sigma(x)\in 
W^{2,\infty}(\T)$ is a position-dependent noise function, 
and $W$ is a 1D Wiener process defined 
on a standard filtred probability 
space $\cS=\bigl(\Omega, \mathcal{F},
\{\mathcal{F}_t\}_{t\in [0,T]},\mathbb{P}\bigr)$, 
henceforth called a \textit{stochastic basis}. 
Formally, by the It\^{o}--Stratonovich 
conversion formula, the Stratonovich differential 
$\sigma\,\pd_x u\circ \d W$ in \eqref{eq:SCH-tmp}---known 
in the literature as a gradient, transport 
or convection noise term---can be expanded 
into the operational form $- \frac12 \sigma(x) 
\pd_x \bk{\sigma(x)\pd_x {u}}\,\d t
+\sigma(x) \pd_x {u} \, \d W$. Moreover, the 
elliptic equation for $P$ can be solved to supply
\begin{equation}\label{eq:P-def}
	P = P[u]:=K*\left(u^2+\frac{1}{2}
	\abs{\pd_x u}^2\right),
	\quad 
	K(x)=\frac{\cosh\left(x-2\pi
	\operatorname{int}\!\left(\tfrac{x}{2\pi}\right)
	-\pi\right)}{2\sinh(\pi)},
\end{equation}
where $K$ is the Green's function of 
$1-\pd_{xx}^2$ on $\T$, $\operatorname{int}(x)$ is 
the integer part of $x$, and $*$ means convolution in $x$. 
Consequently, \eqref{eq:u_ch_ep} takes the form of 
the nonlinear nonlocal SPDE
\begin{equation}\label{eq:u_ch}
	\begin{aligned}
		0 &=\d u + \bigl[u\, \pd_x u
		+\pd_x P \bigr] \, \d t 
		-\frac12\sigma \pd_x \bk{\sigma \pd_x u}\,\d t
		+\sigma \pd_x u \, \d W,\\
		P & = K*\bk{u^2 + \frac12\abs{\pd_x u}^2}.
	\end{aligned}
\end{equation}

We recover the deterministic Camassa--Holm 
(CH) equation by setting $\sigma\equiv 0$ 
in \eqref{eq:u_ch}. Since its introduction in the 
early 1980s \cite{Camassa:1993zr,Fuchssteiner:1981fk}, 
the CH equation has received much attention 
from the mathematical community. 
The CH equation, a nonlinear dispersive 
PDE modelling shallow-water waves, 
is nonlocal, completely integrable and may be written in 
(bi-)Hamiltonian form in terms of the 
momentum variable $m:=\bigl(1-\pd_{xx}^2\bigr)u$. 
Much of the excitement of the CH 
equation is related to its supercritical 
nature---coming from the competition between the dispersive 
and nonlinear terms---which leads to 
the development of singularities in finite 
time (blow-up via wave breaking). 
The question of global well-posedness 
of the CH equation, in different classes of 
appropriately defined weak solutions, is 
widely studied, see for example 
\cite{Bressan:2007fk,Bressan:2006nx,
Coclite:2005tq,Holden:2007wq,Holden:2009aa,Xin:2000qf} (and 
the references therein). 
Indeed, there are two natural classes of $H^1$ 
weak solutions, \textit{dissipative} and 
\textit{conservative}, which differ in how they 
continue the solution past the blow-up time.
Conservative solutions (see, e.g., \cite{Bressan:2007fk}) 
ask that the PDE holds weakly and that the 
total energy is preserved. In contrast, dissipative 
solutions (see, e.g., \cite{Xin:2000qf}) 
are characterized by a drop in the total 
energy at the time of blow-up. 
Starting from general finite-energy 
data $u|_{t=0}=u_0\in H^1$, the CH solution 
operator formally preserves the $H^1$ norm, and 
$H^1$ regularity is also needed to make distributional 
sense of the equation. The solution 
space $H^1$ allows for wave breaking, in the sense  
that the solution $u$ remains bounded while 
its $x$-derivative $\pd_x u$ becomes (negatively) 
unbounded \cite{Camassa:1993zr}. 

\medskip

Stochastic effects, in terms of transport, forcing, or 
uncertain system parameters, are vital for 
developing models of many 
phenomena in fluid dynamics. 
The work of Holm \cite{Holm:2015tc} 
proposes a general approach to deriving SPDEs 
for fluid dynamics from geometric 
mechanics and a stochastic variational principle. 
In particular, he argues that 
``physically relevant" noise arises from 
a suitable perturbation of the integrated Hamiltonian 
of the dynamical system. The corresponding 
stochastic perturbation of the CH equation leads 
to nonlinear SPDEs like \eqref{eq:u_ch}, see 
\cite{Crisan:2018aa} and \cite{Bendall:2021tx}. 
The works \cite{Bendall:2021tx,Crisan:2018aa} 
also investigate blow-up of regular solutions. 
For the related stochastic Hunter--Saxton equation, 
see \cite{Holden:2020aa,Holden:2021vw}.  
We refer to Appendix \ref{sec:derivation-SCH} 
for a short formal derivation of 
the stochastic CH equation \eqref{eq:u_ch}.

\medskip

Let us now turn to the mathematical analysis 
of the stochastic CH equation \eqref{eq:u_ch}. 
Currently, only a few local well-posedness 
results are available. Most 
of them concern the stochastic forcing case, which 
corresponds to \eqref{eq:SCH-tmp} with the transport 
noise  $\sigma(x) \pd_x u \circ \d W$ replaced by 
a lower order It\^{o} term $\sigma(x,u)\, \d W $, either in 
additive ($\sigma(x) \, \d W $) or multiplicative 
($\sigma(u) \, \d W$) form, see the 
works \cite{Chen:2016aa,Chen:2012aa,Chen:2020wp,
Huang:2013ab,Lv:2019wt,Rohde:2021aa,Tang:2018aa,
Tang:2020vf,Zhang:2020vy,Zhang:2021wl}. 
See also \cite{Chen:2021tr} for a global existence 
result if $\sigma\equiv u$ and $m(0)\ge 0$. 

For the CH equation perturbed by transport noise, like 
the term $\sigma \pd_x u \circ \d W$ appearing 
in \eqref{eq:u_ch}, we refer to Albeverio, Brze\'{z}niak, 
and Daletskii \cite{Albeverio:2021uf} for the first local 
well-posedness result (up to wave-breaking). The idea 
in \cite{Albeverio:2021uf} is to transform the equation 
into a PDE with random coefficients and 
apply Kato's operator theory. The work of 
Alonso-Or\'an, Rohde, and Tang \cite{Alonso-Oran:2021wt} 
extends this result to a stochastic 
two-component CH system with transport 
noise (for smooth noise functions $\sigma$). 
Let us also draw attention to a recent study \cite{CL2022} that
investigates the existence of weak solutions for a
two-component CH equation affected
by Markus pure-jump noise. A general Marcus SDE is structured
as follows: $du = a \, \d s+ b \circ \d W+c[u]\, \diamond \, \d L$,
where $L$ represents a pure jump-Lévy process,
and $c[u]\, \diamond \, \d L$ is interpreted within the Markus framework.
The study \cite{CL2022} zeroes in on the pure jump 
component $c[u]\, \diamond \, \d  L$
in this decomposition, particularly when $c[u]=\pd_x u$.
Aside from the fact that examining this case is more
straightforward than dealing with the Wiener noise
$\sigma(x) \pd_x u \circ \d W$, which is the focus of our paper, 
the critical difference lies in the solution class
for analyzing the stochastic (two-component) CH equation.
In their work, the authors of \cite{CL2022} devise solutions in
which $\pd_x u$ is a bounded function. However, when this is
confined to the context of the CH equation, such a solution
class becomes overly restrictive, essentially necessitating
that the initial data satisfy $u_0-\pd_{xx}^2u_0\ge 0$.
This limitation omits crucial solutions involving
peakon-antipeakon interactions, where $\pd_x u$
could potentially blow up or become unbounded.
In contrast, our result is general, applicable to 
any $u_0\in H^1$ (where $\pd_x u$ may not be a bounded function).
Yet, this wide applicability entails a significantly more complex
analytical approach, which we will
elaborate upon later in our discussion.

\medskip

The global existence of properly 
defined weak solutions for the stochastic 
CH equation \eqref{eq:u_ch} is an open problem, 
addressed in this paper for the first time. 
We develop an existence theory for dissipative 
weak solutions for rather general ``non-smooth" noise 
functions $\sigma\in W^{2,\infty}$.  Our main result is 
the following theorem:

\begin{thm}[existence of dissipative solution]
\label{thm:main}
Let $\sigma \in W^{2, \infty}(\T)$, and 
fix some $p_0>4$. For any initial probability 
distribution $\Lambda$ supported 
on $H^1(\T)$, satisfying 
$$
\int_{H^1(\T)}
\norm{v}^{p_0}_{H^1(\T)}{\Lambda}(\d v)<\infty,
$$
there exists a dissipative weak martingale solution 
$\bigl(\tilde{\cS},\tilde{u},\tilde{W}\bigr)$ to the stochastic 
CH equation \eqref{eq:u_ch} 
with random initial data {$\tilde u_0$} distributed 
according to $\Lambda$ (${\tilde u_0}\sim \Lambda$), 
where $\tilde{\cS}=\bigl(\tilde{\Omega}, \tilde{\mathcal{F}},
\{\tilde{\mathcal{F}}_t\}_{t\in [0,T]},
\tilde{\mathbb{P}}\bigr)$ is a stochastic 
basis. Besides, the following energy inequality holds 
$\tilde{\mathbb{P}}$--a.s., for a.e.~$s\in [0,T)$ and every 
$t$ with $s< t\leq T$,
\begin{equation}\label{eq:energybalance-sch}
	\begin{aligned}
		&\int_\T \tilde u^2 + \abs{\pd_x \tilde u}^2
		\,\d x \bigg|_s^t \\
		& \qquad \le  \int_s^t 
		\int_\T \frac14 \pd_{xx}^2 \sigma^2  \tilde u^2
		+\bk{\abs{\pd_x \sigma}^2
		-\frac14 \pd_{xx}^2\sigma^2}
		\,\abs{\pd_x \tilde u}^2
		\,\d x \,\d t' \\
		&\qquad \qquad  
		+\int_s^t \int_\T \pd_x 
		\sigma\bk{\tilde u^2-\abs{\pd_x \tilde u}^2}
		\,\d x \, \d \tilde{W}.
	\end{aligned}
\end{equation}
Specifically, it holds for $s=0$ and any $t\in (0,T]$, with 
$\int_\T \tilde u^2 + \abs{\pd_x \tilde u}^2\,\d x \big|_{s=0}$
replaced by $\int_\T\tilde u_0^2+\abs{\pd_x \tilde u_0}^2\,\d x$.
\end{thm}

Roughly speaking, by a solution to \eqref{eq:u_ch} 
we mean a collection $\bigl(\tilde \cS,\tilde u,\tilde W)$, 
where $\tilde \cS$ is a stochastic basis, $\tilde W$ 
is a Wiener process, and $(\omega,t)\mapsto 
{\tilde u}(\omega,t,\cdot)$ takes values in $H^1(\T)$ and satisfies 
the SPDE \eqref{eq:u_ch} in the weak sense in $x$, 
see Definition~\ref{def:dissp_sol} for details. 
Note that the solutions constructed in 
Theorem \ref{thm:main} are weak 
in the probabilistic sense, as the stochastic 
basis $\tilde \cS$ and the Wiener process $\tilde W$ 
are parts of the unknown solution. We 
refer to these solutions as dissipative 
weak martingale solutions.  
The term ``weak" in the quantifier ``dissipative weak" 
indicates that the solutions are considered 
weak solutions in the PDE sense.  Furthermore, at least in the 
deterministic case ($\sigma=0$), the solutions possess 
the additional property that the total energy decreases over time, 
specifically at a wave breaking time $t_0$. 
The term ``dissipative" also alludes to the methodology 
employed to construct these solutions, which is the vanishing 
viscosity method.

A manifestation of the dissipative nature 
of the solutions is that the total energy 
inequality encodes a fundamental 
right-continuity property; namely, we will 
prove that $\tilde u(t)\to \tilde{u}(t_0)$ 
in $H^1(\T)$, a.s., as $t\downarrow t_0\in [0,T)$. 
In the deterministic setting $\sigma=0$, 
Theorem \ref{thm:main} recovers the main result of Xin 
and Zhang \cite{Xin:2000qf}.

\subsection{Outline of main ideas}
Let us end this introduction by briefly 
expounding the main ideas behind the 
proof of Theorem \ref{thm:main}. 
Although the proof makes use of the vanishing viscosity method 
and weak convergence techniques, there are many substantial 
differences between the deterministic and stochastic situations.  
Adding the viscosity term $\eps \pd_{xx}^2 u$ to \eqref{eq:u_ch}, we first 
construct a regular solution $u_\eps$ to
\begin{equation}\label{eq:u_ch_ep}
	\begin{aligned}
		& 0 = \d u_\ep+ \left[u_\ep \, \pd_x u_\ep 
		+ \pd_x P_\ep - \ep \pd_{xx}^2 u_\ep \right] \, \d t 
		- \frac12 \sigma_\ep \pd_x 
		\bk{\sigma_\ep \pd_x u_\ep}\,\d t
		+\sigma_\ep \pd_x u_\ep \, \d W,
		\\ 
		& P_\ep = P[u_\ep] 
		:=  K*\left(u^2_\ep +\frac{1}{2}
		\abs{\pd_x u_\ep}^2\right).
	\end{aligned}
\end{equation}
This is a non-standard (nonlinear and nonlocal) 
parabolic SPDE. Its global-in-time 
well-posedness does not follow from standard 
parabolic SPDE theory. In \cite{HKP-viscous}, we 
prove the existence and uniqueness of pathwise $H^m_x$ 
solutions for arbitrary $m\in \N$ (as long as the initial data are smooth). 
Notice that in \eqref{eq:u_ch_ep} we have 
replaced $\sigma$ of \eqref{eq:SCH-tmp} with 
$\sigma_\ep \in C^\infty(\T)$, which we require to converge to $\sigma$ 
in $W^{2,\infty}(\T)$ as $\ep \downarrow 0$. This is 
necessary as the well-posedness of $H^m_x$ solutions require 
coefficients $\sigma_\ep \in W^{m + 1, \infty}(\T)$ 
\cite[Theorem 1.2]{HKP-viscous}.

The relevant results from \cite{HKP-viscous} are collected in 
Theorem \ref{thm:bounds1} below. 
In particular, only a few 
$\eps$-uniform statistical estimates are available 
(starting from smooth finite-energy initial data), including
\begin{equation}\label{eq:intro-est-tmp}
	\begin{split}
		& \Ex \norm{u_\eps}_{C_t^\theta L^2_x}^r
		\lesssim 1, \quad
		\text{for some $r>2$ and small $\theta$}, 
		\\ & 
		\norm{q_\eps}_{L^{2+\alpha}_{\omega,t,x}}
		\lesssim 1, \quad 
		\text{for any $\alpha\in [0,1)$},
	\end{split}
\end{equation}
see Sections \ref{sec:sol_def} 
and \ref{sec:apriori}, where the spatial 
gradient $q_\eps:=\pd_x u_\eps $ 
satisfies the nonlinear, second-order transport-type SPDE
\begin{equation}\label{eq:intro-qest-tmp}
	\begin{split}
		 0 & = \d q_{\ep} 
		 +\left(\pd_x \bk{u_{\ep} \, q_{\ep}}
		 -\frac12 q_{\ep}^2+
		 P_\eps-u_{\ep}^2  
		 -\ep\, \pd_{xx}^2 q_{\ep} \right) \,\d t
		 \\ & \qquad \qquad
		 - \frac12 \pd_x \bk{\sigma_\ep\,
		 \pd_x\bk{\sigma_\ep\,q _{\ep}}} \,\d t
		 + \pd_x\bk{\sigma_\ep \, q_{\ep}}\,\d W.
	\end{split}
\end{equation}

The starting point for deducing 
$\eps$-uniform estimates is the SPDE 
satisfied by the total energy $\frac12 \bigl(u_\ep^2
+q_\ep^2\bigr)$, which is formally 
obtained by testing---via the 
temporal (It\^{o}) and spatial chain rules, 
the SPDE \eqref{eq:u_ch_ep} with $u_\eps$ 
and the SPDE \eqref{eq:intro-qest-tmp} with $q_\eps$, 
and then adding the resulting equations, 
noticing some crucial cancellations involving 
cubic terms of $q_\eps$. The end result is
\begin{align}
	& \d \left(\frac{u_\eps^2+q_\eps^2}{2}\right)
	+\pd_x \left[u_\eps \frac{u_\eps^2+q_\eps^2}{2} 
	+u_\eps P_\eps-\frac{u_\eps^3}{2}
	-\frac14 \pd_x\sigma_\ep^2
	\frac{q_\eps^2-u_\eps^2}{2}\right] \, \d t
	\label{eq:intro-visc-energy} 
	\\ & \quad 
	-\pd_{xx}^2\left[ \left (\frac12\sigma_\ep^2
	+\eps\right)\frac{u_\ep^2+q_\ep^2}{2}\right]\, \d t 
	+\left[\pd_x\left(\sigma_\ep \frac{u_\eps^2+q_\eps^2}{2}\right)
	+\pd_x \sigma_\ep \left(\frac{q_\eps^2-u_\eps^2}{2}\right)\right] \,\d W
	\notag \\ & \quad \quad 
	=\frac14\pd_{xx}^2\sigma_\ep^2
	\frac{u_\eps^2}{2} \,\d t
	+\left(\abs{\pd_x\sigma_\ep}^2
	-\frac14 \pd_{xx}^2\sigma_\ep^2\right)
	\frac{q_\eps^2}{2}\,\d t
	-\eps \left(\abs{\pd_x u_\eps}^2
	+\abs{\pd_x q_\eps}^2\right) \, \d t.
	\notag
\end{align}

The second estimate in \eqref{eq:intro-est-tmp} 
implies, passing if necessary to a subsequence,
\begin{equation}\label{eq:intro-weak-conv}
	q_\eps \toepsweak q 
	\,\, \text{in $L^p_{\omega,t,x}$, $p\in [1,3)$}, 
	\quad 
	q_\eps^2 \toepsweak \overline{q^2} 
	\,\, \text{in $L^p_{\omega,t,x}$, $p\in [1,3/2)$}, 
\end{equation}
for some weak limits $q,\overline{q^2}$. 
Throughout this paper, we use overbars to denote weak limits, 
in spaces that often must be understood from the context. 
Only equipped with weak convergence 
of $\bigl\{q_\eps^2=\abs{\pd_x u_\eps}^2
\bigr\}_{\eps>0}$---because of the nonlinearity--- 
it is not possible to pass to the limit $\eps\to 0$ 
in \eqref{eq:u_ch_ep}, \eqref{eq:intro-qest-tmp} 
to obtain a solution of the stochastic 
CH equation \eqref{eq:u_ch}; strong $L^2$ 
convergence of $\left\{q_\eps\right\}_{\eps>0}$ is 
called for. 

An effective (deterministic) strategy  for 
improving the weak convergence to the 
required strong one is to start from a strongly 
convergent sequence of initial data and then 
attempt to propagate that strong convergence through time. 
This ``propagation of compactness" argument is typically 
implemented in the context of DiPerna--Lions 
renormalised solutions \cite{DiPerna:1989aa}; for 
some applications of this strategy, see 
\cite{Feireisl:2004oe,Lions:1998ga} 
(compressible Navier--Stokes equations) and 
\cite{Coclite:2015aa,Coclite:2005tq,Xin:2000qf} (CH equation).

The tailoring of the propagation of 
compactness argument to the stochastic 
CH equation \eqref{eq:u_ch} is 
rather involved. Let us explain 
some of the reasons for this. First, 
we need to use the few available 
estimates \eqref{eq:intro-est-tmp} 
to extract some strong (almost sure) compactness 
in the probability variable $\omega$. 
Indeed, a feature of our approach is that most results 
are derived in a pathwise context, meaning that 
equations and inequalities hold almost 
surely (not only in the weaker statistical mean sense). 
The natural strategy for achieving a.s.~convergence 
is to invoke some nontrivial results of Skorokhod, linked 
to the tightness (weak compactness) 
of probability measures and 
a.s.~representations of random variables, see 
\cite[Theorem 2.4]{DaPrato:2014aa} and, e.g., 
\cite{Bensoussan:1995aa,Debussche:2011aa,
Flandoli:1995aa,Glatt-Holtz:2014aa,Hausenblas:2013aa} 
for some applications of this approach to SPDEs. 
Applying this strategy to the laws $\cL(u_\eps)$ of 
$u_\eps$---defined on the Polish space 
$C([0,T];L^2(\T))$ and whose tightness 
is guaranteed by the first estimate 
in \eqref{eq:intro-est-tmp}---we obtain 
new random variables $\tilde u_\eps$---defined 
on a new probability space and with the same 
laws as the original variables $u_\eps$---which 
converge almost surely to some $\tilde u$:
\begin{equation}\label{eq:intro-ueps-conv}
	\tilde u_\eps \toeps \tilde u 
	\quad \text{in $C([0,T];L^2(\T))$, 
	almost surely.} 
\end{equation}

Next, we wish to apply this strategy to 
improve the ($\omega,t,x$) weak convergence 
\eqref{eq:intro-weak-conv} to a.e.~convergence 
in $\omega$, weak in $(t,x)$. The original Skorokhod 
construction applies to processes taking values in a Polish 
(complete separable metric) space. In our context the 
Skorokhod theorem is not directly applicable, 
because we have to work in spaces 
equipped with the weak topology, like 
$L^p_{t,x} - w$, which are not Polish. Therefore we use 
a recent version of the Skorokhod theorem---due 
to Jakubowski \cite{Jakubowski:1997aa}---that 
applies to so-called quasi-Polish spaces, where 
\textit{quasi-Polish} refers to a Hausdorff space  
that exhibits a continuous injection into a Polish space. 
It turns out that separable Banach spaces 
equipped with the weak topology as well as dual 
spaces of separable Banach spaces (equipped with the 
weak-star topology) are quasi-Polish. 
For relevant background material on 
quasi-Polish spaces, see Appendix \ref{sec:qpolish}.
We refer to Brze\'zniak and Ondrej\'at 
\cite{Ondrejat:2010aa,Brzezniak:2013ab} and 
\cite{Breit:2016aa,Brzezniak:2013aa,Brzezniak:2011aa,
Punshon-Smith:2018aa,SmithTrivisa:2018,Wang-StochNS:2015} 
for some applications of the Skorokhod--Jakubowski theorem 
to different SPDEs (this list is far from complete).

The second estimate in \eqref{eq:intro-est-tmp} implies 
that the laws $\cL(q_\eps)$ and $\cL(q_\eps^2)$ are tight
as probability measures on the quasi-Polish 
space $L^p([0,T]\times \T)-w$, 
respectively for $p\in [1,3)$ ($q_\eps$) and 
$p\in [1,3/2)$ ($q_\eps^2$). An application of 
the Skorokhod--Jakubowski theorem 
supplies new random variables $\tilde q_\eps$ and 
$\tilde q_\eps^2$ defined on the same 
probability space as $\tilde u_\eps$ and with 
the same laws as the original variables $q_\eps$ 
and $q_\eps^2$, such that (extracting a 
subsequence if necessary and for the same 
values of $p$ as before)
\begin{equation}\label{eq:intro-weak-conv2}
	\tilde q_\eps \toepsweak \tilde q 
	\,\, \text{in $L^p_{t,x}$, a.s.}, 
	\quad 
	\tilde q_\eps^2 \toepsweak \overline{\tilde q^2} 
	\,\, \text{in $L^p_{t,x}$, a.s.},
\end{equation}
for some limits $\tilde q$ and $\overline{\tilde q^2}$, 
see Section \ref{sec:Jak-Skor}. 

It is of vital importance to us that products like 
$S'(\tilde q_\ep)\tilde P_\ep$ converge weakly, for a 
suitable class of linearly growing nonlinearities $S(\cdot)$, 
where $\tilde P_\ep$ is defined in \eqref{eq:u_ch_ep}. 
Since $S'(\tilde q_\ep)$ converges weakly, $\tilde P_\ep$ 
must converge strongly. This strong convergence does not follow 
from \eqref{eq:intro-weak-conv2}, as we are missing 
strong temporal compactness for $\tilde q_\eps^2$. 
In the deterministic theory \cite{Xin:2000qf}, 
one establishes directly uniform $W^{1,1}_{t,x}$ estimates 
for $\tilde P_\ep$, which implies strong convergence. 
This strategy does not work in the stochastic setting. A natural 
modification of this strategy, based on the derivation 
of uniform H\"older continuity in $t$, does not 
seem to accomplish the task either, even if the spatial topology is weak. 
As a result, we cannot apply the often-used compactness approach 
based on tightness in the (quasi-Polish) 
space $C([0,T];L^p(\T)-w)$, used by many of the 
references above. 

These obstructions have 
motivated us to introduce the locally convex space
$L^p\bigl(L^p_w\bigr)=L^p([0,T];L^p(\T)-w\bigr)$, 
which is quasi-Polish (see Appendix \ref{sec:qpolish}). 
The space $L^p\bigl(L^p_w\bigr)$ can account for 
strong temporal and weak spatial convergence of the energy variable 
$\tilde q_{\ep_n}^2$. To this end, we 
formulate a new tightness criterion in 
$L^p\bigl(L^p_w\bigr)$, which we believe is 
of independent interest: the probability laws 
of a sequence $\left\{Q_n\right\}_{n\in \N}$ 
of random variables is tight 
on $L^p\bigl(L^p_w\bigr)$ provided
\begin{align*}
	&\mathrm{(i)}\,\, \, 
	\Ex \norm{Q_n}_{L^p([0,T];L^p(\T))}
	\lesssim 1,
	\\ &
	\mathrm{(ii)}\,\, \, 
	\Ex \norm{Q_n}_{L^{\bar p}([0,T];L^1(\T))}
	\lesssim 1, \quad \text{for some $\bar p>p$},
\end{align*}
and, for all $\varphi\in C^\infty(\T)$ and $\vartheta>0$,  
\begin{align*}
	\mathrm{(iii)}\, \, \, \Ex \sup_{\tau\in (0,\vartheta)}
	\int_0^{T-\tau}\abs{\, \int_{\T} \varphi(x)\bigl(Q_n(t+\tau,x)
	-Q_n(t,x)\bigr)\, \d x}\, \d t 
	\lesssim_\varphi \vartheta^\alpha,
\end{align*}
for some $\alpha\in (0,1)$. 
We verify these conditions for the energy 
variable $\tilde q_\ep^2$, thereby supplying 
the following critical improvement over \eqref{eq:intro-weak-conv2}: 
$\tilde q_\eps^2 \to \overline{\tilde q^2}$ 
in $L^p\bigl(L^p_w\bigr)$ a.s.; thus, passing to a subsequence, 
$\tilde q_\eps^2$ converges weakly in $x$ and pointwise 
in $(\tilde \omega,t)$. We refer to 
Sections \ref{sec:apriori} and \ref{sec:Jak-Skor} 
for the details.

In Section \ref{sec:equality-of-laws}, we 
prove several results that transfer 
the available a priori estimates and 
the SPDE \eqref{eq:u_ch_ep} to the new probability 
space (for the new variables $\tilde u_\eps$, 
$\tilde q_\eps$, $\tilde W_\eps$ and their limits). 
Equipped with \eqref{eq:intro-ueps-conv} 
and \eqref{eq:intro-weak-conv2}, we  
send $\eps\to 0$ in the SPDE 
\eqref{eq:u_ch_ep} (on the new probability space) to 
produce a solution $\tilde u$ of an SPDE that looks like 
the stochastic CH equation \eqref{eq:u_ch} 
but with the nonlinearity $\overline{\tilde q^2}$ instead 
of the required one $\tilde q^2=
\abs{\pd_x \tilde u}^2 $, see Section \ref{sec:existence}. 

The final Section \ref{sec:wk_limits_section} 
is devoted to the proof that $\overline{\tilde q^2}=\tilde{q}^2$ 
a.e.~in $(\omega,t,x)$, and thereby the validity 
of Theorem \ref{thm:main}. The proof amounts 
to upgrading the $(t,x)$ weak convergence \eqref{eq:intro-weak-conv2} 
to strong convergence via a study of the \textit{defect measure}
\begin{equation}\label{eq:intro-defect}
	\bD=\bD(\omega,t,x)=
	\frac12\left(\overline{\tilde q^2}-\tilde q^2\right)\geq 0.
\end{equation}
The idea is to derive a transport-type SPDE 
(up to an inequality) for the evolution of $\bD$, so that 
if $\bD$ is time-continuous at $t=0$ with 
$\bD(0)=0$ (assuming strong compactness at $t=0$), 
then $\bD(t)$ is zero at all later times $t>0$.  
Roughly speaking, an SPDE (up to an inequality) for 
$\frac12\bigl(\tilde u^2+\overline{\tilde q^2}\bigr)$ 
is obtained using \eqref{eq:intro-ueps-conv}, 
\eqref{eq:intro-weak-conv2} to pass to 
the limit in the total energy balance 
\eqref{eq:intro-visc-energy} (again 
written on the new probability space). On the other hand, 
by formally repeating the derivation of 
the energy balance \eqref{eq:intro-visc-energy} 
for the limits $\tilde u$, $\tilde q$, relying 
on the SPDEs obtained by sending $\eps\downarrow 0$ 
in \eqref{eq:u_ch_ep}, \eqref{eq:intro-qest-tmp}, we 
arrive at an SPDE for $\frac12 \bigl( \tilde u^2
+\tilde q^2\bigr)$, and therefore an inequality for 
the defect measure $\bD$, which takes the form
\begin{equation}\label{eq:intro-defect-SPDE}
	\begin{split}
		& \pd_t  \bD 
		+\pd_x \left(\tilde u\, \bD 
		-\frac14 \pd_x\sigma^2\,\bD\right)
		-\frac12 \pd_{xx}^2\bigl(\sigma^2\,\bD\bigr)
		+\bigl[\pd_x\left(\sigma\, \bD\right)
		+\pd_x \sigma \, \bD \bigr] \, \dot{\tilde W} 
		\\ & \qquad\quad 
		\leq \left(\abs{\pd_x\sigma}^2
		-\frac14 \pd_{xx}^2\sigma^2
		-\pd_x \tilde u\right)\, \bD
		\quad \text{in $\cD^\prime_{t,x}$, 
		almost surely},
	\end{split}	
\end{equation}
where $\cD^\prime_{t,x}=\cD^\prime([0,T)\times \T)$ 
denotes the space of distributions on $[0,T)\times \T$.

Unfortunately, the arguments leading 
up to \eqref{eq:intro-defect-SPDE} are only formal. 
Recalling \eqref{eq:intro-est-tmp}, we do not 
have enough integrability on $\tilde q$, $\overline{\tilde q^2}$ 
to give sense to the terms $\tilde q^3$ and 
$\tilde q \, \overline{\tilde q^2}$ arising during 
the derivation of \eqref{eq:intro-defect-SPDE}.
The way to overcome this difficulty is 
to work with renormalised formulations of 
the SPDEs for $\tilde q_\eps, \tilde q$ 
based on linearly growing approximations 
$S_\ell(v)$ of $v^2$ and eventually send $\ell\to \infty$. 
More precisely, we split $v$ into 
its positive $v_+$ and negative parts $v_-$ 
(so that $v^2=v_+^2+v_-^2)$ and then work with 
the SPDEs satisfied by the nonlinear compositions 
$S_\ell\bigl((\tilde q_\eps)_\pm\bigr)$, 
$S_\ell\bigl(\tilde q_\pm\bigr)$. 
In passing, let us mention that this forces 
us to accomodate a countable product of quasi-Polish 
spaces, as we need to apply the Skorokhod--Jakubowski 
procedure to all members of the 
sequence $\left\{S_\ell\bigl((q_\eps)_{\pm}
\bigr)\right\}_{\ell\in \N}$ simultaneously. 
Countable products of quasi-Polish spaces are discussed
in Appendix \ref{sec:qpolish}. 

Again drawing an analogy to the deterministic theory 
\cite{Coclite:2005tq,Coclite:2015aa,Xin:2000qf}, here 
we run into another difficulty linked to the stochastic 
part of the problem. Namely, the temporal irregularity 
of the noise induces structural changes in 
the equation that make it impossible to work 
with the familiar $W^{2,\infty}_{\operatorname{loc}}(\R)$ 
approximations $S_\ell(v)=v^2 \mathds{1}_{\{\abs{v}\leq \ell\}}
+\ell\bk{2\abs{v}-\ell}\mathds{1}_{\{\abs{v}>\ell\}}$ of 
$v\mapsto v^2$. This adds further complications to the 
analysis. See Section \ref{sec:Jak-Skor} for further details.

A further intricacy arising during the 
derivation of \eqref{eq:intro-defect-SPDE} is 
the passage to the limit in stochastic integrals of the form 
$\int_0^t \int_{\T} S'(\tilde q_\ep)\tilde q_\ep \, \d x 
\, \d \tilde W_\ep$, for some class of nonlinear 
functions $S(\cdot)$. Here, $\tilde W_\ep$ is a sequence
of Wiener processes converging uniformly to a limit 
process $\tilde W$, a.s., while $\int_{\T} S'(\tilde q_\ep)
\tilde q_\ep \, \d x$ converges just weakly in $L^p_t$, a.s., 
towards $\int_{\T} \overline{S'(\tilde q)\tilde q} \, \d x$. 
The absence of strong temporal compactness hinders the 
application of Lemma 2.1 of \cite{Debussche:2011aa}, 
which is regularly used to certify convergence 
of stochastic integrals. We manage this issue by 
once more making vital use of the quasi-Polish space 
$L^p\bigl(L^p_w\bigr)$ and the tightness criterion 
provided by the conditions (i), (ii), and (iii). 
The details are worked out in 
Section \ref{sec:wk_limits_section}. 

The renormalised SPDEs are derived by regularising 
non-smooth processes via convolution 
against a spatial mollifier $J_\delta(x)$. 
Sending $\delta \to 0$, we handle most of the error 
terms using standard DiPerna--Lions 
estimates \cite{DiPerna:1989aa}, except for some 
unique terms coming from the interconnection between the 
martingale part of the equations and second-order 
Stratonovich--It\^{o} correction terms. 
The corresponding commutator estimates---collected 
in Appendix \ref{sec:commutator-est}---are proved 
in the paper \cite{HKP-viscous} by the last three authors. 
Similar estimates have been used recently in 
\cite{Punshon-Smith:2018aa} and \cite{Holden:2020aa}.

It remains to send $\ell\to \infty$ to recover 
a useful version of the SPDE inequality 
\eqref{eq:intro-defect-SPDE} for the defect measure $\bD$. 
The renormalisations $S_\ell(v_\pm)$ 
give rise to a number of intricate 
error terms involving the approximation 
parameter $\ell$, several of them linked 
to the stochastic nature of the problem. 
For the deterministic CH equation \cite{Xin:2000qf}, this 
part of the analysis relies crucially on knowing that 
the viscous solutions $q_\eps$ obey a one-sided 
gradient bound of the Oleinik-type: 
$q_\eps(t)=\pd_x u_\eps(t) \lesssim 1+\tfrac{1}{t}$ 
(that is independent of $\eps$), a further 
reflection of the dissipative nature of the solutions. 
No such bound is currently known for 
the stochastic CH equation.  However, let us 
mention that recently \cite{Holden:2021vw} 
it was discovered that dissipative solutions of the related 
stochastic Hunter--Saxton equation \cite{Holden:2020aa} 
satisfy a one-sided gradient bound of the form  
$\pd_x u(\omega,t,x) \le K(\omega, t)$, 
where the process $K(\omega,t )>0$ 
exhibits an exponential moment bound in the 
sense that $\Ex \exp\bigl(p\int_t^T K(s)\,\d s\bigr) 
\lesssim {t^{-2p}}$ for small times $t$, for some $p\ge 1$. 
We have not been able to establish a similar bound for 
the stochastic CH equation. Here we will instead  
rely on an observation due to the third author 
and Coclite \cite{Coclite:2015aa} for the deterministic 
CH equation, which makes it possible 
to rigorously derive an SPDE inequality 
for the ``positive part" of the defect measure,
$\bD_+=\frac12\bigl(\overline{\tilde q_+^2}
-\tilde q_+^2\bigr)$, without using an Oleinik bound. 
The detailed analysis of the defect measure is found in 
Section \ref{sec:wk_limits_section}.

The remaining part of the paper is divided 
into six sections and three appendices, which 
together establishes Theorem \ref{thm:main}. 

\section{Preliminaries and solution concepts}
\label{sec:sol_def}

We refer to \cite[Chapter 1]{Chow:2015aa} for 
notation and background material on 
stochastic analysis and SPDEs, including 
stochastic integrals, It\^o's chain rule, and 
martingale inequalities like the one of 
Burkholder--Davis--Gundy (BDG). For a more general context 
of cylindrical Wiener processes, see \cite{DaPrato:2014aa}. 
For some key concepts linked to probability measures (on 
topological spaces), weak compactness and tightness, see 
the book \cite{Bogachev-measures:2018}.
For basic properties of Bochner spaces like $L^p(\Omega;X)
= L^p\bigl(\Omega,\mathcal{F},\mathbb{P};X\bigr)$, 
where $X$ is a Banach space, we refer 
to \cite[Chapters 1 \& 2]{BanachI:2016}. 
On several occasions we will use \cite[Lemma 2.1]{Debussche:2011aa} to 
lay the foundations for the convergence 
of stochastic integrals. The reader can find 
a primer on quasi-Polish spaces and the 
Skorokhod--Jakubowski theorem \cite{Jakubowski:1997aa} 
in Appendix \ref{sec:qpolish}. Quick background reading can be found 
in, e.g., \cite{Brzezniak:2013ab,Brzezniak:2016wz,Ondrejat:2010aa}, 
some results from which are quoted 
in the aforementioned appendix. The definition and properties of the 
space $C\bigl([0,T];H^1(\T)-w\bigr)$, which 
is quasi-Polish and used herein, can be found in 
\cite{Brzezniak:2013ab,Brzezniak:2016wz,Ondrejat:2010aa}.

\medskip

This section presents the solution concept used 
in Theorem \ref{thm:main} and the one used 
for the viscous SPDE \eqref{eq:u_ch_ep}, starting 
with the notion of a $H^m$-regular martingale solution 
of the viscous equation. Here $\eps>0$ is fixed, and therefore 
we write $u$ instead of $u_\eps$ for the solution of 
\eqref{eq:u_ch_ep}.  

\begin{defin}[$H^m$ martingale solution of 
viscous SPDE]\label{def:wk_sol_visc} 
Fix any integer $m\ge 1$ and some $p_0>4$. 
Let $\Lambda$ be a probability 
measure on $H^{m}(\T)$, satisfying
\begin{equation}\label{eq:u0-cond-Hmp0}
	\int_{H^{m}(\T)} 
	\norm{v}_{H^{m}(\T)}^{p_0} 
	 \Lambda(\d v)<\infty.
\end{equation}
The triple $\bigl(\mathcal{S},u,W \bigr)$ 
is a $H^m$ martingale solution of \eqref{eq:u_ch_ep} with 
initial law $\Lambda$ if the following conditions hold:
\begin{itemize}
	\item[(a)] $\mathcal{S}=\bigl(\Omega, 
	\mathcal{F},\{\mathcal{F}_t\}_{t \ge 0},
	\mathbb{P}\bigr)$ is a stochastic basis;
	
	\item[(b)] $W$ is a standard Wiener 
	process on $\mathcal{S}$;

	\item[(c)] $u:\Omega\times[0,T]\to H^1(\T)$ is 
	adapted, with $u\in L^{p_0}\left(\Omega;
	C([0,T];H^1(\T))\right)$. 
	Moreover,  $u\in L^2([0,T];H^{m+1}(\T))
	\cap L^\infty([0,T];H^m(\T))$ a.s.~and
	$$
	u\in L^2\left(\Omega;L^2([0,T];H^2(\T))\right).
	$$
	
	\item[(d)] initial data --- the law of $u_0:=u(0)$ 
	on $H^{m}(\T)$ is $\Lambda$, 
	i.e., $\bk{u(0)}_*\mathbb{P}=\Lambda $;
	
	\item[(e)] for all $t \in [0,T]$ and all 
	$\varphi \in C^1(\T)$, the following equation 
	holds $\mathbb{P}$-almost surely (in the sense of It\^o):
	\begin{equation}\label{eq:u_ep_weakeq}
		\begin{aligned}
  			&\int_\T u(t) \varphi \,\d x
  			-\int_\T u_0 \varphi\,\d x\\
  			& \qquad 
  			= \int_0^t\int_\T - u\,\pd_x u \,\varphi 
  			+  \bigl[P -\ep \pd_x u \bigr]
  			\pd_x \varphi \,\d x\, \d s
  			\\ & \qquad\qquad  
  			-\frac12\int_0^t\int_\T 
  			\sigma \pd_x u
			\pd_x\bk{\sigma \varphi} 
  			 \,\d x \,\d s
  			-\int_0^t\int_\T 
  			\sigma \pd_x u
			\varphi\,\d x \, \d W(s),
  			\\  
  			& P = P[u]:=K*\left(u^2+\frac{1}{2}
  			\abs{\pd_x u}^2\right).
		\end{aligned}
	\end{equation}
\end{itemize}
\end{defin}

If $\bigl(\mathcal{S},W\bigr)$ is not a part 
of the unknown solution but fixed in advance, 
we speak of a probabilistic strong or pathwise solution. 
According to the famous Yamada--Watanabe principle, 
a martingale solution of an SPDE is probabilistic strong if 
the SPDE exhibits a pathwise uniqueness result.

\begin{defin}[strong $H^m$ solution of viscous equation]
\label{def:st_sol}
Let $u_0 \in L^{p_0}(\Omega;H^{m}(\T))$ for 
some $p_0>4$, and consider a fixed stochastic basis  
$\mathcal{S} = \bigl(\Omega, \mathcal{F},
\{\mathcal{F}_t\}_{t \in [0,T]}, \mathbb{P}\bigr)$. 
We say that $u$, defined on $\mathcal{S}$, is 
a strong $H^m$ solution to \eqref{eq:u_ch_ep} with initial 
data $u(0)=u_0$ if, for a given Wiener 
process $W$ defined on $\cS$, the triple 
$\bigl(\mathcal{S},u,W \bigr)$ constitutes 
a $H^m$ martingale solution to \eqref{eq:u_ch_ep} 
with initial distribution $\Lambda
=\bk{u_0}_*\mathbb{P}=P\circ u_0^{-1}$.
\end{defin}

The viscous equation \eqref{eq:u_ch_ep} is 
strongly well-posed \cite{HKP-viscous}. 
The following theorem gathers the main results 
from \cite{HKP-viscous}.

\begin{thm}[strong well-posedness of viscous SPDE]
\label{thm:bounds1}
Fix $\eps>0$. Suppose $u_0 \in L^{p_0}(\Omega;H^m(\T))$ 
for some $p_0>4$. There exists a unique strong $H^m$ solution 
to \eqref{eq:u_ch_ep} with initial condition $u_0$. 
Denoting this solution by $u_\eps$, the 
following properties and $\eps$-uniform bounds hold:
\begin{itemize}
	\item[(i)] Total energy balance --- 
	for any $0\le s\leq t\leq T$, 
	\begin{equation}\label{eq:energybalance}
		\begin{aligned}
			&\int_\T u_\ep^2 + 
			\abs{\pd_x u_\ep}^2\,\d x \bigg|_s^t 
 			+2\ep \int_s^t \int_\T \abs{\pd_x u_\ep}^2
 			+\abs{\pd_{xx}^2 u_\ep}^2\,\d x\,\d t'
 			\\ &\qquad = \int_s^t \int_\T 
 			\frac14 \pd_{xx}^2 \sigma_\ep^2  u_\ep^2
 			+\bk{\abs{\pd_x \sigma_\ep}^2
 			-\frac14 \pd_{xx}^2 \sigma_\ep^2}
 			\,\abs{\pd_x u_\ep}^2 \,\d x \,\d t'
 			\\ &\qquad \qquad 
			+\int_s^t \int_\T \pd_x \sigma_\ep\bk{ u^2 
 			-\abs{\pd_x u_\ep}^2}\,\d x \, \d W, 
 			\quad \text{$\tilde{\mathbb{P}}$--almost surely}.
 		\end{aligned}
	\end{equation}
	Furthermore, there exists an 
	$\ep$-independent positive constant
	$$
	C=C\bk{p_0,T,\norm{\sigma}_{W^{2,\infty}(\T)},
	\norm{u_0}_{L^{p_0}(\Omega;H^1(\T))}}
	$$ 
	such that
	\begin{equation}\label{eq:energybound}
		\begin{split}
			& \Ex \norm{u_{\ep}}_{L^\infty([0,T];H^1(\T))}^{p_0}
			\leq C \quad \text{and} \quad
			\\ & \quad
			\Ex \abs{2\eps \int_0^T\int_{\T}
			\abs{\pd_x u_{\ep}}^2
			+\abs{\pd_{xx}^2 u_{\ep}}^2
			\,\d x\, \d t}^{\frac{p_0}{2}}\le C.
		\end{split}
	\end{equation}

	\item[(ii)] For any $\theta \in 
	\bigl[0,\frac{p-2}{4p}\bigr)$, $p\in [2,p_0]$, 
	there exists an $\ep$-independent constant 
	$C=C\bk{\theta,T,\norm{\sigma}_{W^{2,\infty}(\T)},
	\norm{u_0}_{L^2(\Omega;H^1(\T))}}>0$ 
	such that
	\begin{equation}\label{eq:Holder-viscous}
		\Ex \norm{u_{\ep}}_{C^\theta([0,T];
		L^2(\T))}^{2/(1-4 \theta)} \le C.
	\end{equation}

	\item[(iii)] The laws of $\left\{u_\eps\right\}_{\eps>0}$ 
	form a (uniformly in $\eps$) tight sequence of 
	probability measures on the space 
	$C\bigl([0,T];H^1(\T)-w\bigr)$.
\end{itemize}
\end{thm}

Finally, we define the solution concept 
used in Theorem \ref{thm:main} for 
the stochastic CH equation \eqref{eq:u_ch}. 

\begin{defin}[dissipative weak martingale 
solution]\label{def:dissp_sol}
Let $\Lambda$ be a probability measure 
on $H^{1}(\T)$ with finite $p_0$th 
moment for some $p_0>4$, i.e., 
$$
\int_{H^{1}(\T)} \norm{v}_{H^{1}(\T)}^{p_0}
\Lambda(\d v) < \infty.
$$
The triple $\bigl(\mathcal{S},u,W \bigr)$ 
is a dissipative weak martingale 
solution to the stochastic CH equation 
\eqref{eq:u_ch} with initial 
distribution $\Lambda$ if:

\begin{itemize}
	\item[(a)] $\mathcal{S}=\bigl(\Omega, 
	\mathcal{F},\{\mathcal{F}_t\}_{t \ge 0},
	\mathbb{P}\bigr)$ is a stochastic basis;

	\item[(b)] $W$ is a standard Wiener process on $\cS$;

	\item[(c)] $u:\Omega\times [0,T]\to L^2(\T)$ is 
	a progressively measurable stochastic process 
	with paths $u(\omega)\in C([0,T];L^2(\T)) 
	\cap C([0,T];H^1(\T)-w)$, 
	for $\mathbb{P}$-a.e.~$\omega\in \Omega$. 
	Moreover, $u$ belongs to the space 
	$L^2\bigl(\Omega;L^\infty([0,T];H^1(\T))\bigr)$;

	\item[(d)] initial data --- $\bk{u(0)}_*\mathbb{P}=\Lambda $;

	\item[(e)] the following equation holds in the sense 
	of It\^o, $\mathbb{P}$-almost surely, for 
	all $t \in [0,T]$ and for all $\varphi \in C^2(\T)$,
	\begin{equation}\label{eq:u_weakeq}
		\begin{aligned}
			& \d \int_\T u\varphi \,\d x 
			= \int_\T \left[\frac12 u^2 + P\right] \pd_x \varphi
			\,\d x\, \d t \\
			& \qquad +\frac12  \int_\T  
			u \pd_x\bk{\pd_x \bk{\sigma \varphi}\sigma}\,\d x \,\d t 
			+ \int_\T u \pd_x\bk{\sigma \varphi}\,\d x \, \d W ,
			\\ & 
			P = P[u] := K*\left(u^2
			+\frac{1}{2} \abs{\pd_x u}^2\right);
		\end{aligned}
	\end{equation}

	\item[(f)] temporal right-continuity 
	in $H^1(\T)$ --- for a.e.~$(\omega,t_0) 
	\in\Omega \times [0,T]$,
	$$
	\lim_{t \downarrow t_0} 
	\norm{u(t) - u(t_0)}_{H^1(\T)} = 0.
	$$
\end{itemize}
\end{defin}

At a time $t=t_0$ of wave breaking, a dissipative 
solution $u$ is not going to be time-continuous 
in $H^1$, but merely right-continuous.  
The right-continuity condition (f) in 
Definition \ref{def:dissp_sol} manifests 
the energy inequality \eqref{eq:energybalance-sch} 
and the dissipative nature of the considered solution class. 

Currently, no pathwise uniqueness result is known 
for the stochastic CH equation \eqref{eq:u_ch_ep}. 
As a result, we cannot rely on the 
Yamada--Watanabe principle to upgrade 
martingale solutions to strong solutions. 

\section{Some a-priori estimates}\label{sec:apriori}

Recall that $u_\eps$ denotes the $H^m$ regular solution 
of the viscous SPDE \eqref{eq:u_ch_ep} with 
initial data $u(0)=u_0$, whose existence, uniqueness 
and basic properties are given by Theorem \ref{thm:bounds1}, 
under the assumptions that $\sigma \in W^{2,\infty}$ 
and $u_0 \in L^{p_0}_{\omega}H^m_x$ for some $p_0>4$.  
This section collects some straightforward 
consequences of the $\eps$-uniform bounds 
listed in Theorem \ref{thm:bounds1}, 
which will be used in Section \ref{sec:wk_limits_section}. 

\begin{lem}[basic estimates]\label{thm:P-u2_bound}
Let $u_\eps$ be the $H^m$ regular solution 
of the viscous SPDE \eqref{eq:u_ch_ep} with 
initial data $u(0)=u_0$ satisfying 
\eqref{eq:u0-cond-Hmp0} (for an arbitrary $m>1$). 
There exists a constant 
$$
C=C\bk{T,\norm{\sigma}_{W^{2,\infty}(\T)},
\norm{u_0}_{L^{p_0}(\Omega;H^1(\T))}},
$$ 
independent of $\ep>0$, such that
\begin{align*}
	\Ex \norm{u_\ep}_{L^\infty([0,T]\times \T)}^{p_0}
	\leq C, \quad
	\Ex \norm{P_\eps}_{L^\infty([0,T]\times \T)}^p\leq C,
\end{align*}
for any $p\in \bigl[1,p_0/2\bigr]$, where $P_\ep=P[u_\ep]$ is 
defined in \eqref{eq:u_ch_ep}.
\end{lem}

\begin{proof}
The first part is a direct consequence of 
the $L^{p_0}_\omega L^\infty_tH^1_x$ 
bound \eqref{eq:energybound}
and the one-dimensional embedding 
$H^1(\T) \hookrightarrow L^\infty (\T)$. 
By the definition of $P_\eps$ and because 
$\abs{K(x)}\lesssim 1$ for all $x\in \T$,
\begin{align*}
	\abs{P_\eps(\omega,t,x)}^p
	& \lesssim \norm{u_\eps (\omega,t,\cdot)}_{L^2(\T)}^{2p}
	+\norm{\pd_x u_\eps (\omega,t,\cdot)}_{L^2(\T)}^{2p}
	\\ & \lesssim \norm{u_\eps(\omega,\cdot,
	\cdot)}_{L^\infty([0,T];H^1(\T))}^{2p}.
\end{align*}
The second estimate now follows by 
taking the expectation and again 
using \eqref{eq:energybound}, recalling 
the assumption $p\leq p_0/2$.
\end{proof}

Consider any function 
$S\in W^{2,\infty}_{\operatorname{loc}}(\R)$ that satisfies
\begin{equation}\label{eq:entropies}
	\begin{aligned}
		& \abs{S(v)}\lesssim \abs{v}^2, 
		\quad 
		\abs{S'(v)}\lesssim \abs{v}, 
		\quad
		\abs{S''}\lesssim 1, 
		\\ & 
		\quad \text{and} \quad 
		\abs{S(v)v-\frac12 S'(v)v^2}\lesssim \abs{v}^2, 
		\quad \forall v\in \R.
	\end{aligned}
\end{equation}
The goal is to compute the differential $\d S(q_\ep)$,
recalling that $q_\eps=\pd_x u_\eps$ 
is the spatial gradient of $u_\eps$.
This requires us to apply the It\^o formula to
\eqref{eq:intro-qest-tmp}. However, $q_\ep$
is known to have continuous paths only in the
infinite dimensional space $L^2(\T)$, and the
equation involves terms such as $q_\ep^2$ which
brings it outside the scope where standard
Hilbert space-valued It\^o formulas apply.
Instead, we convolve \eqref{eq:u_ch_ep} by taking
$\varphi$ in \eqref{eq:u_ep_weakeq} to be a spatial Friedrichs
mollifier $J_\delta$. This gives us the equation for
$u_{\ep,\delta} = u_\ep *J_\delta$, with $u_{\ep,\delta}$ continuous 
in $t$ for each fixed $x$.  Taking a classical
spatial derivative gives us an equation for $q_{\ep, \delta} = q_\ep *J_\delta$,
which is \eqref{eq:intro-qest-tmp} mollified against $J_\delta$. 
These equations can be interpreted pointwise in $x$, and 
the real-valued It\^o formula can be applied for each fixed $x$.

The It\^o formula is classically stated for $C^2$ nonlinearities,
but it can be extended by approximation to functions in $W^{2,\infty}$
\cite[Theorem 71]{Protter-book:2005} (and, in fact, to even
rougher functions in some cases, like in the Tanaka formula).
Throughout the paper, we will be applying the It\^{o} formula
to nonlinearities $S$ from the $W^{2,\infty}$ class 
satisfying \eqref{eq:entropies}. 

The entire argument, including taking the mollification 
limit ($\delta \downarrow 0$), is executed for a similar equation 
across \eqref{eq:mollifying_example-q} -- \eqref{eq:S_equation_limit} 
in the forthcoming Lemma \ref{thm:limiteq_2}.  That argument 
can be applied here to any nonlinear function $S$ satisfying 
\eqref{eq:entropies}, recalling that $q_\ep$ is 
more regular (integrable) than the solution in 
Lemma \ref{thm:limiteq_2}. For a fixed $\ep>0$, we have 
$q_\ep \in L^2([0,T];H^{m -1}(\T))$, a.s., for any finite $m$. 
The only terms here not present in Lemma \ref{thm:limiteq_2} 
are $\ep S'(q_{\ep,\delta}) \pd_{xx}^2 q_{\ep,\delta}$
and $S'(q_{\ep,\delta}) P_\ep*J_\delta$. In view of 
the regularity of $q_\ep$, we have $S'(q_{\ep,\delta}) 
\to S'(q_{\ep})$, $\pd_{xx}^2 q_{\ep, \delta}
\to \pd_{xx}^2 q_{\ep}$ a.s.~in $L^2([0,T]\times \T)$.
A similar reasoning applies to the other term: since $P_\ep$ 
belongs a.s.~to $L^\infty([0,T]\times \T)$, 
it follows that the convolution $P_\ep*J_\ep$ 
converges to $P_\ep$ a.s.~in $L^2([0,T] \times \T)$.
Apart from these terms, the steps are the same,
and we will not repeat them here.  Similar arguments are 
also carried out for $\d \bk{u_\ep^2 + q_\ep^2}$ 
and for the squared-difference of two 
solutions in \cite[Theorem 7.6, Lemma 7.7]{Holden:2021vw} 
(see also Lemma \ref{thm:cauchy_CtH1_2}).

This argument, which combines mollification 
with the real-valued It\^o formula, leads us to the SPDE
\begin{equation}\label{eq:SPDE-S(qe)-tmp1}
	\begin{split}
		0 & = \d S(q_{\ep})
		+\left(S'(q_{\ep})\pd_x \bk{u_{\ep} \, q_{\ep}}
		-\frac12 S'(q_{\ep})q_{\ep}^2 \right)\, \d t
		+ S'(q_{\ep}) \bk{P_\eps-u_{\ep}^2}\, \d t  
		\\ & \qquad
		-\ep S'(q_{\ep})\pd_{xx}^2 q_{\ep}\,\d t
		-\frac12 S'(q_{\ep})\pd_x \bk{\sigma_\ep\,
		\pd_x\bk{\sigma_\ep\,q _{\ep}}} \,\d t
		\\ & \qquad 
		+S'(q_{\ep})\pd_x\bk{ \sigma_\ep \, q_{\ep}}\,\d W
		-\frac12 S''(q_{\ep})\,
		\bigl|\pd_x\bk{\sigma_\ep \, q_{\ep}}\bigr|^2\,\d t.
	\end{split}	
\end{equation}
In this paper, we make  
repeated use of the following identities:
\begin{equation}\label{eq:basic-identities}
	\begin{split}
		S'(q_{\ep})\pd_{xx}^2 q_{\ep} & =\pd_{xx}^2 S(q_\ep)
		-S''(q_{\ep}) \abs{\pd_x q_\ep}^2,
		\\ 
		S'(q_{\ep})\pd_x \bigl(u_\ep q_\ep\bigr)
		& =\pd_x \bigl(u_\ep S(q_\ep)\bigr)
		-\bigl(S(q_\ep)-S'(q_\ep)q_\ep\bigr)\pd_x u_\ep,
		\\ 
		S'(q_{\ep})\pd_x \bigl(\sigma_\ep q_\ep\bigr)
		& =\pd_x \bigl(\sigma_\ep S(q_\ep)\bigr)
		-\bigl(S(q_\ep)-S'(q_\ep)q_\ep\bigr)\pd_x \sigma_\ep,
		\\ 
		S'(q_{\ep})\pd_x\bigl(\sigma_\ep 
		\pd_x(\sigma_\ep q_{\ep})\bigr)
		& =\pd_{xx}^2\bigl(\sigma_\ep^2 S(q_{\ep}) \bigr)
		-\pd_x\left(\frac12\pd_x \sigma_\ep^2 
		\bigl(3S(q_{\ep})-2S'(q_\ep)q_\ep\bigr)\right)
		\\ & \qquad 
		+ \frac12\pd_{xx}^2\sigma_\ep^2
		\bigl(S(q_{\ep})-S'(q_{\ep}) q_{\ep}\bigr) 
		\\ & \qquad 
		-S''(q_\ep)\left(\abs{\pd_x \bigl(\sigma_\ep\, q_\ep\bigr)}^2
		-\abs{\pd_x \sigma_\ep\, q_\ep}^2\right).
	\end{split}
\end{equation}
Inserting \eqref{eq:basic-identities} into 
\eqref{eq:SPDE-S(qe)-tmp1}, we obtain
\begin{equation}\label{eq:SPDE-S(qe)}
	\begin{split}
		0 & = \d S(q_{\ep})
		+\pd_x \left[u_\ep\, S(q_\ep)
		+\frac14 \pd_x \sigma_\ep^2 
		\bigl(3S(q_{\ep})-2S'(q_\ep)q_\ep\bigr) \right]\, \d t
		\\ & \qquad 
		-\pd_{xx}^2\left[ \left (\frac12\sigma_\ep^2
		+\eps\right)S(q_\ep)\right]\, \d t
		+S''(q_{\ep})\, \ep\abs{\pd_x q_\ep}^2 \, \d t
		\\ & \qquad +\Biggl[ S'(q_{\ep}) \bk{P_\eps-u_{\ep}^2}
		-\left(S(q_{\ep})q_{\ep}-\frac12 S'(q_{\ep})q_{\ep}^2
		\right)
		\\ & \qquad \qquad 
		-\frac14 \pd_{xx}^2\sigma_\ep^2 
		\bigl(S(q_{\ep})-S'(q_{\ep}) q_{\ep}\bigr)
		-\frac12 \abs{\pd_x \sigma_\ep}^2
		\, S''(q_\ep) \,q_\ep^2\Biggr] \, \d t
		\\ & \qquad
		+\Bigl[\pd_x \bigl(\sigma_\ep \, S(q_\ep)\bigr)
		-\pd_x\sigma_\ep \, \bigl(S(q_\ep)-S'(q_\ep)q_\ep\bigr)
		\Bigr]\, \d W.
	\end{split}	
\end{equation}
noticing the cancellation of the two terms 
involving $\bigl|\pd_x\bk{\sigma_\ep \, q_{\ep}}\bigr|^2$. 

For use in upcoming sections, let us also state 
the SPDE satisfied by $S(u_e)$:
\begin{equation}\label{eq:SPDE-S(ue)}
	\begin{split}
		0 & = \d S(u_\ep)+\pd_x\left[u_\ep \, S(u_\ep)
		+S'(u_\ep)P_\eps-\int^{u_\ep}S(\xi)\, \d\xi
		+ \frac14 \pd_x\sigma_\ep^2S(u_{\ep})\right] \, \d t
		\\ & \qquad 
		-\pd_{xx}^2\left[ \left (\frac12\sigma_\ep^2
		+\eps\right)S(u_\ep)\right]\, \d t
		+S''(u_\ep)\, \ep \abs{\pd_x u_\ep}^2\, \d t
		\\ & \qquad
		+\left[\frac14\pd_{xx}^2 \sigma_\ep^2 \bigl(S(u_{\ep})
		-S'(u_{\ep}) u_{\ep}\bigr)
		- S''(u_\ep)\, q_\ep P_\ep\right]\, \d t
		\\ & \qquad
		+\Bigl[\pd_x \bigl(\sigma_\ep S(u_\ep)\bigr)
		- \pd_x \sigma_\ep \, S(u_\ep) \Bigr] \, \d W. 
	\end{split}
\end{equation}
This equation can be derived as before, using 
\eqref{eq:basic-identities} with $q_\ep$ replaced by 
$u_\ep$, rewriting the last identity \eqref{eq:basic-identities} as
\begin{equation*}
	\begin{split}
		S'(u_{\ep})\pd_x\bigl(\sigma_\ep \, 
		\pd_x(\sigma_\ep\, u_{\ep})\bigr)
		& = \pd_{xx}^2\bigl(\sigma_\ep^2 S(u_{\ep}) \bigr)
		-\pd_x\left(\frac12 \pd_x \sigma_\ep^2S(q_{\ep})\right)
		\\ & \quad 
		-\frac12 \pd_{xx}^2 \sigma_\ep^2
		\bigl(S(u_{\ep})-S'(u_{\ep}) u_{\ep}\bigr)
		-S''(u_\ep)\abs{\sigma_\ep \pd_x u_\ep}^2.
	\end{split}	
\end{equation*}
The SPDE \eqref{eq:intro-visc-energy} for the 
total energy balance follows from \eqref{eq:SPDE-S(qe)} 
and \eqref{eq:SPDE-S(ue)}.

\medskip

We are now in a position to derive a higher integrability 
property of $q_\eps=\pd_x u_\ep$. This property will ensure that the 
weak limit $\overline{q^2}$ in \eqref{eq:intro-weak-conv} 
does not concentrate into a measure but remains 
(at least) in $L^1_{\omega,t,x}$.

\begin{prop}[higher integrability]
\label{thm:2plusalpha_apriori}
Let $u_\eps$ be the $H^m$ regular solution 
of the viscous SPDE \eqref{eq:u_ch_ep} with 
initial data $u(0)=u_0$ satisfying 
\eqref{eq:u0-cond-Hmp0} (for an arbitrary $m>1$), 
and denote by $q_\ep=\pd_x u_\eps$ the spatial 
gradient of $u_\ep$. For fixed $\alpha\in (0,1)$, there 
exists a constant
$$
C=C\bk{\alpha,T,\norm{\sigma}_{W^{2,\infty}(\T)},
\norm{u_0}_{L^{p_0}(\Omega;H^1(\T))}},
$$ 
independent of $\ep>0$, such that
\begin{equation}\label{eq:higher-int}
	\Ex \norm{q_\ep}_{L^{2+\alpha}
	([0,T]\times \T)}^{2 + \alpha}\leq C.
\end{equation}
\end{prop}

\begin{proof}
Consider the function 
$S(v):= v\bk{\abs{v}+1}^{\alpha}$, which satisfies
\begin{align*}
	S'(v) &= \bk{\abs{v}+1}^{\alpha}
	+\alpha \abs{v} \bk{\abs{v}+1}^{\alpha - 1},
	\intertext{and} 
	S''(v) &= \alpha  \,\sgn(v) 
	\bk{\abs{v}+1}^{\alpha-2}
	\bk{2+\bk{ \alpha + 1} \abs{v}}.
\end{align*}
Clearly, $\abs{S''(v)} \le C$ for all $v\in \R$.

Integrating the SPDE {\eqref{eq:SPDE-S(qe)}} 
for $S(q_\ep)$ over $x\in \R$ gives
\begin{equation}\label{eq:2alpha_1}
	L\, \d t  = \d \int_\T S(q_{\ep}) \,\d x
	+\bk{I_1 + I_2} \,\d t+I_3\,\d W,
\end{equation}
where
\begin{align*}
	L & =\int_\T \left(S(q_{\ep})q_{\ep}
	-\frac12 S'(q_{\ep})q_{\ep}^2\right)\,\d x,
	\\ 
	I_1 & = \int_{\T} S'(q_{\ep}) \bk{P_\eps-u_{\ep}^2}
	+S''(q_{\ep})\, \ep\abs{\pd_x q_\ep}^2 \, \d x,
	\\ 
	I_2 & = -\int_\T \bigl(S(q_{\ep})-S'(q_{\ep}) q_{\ep}\bigr)
	\pd_x \bigl(\sigma_\ep \, \pd_x \sigma_\ep\bigr)
	+\frac12 S''(q_\ep)\abs{\pd_x \sigma_\ep\, q_\ep}^2\, \d x,
	\\ 
	I_3 & = -\int_\T \bigl(S(q_\ep)-S'(q_\ep)q_\ep\bigr)
	\pd_x\sigma_\ep\, \d x.
\end{align*}

One can verify that 
$I_3 \in L^2(\Omega\times [0,T])$, so 
that $\Ex \int_0^t I_3 \,\d W=0$. Let us argue 
in some more detail for the square-integrability of $I_3$, 
observing first that
\begin{equation}\label{eq:higher-int-tmp1}
	\abs{S(v)-S'(v)v}= \abs{\alpha v \abs{v}
	\bk{\abs{v}+1}^{\alpha-1}}
	\lesssim_\alpha 1+\abs{v}^{1+\alpha}
	\lesssim 1+\abs{v}^2,
\end{equation}
so that
\begin{align*}
	&\Ex \int_0^T \abs{\int_{\T} I_3 \, \d x}^2\, \d t
	\lesssim_{\sigma,\alpha,T}
	1+\Ex \norm{q_\eps}_{L^\infty([0,T];L^2(\T))}^4
	\overset{\eqref{eq:energybound}}{\lesssim} 1.
\end{align*}

Continuing, 
\begin{align*}
	I_1 & \leq \int_{\T} \ep  S''(q_{\ep}) 
	\,\abs{\pd_x q_{\ep}}^2\,\d x   
	+ \frac12\int_\T \abs{S'(q_{\ep})}^2\,\d x 
	+ \frac12\int_\T\bk{P_\ep-u^2_{\ep}}^2\,\d x.
\end{align*}
By $\abs{S''} \lesssim 1$, \eqref{eq:energybound} and 
Lemma \ref{thm:P-u2_bound}, we thus arrive at
\begin{equation}\label{eq:hi-I1}
	\Ex \int_0^t \abs{I_1}\,\d s \lesssim 1.
\end{equation}

Next, by \eqref{eq:higher-int-tmp1}, 
$\abs{S''} \lesssim 1$ and \eqref{eq:energybound}, we obtain
\begin{equation}\label{eq:hi-I2}
	\Ex \int_0^t \abs{I_2} \,\d s 
	\lesssim_{\sigma,\alpha,T}
	1+\Ex \int_0^t\int_\T 
	\abs{q_{\ep}}^{2}\,\d x\,\d s
	\lesssim_{\sigma,\alpha,T} 1.
\end{equation}

Finally, regarding $L$ in \eqref{eq:2alpha_1}, note that 
$$
S(v)v-\frac12 S'(v)v^2
=\frac12 v^2 \bk{\abs{v}+1}^\alpha 
-\frac{\alpha}2 \abs{v}^3 
\bk{\abs{v}+1}^{\alpha-1}
\geq \frac{1 - \alpha}{2} \abs{v}^{2+\alpha};
$$
hence
$$
L\ge \frac{1-\alpha}{2} 
\int_\T \abs{q_{\ep}}^{2+\alpha}\,\d x.
$$
After integrating \eqref{eq:2alpha_1} 
in time, making use of the estimates 
\eqref{eq:hi-I1}, \eqref{eq:hi-I2} 
and also $S(v)\lesssim_\alpha 
1+ \abs{v}^2$, we arrive at
\begin{align*}
	&\frac{1-\alpha}2 \Ex \int_0^T \int_\T 
	\abs{q_\ep}^{2+\alpha}\,\d x 
	\\ & \qquad 
	\lesssim_{\sigma,\alpha,T} 
	1+\Ex \int_\T \abs{q_\ep(T,x)}^2\,\d x
	+\Ex \int_\T \abs{q_\eps(0,x)}^2\,\d x
	\overset{\eqref{eq:energybound}}{\lesssim} 1.
\end{align*}
This concludes the proof.
\end{proof}

The next result has no counterpart in the 
deterministic theory, see \cite{Coclite:2005tq,Coclite:2015aa,Xin:2000qf}. 
It is going to play an important role in the 
upcoming convergence analysis, as it will 
allow passing to the limit in products like $q_\ep P_\ep$ 
towards $q P$, where $P_\ep = K*\bigl(u^2_\ep
+\frac{1}{2}\abs{q_\ep}^2\bigr)$ and 
$P=K*\bigl(u^2_\ep +\frac{1}{2}\overline{q^2}\bigr)$.  
The deterministic approach of establishing $\ep$-uniform
$W^{1,1}_{t,x}$ estimates---to enforce strong convergence 
of $P_\ep$---does not work in the stochastic setting. 
Besides, the nonlinear quantity $q_\ep^2$ does 
not exhibit weak temporal H\"older continuity 
(uniformly in $\ep$), which would be needed 
for a traditional stochastic compactness argument.

\begin{prop}[temporal translation estimate]
\label{prop:translation-time}
Let $u_\eps$ be the $H^m$ solution 
of the viscous SPDE \eqref{eq:u_ch_ep} with 
initial data $u(0)=u_0$ satisfying \eqref{eq:u0-cond-Hmp0} ($m>1$), 
and set $q_\ep=\pd_x u_\eps$. Fix a nonlinear 
function $S\in W^{2,\infty}_{\operatorname{loc}}(\R)$ 
that satisfies \eqref{eq:entropies}. Set $Q_\ep=S(q_{\ep})$. 
Then, for all $\vartheta\in (0,T\wedge 1)$ and 
$\varphi\in C^\infty(\T)$, 
\begin{equation}\label{eq:translation-time}
	\Ex \sup_{\tau\in (0,\vartheta)}
	\int_0^{T-\tau} \abs{\, 
	\int_{\T} \varphi(x) 
	\bigl(Q_\ep(t+\tau)-Q_\ep(t)\bigr)\, \d x} \, \d t
	\lesssim \norm{\varphi}_{C^2(\T)}
	\, \vartheta^{\frac12}.
\end{equation}
\end{prop}

\begin{rem}
In view of the regularity of $Q_\varepsilon$ 
(which follows from Lemma \ref{thm:cauchy_CtH1_2}), 
the function
$\tau \mapsto \int_0^{T-\tau} \abs{\, \int_{\T} \varphi(x) 
\bigl(Q_\ep(t+\tau)-Q_\ep(t)\bigr)\, \d x} \, \d t$ 
is a.s.~continuous. Therefore, it follows that 
the supremum can be equivalently taken 
over $\mathbb Q \cap (0,\vartheta)$, so that 
the resulting object is a random variable.
\end{rem}

\begin{proof}
The nonlinear composition $Q_\ep=S(q_{\ep})$ 
satisfies the SPDE \eqref{eq:SPDE-S(qe)}. For 
any $t\in [0,T]$ and $\tau>0$ with $t+\tau\leq T$, we 
obtain (easily via integration by parts) 
\begin{align*}
	& \abs{\int_{\T} \varphi(x) 
	\bigl(Q_\ep(t+\tau)-Q_\ep(t)\bigr)\, \d x}
	\\ & \qquad 
	\leq  \sum_{i=1}^{9} \int_t^{t+\tau}\int_{\T} I_\ep^{(i)}
	\, \d x \, \d s
	+\sum_{i=10}^{11}\abs{\,\int_t^{t+\tau}\int_{\T} I_\ep^{(i)}
	\, \d x \, \d W},
\end{align*}
where 
\begin{align*}
	& I_\ep^{(1)}=\abs{u_\ep}\abs{S(q_\ep)}\abs{\pd_x \varphi},
	\quad 
	I_\ep^{(2)}=\frac14 \abs{\pd_x \sigma_\ep^2} 
	\abs{3S(q_\ep)-2S'(q_\ep) q_\ep}\abs{\pd_x \varphi},
	\\ & 
	I_\ep^{(3)}=\abs{\frac12\sigma_\ep^2+\eps}\abs{S(q_\ep)}
	\abs{\pd_{xx}^2 \varphi},\quad 
	I_\ep^{(4)} = \abs{S''(q_\ep)}\,
	\ep\abs{\pd_x q_\ep}^2\abs{\varphi},
	\\ & 
	I_\ep^{(5)}=\abs{S'(q_{\ep})}\abs{P_{\ep}}\abs{\varphi},
	\quad 
	I_\ep^{(6)}=\abs{S'(q_{\ep})}\abs{u_{\ep}^2}\abs{\varphi},
	\\ & 
	I_\ep^{(7)}=\abs{S(q_{\ep})q_\ep
	-\frac12 S'(q_\ep)q_\ep^2}\abs{\varphi},
	\\ &
	I_\ep^{(8)}=\frac14 \abs{\pd_{xx}^2\sigma_\ep^2} 
	\abs{S(q_\ep)-S'(q_\ep) q_\ep}\abs{\varphi},
	\quad 
	I_\ep^{(9)}=\frac12 \abs{\pd_x \sigma_\ep}^2
	\abs{S''(q_\ep)} \abs{q_\ep^2}\abs{\varphi},
	\\ &
	I_\ep^{(10)}=\sigma_\ep\, S(q_\ep)\, \pd_x \varphi,
	\quad
	I_\ep^{(11)}=\pd_x\sigma_\ep\, 
	\bigl(S(q_\ep)-S'(q_\ep)q_\ep\bigr)\,\varphi.
\end{align*}
This implies that 
\begin{align}
	& \Ex \sup_{\tau\in (0,\vartheta)}
	\int_0^{T-\tau} \abs{\, 
	\int_{\T} \varphi(x)\bigl(Q_\ep(t+\tau)-Q_\ep(t)\bigr)
	\, \d x} \, \d t
	\label{eq:translation-time-tmp0}
	\\ & \qquad 
	\lesssim \vartheta \sum_{i=1}^{9} 
	\Ex \norm{I_\ep^{(i)}}_{L^1([0,T]\times \T)}
	\notag 
	\\ & \qquad \qquad
	+\sum_{i=10}^{11}\int_0^{T-\vartheta}
	\Ex \sup_{\, \tau\in (0,\vartheta)}
	\abs{\, \int_t^{t+\tau}\int_{\T} I_\ep^{(i)}
	\, \d x \, \d W(s)}\,\d t.
	\notag
\end{align} 

In what follows, we use the trivial fact that 
$$
\norm{\varphi}_{L^\infty(\T)},
\norm{\pd_x \varphi}_{L^\infty(\T)},
\norm{\pd_{xx}^2 \varphi}_{L^\infty(\T)}
\leq \norm{\varphi}_{C^2(\T)}, 
$$

For $i\in \{2,3,8,9\}$, recalling $\sigma\in W^{2,\infty}(\T)$ 
and \eqref{eq:entropies}, 
$$
I_\ep^{(i)}\lesssim \abs{q_\ep}^2\norm{\varphi}_{C^2(\T)},
$$
and thus, by \eqref{eq:energybound} (as $L^\infty_tL^2_x
\hookrightarrow L^2_{t,x}$), 
$$
\Ex \norm{I_\ep^{(i)}}_{L^1([0,T]\times \T)}
\lesssim \Ex \left[\norm{q_\ep}_{L^2([0,T]\times \T)}^2\right]
\norm{\varphi}_{C^2(\T)}
\lesssim \norm{\varphi}_{C^2(\T)}.
$$

Similarly, for $i=7$,
$$
\Ex \norm{I_\ep^{(7)}}_{L^1([0,T]\times \T)}
\lesssim \Ex \left[\norm{q_\ep}_{L^2([0,T]\times \T)}^2\right]
\norm{\varphi}_{C^2(\T)}
\lesssim \norm{\varphi}_{C^2(\T)}.
$$

For $i=5$, we use \eqref{eq:entropies}, 
Lemma \ref{thm:P-u2_bound} and \eqref{eq:energybound} 
(as $L^\infty_tL^2_x\hookrightarrow L^1_{t,x}$),
\begin{align*}
	& \Ex \norm{I_\ep^{(5)}}_{L^1([0,T]\times \T)}
	\lesssim \Ex\left[ \int_0^T \int_\T 
	\abs{q_\ep} \, \d x\, \d t 
	\norm{P_\ep}_{L^\infty([0,T]\times\T)}\right]
	\norm{\varphi}_{C^2(\T)}
	\\ & \quad \leq 
	\left(\Ex \norm{q_\ep}_{L^1([0,T]\times \T)}^2\right)^{1/2}
	\left(\Ex \norm{P_\ep}_{L^\infty([0,T]\times\T)}^2
	\right)^{1/2} \norm{\varphi}_{C^2(\T)} 
	\\ & \quad \lesssim \norm{\varphi}_{C^2(\T)}.
\end{align*}

Similarly, for $i=6$,
\begin{align*}
	& \Ex \norm{I_\ep^{(6)}}_{L^1([0,T]\times \T)}
	\\ & \quad \lesssim 
	\left(\Ex \norm{q_\ep}_{L^1([0,T]\times \T)}^2\right)^{1/2}
	\left(\Ex \norm{u_\ep}_{L^\infty([0,T]\times\T)}^4
	\right)^{1/2} \norm{\varphi}_{C^2(\T)}
	\\ & \quad \lesssim \norm{\varphi}_{C^2(\T)}.
\end{align*}

For $i=1$, by  Lemma \ref{thm:P-u2_bound} 
and \eqref{eq:energybound} (as $L^\infty_tL^2_x
\hookrightarrow L^2_{t,x}$), 
\begin{align*}
	& \Ex \norm{I_\ep^{(1)}}_{L^1([0,T]\times \T)}
	\lesssim \Ex\left[\norm{u_\ep}_{L^\infty([0,T]\times \T)}
	\norm{q_\ep}_{L^2([0,T]\times \T)}^2\right]
	\norm{\varphi}_{C^2(\T)}
	\\ & \quad \leq 
	\left(\Ex\norm{u_\ep}_{L^\infty([0,T]\times \T)}^2\right)^{1/2}
	\left(\Ex\norm{q_\ep}_{L^2([0,T]\times \T)}^4\right)^{1/2}
	\norm{\varphi}_{C^2(\T)}
	\\ & \quad \lesssim \norm{\varphi}_{C^2(\T)}.
\end{align*}

For $i=4$, by the second part of \eqref{eq:energybound},
\begin{align*}
	& \Ex \norm{I_\ep^{(4)}}_{L^1([0,T]\times \T)}
	\lesssim \Ex \int_0^T \int_\T \ep\abs{\pd_{xx}^2 u_\ep}^2
	\, \d x \, \d t \norm{\varphi}_{C^2(\T)}
	\lesssim \norm{\varphi}_{C^2(\T)}.
\end{align*}

Finally, we turn to the stochastic 
integrals ($i=11,12$). By $\sigma\in W^{2,\infty}(\T)$, 
\eqref{eq:entropies} and the BDG inequality,
\begin{align*}
	& \Ex \sup_{\, \tau\in (0,\vartheta)}
	\abs{\, \int_t^{t+\tau}\int_{\T} I_\ep^{(i)}
	\, \d x \, \d W(s)}\,\d t
	\\ & \quad \lesssim \Ex \left[\left(\int_t^{t+\vartheta}
	\left(\int_{\T} q_\ep^2\, \d x\right)^2\,ds\right)^{1/2}
	\right] \norm{\varphi}_{C^2(\T)}
	\\ & \quad \leq \vartheta^{\frac12}
	\Ex \norm{q_\ep}_{L^\infty([0,T];L^2(\T))}^2
	\norm{\varphi}_{C^2(\T)}
	\overset{\eqref{eq:energybound}}{\lesssim} 
	\vartheta^{\frac12}\norm{\varphi}_{C^2(\T)}.
\end{align*}

Given \eqref{eq:translation-time-tmp0}, 
the above estimates yield \eqref{eq:translation-time}.
\end{proof}

\section{Tightness and a.s.~representations}\label{sec:Jak-Skor}

\subsection{Renormalisations}
As explained in the introduction, we 
wish to use the $\eps$-uniform a 
priori estimates \eqref{eq:energybound}, 
\eqref{eq:Holder-viscous}, \eqref{eq:higher-int} 
to extract a.s.~convergence properties of 
$u_\ep$ and of the spatial gradient $q_\eps=\pd_x u_\ep$, as 
well as of nonlinear quantities like $q_\eps^2$. 
Verifying the tightness of the different 
probability laws, among which the one for $q_\eps^2$ 
is the most challenging, we construct 
Skorokhod a.s.~representations 
of $u_\ep$, whose laws are defined on the Polish space 
$C_tL^2_x$, and Jakubowski a.s.~representations 
of $q_\ep$ and several infinite 
sequences of nonlinear compositions 
of $q_\eps$, whose laws are defined on suitable 
quasi-Polish spaces like $L^p_{t,x}-w$. 
In addition, for the energy variable $q_\ep^2$, 
we construct representations in 
the quasi-Polish spaces $L^p\bigl(L^p_w\bigr)$, 
for some $p$, which supplies a crucial strong 
convergence property in $t$. 
We refer to Section \ref{sec:qpolish} for 
quasi-Polish spaces and their properties. 

In what follows, we fix a sequence $\{\ep_n\}_{n=1}^\infty$ 
of positive numbers such that $\ep_n \to 0$ 
as $n\to \infty$. Let us introduce 
the random mappings:
\begin{equation}\label{eq:random-variables-I}
	\begin{split}	
		& F_{\ep_n}^{q}=
		\left(q_{\ep_n},\bigl(q_{\ep_n}\bigr)_+,
		\bigl(q_{\ep_n}\bigr)_-\right),
		\quad 
		F_{\ep_n}^{q^2}=
		\left(q_{\ep_n}^2,\bigl(q_{\ep_n}\bigr)^2_{+},
		\bigl(q_{\ep_n}\bigr)^2_{-}\right).
	\end{split}
\end{equation}
Here, we denote by $f_+$ and $f_-$ the positive and 
negative parts of a function $f$, so that 
$f(v)=f_+(v)+f_-(v)=\max(f(v),0)+\min(f(v),0)$. 
Furthermore, the notation $\bigl(q_{\ep_n}\bigr)^2_{\pm}$ is 
a concise representation of $\bigl((q_{\ep_n})_{\pm}\bigr)^2$.

To execute various renormalisation procedures, we 
shall need to take limits as $n \to \infty$ of 
infinite sequences of nonlinear compositions 
of $q_{\ep_n}$, like $S_\ell(q_{\ep_n})$, where 
$\left\{S_\ell(v)\right\}_{\ell\in \N}$
is a sequence that approximates $\frac12 v^2$ up to some 
cut-off $\abs{v} \le \ell$. The Skorokhod--Jakubowski 
procedure extracts \textit{simultaneously} 
the a.s.~convergence of all these variables. 
As a preparation, we introduce several sequences 
of random mappings that are built from the approximation 
$\left\{S_\ell(v)\right\}_{\ell=1}^\infty$ 
of $\frac12 v^2$ on $\R$:
\begin{align}\label{eq:entropies_S_ell}
	S_\ell(v)
	=
	\begin{cases}
		\frac12 v^2, & \abs{v} \le \ell 
		\\
		-\frac1{6\ell}\abs{v}^3+v^2
		-\frac12 \ell \abs{v} 
		+\frac16\ell^2,  & \ell<\abs{v}<2 \ell
		\\
		\frac32\ell \abs{v}-\frac76 \ell^2, 
		& \abs{v}\ge 2\ell
	\end{cases}.
\end{align}
Each function $S_\ell$ is convex and satisfies
\begin{align}\label{eq:entrop1}
	S'_\ell(v)
	=
	\begin{cases}
		v, & \abs{v} \le \ell
		\\
		\sgn(v) \bk{2\abs{v}-\frac1{2\ell}v^2-\frac12 \ell},
		& \ell<\abs{v}<2 \ell
		\\ 
		\frac{3}2 \sgn(v) \ell,  
		& \abs{v} \ge 2 \ell
	\end{cases}
\end{align}
and
\begin{equation}\label{eq:Sell-2nd-der}
	S''_\ell(v)= 
	\begin{cases} 
		1, & \abs{v} \le \ell 
		\\ 
		\frac1\ell \bk{2\ell-\abs{v}}, 
		& \ell<\abs{v}<2 \ell 
		\\
		0, & \abs{v} \ge 2 \ell
	\end{cases}.	
\end{equation}

The random mappings that we introduce below 
are motivated by the need to pass to the weak limit in 
various nonlinear compositions of $q_{\ep_n}$, based on 
\begin{equation}\label{eq:bf_F2}
	\begin{aligned}
		& S_\ell(v_\pm),
		\quad 
		S_\ell(v_\pm)',
		\quad 
		S_\ell(v_\pm)''v^2,
		\\ & 
		\quad
		S_\ell(v_\pm)-S_\ell(v_\pm)' v,
		\quad 
		S_\ell(v_\pm) v
		-\frac12 S_\ell(v_\pm)'v^2.
	\end{aligned}
\end{equation}

Clearly, $\bk{v_\pm}'=\one{\abs{v_\pm}>0}$, $\bk{v_\pm^2}'
=2v_\pm$, $\bk{v_\pm^2}''=2\one{\abs{v_\pm}>0}$,
and so $S(v)=v_\pm^2$ belongs to 
$W^{2,\infty}_{\operatorname{loc}}(\R)$ 
and satisfies \eqref{eq:entropies}. 
Using \eqref{eq:entropies_S_ell}, 
\eqref{eq:entrop1}, \eqref{eq:Sell-2nd-der} 
and the chain rule, we can readily compute 
the following nonlinear compositions:
\begin{equation}\label{eq:entrop2}
	\begin{aligned}
		&S_\ell(v_\pm)'=S'_\ell(v_\pm),
		\qquad 
		S_\ell(v_\pm)''=S''_\ell(v_\pm)\one{\abs{v_\pm}>0},
		\\ &
		S_\ell(v_\pm)-S_\ell(v_\pm)'v
		= 
		\begin{cases}
			-\frac12 v_\pm^2, & \abs{v_\pm} \le \ell
			\\
			\frac1{3\ell}\abs{v_\pm}^3-{v_\pm}^2
			+\frac16 \ell^2, 
			& \ell<\abs{v_\pm}<2\ell
			\\
			- \frac76\ell^2, 
			& \abs{v_\pm}\ge 2\ell
		\end{cases},
		\\
		& 3S_\ell(v_\pm)-2S_\ell(v_\pm)'v 
		\\ & \qquad =
		\begin{cases}
			-\frac12 v_\pm^2, & \abs{v_\pm} \le \ell
			\\
			\frac1{2\ell}\abs{v_\pm}^3-{v_\pm}^2
			-\frac12  \abs{v_\pm}\ell+\frac12\ell^2,
			& \ell<\abs{v_\pm}<2 \ell
			\\ 
			\frac32\ell \abs{v_\pm}-\frac72\ell^2, 
			& \abs{v_\pm} \ge 2 \ell
		\end{cases},
		\\ & 
		S_\ell(v_\pm) v-\frac12S_\ell(v_\pm)'v^2
		\\ & \qquad = 
		\begin{cases}
			0,& \abs{v_\pm} \le \ell
			\\ 
			\frac1{12 \ell} \abs{v_\pm}^3 v_\pm
			-\frac14 \abs{v_\pm}v_\pm\ell
			+\frac16 v_\pm\ell^2, 
			& \ell<\abs{v_\pm}<2 \ell
			\\ 
			\frac34 \abs{v_\pm}v_\pm\ell-\frac76 v_\pm \ell^2, 
			& \abs{v_\pm} \ge 2 \ell
		\end{cases}.
	\end{aligned}
\end{equation}
In particular, this implies that 
$S_\ell \in W^{3,\infty}_{\operatorname{loc}}(\R)$, 
$\abs{S_\ell(v)}\lesssim_\ell \abs{v}$, 
$\abs{S'_\ell(v)}\lesssim_\ell 1$, $\abs{S''_\ell(v)}\lesssim 
\mathds{1}_{\{\abs{v}\le 2\ell\}}$, and 
$\abs{S_\ell(v)v-\frac12 S'_\ell(v) v^2}
\lesssim_\ell \abs{v}^2$, so that 
\eqref{eq:entropies} is satisfied with $S=S_\ell$.  
The nonlinear compositions $v\mapsto S_\ell(v_\pm)$ belong 
to $W^{2,\infty}_{\operatorname{loc}}(\R)$ 
and cater to similar bounds, see also 
Remark \ref{rem:pos-neg-parts}. 

Notice that $v\mapsto S_\ell(v_\pm)''v^2= S''_\ell(v_\pm)v^2$ 
is a continuous function (but $S_\ell(v_\pm)''$ is not). 
Later, we will also need to know that the function
$\beta(v)=S_\ell(v_\pm)'v$ belongs to 
$W^{2,\infty}_{\operatorname{loc}}(\R)$ (although $S_\ell(v_\pm)'$
does not) and satisfies \eqref{eq:entropies} 
with $S=\beta$. Indeed,
\begin{equation}\label{eq:beta-def}
	\begin{split}
		& \beta'(v)= S''_\ell(v_\pm)\one{\abs{v_\pm}>0}v+S'_\ell(v_\pm),
		\\ & \beta''(v)
		=S'''_\ell(v_\pm)\one{\abs{v_\pm}>0}v
		+2S''_\ell(v_\pm)\one{\abs{v_\pm}>0},
	\end{split}
\end{equation}
so that $\abs{\beta(v)}\lesssim_\ell \abs{v}$,
$\abs{\beta'(v)}\lesssim_\ell 1$, 
and $\abs{\beta''(v)}\lesssim_\ell 1$.

\begin{rem}
One may wonder about the specific choice \eqref{eq:entropies_S_ell} 
of renormalisations (entropies), which admittedly comes 
across as complicated. At this point, we run into a new difficulty 
compared to the deterministic CH equation \cite{Xin:2000qf}. 
The particular form of the noise in the stochastic CH 
equation \eqref{eq:u_ch} leads to some key 
structural changes in the equation satisfied by $S_\ell(q_\ep)$, 
which prevents us from using the simple entropies of 
\cite{Xin:2000qf} (linearly growing $W^{2,\infty}$ 
approximations of $v^2$). The entropies \eqref{eq:bf_F2} are 
carefully constructed to allow for the control of some 
delicate error terms involving weak limits 
linked to the defect measure \eqref{eq:intro-defect}, 
see Remark \ref{rem:delicate}.
\end{rem}

\subsection{Random mappings and path spaces}
For $\ell\in \N$, we introduce the random mappings
\begin{equation}\label{eq:random-variables-II}
	F_{\ep_n}^{\xi_\ell,\pm}
	=\xi_\ell\big|_{v=q_{\ep_n}},
	\quad \xi_\ell \in \mathbb{S}_\ell^\pm,
\end{equation}
where $\mathbb{S}_\ell^\pm$ denote the collections
\begin{equation}\label{eq:nonlinear-collection}
	\mathbb{S}_\ell^\pm
	=\Bigl\{
	S_\ell(v_\pm),
	\, S_\ell'(v_\pm)v, 
	\, S_\ell''(v_\pm)v^2, 
	\, S_\ell(v_\pm)v,\, S_\ell'(v_\pm)v^2,
	\, S_\ell'(v_\pm)
	\Bigr\}
\end{equation}
of nonlinear functions. We also make 
use of $F_{\ep_n}^{\mathbb{S}}$ as a notation 
for the gathering of all these $\ell$-dependent mappings:
\begin{equation}\label{eq:random-variables-III}
	F_{\ep_n}^{\mathbb{S}}
	=\left\{
	\left\{F_{\ep_n}^{\xi_\ell,+},
	\,\, \xi_\ell\in \mathbb{S}_\ell^+
	\right\}_{\ell\in \N},
	\, \,
	\left\{F_{\ep_n}^{\xi_\ell,-},
	\,\, \xi_\ell\in \mathbb{S}_\ell^-
	\right\}_{\ell\in \N}
	\right\}.
\end{equation}

Finally, we use $X_n$ as a collective symbol for 
\textit{all} the random mappings just introduced:
\begin{equation}\label{eq:Xn-def}
	X_n=\left(u_{\ep_n},
	F_{\ep_n}^{q},
	F_{\ep_n}^{q^2},
	W,z_n,
	F_{\ep_n}^{\mathbb{S}}
	\right),
\end{equation} 
where $W$ is the Wiener process appearing 
in \eqref{eq:u_ch_ep} and $\left\{z_n\right\}_{n=1}^\infty$ 
is a sequence of $C^\infty$ approximations of the initial 
data $u_0$, satisfying 
\begin{equation}\label{eq:u0-approx}
	z_n \in L^{p_0}\bigl(\Omega;C^\infty(\T)\bigr),
	\quad z_n \ton u_0 \quad \text{in 
	$L^{p_0}\bigl(\Omega;H^1(\T)\bigr)$},
\end{equation} 
recalling the assumption $p_0>4$ from 
Theorem \ref{thm:bounds1}. 

The goal is to establish the tightness of the 
joint probability laws $\mu_n=\cL(X_n)$ 
of the random mappings $X_n:\bigl(\Omega,\mathcal{F},
\mathbb{P}\bigr) \to \bigl(\mathcal{X},
\mathcal{B}_{\mathcal{X}} \bigr)$. 
To this end, we need to specify 
$\mathcal{X}$---the path space for $\mu_n$---and the
$\sigma$-algebra $\mathcal{B}_\mathcal{X}$. 
Denote the factors of the infinite vector $X_n$ 
by $X_n^{(l)}$ and the corresponding factor spaces 
by $\mathcal{X}_{l}$, $l\in \N$. For example, 
$X_n^{(1)}=u_{\ep_n}$ and $X_n^{(5)}=q_{\ep_n}^2$, 
cf.~\eqref{eq:random-variables-I}. 
Whenever needed, we also use superscript symbols 
on $X_n$ to identify the corresponding 
factor of $X_n$, for example, $X_n^{u}=X_n^{(1)}$ 
and $X_n^{q^2}=X_n^{(5)}$, while $X_n^{S_\ell(v_+)}$ 
would refer to $X_n^{(7+\ell)}$ 
with $X_n^{(7)}=\bigl(q_{\ep_n}\bigr)^2_-$, 
see \eqref{eq:random-variables-II} and 
\eqref{eq:random-variables-I}. 
Denote by $ \mu^{(l)}$ the corresponding 
marginals of $\mu_n$, defined on $\bigl(\mathcal{X}_{l},
\mathcal{B}_{\mathcal{X}_{l}}\bigr)$. Similarly, 
we will write $\mu_n^u$ instead of $\mu_n^{(1)}$ 
for the marginal linked to $X_n^{(1)}=u_{\ep_n}$, 
and so forth, and the same for the factor spaces $\mathcal{X}_l$. 

\begin{rem}
The notation just introduced may appear overwhelming. 
Fortunately, most of it will be utilised only in this section.
\end{rem}

For a fixed number $r\in \bigl[1,\frac32\bigr)$ (close to $3/2$), 
we specify the following spaces for the marginals:
\begin{equation}\label{eq:pathspaces}
	\begin{split}
		& \mathcal{X}_u= C_tL^2_x,
		\quad \mathcal{X}_W= C_t, 
		\quad \mathcal{X}_{u_0}= H^1_x,
		\\ &
		\mathcal{X}_q=L^{2r}_{t,x}-w,
		\quad
		\mathcal{X}_{q_\pm}=L^{2r}_{t,x}-w,
		\\ &
		\mathcal{X}_{q^2}=L^r\bigl(L^r_w\bigr),
		\quad
		\mathcal{X}_{q_\pm^2}=L^r\bigl(L^r_w\bigr),
		\\ & 
		\mathcal{X}_{\xi}=L^{2r}(L^{2r}_w),
		\,\, \xi=S_\ell(v_\pm),\, S_\ell(v_\pm)'v,
		\, \, \ell\in \N, 
		\\ & 
		\mathcal{X}_{\xi}=L^{2r}_{t,x}-w,
		\,\, \xi=\, S_\ell(v_\pm)',
		\, \, \ell\in \N, 
		\\ & 
		\mathcal{X}_{\xi}=L^r_{t,x}-w,
		\,\, \xi= S_\ell(v_\pm)''v^2,
		\, S_\ell(v_\pm)v,\, S_\ell(v_\pm)'v^2,
		\,\, \ell\in \N.		
	\end{split}
\end{equation}

Here, $C_t=C([0,T])$, $H^1_x=H^1(\T)$, 
and $C_tL^2_x=C([0,T];L^2(\T))$ are all 
Polish spaces. Furthermore, $L^p_{t,x}-w=L^p([0,T]\times \T)-w$, 
for any $p\in [1,\infty)$, denotes the $L^p$ space 
equipped with the weak topology, which is quasi-Polish. 
For the energy variables $q_{\ep_n}^2$ and $(q_{\ep_n})_\pm^2$, we use 
the space $L^r\bigl(L^r_w\bigr)=L^r\bigl([0,T];L^r(\T)-w\bigr)$, 
which is quasi-Polish as well, see 
Section \ref{sec:qpolish} and \eqref{eq:LptLpx-weak} 
for details. Notice that the topology of 
$L^r\bigl(L^r_w\bigr)$ is strong in $t$ and weak in $x$. 
Similarly, we use $L^{2r}(L^{2r}_w)$ for 
the variables $S_\ell\bigl((q_{\ep_n})_\pm\bigr)$ 
and $S_\ell'\bigl((q_{\ep_n})_\pm\bigr)q_{\ep_n}$ (linearly 
growing approximations of $\frac12 (q_{\ep_n})_\pm^2$).

\begin{rem}
The spaces prescribed in \eqref{eq:pathspaces} reflect 
some minimum requirements for convergence in 
Section \ref{sec:wk_limits_section}. 
The significance of the peculiar ``strong-weak" spaces 
$L^r\bigl(L^r_w\bigr)$, $L^{2r}\bigl(L^{2r}_w\bigr)$ 
will become clear during the proofs of 
Lemmas \ref{lem:energy-of-weak-limit}, 
\ref{lem:strong-conv-tPn}, and \ref{thm:limiteq_1}. 
Roughly speaking, these spaces will allow us to 
pass to the limit in delicate product terms like 
$S'(\tilde q_{\ep_n})\, \tilde P_{\ep_n}$ as well as
in various stochastic integrals.
\end{rem}
		
The path space $\mathcal{X}$ for the joint 
laws $\left\{\mu_n\right\}_{n\in \N}$ is taken as 
\begin{equation}\label{eq:pathspaces-joint}
	\mathcal{X}=\prod_{l=1}^\infty \mathcal{X}_l,
	\qquad \mathcal{B}_{\mathcal{X}}
	=\mathcal{B}(\mathcal{X}),
\end{equation}
which carries the product topology 
for its infinitely many factors.

\begin{rem}
Each factor space $\mathcal{X}_l$ in $\mathcal{X}$ is 
either Polish or quasi-Polish. 
Polish spaces are quasi-Polish and countable 
products of quasi-Polish spaces are quasi-Polish, 
see Lemma \ref{thm:a_0}. Generally, for 
a quasi-Polish space $\mathcal{Y}$ there 
are two natural candidates for 
the $\sigma$-algebra $\mathcal{B}_{\mathcal{Y}}$, 
the Borel $\sigma$-algebra $\mathcal{B}_{\mathcal{Y}}
=\mathcal{B}(\mathcal{Y})$ or the $\sigma$-algebra 
$\mathcal{B}_{\mathcal{Y}}=\mathcal{B}_f$
generated by the separating sequence 
$f=\left\{f_l\right\}_{l\in \N}$ defining 
the space, see Definition \ref{def:quasipolish}.  
In general, $\mathcal{B}_f\subset 
\mathcal{B}(\mathcal{Y})$, see Lemma \ref{thm:a_4}.
However, for each space in \eqref{eq:pathspaces} 
we use the Borel $\sigma$-algebra, as it happens to 
coincide with the one generated by the separating 
sequence (see Lemma \ref{lem:qpolish-borel-sigma-algebra}). 
The space $\mathcal{X}$ for the joint laws 
is equipped with the product topology. For 
a quasi-Polish product space like $\mathcal{X}$, the 
Borel $\sigma$-algebra $\mathcal{B}(\mathcal{X})$ 
for the product topology is likely to differ from 
the product of the individual Borel $\sigma$-algebras 
(although they do coincide if $\mathcal{X}$ is Polish), see  
Lemma \ref{thm:a_4}. However, as is shown 
in \cite{Jakubowski:1997aa}, this is not a problem as 
long as we work with random mappings with tight laws. 
To be specific, we take $\mathcal{B}_{\mathcal{X}}
=\mathcal{B}(\mathcal{X})$.
\end{rem}

Consider the random map 
$X_n=\bigl\{X_n^{(l)}\bigr\}_{l\in \N}
:\bigl(\Omega,\mathcal{F},
\mathbb{P}\bigr) \to \bigl(\mathcal{X},
\mathcal{B}_{\mathcal{X}} \bigr)$ defined by 
\eqref{eq:Xn-def}. By Theorem \ref{thm:bounds1}, one can 
check that each factor $X_n^{(l)}$ is a random 
variable (Borel measurable). We only prove this 
for the nonlinear parts of $X_n$ involving $q_{\ep_n}$ 
(the other parts are simpler). By construction, 
$\Omega\overset{q_{\varepsilon_n}}{\xrightarrow{\hspace*{0.3cm}}} 
C([0,T];H^1(\T))$ is a random variable. 
Since all of our nonlinear functions 
or entropies (here denoted by the 
generic placeholder $\beta$) are continuous 
real-valued functions and satisfy (at least) the 
bound $\abs{\beta(v)}\lesssim \abs{v}^{3-1}$, Nemytskii 
theory ensures then that these entropies $\beta$, when 
viewed as operators, are bounded and continuous 
from $L^3_{t,x}$ into $L^{3/2}_{t,x}$. 
Because $L^{3/2}_{t,x}$ embeds continuously in 
$L^r_{t,x}-w$ and $C_tH^1_x$ embeds continuously 
in $L^3_{t,x}$, the composition 
$\Omega\overset{\beta(q_{\varepsilon_n})}{\xrightarrow{\hspace*{0.3cm}}}
L^r_{t,x}-w$ is a random variable.

\subsection{Compactness and tightness criteria}

The goal is to establish the tightness 
of the joint laws of $X_n$. The most difficult 
part is to verify the tightness of the laws 
of the energy variables $q_{\ep_n}^2$ and 
$(q_{\ep_n})_{\pm}^2$---and 
similarly also $S_\ell\bigl((q_{\ep_n})_\pm\bigr)$, 
$S_\ell'\bigl((q_{\ep_n})_\pm\bigr)q_{\ep_n}$---which 
take values in a quasi-Polish space of the form 
$L^{p_1}\bigl(L^{p_2}_w\bigr)$, for some $p_1,p_2\in (1,\infty)$, 
see \eqref{eq:LptLpx-weak}. This space encodes strong temporal and 
weak spatial compactness. The strong  
$t$-compactness of $q_{\ep_n}^2$ is essential 
for our analysis, noting that there is no hope of establishing 
uniform H\"older continuity in $t$, even if the spatial 
topology is weak. This excludes the traditional compactness approach 
based on tightness in the space $C([0,T];L^p(\T)-w)$, 
used by many of the references listed in Section \ref{sec:intro}.
Indeed, the space $L^{p_1}\bigl(L^{p_2}_w\bigr)$ was 
carefully singled out to resolve this particular 
predicament of the energy variable. 

The following result, which is of independent interest, provides
general criteria for compactness in $L^{p_1}\bigl(L^{p_2}_w\bigr)$.
These criteria will be later used in the analysis of tightness.
For the space $L^{p_1}\bigl(L^{p_2}_w\bigr)$ there exists a sequence
of continuous functionals that separate points and generate
the Borel $\sigma$-algebra. This fact is discussed in
Appendix \ref{sec:qpolish} and can be found in \eqref{eq:LptLpx-weak}.
Based on this, Jakubowski \cite[page 169]{Jakubowski:1997aa}
states that the notions of compactness and
sequential compactness are equivalent.

The lemma stated below identifies conditions that ensure
the \textit{relative} sequential compactness of a subset in
$L^{p_1}\bigl(L^{p_2}_w\bigr)$. In Appendix \ref{sec:SJThm_appendix}, we show
that the notions of \textit{relative compactness}
and \textit{relative sequential compactness} are also the same
in quasi-Polish spaces like $L^{p_1}\bigl(L^{p_2}_w\bigr)$.
Whence, the closure $\overline{\mathcal{K}}$ 
of a relatively sequentially compact
set $\mathcal{K}$ can be used to verify the tightness
condition of Jakubowski's theorem \cite{Jakubowski:1997aa}
(see Theorem \ref{thm:Jakubowski}).

\begin{lem}[compactness criterion]\label{lem:new-compact}
Fix some integrability indices $p_1,p_2\in (1,\infty)$,
and consider the space $L^{p_1}\bigl(L^{p_2}_w\bigr)
= L^{p_1}\bigl([0,T];L^{p_2}(\T)-w\bigr)$,
cf.~\eqref{eq:LptLpx-weak}. Let $\mathcal{K}$ be a subset
of $L^{p_1}\bigl(L^{p_2}_w\bigr)$ for which
the following conditions hold uniformly in $Q\in \mathcal{K}$:
\begin{align*}
        & \mathrm{(i)}\, \, \,
        \norm{Q}_{L^{p_1}([0,T];L^{p_2}(\T))}\lesssim 1,
        \\ &
        \mathrm{(ii)}\, \, \,
        \norm{Q}_{L^{\bar p_1}([0,T];L^1(\T))}\lesssim 1,
        \quad \text{for some $\bar p_1>p_1$},
        \\ &
        \mathrm{(iii)}\, \, \,
        \int_0^{T-\tau} \abs{\,\int_{\T}
        \varphi(x) \bigl(Q(t+\tau,x)-Q(t,x)\bigr)\, \d x}\, \d t
        \totau 0, \quad \forall \varphi\in C^\infty(\T).
\end{align*}
Then $\mathcal{K}$ is relatively sequentially compact 
in $L^{p_1}\bigl(L^{p_2}_w\bigr)$.
\end{lem}

\begin{rem}
Note carefully how, in $\mathrm{(ii)}$, some higher temporal
integrability is traded for low spatial integrability.
This flexibility is important for us. However, in other applications,
if one is happy with the temporal integrability provided by $\mathrm{(i)}$,
then $\mathrm{(ii)}$ can be dropped at the expense of getting
compactness in $L^p\bigl(L^{p_2}_w\bigr)$, $\forall p<p_1$.
\end{rem}

\begin{proof}
Consider a subset $\mathcal{K}\subset L^{p_1}\bigl(L^{p_2}_w\bigr)$
for which (i), (ii) and (iii) hold. To establish the lemma, we must
demonstrate that for any sequence $\left\{Q_n\right\}_{n\in \N}$
in $\mathcal{K}$, it is possible to find a subsequence that
converges in $L^{p_1}\bigl(L^{p_2}_w\bigr)$.

By (i), there exists a subsequence
$\left\{Q_{n_j}\right\}_{j\in \N}$
of $\left\{Q_n\right\}_{n\in \N}$
that converges weakly to some $Q$
in $L^{p_1}([0,T];L^{p_2}(\T))$:
\begin{equation}\label{eq:weak-limit-Q}
        \int_0^T \int_\T \psi(t)\varphi(x)Q_{n_j}(t,x)\, \d x\, \d t
        \toj \int_0^T \int_\T \psi(t)\varphi(x)Q(t,x)\, \d x\, \d t,
\end{equation}
for all $\psi\in L^{p_1'}([0,T])$ and $\varphi\in L^{p_2'}(\T)$,
$\frac{1}{p_1}+\frac{1}{p_1'}= \frac{1}{p_2}+\frac{1}{p_2'}=1$.

Let $J_\delta$ be a standard (Friedrichs)
mollifier in $x$ and set
$$
Q_{n_j,\delta}=Q_{n_j}*J_\delta, \quad
Q_{\delta}=Q*J_\delta.
$$
Then
\begin{equation}\label{eq:conv-weak-conv}
        \begin{split}
                & Q_{n_j,\delta}\tojweak Q_{\delta}\quad
                \text{in $L^{p_1}([0,T];L^{p_2}(\T))$,
                for each fixed $\delta$},
                \\ &
                Q_{\delta}\todelta Q
                \quad \text{in $L^{p_1}([0,T];L^{p_2}(\T))$},
        \end{split}
\end{equation}
where the second convergence comes from basic
properties of mollifiers (in $x$) and, via
(i), Lebesgue's dominated convergence theorem in $t$.
The first convergence can be proved using a basic
property of the convolution product. Indeed, we have
\begin{align*}
        & \abs{\int_0^T \int_\T \psi(t)\varphi(x)
        \left(Q_\delta(t,x)-Q_{n_j,\delta}(t,x)\right)\, \d x\, \d t}
        \\ & \qquad =
        \abs{\int_0^T \int_\T \psi(t)\varphi_\delta(x)
        \left(Q(t,x)-Q_{n_j}(t,x)\right)\, \d x\, \d t}
        \toj 0,
\end{align*}
for each fixed $\delta>0$, recalling that the
algebraic tensor product $L^{p_1'}\otimes L^{p_2'}$
is dense in $L^{p_1'}(L^{p_2'})$. 

By the translation estimate (iii) with
$\varphi(x) =J_\delta(y-x)$, for any $y\in \T$,
\begin{align*}
        \int_0^{T-\tau} \abs{Q_{n_j,\delta}(t+\tau,y)
        -Q_{n_j,\delta}(t,y)}\, \d t
        \totau 0,
\end{align*}
uniformly in $j$. Using (i) and
Vitali's convergence theorem (in $y$),
\begin{equation}\label{eq:temp-translation-y}
        \int_0^{T-\tau}\!\!\int_\T\abs{Q_{n_j,\delta}(t+\tau,y)
        -Q_{n_j,\delta}(t,y)} \,\d y\, \d t\totau 0,
\end{equation}
uniformly in $j$.

Next, by (i), we have
$\norm{\pd_x Q_{n_j,\delta}}_{L^{p_1}([0,T];L^1(\T))}
\lesssim_\delta 1$ and thus
$$
\norm{Q_{n_j,\delta}}_{L^{p_1}([0,T];W^{1,1}(\T))}
\lesssim_\delta 1.
$$
By (ii), we also deduce that
$$
\norm{Q_{n_j,\delta}}_{L^{\bar{p}_1}([0,T];L^1(\T))}
\lesssim 1, \quad \text{where $\bar{p}_1>p_1$}.
$$

Consider the compact embedding
$$
W^{1,1}(\T)\doublehookrightarrow L^1(\T),
$$
and now note that $\left\{Q_{n_j,\delta}\right\}_{j\in \N}$
is bounded in
$$
L^{\bar{p}_1}([0,T];L^1(\T))\bigcap
L^1([0,T];W^{1,1}(\T)),
$$
uniformly in $j$, for each fixed $\delta$.
Besides, from \eqref{eq:temp-translation-y},
$$
\norm{Q_{n_j,\delta}(\cdot+\tau,\cdot)
-Q_{n_j,\delta}}_{L^1([t_1,t_2];L^1(\T))}
\totau 0,
$$
for all $0<t_1<t_2<T$, uniformly in $j$,
for each fixed $\delta$.
By \cite[Theorem 4]{Simon:1987vn}, we may
therefore assume that there exists a
limit $\overline{Q}_\delta\in L^{p_1}([0,T];L^1(\T))$
such that
\begin{equation}\label{eq:conv-strong-conv}
        Q_{n_j,\delta}\toj \overline{Q}_\delta
        \quad \text{in $L^{p_1}([0,T];L^1(\T))$},
\end{equation}
for each fixed $\delta$. However, by the
uniqueness of the weak limit
in \eqref{eq:conv-weak-conv},
$$
\overline{Q}_{\delta}=Q_{\delta}=Q*J_\delta.
$$
In fact, all subsequences extracted
from $\left\{Q_{n_j,\delta}\right\}_{j\in \N}$ have further
subsequences that converge to the same
limit $Q*J_\delta$, and therefore
the original sequence also converges to that limit.

Let us verify that $Q_{n_j}\toj Q$
in $L^{p_1}\bigl(L^{p_2}_w\bigr)$,
where $Q$ is defined in \eqref{eq:weak-limit-Q}.
Fix any $\varphi\in C^\infty(\T)$.
We proceed as follows:
\begin{align*}
        I_\varphi(j) & =\int_0^T \abs{\, \int_\T \varphi(x)
        \bigl( Q(t,x)-Q_{n_j}(t,x)\bigr)\, \d x }^{p_1}\, \d t
        \\ & \lesssim
        \underbrace{\int_0^T \abs{\, \int_\T
        \varphi(x)\bigl( Q_\delta(t,x)-Q_{n_j,\delta}(t,x)\bigr)
        \, \d x }^{p_1}\, \d t}_{=:I_1(j,\delta)}
        \\ & \qquad
        + \underbrace{\int_0^T \abs{\int_\T \varphi(x)
        \bigl(Q(t,x)-Q_\delta(t,x)\bigr)
        \, \d x }^{p_1}\, \d t}_{=:I_2(\delta)}
        \\ & \qquad
        +\underbrace{\int_0^T \abs{\int_\T \varphi(x)\bigl(
        Q_{n_j}(t,x)-Q_{n_j,\delta}(t,x)\bigr)
        \, \d x }^{p_1}\, \d t}_{=:I_3(\delta,j)}.
\end{align*}
By \eqref{eq:conv-strong-conv},
\begin{align*}
        I_1(j,\delta)
        \leq \norm{\varphi}_{L^\infty}
        \norm{Q_\delta-Q_{n_j,\delta}}_{L^{p_1}([0,T];L^1(\T))}^{p_1}
        \toj 0,
\end{align*}
for each fixed $\delta>0$. Next, we show that
$I_2$ and $I_3$ tend to zero as $\delta\to 0$,
uniformly in $j$. By H\"older's inequality,
\eqref{eq:conv-weak-conv}, and
$C^\infty(\T)\hookrightarrow L^{p_2'}(\T)$
\begin{align*}
        I_2(\delta)
        & \leq
        \norm{Q-Q_{\delta}}_{L^{p_1}([0,T];L^{p_2}(\T))}^{p_1}
        \norm{\varphi}_{L^{p_2'}(\T)}^{p_1}
        \todelta 0,
\end{align*}
uniformly in $j$. Finally, by (i), a basic property
of the convolution product, H\"older's inequality,
and $C^\infty(\T)\hookrightarrow L^{p_2'}(\T)$,
it follows that
\begin{align*}
        I_3(j,\delta) & =
        \int_0^T \abs{\int_\T Q_{n_j}(t,x)
        \bigl(\varphi(x)-\varphi_\delta(x)\bigr)\, \d x }^{p_1}\, \d t
        \\ & \leq
        \norm{Q_{n_j}}_{L^{p_1}([0,T];L^{p_2}(\T))}^{p_1}
        \norm{\varphi_\delta-\varphi}_{L^{p_2'}(\T)}^{p_1}
        \\ & \lesssim
        \norm{\varphi_\delta-\varphi}_{L^{p_2'}(\T)}^{p_1}
        \todelta 0,
        \quad \text{uniformly in $j$}.
\end{align*}

Summarising, to any given $\kappa>0$, we can choose $\delta\leq\delta_0$,
for a small enough $\delta_0=\delta_0(\kappa)$, such that
$I_2(\delta)+I_3(\delta,j)\leq \kappa/2$ for all $j$,
and then choose an integer $j_0=j_0(\delta_0)$ such that $j\geq j_0$
implies $I_1(j,\delta_0)\leq \kappa/2$, and thus
$I_\varphi(j)\leq \kappa$ for all $j\ge j_0$.
In other words, $I_\varphi(j)\toj 0$, for any $\varphi\in C^\infty(\T)$.
By density of $C^\infty(\T)$ in $L^{p_2'}(\T)$, this convergence
holds for all $\varphi\in L^{p_2'}(\T)$, which concludes the proof.
\end{proof}

We use the previous lemma to formulate a 
tightness criterion in $L^{p_1}\bigl(L^{p_2}_w\bigr)$.

\begin{lem}[tightness criterion]\label{lem:tightness-LpLpx-weak}
Fix $p_1,p_2\in (1,\infty)$ and consider 
the quasi-Polish space $L^{p_1}\bigl(L^{p_2}_w\bigr)$, 
cf.~\eqref{eq:LptLpx-weak}. Let $\left\{Q_n\right\}_{n\in \N}$ 
be a sequence of random variables, defined 
on a standard probability space 
$\bigl(\Omega,\mathcal{F},\mathbb{P}\bigr)$, 
that take values in $L^{p_1}\bigl(L^{p_2}_w\bigr)$. 
Suppose the following conditions 
hold (uniformly in $n\in \N$):
\begin{align*}
	&\mathrm{(i)}\,\, \, 
	\Ex \norm{Q_n}_{L^{p_1}([0,T];L^{p_2}(\T))}
	\lesssim 1,
	\\ &
	\mathrm{(ii)}\,\, \, 
	\Ex \norm{Q_n}_{L^{\bar p_1}([0,T];L^1(\T))}
	\lesssim 1, \quad \text{for some $\bar p_1>p_1$},
\end{align*}
and, for all $\varphi\in C^\infty(\T)$ 
and $\vartheta\in \bigl(0,T\wedge 1\bigr)$,  
\begin{align*}
	\mathrm{(iii)} \, \, \, \Ex \sup_{\tau\in (0,\vartheta)}
	\int_0^{T-\tau}\abs{\, \int_{\T} 
	\varphi(x)\bigl(Q_n(t+\tau,x)-Q_n(t,x)\bigr)\, \d x}\, \d t 
	\leq  C_\varphi\vartheta^\alpha,
\end{align*}
for some $\alpha\in (0,1)$ and a 
constant $C_\varphi$ independent of $n$. Then the sequence 
$\left\{\mathcal{L}(Q_n)\right\}_{n\in \N}$ 
of probability laws is tight on $L^{p_1}\bigl(L^{p_2}_w\bigr)$.
\end{lem} 

\begin{proof}
We will verify the tightness of the laws 
$\mu_n=\mathcal{L}(Q_n)$ on $L^{p_1}\bigl(L^{p_2}_w\bigr)$ 
by using Lemma \ref{lem:new-compact} to 
produce, for each $\kappa>0$, a compact 
set $\mathcal{K}_{\kappa}$ in $L^{p_1}\bigl(L^{p_2}_w\bigr)$ 
such that $\mu_n\bigl(\mathcal{K}_{\kappa}^c\bigr)
=\prob\bigl(X_n\in \mathcal{K}_{\kappa}^c\bigr)\le \kappa$, 
uniformly in $n$.

For arbitrary sequences $\left\{b_k\right\}_{k\in \N}$, 
$\left\{\nu_l\right\}_{l\in \N}$, 
$\left\{\vartheta_l\right\}_{l\in \N}$ of 
positive numbers, with $\nu_l,\vartheta_l\to 0$ 
as $l\to\infty$, and an arbitrary function sequence 
$\left\{\varphi_k\right\}_{k\in \N}$ that is 
dense in $C^\infty(\T)$ (for the 
uniform topology), we introduce the set
\begin{align*}
	K_a & = \bigcap_{k=1}^\infty
	\Biggl\{Q\in L^{p_1}\bigl(L^{p_2}_w\bigr): 
	\norm{Q}_{L^{p_1}([0,T];L^{p_2}(\T))}+ 
	\norm{Q}_{L^{\bar{p}_1}([0,T];L^1(\T))}
	\\ & \qquad\quad
	+\sup_{l\in \N} \frac{1}{\nu_l}
	\sup_{\tau\in (0,\vartheta_l)}
	\int_0^{T-\tau} \abs{\, \int_{\T} 
	\varphi_k(x)\bigl(Q(t+\tau,x)-Q(t,x)\bigr)\, \d x}\, \d t
	\leq a\, b_k\Biggr\},
\end{align*}
for $a>0$. By Lemma \ref{lem:new-compact}, the 
set $K_a$ is relatively compact in $L^{p_1}\bigl(L^{p_2}_w\bigr)$, 
for each $a>0$. By the Chebyshev inequality and the assumptions 
(i), (ii) and (iii),
\begin{multline*}
	\mathbb{P}\bigl( Q_n \in \mathcal{K}_a^c\bigr)
	\leq\sum_{k=1}^\infty{\frac3{\, a\,b_k}}\Ex 
	\norm{Q_n}_{L^{p_1}([0,T];L^{p_2}(\T))}
	+\sum_{k=1}^\infty\frac{3}{\, a\, b_k}
	\Ex\norm{Q_n}_{L^{\bar{p}_1}([0,T];L^1(\T))}
	\\  
	+\sum_{k,l=1}^\infty\frac3{\, a\, b_k \, \nu_l}
	\Ex \sup_{\tau\in (0,\vartheta_l)}
	\int_0^{T-\tau} \abs{\,\int_{\T} 
	\varphi_k\bigl(Q_n(t+\tau,x)-Q_n(t,x)\bigr)\, \d x}\, \d t
	\\ 
	\leq \frac{C_1}{a}\sum_{k=1}^\infty\frac{1}{b_k}
	+\frac{1}{a}\left(\, \sum_{k=1}^\infty 
	\frac{C_{\varphi_k}}{b_k}\right)
	\left(\, \sum_{l=1}^\infty 
	\frac{\vartheta_l^\alpha}{\nu_l}\right),
\end{multline*}
for some $n$-independent constant $C_1$. 

Particularising $b_k=2^{k+1}\max\left(C_1,C_{\varphi_k}\right)$, 
and $\tau_l,\vartheta_l\to 0$ with
$\nu_l=\vartheta_l^\alpha\, 2^l$, we obtain 
$\mathbb{P}\bigl(Q_n \in \mathcal{K}_a^c\bigr)<\frac{1}{a}$, 
which can be made $\leq \kappa$ by taking $a$ large. 
As a result, we can specify the required compact 
$\mathcal{K}_\kappa$ as the closure of $K_a$, 
for some $a=a(\kappa)$, such that 
$\mu_n\bigl(\mathcal{K}_\kappa^c\bigr)\leq \kappa$.
\end{proof}

\subsection{Tightness and a.s.~representations}

We are now in a position to verify 
the crucial tightness property of $X_n$.

\begin{lem}[tightness]\label{thm:jointlaw_tightness}
Consider the random variables 
$X_n:\bigl(\Omega,\mathcal{F},
\mathbb{P}\bigr) \to \bigl(\mathcal{X},
\mathcal{B}_{\mathcal{X}} \bigr)$ defined by 
\eqref{eq:Xn-def} and 
\eqref{eq:random-variables-I}, 
\eqref{eq:random-variables-II}, 
\eqref{eq:random-variables-III}, 
\eqref{eq:u0-approx}, 
\eqref{eq:pathspaces}, 
\eqref{eq:pathspaces-joint}. 
Then the sequence $\left\{\mu_n=\mathcal{L}(X_n)\right\}_{n\in \N}$ 
of joint laws is tight, as probability measures 
defined on $\bigl(\mathcal{X},\mathcal{B}_{\mathcal{X}} \bigr)$.
\end{lem}

\begin{proof}
For each $\kappa>0$, we must produce a compact 
set $\mathcal K_{\kappa}\subset \mathcal{X}$ such that
\begin{equation}\label{eq:tight-qpolish}
	\mu_n\bigl(\mathcal K_{\kappa}\bigr)
	>1-\kappa \quad 
	\Longleftrightarrow \quad
	\mu_n\bigl(\mathcal K_{\kappa}^c\bigr)
	=\prob\bigl(X_n\in \mathcal K_{\kappa}^c\bigr)
	\le \kappa,
\end{equation}
uniformly in $n$. By Tychonoff's theorem, 
the tightness of the joint laws on $\mathcal{X}$ 
follows from the tightness of the product measures 
$\bigotimes_{l\in \N} \mu_n^{(l)}$ on $\mathcal{X}$ (with the 
product $\sigma$-algebra). In other words, to 
prove \eqref{eq:tight-qpolish} it is sufficient to find compact sets 
$\mathcal{K}_{l,\kappa}\subset \mathcal{X}_l$ such that 
$\mu^{(l)}_n\bigl(\mathcal{K}_{l,\kappa}^c\bigr)\leq \kappa$, 
for arbitrary $\kappa>0$, for each $l\in \N$.

As most of the cases can be treated similarly, 
we will carry out the tightness analysis of $\mu^\xi_n$ only for 
$$
\xi=u,\,q,\, q^2, \, u_0,\, W,\, S_\ell(q_+),
\, S_\ell(q_+)'q, \, S_\ell(q_+)q,
$$
thereby making up for each path space 
in \eqref{eq:pathspaces} at least once, 
see also \eqref{eq:random-variables-I}, 
\eqref{eq:random-variables-II}, 
\eqref{eq:random-variables-III}, 
and \eqref{eq:u0-approx}.

First, we verify the tightness 
on $\mathcal{X}_u=C_tL^2_x$ of the laws 
$\mu^u_n$ of $u_{\eps_n}$ using the $L^{p_0}_\omega L^\infty_t 
H^1_x$ bound \eqref{eq:energybound} and 
the H\"older continuity estimate \eqref{eq:Holder-viscous}. 
It is enough to verify tightness 
via relative compactness. For $a>0$, set
\begin{align}
	K(a) & = \Bigl\{f \in C([0,T];L^2(\T))
	\bigcap L^\infty([0,T];H^1(\T)):
	\notag \\ & \qquad\qquad 
	\norm{f}_{L^\infty([0,T];H^1(\T))} 
	+\norm{f}_{C^\theta([0,T];L^2(\T))} \leq a 
	\Bigr\},
	\label{eq:Ka-def}
\end{align}
where $\theta \in (0,1/4)$ is fixed (and constrained 
by (ii) of Theorem \ref{thm:bounds1}). By 
the Arzel\`a--Ascoli theorem \cite[Lemma 1]{Simon:1987vn}, 
$K(a)$ is relatively compact in $C([0,T];L^2(\T))$. 
By the Chebyshev inequality,
\begin{align*}
	\mathbb{P}\bigl( u_{\eps_n} \in 
	K(a)^c\bigr)
	& < \frac{1}{a}\Ex 
	\norm{u_{\eps_n}}_{L^\infty([0,T];H^1(\T))}
	\\ & \qquad +\frac{1}{a}
	\Ex \norm{u_{\eps_n}}_{C^\theta([0,T];L^2(\T))}
	\overset{\eqref{eq:energybound}, 
	\eqref{eq:Holder-viscous}}{\lesssim} 
	\frac{1}{a},
\end{align*}
which can be made $\leq \kappa$ 
by taking $a$ large. Hence, we can specify 
the required compact $\mathcal{K}_{u,\kappa}$ 
as the closure of $K(a)$, for some $a=a(\kappa)$, 
such that $\mu^u_n\bigl(\mathcal{K}_{u,\kappa}^c\bigr)
\leq \kappa$.

Second, we consider $\mu^\xi_n$, 
$\xi=S_\ell(q_+)q$, $\ell\in \N$. 
By \eqref{eq:entropies_S_ell}, $\abs{S_\ell(v_+)v} 
\lesssim_\ell \abs{v}^2$. In view of 
Proposition \ref{thm:2plusalpha_apriori},
\begin{equation}\label{eq:F-pm-i-l-est}
	\Ex \norm{S_\ell\bigl((q_{\ep_n})_+
	\bigr)q_{\ep_n}}_{L^{r}([0,T]\times \T)}^{r}
	\lesssim 1, \quad \ell\in \N,
\end{equation}
where the integrability index $r$ 
appears in \eqref{eq:pathspaces}. 
Let $a$ be a positive number and 
consider the set
\begin{align*}
	K(a) = \Bigl\{f \in L^{r}([0,T]\times \T): 
	\norm{f}_{L^{r}([0,T]\times \T))}^{r}\leq a\Bigr\}.
\end{align*}
By the Banach--Alaoglu theorem and 
reflexivity of $L^{r}_{t,x}$, $K(a)$ 
is a compact subset of $\mathcal{X}_\xi
=L^{r}_{t,x}-w$, $\xi=S_\ell(q_+)q$, 
cf.~\eqref{eq:pathspaces}. 
By Chebyshev's inequality and \eqref{eq:F-pm-i-l-est},
\begin{align*}
	\mathbb{P}\Bigl(S_\ell\bigl((q_{\ep_n})_+\bigr)
	q_{\ep_n} \in K(a)^c\Bigr)<\frac{1}{a}\Ex 
	\norm{S_\ell\bigl((q_{\ep_n})_+
	\bigr)q_{\ep_n})}_{L^{r}([0,T]\times \T)}^{r} 
	\lesssim \frac{1}{a},
\end{align*}
which can be made $\leq \kappa$ for large $a$.  
Thus, we pick $K(a)$, for some $a=a(\kappa)$, as 
the required compact $\mathcal{K}_{\xi,\kappa}$ for which 
$\mu^\xi_n\bigl(\mathcal{K}_{\xi,\kappa}^c\bigr)
\leq \kappa$, for $\xi=S_\ell(q_+)q$, $\ell\in \N$.  
Similarly, we can construct a compact subset
$\mathcal{K}_{q,\kappa}$ of $\mathcal{X}_q
=L^{2r}_{t,x}-w$ such that 
$\mu^q_n\bigl(\mathcal{K}_{q,\kappa}^c\bigr)
\leq \kappa$. 

Since the law of $W$ is tight as a 
Radon measure on the (Polish) space $C([0,T])$, 
there is a compact subset $\mathcal{K}_{W,\kappa}$ 
of $C([0,T])$ such that $\mu^W_n
\bigl(\mathcal{K}_{W,\kappa}^c\bigr)\leq \kappa$.

By the hypothesis \eqref{eq:u0-approx}, 
$\Ex \norm{z_n - u_0}_{H^1(\T)}^{p_0} \ton 0$. 
Therefore, by Chebyshev's inequality, we deduce 
the tightness of the laws of $z_n$, 
that is, there exists a compact set $\mathcal{K}_{u_0,\kappa}$ 
in the space $H^1(\T)$ such that $\mu^{u_0}_n
\bigl(\mathcal{K}_{u_0,\kappa}^c\bigr)\leq \kappa$.

Next, let us consider the tightness of $\mu^{q^2}_n$ on 
the ``strong in $t$ and weak in $x$" path 
space $\mathcal{X}_{q^2}=L^r\bigl(L^r_w\bigr)$. 
We will apply Lemma \ref{lem:tightness-LpLpx-weak} to $Q_n=q_{\ep_n}^2$ 
with $p_1=p_2=r$, recalling that $r<3/2$ is fixed 
in \eqref{eq:pathspaces}. Note that the first condition (i) 
of the lemma is satisfied by the higher integrability 
estimate \eqref{eq:higher-int}. To verify 
(ii), we use the estimate \eqref{eq:energybound} as follows:
\begin{multline*}
	\left(\Ex \norm{Q_n}_{L^p([0,T];L^1(\T))}\right)^p
	=\Ex \int_0^T \norm{Q_n(t)}_{L^1(\T)}^p\, \d t
	\\ = \Ex \int_0^T \norm{q_{\ep_n}(t)}_{L^2(\T)}^{2p}\, \d t
	\lesssim_T 
	\Ex \norm{q_{\ep_n}}_{L^\infty([0,T];L^2(\T))}^{2p}
	\lesssim 1,	
\end{multline*}
for any $p\in \bigl[1,p_0/2\bigr]$, where $p_0>4$ is 
fixed in Theorem \ref{thm:main}. The final condition 
(iii) is satisfied by Proposition \ref{prop:translation-time} 
with $S(v)=v^2$. 

Similarly, applying Lemma \ref{lem:tightness-LpLpx-weak} 
to $Q_n=S_\ell\bigl((q_{\ep_n})_+\bigr)$ with $p_1=p_2=2r$, 
we deduce the tightness 
of $\mu^{S_\ell(q_+)}_n$ on $\mathcal{X}_{S_\ell(q_+)}
=L^{2r}\bigl(L^{2r}_w\bigr)$. The first condition (i) 
is satisfied by the higher integrability 
estimate \eqref{eq:higher-int}, recalling that 
$\abs{S_\ell\bigl((q_{\ep_n})_+\bigr)}
\lesssim_\ell \abs{q_{\ep_n}}$, 
cf.~\eqref{eq:entropies_S_ell}.
To verify (ii), note that
\begin{align*}
	\left(\Ex \norm{Q_n}_{L^p([0,T];L^1(\T))}\right)^p
	& \lesssim_\ell 
	\Ex \int_0^T \norm{q_{\ep_n}(t)}_{L^1(\T)}^{p}\, \d t
	\\ & 
	\lesssim_{T} 
	\Ex \norm{q_{\ep_n}}_{L^\infty([0,T];L^2(\T))}^{p}
	\lesssim 1,	
\end{align*}
for any $p\in [1,p_0]$ (keep in mind that $p_0>4$ and $2r<3$). 
The condition (iii) is satisfied by 
Proposition \ref{prop:translation-time}, 
which can be applied because $v\mapsto S_\ell(v_+)
\in W^{2,\infty}_{\loc}(\R)$, see 
\eqref{eq:entrop2}. Likewise, we can 
apply Lemma \ref{lem:tightness-LpLpx-weak} 
to $Q_n=S_\ell\bigl((q_{\ep_n})_+\bigr)'q_{\ep_n}$ 
(still with $p_1=p_2=2r$), to deduce the tightness 
of $\mu^{S_\ell(q_+)'q}_n$ on $\mathcal{X}_{S_\ell(q_+)'q}
=L^{2r}\bigl(L^{2r}_w\bigr)$. Here, note carefully that 
Proposition \ref{prop:translation-time} applies 
owing to the fact that the map $v\mapsto S_\ell(v_+)'v$ 
belongs to $W^{2,\infty}_{\loc}(\R)$, 
see \eqref{eq:beta-def}.
\end{proof}

Given Lemma \ref{thm:jointlaw_tightness} (tightness), the 
following theorem is an immediate consequence 
of the main result of Jakubowski \cite{Jakubowski:1997aa}, recalled 
as Theorem \ref{thm:Jakubowski} in the appendix. 
We refer to \cite{Brzezniak:2013aa,Brzezniak:2011aa,
Brzezniak:2013ab,Ondrejat:2010aa} for the first 
applications of the Jakubowski theorem to SPDEs. We rely on the 
Jakubowski version of Skorokhod's 
representation theorem because of the non-metrisable weak 
topologies in \eqref{eq:pathspaces-joint}. 

\begin{prop}[Skorokhod--Jakubowski representations]\label{thm:skorohod}
Fix a sequence $\{\ep_n\}$ of positive numbers 
with $\ep_n \to 0$ as $n\to \infty$, and 
consider the corresponding strong $H^m$ 
solutions $u_{\ep_n}$ of the viscous 
SPDE \eqref{eq:u_ch_ep} with initial data 
$u_{\ep_n}(0)=z_n$, cf.~\eqref{eq:u0-approx}. 
Denote the spatial gradient by $q_{\ep_n}=\pd_x u_{\ep_n}$. 
Consider the random mappings 
$X_n:\bigl(\Omega,\mathcal{F},
\mathbb{P}\bigr) \to \bigl(\mathcal{X},
\mathcal{B}_{\mathcal{X}} \bigr)$ defined by 
\eqref{eq:Xn-def} and 
\eqref{eq:random-variables-I}, 
\eqref{eq:random-variables-II}, 
\eqref{eq:random-variables-III}, 
\eqref{eq:u0-approx}, 
\eqref{eq:pathspaces}, 
\eqref{eq:pathspaces-joint}. 
There exist a new probability space 
$\bigl(\tilde{\Omega},\tilde{\mathcal{F}},
\tilde{\mathbb{P}}\bigr)$ and 
$\mathcal{X}$-valued random variables
\begin{equation}\label{eq:tilde-Xn}
	\begin{split}
		& \tilde X_n=
		\left(
		\tilde u_n,
		\tilde F_n^{q},
		\tilde F_n^{q^2},
		\tilde W_n, \tilde u_{0,n},
		\tilde F_n^{\mathbb{S}}
		\right),
		\quad
		\tilde X=
		\left(
		\tilde u,
		\overline{\tilde F^{q}},
		\overline{\tilde F^{q^2}},
		\tilde W, \tilde u_0,
		\overline{\tilde F^{\mathbb{S}} \,}
		\right),
	\end{split}
\end{equation}
defined on $\bigl(\tilde{\Omega},\tilde{\mathcal{F}},
\tilde{\mathbb{P}}\bigr)$, such that along a 
subsequence (notationally not relabelled) 
the joint laws of $X_n$ and $\tilde X_n$ coincide 
for all $n$, and $\tilde X_n \ton \tilde X$ 
almost surely, in the product topology on $\mathcal{X}$. 
In the first line of \eqref{eq:tilde-Xn}, we have
\begin{equation}\label{eq:SJ-nonlinear}
	\begin{split}
		& \tilde F_{\ep_n}^{\xi_\ell,\pm}
		=\xi_\ell\big|_{v=\tilde q_{\ep_n}},
		\quad \xi_\ell \in \mathbb{S}_\ell^\pm, \quad \ell\in \N,
		\\ &
		\tilde F_{\ep_n}^{\mathbb{S}}
		=\left\{
		\left\{
		\tilde F_{\ep_n}^{\xi_\ell,+},
		\, \, \xi_\ell\in \mathbb{S}_\ell^+
		\right\}_{\ell\in \N},
		\, \,
		\left\{
		\tilde F_{\ep_n}^{\xi_\ell,-},
		\, \, \xi_\ell\in \mathbb{S}_\ell^-
		\right\}_{\ell\in \N}
		\right\},
	\end{split}
\end{equation}
where $\mathbb{S}_\ell^\pm$ denote 
the collections of nonlinearities 
given by \eqref{eq:nonlinear-collection}.
In the second line of \eqref{eq:tilde-Xn}, 
the ``overline" should be understood as sitting 
over each component of the overlined quantity; 
for example, $\overline{\tilde F^{q}}=
\left(\, \tilde q,\overline{\tilde q_+},
\overline{\tilde q_-}\, \right)$ and 
$\overline{\tilde F^{q^2}}=\left(\, \overline{\tilde q^2},
\overline{\tilde q_+^2},\overline{\tilde q_-^2}\right)$, 
with $\overline{\tilde q}=\tilde q$. 
More explicitly, we have the 
following $\tilde{\mathbb{P}}$--a.s.~convergences: 
\begin{equation}\label{eq:weak-conv-tilde}
	\begin{split}
		&\tilde u_n \ton \tilde u \, \, \text{in $C_tL^2_x$}, 
		\quad 
		\tilde W_n \ton \tilde W \,\,
		\text{in $C_t$}, \,\, 
		\tilde u_{0,n} \ton \tilde u_0 \,\,
		\text{in $H^1_x$},
		\\ & 
		\tilde q_n \tonweak \tilde q, 
		\,\, \bigl(\tilde q_n\bigr)_\pm \tonweak 
		\overline{\tilde q_\pm} \,\, 
		\text{in $L^{2r}_{t,x}$},
		\\ & 
		\tilde q_n^2\ton \overline{\tilde q^2} 
		\,\, 
		\text{in $L^{r}\bigl(L^r_w\bigr)$},
		\quad
		(\tilde q_n)_\pm^2
		\tonweak \overline{{\tilde q_\pm}^2} 
		\,\, \text{in $L^{r}\bigl(L^r_w\bigr)$},
		\\ & 
		S_\ell\bigl((\tilde q_n)_\pm\bigr)
		\ton  
		\overline{S_\ell\bigl(\tilde q_\pm\bigr)}
		\,\,
		\text{in $L^{2r}(L^{2r}_w)$}, \quad \ell\in \N,
		\\ &
		S_\ell\bigl((\tilde q_n)_\pm\bigr)' \tilde q_n
		\ton  
		\overline{S_\ell\bigl(\tilde q_\pm\bigr)'\tilde q}
		\, \,
		\text{in $L^{2r}(L^{2r}_w)$}, \quad \ell\in \N,
		\\ &
		S_\ell\bigl((\tilde q_n)_\pm\bigr)''\tilde q_n^2
		\tonweak 
		\overline{S_\ell\bigl(\tilde q_\pm\bigr)''\tilde q^2} \,\,
		\text{in $L^r_{t,x}$}, \quad \ell\in \N,
		\\ &
		S_\ell\bigl((\tilde q_n)_\pm\bigr)\tilde q_n
		\tonweak 
		\overline{S_\ell\bigl(\tilde q_\pm\bigr)\tilde q} \,\,
		\text{in $L^r_{t,x}$}, \quad \ell\in \N,
		\\ &
		S_\ell\bigl((\tilde q_n)_\pm\bigr)'\tilde q_n^2
		\tonweak 
		\overline{S_\ell\bigl(\tilde q_\pm\bigr)'\tilde q^2} \,\,
		\text{in $L^r_{t,x}$},
		\quad \ell\in \N,
		\\ & 
		S_\ell\bigl((\tilde q_n)_\pm\bigr)'
		\tonweak 
		\overline{S_\ell\bigl(\tilde q_\pm\bigr)'}
		\,\,
		\text{in $L^{2r}_{t,x}$}, \quad \ell\in \N.
	\end{split}
\end{equation}
\end{prop}

\begin{proof}
An application of Theorem \ref{thm:Jakubowski} 
supplies all the claims of the proposition, except 
the one that the nonlinear composition 
variables take the the explicit 
form \eqref{eq:SJ-nonlinear}. However, this 
follows from Lemma \ref{lem:SJ-nonlinear}.
\end{proof}

\begin{rem}\label{rem:weak-notation}
As we shall henceforth be working 
in the new probability space, for brevity, we 
drop the tilde under the overline indicating a weak limit. 
For example,  $\overline{\tilde F^{q}}=
\left(\, \tilde q,\overline{q_+},
\overline{q_-}\, \right)$ and 
$\overline{\tilde F^{q^2}}=\left(\, \overline{q^2},
\overline{q_+^2},\overline{q_-^2}\right)$. 
Similarly, instead of $\overline{S_\ell(\tilde q_\pm)}$, 
we write $\overline{S_\ell(q_\pm)}$, and so forth 
with the other nonlinear compositions.
\end{rem}

\section{Properties of a.s.~representations}
\label{sec:equality-of-laws}
The strong $H^m$ solution $u_{\ep_n}$ of the 
SPDE \eqref{eq:u_ch_ep} possesses several consequential 
bounds, see Theorem \ref{thm:bounds1}, Lemma \ref{thm:P-u2_bound} 
and Proposition \ref{thm:2plusalpha_apriori}. 
In this section, we wish to transfer these  
bounds to $\tilde u_n$ (the Skorokhod--Jakubowski 
representation from Proposition \ref{thm:skorohod}). 
At the moment, we do not have the SPDE satisfied by 
$\tilde u_n$, so we cannot derive them as before. 
Instead, as is often done in the literature, we will 
appeal to the fact that $u_{\ep_n}$ and $\tilde u_n$ share 
the same probability law and invoke the Kuratowski--Lusin--Souslin 
(KLS) theorem. We refer to \cite[Corollary A.2]{Ondrejat:2010aa} and 
\cite[Proposition C.2]{Brzezniak:2013ab} for 
the quasi-Polish version of this theorem (cf.~Lemma 
\ref{thm:LS4qP}). The KLS theorem allows 
one to assert that spaces of higher integrability/regularity 
are Borel subsets of the postulated path spaces \eqref{eq:pathspaces}. 
The law shared by $u_{\ep_n}$ and $\tilde u_n$ can then 
be integrated against over these better function spaces to derive 
bounds for $\tilde u_n$ from those of $u_{\ep_n}$. 

\begin{lem}[spatial gradient]\label{thm:lim_iden_simple}
Let $\tilde{u}_n,\tilde{q}_n,\tilde{u},\tilde{q}$ 
be the Skorokhod--Jakubowski representations 
from Proposition \ref{thm:skorohod}. 
There is an event $\tilde{\Omega}_0$,  with
$\tilde{\mathbb{P}}(\tilde{\Omega}_0)=1$,
such that for any $\tilde{\omega}\in
\tilde{\Omega}_0$ there exist 
sets $E_n\bigl(\tilde{\omega}\bigr), 
E\bigl(\tilde{\omega}\bigr)\subset [0,T]\times \T$ 
of full measure on which the weak spatial 
derivatives of $\tilde{u}_n$, $\tilde{u}$ 
are $\tilde{q}_n$, $\tilde{q}$, respectively, i.e., 
for $\tilde \omega\in \tilde \Omega_0$,
\begin{equation}\label{eq:spatial-grad}
	\begin{split}
		& \pd_x \tilde u_n(\tilde \omega,t,x) 
		=\tilde q_n(\tilde \omega,t,x)
		\,\,\, \text{for a.e.~$(t,x)\in E_n(\tilde \omega)$},
		\\ & \pd_x \tilde u(\tilde \omega,t,x) 
		=\tilde q(\tilde \omega,t,x)
		\,\,\, \text{for a.e.~$(t,x)\in E(\tilde \omega)$}.
	\end{split}
\end{equation}
\end{lem}

\begin{proof}
We first show that $\tilde{\mathbb{P}}$--a.s., 
for every $\psi \in C([0,T];C^1(\T))$, 
\begin{equation}\label{eq:spatial-grad-tmp1}
	-\int_0^T\int_\T \tilde{u}_n \pd_x \psi \,\d x \,\d t 
	=\int_0^T \int_\T \tilde{q}_n \psi \,\d x\,\d t,
\end{equation}
which implies the first claim in \eqref{eq:spatial-grad}.
Let $\{\psi_j\}_{j = 1}^\infty$ be a countable dense subset 
of $C([0,T];C^1(\T))$, and consider the continuous 
mappings 
\begin{align*}
	& F_j: C([0,T];L^2(\T)) \times 
	\bigl(L^{2r}([0,T] \times \T)-w \bigr)\to \R, 
	\\ &
	F_j(u,q)=\int_0^T \int_\T u\, \pd_x \psi_j \, \d x \,\d t
	+\int_0^T \int_\T q\, \psi_j\,\d x \,d t.
\end{align*}
By continuity of $F_j$, Remark 
\ref{rem:cont-map-measurable}, and the equality 
of joint laws, 
\begin{align*}
	\tilde{\mathbb{P}}\left(\left\{F_j\bigl(\tilde{u}_n,
	\tilde{q}_n\bigr)=0\right\}\right) 
	=\mathbb{P}\left(\left\{F_j\bigl(u_{\ep_n},
	q_{\ep_n}\bigr) = 0\right\}\right) = 1.
\end{align*}
Since there are countably many pairs $(n,j)$, 
there is a set $\tilde \Omega_0$ of full 
$\tilde{\mathbb{P}}$--measure such that 
$F_j\bigl(\tilde{u}_n(\tilde\omega),
\tilde{q}_n(\tilde \omega)\bigr)=0$ for all $(n,j)$, 
$\tilde \omega\in \tilde \Omega_0$. This 
implies \eqref{eq:spatial-grad-tmp1}.

Proposition \ref{thm:skorohod} gives the
$\tilde{\mathbb{P}}$--a.s.~convergences 
$\tilde{u}_n \ton \tilde{u}$ in $C([0,T];L^2(\T))$ and
$\tilde{q}_n\tonweak\tilde{q}$ in $L^{2r}([0,T]\times\T)$ jointly. 
For a fixed $j$, we obtain a.s.~that 
$\int_0^T \int_\T \tilde{q}_n \psi_j\,\d x\,\d t  
\ton \int_0^T \int_\T \tilde{q}\, \psi_j \,\d x\,\d t $. 
Similarly, we have the a.s.~convergence 
$\int_0^T \int_\T \tilde{u}_n\pd_x \psi_j \,\d x \,\d t 
\ton \int_0^T \int_\T \tilde{u}\, \pd_x \psi_j\,\d x \,\d t$. 
By density,  we thus arrive at
\begin{align*}
	\int_0^T \int_\T \psi\, \tilde{q}\, \d x \,\d t
	 =-\int_0^T \int_\T \pd_x \psi\, \tilde{u}\,\d x\,\d t, 
	 \quad \text{$\tilde{\mathbb{P}}$--a.s.},
\end{align*}
for any $\psi \in C([0,T] ; C^1(\T))$. 
This establishes the second claim 
in \eqref{eq:spatial-grad}.
\end{proof}

\begin{lem}[regularity]\label{thm:smoothness_n}
Let $\tilde{u}_n$, $\tilde{q}_n$ be the Skorokhod--Jakubowski 
representations from Proposition \ref{thm:skorohod}. 
Then $\tilde{u}_n \in L^2([0,T];H^m(\T))$, 
$\tilde{\mathbb{P}}$--a.s., for any $m\geq 1$, 
and thus $\tilde{q}_n(t,x)=\pd_x \tilde{u}_n(t,x)$ 
pointwise in $(t,x)$, $\tilde{\mathbb{P}}$--a.s.
\end{lem}

\begin{proof}
Recall that $u_{\ep_n}$ is the strong solution 
of the viscous SPDE \eqref{eq:u_ch_ep} with initial data 
$u_{\ep_n}(0)=z_n$, cf.~\eqref{eq:u0-approx}. 
By Definition \ref{def:wk_sol_visc}, $u_{\ep_n}$ 
belongs to $L^2([0,T];H^{m}(\T))$, a.s., for any $m$. 
Since the intersection $L^2([0,T];H^{m}(\T))\cap C([0,T];L^2(\T))$ 
injects continuously into the path space $C([0,T];L^2(\T))$,
cf.~\eqref{eq:pathspaces}, under the identity map, its 
image under the injection is Borel in $C([0,T];L^2(\T))$, 
according to the KLS theorem (cf.~Lemma \ref{thm:LS4qP}). 
Therefore, by the equality of laws, also the variable $\tilde{u}_n$ 
belongs to $L^2([0,T];H^{m}(\T))$, a.s., for any $m$. 
Since $\tilde{q}_n$ is the weak $x$-derivative 
of $\tilde{u}_n$ (cf.~Lemma \ref{thm:lim_iden_simple}), 
and we have the inclusion $H^m(\T) \hookrightarrow C^{m-1/2}(\T)$, 
this $x$-derivative is classical.
\end{proof}

\begin{lem}[a priori estimates]\label{thm:q_bounds_skorohod}
Let $p_0>4$ be as specified in Theorem \ref{thm:bounds1} 
and $r\in[1,3/2)$ as fixed in \eqref{eq:pathspaces}. 
Let $\tilde{u}_n,\tilde{q}_n$ be the Skorokhod--Jakubowski 
representations from Proposition \ref{thm:skorohod}. There exists a 
constant $C$, independent of $n$, such that 
\begin{align*}
	& \tilde{\Ex}\norm{\tilde u_n}_{L^\infty([0,T];H^1(\T))}^{p_0}
	\le C, \quad
	\tilde{\Ex}\norm{\tilde{q}_n}_{L^\infty([0,T];L^2(\T))}^{p_0}\leq C,
	\\ & \quad \text{and} \quad
	\tilde{\Ex}\norm{\tilde{q}_n}_{L^{2r}([0,T]\times\T)}^{2r}\leq C.
\end{align*}
\end{lem}

\begin{proof}
By the continuous injection of the Polish space 
$\mathcal{Y}=C([0,T];H^1(\T))$ into the path space 
$\mathcal{X}^u=C([0,T];L^2(\T))$, the KLS theorem ensures that 
$\mathcal{Y}$ is a Borel subset of $\mathcal{X}^u$, and 
thus the equality of laws implies the first estimate:
\begin{align*}
	\tilde{\Ex}\norm{\tilde{u}_n}_{L^\infty([0,T];H^1(\T))}^{p_0}
	& = \int_{\mathcal{Y}} \norm{v}_{L^\infty([0,T];H^1(\T))}^{p_0}
	\, \d \mu_n^{u}(v)
	\\ & =\Ex\norm{u_{\ep_n}}_{L^\infty([0,T];H^1(\T))}^{p_0}
	\overset{\eqref{eq:energybound}}{\lesssim} 1,
\end{align*}
recalling that $\mu_n^{u}$ denotes the 
law of $u_{\ep_n}$, cf.~Section \ref{sec:Jak-Skor}.

The second estimate is a consequence of the first 
and Lemmas \ref{thm:lim_iden_simple}, \ref{thm:smoothness_n}. 
Since the injection $L^{2r}([0,T]\times \T)
\hookrightarrow L^{2r}([0,T] \times \T)-w$ 
is continuous, we can use the KLS theorem on quasi-Polish spaces 
(Lemma \ref{thm:LS4qP}) and the equality of laws to deduce 
the third estimate: 
$\tilde{\Ex}\norm{\tilde{q}_n}_{L^{2r}([0,T]\times \T)}^{2r}
=\Ex\norm{q_{\ep_n}}_{L^{2r}([0,T]\times \T)}^{2r}
\overset{\eqref{eq:higher-int}}{\lesssim} 1$.
\end{proof}

By the a.s.~convergence \eqref{eq:weak-conv-tilde} 
and a weak compactness argument (in $\omega,t,x$), it follows 
that the limit $\tilde q=\pd_x \tilde u$ continues 
to satisfy the third estimate of Lemma \ref{thm:q_bounds_skorohod}. 
Because of non-reflexivity, the other ($L^\infty$) 
estimates are more delicate. We have the following result: 

\begin{lem}[a priori estimates for limits]\label{thm:u_in_Linfty}
Let $\tilde{u}$, $\tilde{q}$ be the a.s.~limits 
from Proposition \ref{thm:skorohod}, $p_0>4$ be as 
specified in Theorem \ref{thm:bounds1}, 
and $r\in [1,3/2)$ as fixed in \eqref{eq:pathspaces}. 
There exists a constant $C$ such that 
\begin{equation}\label{eq:tu-Lp-Linfty-H1}
	\tilde{\Ex}\norm{\tilde u}_{L^\infty([0,T];H^1(\T))}^{p_0}
	\leq C.
\end{equation}
Besides, $\tilde u\in C\bigl([0,T];H^1(\T)-w\bigr)$, 
$\tilde{\mathbb{P}}$--almost surely. Finally, 
\begin{equation}\label{eq:tu-Lp-Linfty-H1-new}
	\tilde{\Ex}\norm{\tilde{q}}_{L^\infty([0,T];L^2(\T))}^{p_0}\leq C,
	\quad
	\tilde{\Ex}\norm{\tilde{q}}_{L^{2r}([0,T]\times\T)}^{2r}\leq C.
\end{equation}
\end{lem}

\begin{proof}
The first part of \eqref{eq:tu-Lp-Linfty-H1-new} 
comes from \eqref{eq:tu-Lp-Linfty-H1} and 
Lemma \ref{thm:lim_iden_simple}. The second part 
is a corollary of the corresponding estimate 
in Lemma \ref{thm:q_bounds_skorohod} 
and the considerations given 
before Lemma \ref{thm:u_in_Linfty}.

The rest of the proof is devoted to 
\eqref{eq:tu-Lp-Linfty-H1} and the claim 
about weak time-continuity. By Lemma \ref{thm:q_bounds_skorohod}, 
$\tilde{\Ex}\norm{\tilde{u}_n}_{L^{\bar r}([0,T];H^1(\T))}^{p_0}
\leq CT^{p_0/\bar r}$, for any $\bar r\in [1,\infty)$. 
In other words, $\left\{\tilde{u}_n\right\}_{n\in \N}
\subset_b L^{p_0}\bigl(\tilde{\Omega};L^{\bar r}([0,T];H^1(\T))\bigr)$ 
for any finite $\bar r$. By standard duality theory in 
Lebesgue--Bochner spaces (see, e.g., 
\cite[page 98, Theorem 1]{Diestel:1977aa}), 
$$
\bigl(L^{p_0}\bigl(\tilde{\Omega};
L^{\bar r}([0,T];H^1(\T))\bigr)\bigr)^\ast 
=L^{p_0'}\bigl(\tilde{\Omega};
L^{\bar r'}([0,T];H^{-1}(\T))\bigr),
$$
for any $\bar r\in (1,\infty)$, 
$\bar r'=\frac{\bar r}{\bar r-1}$, $p_0'=\frac{p_0}{p_0-1}$. 
The space $L^{p_0}\bigl(\tilde{\Omega};L^{\bar r}([0,T];H^1(\T))\bigr)$ 
is reflexive. Thus, by Kakutani's theorem on reflexive spaces, 
up to subsequences,
\begin{align}\label{eq:kakutani}
	\tilde{u}_n \tonweak v^{(\bar r)}\quad
	\text{in $L^{p_0}\bigl(\tilde{\Omega};
	L^{\bar r}([0,T];H^1(\T))\bigr)$},
\end{align}
for $\bar r\in (1,\infty)$, where the limit $v^{(\bar r)}$ 
depends possibly on $\bar r$. Besides, 
\begin{equation*}
	\tilde{\Ex} \norm{v^{(\bar r)}}^{p_0}_{L^{\bar r}([0,T];H^1(\T))} 
	\leq \liminf_{n\to\infty} 
   \tilde{\Ex}\norm{\tilde{u}_n}^{p_0}_{L^{\bar r}([0,T];H^1(\T))}
   \leq C T^{p_0/\bar r}.
\end{equation*}
The continuity of the embedding  
$$
L^{p_0}\bigl(\tilde{\Omega};L^{r_2}([0,T];H^1(\T))\bigr) 
\hookrightarrow L^{p_0}\bigl(\tilde{\Omega};
L^{r_1}([0,T];H^1(\T))\bigr),
$$ 
for $1<r_1<r_2<\infty$, implies that $v^{(\bar r)}$ does not, in 
fact, depend on $\bar r$. Therefore, we will write $v$ instead of 
$v^{(\bar r)}$ in the following.

By the monotone convergence theorem,
$$
\tilde{\Ex} \lim_{\bar r\to\infty}
\norm{v}_{L^{\bar r}([0,T];H^1(\T))}^{p_0} \leq C.
$$
Since the $L^{\bar r}_t$ norm depends continuously on the 
index $\bar r$ for any measurable function $f:[0,T]\to H^1(\T)$ for which 
$\lim\limits_{\bar r\to\infty}\left(\int_0^T
\norm{f(t)}^{\bar r}_{H^1(\T)}\, ¨
\d t \right)^{1/\bar r}<\infty$, it follows that
\begin{equation}\label{eq:almost-there}
	\tilde{\Ex} \norm{v}_{L^\infty([0,T];H^1(\T))}^{p_0}\leq C
	\quad \text{and} \quad
	\text{$v \in L^\infty([0,T];H^1(\T))$, 
	$\tilde{\mathbb{P}}$-a.s.}
\end{equation}

It remains to identify $v$ with the 
$\tilde{\mathbb{P}}$--almost sure 
Skorokhod--Jakubowski limit $\tilde{u}$ 
in $C([0,T];L^2(\T))$, see \eqref{eq:weak-conv-tilde}. 
Consider the following test functions:
\begin{equation}\label{eq:testfunctions1}
	\phi(\tilde{\omega},t,x)
	=\psi(\tilde{\omega})\vartheta(t,x),
	\quad \psi\in L^\infty(\tilde{\Omega}), 
	\, \, \vartheta\in L^{\bar r'}([0,T];L^2(\T)).
\end{equation}
From \eqref{eq:kakutani},
$$
\tilde{\Ex}\bk{ \psi
\int_0^T \int_\T \vartheta(t,x) 
\bigl(\tilde{u}_n(t,x)-v(t,x)\bigr)\,\d x\,\d t}\ton 0.
$$
On the other hand, by \eqref{eq:weak-conv-tilde}, 
$$
\psi
\int_0^T \int_\T \vartheta(t,x)
\bigl(\tilde{u}_n(t,x)-\tilde{u}(t,x)\bigr)
\,\d x\,\d t\ton 0, \quad 
\text{$\tilde{\prob}$--a.s.}
$$
By Lemma \ref{thm:q_bounds_skorohod}, 
we have the moment bound
\begin{align*}
	&\tilde{\Ex}\abs{\psi 
	\int_0^T\int_\T 
	\tilde{u}_n(\tilde{\omega},t,x)\vartheta(t,x)
	\,\d x\,\d t}^{p_0}
	\\ & \quad \leq  
	\norm{\psi}_{L^\infty(\tilde{\Omega})}^{p_0}
	\norm{\vartheta}^{p_0}_{L^1([0,T];L^2(\T))}
	\tilde{\Ex}\norm{\tilde{u}_n}^{p_0}_{L^\infty([0,T];H^1(\T))}
	\leq C\bigl(\psi,\vartheta\bigr),
\end{align*}
and so, by Vitali's convergence theorem, 
$$
\tilde{\Ex}\abs{\psi\int_0^T \int_\T 
\vartheta(t,x) \bigl(\tilde{u}_n(t,x)
-\tilde{u}(t,x)\bigr) \,\d x\,\d t}^{p}
\ton 0,
$$
for any $1\leq p<p_{0}$. Consequently,
\begin{equation}\label{eq:almost-there2}
	\tilde{\Ex} \bk{\psi 	 
	\int_0^T \int_\T \vartheta(t,x) 
	\bigl(\tilde{u}(t,x)-v(t,x)\bigr)\,\d x\,\d t}=0,
\end{equation}
for $\psi$, $\vartheta$ as 
in \eqref{eq:testfunctions1}.

We use $I_z(\vartheta)$ as short-hand for
$\int_0^T\int_\T \vartheta(t,x) 
z(\cdot,t,x) \,\d x\,\d t$, where $z=\tilde u,v$. 
Clearly, by \eqref{eq:almost-there}, 
$I_v(\vartheta)\in L^p(\tilde{\Omega})$, for any
$1\leq p<p_{0}$ Since \eqref{eq:almost-there2} implies that 
$I_{\tilde u}(\vartheta)=I_v(\vartheta)$, almost surely, it 
follows that also $I_{\tilde u}(\vartheta)\in L^p(\tilde{\Omega})$, 
for each fixed $\vartheta$. We conclude that for 
any $\vartheta \in L^{\bar r'}([0,T];L^2(\T))$, 
with $1<\bar r'<\infty$, there exists 
a full $\tilde{\mathbb{P}}$--measure 
set $\tilde{\Omega}_\vartheta$ on which 
$I_{\tilde u}(\vartheta)=I_{v}(\vartheta)$. 
By separability of $L^{\bar r'}([0,T];L^2(\T))$, 
we deduce that for any $1<\bar r'<\infty$ 
there exists a full $\tilde{\mathbb{P}}$--measure 
set on which the identity 
$I_{\tilde u}(\vartheta)=I_{v}(\vartheta)$ 
holds for all $\vartheta \in L^{\bar r'}([0,T];L^2(\T))$. 
We can take this set to 
be the countable intersection of 
$\tilde{\Omega}_\vartheta$ associated 
with a countable dense subset of 
$\vartheta$ in $L^{r'}([0,T];L^2(\T))$. 
This shows that $\tilde{u} = v$, 
$\tilde{\mathbb{P}}\otimes \d t \otimes \d x$--almost 
everywhere. We also have \eqref{eq:almost-there} 
for $\tilde{u}=v$, thereby concluding the proof 
of \eqref{eq:tu-Lp-Linfty-H1}. 

Finally, let us prove the claim that $\tilde u$ is 
weakly time-continuous. By Lemma \ref{thm:jointlaw_tightness} 
(tightness), see also \eqref{eq:Ka-def}, for any $L\in\N$ 
there exists an $a_L>0$ such that $\inf\limits_{n\in \N}\prob
\bigl(\left\{u_{\ep_n}\in K(a_L)\right\}\bigr)>1-1/L$. 
Thus, by the equality of laws,
$$
\inf_{n\in \N}
\tilde{\prob}\bigl( 
\left\{\tilde{u}_n\in K(a_L)\right\}\bigr)>1-1/L.
$$
Pick an arbitrary subsequence $\{n_j\}_{j\in\N}$ and 
set $A_{j,L}=\left\{\tilde{u}_{n_j}\in K(a_L)\right\}$.
Then $\liminf_j\tilde{\prob}\left(A_{j,L}\right)>1-1/L$ and so
$$
\tilde{\prob}\left(\limsup_j A_{j,L}\right)>1-1/L,
$$
where $\limsup_j A_{j,L}=\bigcap_{J=1}^\infty\bigcup_{j>J} A_{j,L}$.
Introduce the $C_tL^2_x$ convergence set
$$
F=\left\{\tilde{\omega}\in\tilde{\Omega}: 
\tilde{u}_n(\tilde{\omega})
\ton \tilde{u}(\tilde{\omega})
\,\, \text{in $C_tL^2_x$}\right\}.
$$ 
By the first part of \eqref{eq:weak-conv-tilde}, 
$\tilde{\prob}(F)=1$. Therefore,
$$
\tilde{\prob}\left(F \cap \limsup_j A_{j,L}\right)>1-1/L.
$$
Select an arbitrary $\tilde{\omega}_0\in F\cap 
\limsup_j A_{j,L}$. By construction,
there is a subsequence $\left\{n_{j_k}\right\}_{k\in \N}$ 
(depending on $\tilde{\omega}_0$) such that 
$\tilde{u}_{n_{j_k}}(\tilde{\omega}_0)
\in K(a_L)$ for all $k\in \N$. 
Besides, we have $\tilde{u}_{n_{j}}(\tilde{\omega}_0)
\toj \tilde{u}(\tilde{\omega}_0)$ in $C_tL^2_x$. 
This implies that $\tilde{u}_{n_{j_k}}(\tilde{\omega}_0)
\tok \tilde{u}(\tilde{\omega}_0)$ in $C_tL^2_x$ and 
whence $\tilde{u}_{n_{j_k}}(\tilde{\omega}_0)\tok 
\tilde{u}(\tilde{\omega}_0)$ in $C_tH^1_x-w$. 
Since $\left\{n_j\right\}_{j\in \N}$ was arbitrary, 
and $K(a_L)$ is metrisable in $C_tH^1_x-w$, this leads 
us to conclude that 
$$
\tilde{u}_n(\tilde{\omega}_0)\ton 
\tilde{u}(\tilde{\omega}_0)
\,\, \text{in $C_tH^1_x-w$}
\quad \text{and} \quad
\tilde{u}(\tilde{\omega}_0)\in K(a_L). 
$$
In other words, the convergence set
$$
M_L=\left\{\tilde{\omega} \in \tilde{\Omega}: 
\tilde{u}_n(\tilde{\omega})\ton 
\tilde{u}(\tilde{\omega}) 
\,\, \text{in $C_tH^1_x-w$}, 
\,\,\tilde{u}(\tilde{\omega})\in K(a_L)\right\} 
$$
satisfies $M_L\supseteq F \cap \limsup_j A_{j,L}$, 
and so $\tilde{\prob}(M_L)>1 - 1/L$, for any $L\in\N$. 
Set 
$$
M=\left\{\tilde{\omega} \in \tilde{\Omega}: 
\tilde{u}_n(\tilde{\omega})\ton 
\tilde{u}(\tilde{\omega}) 
\,\, \text{in $C_tH^1_x-w$}\right\}.
$$
Then 
$M\supseteq M_L \supseteq 
F \cap \limsup_j A_{j,L} \in \tilde{\mathcal{F}}$, 
for any $L\in \N$, and therefore we obtain 
$M \supseteq \cup_{L=1}^\infty \bigl( 
F \cap \limsup_j A_{j,L}\bigr)
\in \tilde{\mathcal{F}}$. This implies that
$$
\tilde{\prob}\left(\cup_{L=1}^\infty \bigl( 
F \cap \limsup_j A_{j,L}\bigr)\right)
\geq 
\tilde{\prob}\left( 
F \cap \limsup_j A_{j,L}\right)
> 1-1/L,
$$
for each $L\in \N$, i.e., 
$\tilde{\prob}\left(\cup_{L=1}^\infty \bigl( 
F \cap \limsup_j A_{j,L}\bigr)\right)=1$. 
By the completeness of 
$\bigl(\tilde{\Omega},\tilde{\mathcal{F}},
\tilde{\mathbb{P}}\bigr)$, it follows that 
$M\in \tilde{\mathcal{F}}$ and  $\tilde{\mathbb{P}}(M)=1$; 
thus the second claim of the lemma 
follows: $\tilde u\in C_tH^1_x-w$, $\tilde{\prob}$--a.s.
\end{proof}

\begin{rem}
The use of the path space 
$\mathcal{X}_{u,\mathrm{new}}= C([0,T];H^1(\T)-w)$ 
instead of $\mathcal{X}_u = C([0,T];L^2(\T))$ could have 
simplified some of the work in Lemma \ref{thm:u_in_Linfty} and 
other places. However, we would 
have needed to provide additional steps to establish the 
tightness on $\mathcal{X}_{u,\mathrm{new}}$. 
It is also possible to operate with two different path spaces for $u$, 
each reflecting different topologies. 
However, this approach would necessitate 
additional steps to ensure that the two  
Skorokhod--Jakubowski representations of $u_n$ (and their limits) 
in these path spaces can be properly identified. 
While using $\mathcal{X}_{u,\mathrm{new}}$ 
instead of $\mathcal{X}_u$ might have 
simplified certain aspects of the analysis, there is 
a trade-off between simplicity and the 
added complexity in establishing the tightness and 
identifying the limits in the chosen path spaces.
Therefore, we have opted to use the space 
$\mathcal{X}_u$.
\end{rem}

The next result is a product of Lemmas \ref{thm:q_bounds_skorohod} 
and \ref{thm:u_in_Linfty}, see also Lemma \ref{thm:P-u2_bound}.

\begin{lem}[additional a priori estimates]
\label{thm:P-u2_bound-tilde}
Let $\tilde{u}_n$, $\tilde{u}$, and $\overline{q^2}$ 
be the Skorokhod--Jakubowski representations 
from Proposition \ref{thm:skorohod}, see also 
Remark \ref{rem:weak-notation}. There exists a 
constant $C$, independent of $n$, such that 
$\Ex \norm{\tilde u_n}_{L^\infty([0,T]\times \T)}^{p_0}\leq C$, 
where $p_0>4$ is fixed in Theorem \ref{thm:bounds1}, 
and $\Ex \norm{P[\tilde u_n]}_{L^\infty([0,T]\times \T)}^2\leq C$, 
where $P[\cdot]$ is defined in \eqref{eq:u_ch_ep}. 
In particular, we have
$$
\tilde{\Ex} \norm{K*\bk{\tilde u_n^2+\tfrac12
\tilde q_n^2}}_{L^\infty([0,T]\times \T)}^p \le C,
\quad p\in \bigl[1,p_0/2\bigr].
$$
The same bounds hold with $\tilde u_n$ 
replaced by its a.s.~limit $\tilde u$. 
\end{lem}

The final lemma of this section collects some 
integrability estimates (in $\tilde \omega,t,x)$ 
for the a.s.~weak limit $\overline{q^2}$. The 
second estimate will play a role in upcoming discussions 
about the martingale property of 
a stochastic integral and the weak convergence of 
of some specific product terms.

\begin{lem}[additional a priori estimates for limits]
\label{thm:P-u2_weak}
Let $\overline{q^2}= \overline{q^2}(\tilde \omega,t,x)$ 
be the Skorokhod--Jakubowski representation 
from Proposition \ref{thm:skorohod}. There is a 
constant $C$ such that 
\begin{equation}\label{eq:Lr-HnegEst}
	\tilde{\Ex}\int_0^T\int_\T
	\abs{\overline{q^2}}^r\, \d x\, \d t\leq C, 
	\quad 
	\tilde{\Ex}\int_0^T 
	\norm{\overline{q^2}(t)}_{H^{-1}(\T)}^p\,\d t \leq C, 
\end{equation}
for any $p\in \bigl[1,p_0/2\bigr]$, where $p_0>4$ and $r\in [1,3/2)$ are 
specified in Theorem \ref{thm:bounds1} and 
\eqref{eq:pathspaces}, respectively.
Furthermore, 
\begin{equation}\label{eq:Kstar-weak-q2}
	\tilde{\Ex} \int_0^T\norm{K*\bk{\tilde{u}^2+\tfrac12
	\overline{q^2}}(t)}_{L^\infty(\T)}^p\, \d t \leq C,
	\quad p\in \bigl[1,p_0/2\bigr].
\end{equation}
\end{lem}

\begin{proof}
By Lemma \ref{thm:q_bounds_skorohod}, 
$\left\{\tilde q_n^2\right\}_{n\in \N}$ is uniformly bounded in 
$L^r(\tilde \Omega\times [0,T]\times\T)$. Thus, by a weak 
compactness argument and \eqref{eq:weak-conv-tilde}, 
we may assume that
\begin{equation}\label{eq:Hneg-bound-oq2}
	\overline{q^2}\in L^r(\tilde \Omega \times [0,T]\times \T)
	\quad \text{and} \quad
	\tilde q_n^2\tonweak \overline{q^2} 
	\quad \text{in $L^r(\tilde \Omega\times [0,T]\times\T)$},
\end{equation}
which implies that the first part of \eqref{eq:Lr-HnegEst} holds.

Next, by \eqref{eq:weak-conv-tilde}, 
$\tilde q_n^2 \ton \overline{q^2}$ 
in $L^r\bigl(L^r_w\bigr)$ a.s. Since $L^r_w(\T) 
\hookrightarrow H^{-1}(\T)$, this also implies 
the convergence 
\begin{equation}\label{eq:Hneg-conv}
	\tilde q_n^2 \ton \overline{q^2} 
	\quad \text{in $L^r_t\bigl(H^{-1}_x\bigr)
	=L^r([0,T];H^{-1}(\T))$, a.s.}  
\end{equation}
We can use Lemma \ref{thm:q_bounds_skorohod} 
to deduce the $n$-uniform bound
\begin{align}
	& \norm{\tilde q_n^2}_{L^p(\tilde \Omega 
	\times [0,T];H^{-1}(\T))}^p
	=\tilde \Ex \int_0^T \norm{\tilde q_n^2(t)}_{H^{-1}(\T)}^p\, \d t
	\label{eq:Hneg-bound-qn2}
	\\ & \qquad
	=\tilde \Ex \int_0^T \sup_\varphi 
	\abs{\int_{\T} \tilde q_n^2\, \varphi
	\, \d x}^p \, \d t
	\lesssim_T \tilde \Ex 
	\norm{\tilde q_n}_{L^\infty([0,T];L^2(\T))}^{2p}\, \d t
	\lesssim 1,
	\notag
\end{align}
for any $p\in \bigl[1,p_0/2\bigr]$. Here, the supremum 
runs over all $\varphi \in H^1(\T)$ 
for which $\norm{\varphi}_{H^1(\T)}\leq 1$. Since 
$H^1(\T)\hookrightarrow L^\infty(\T)$, we have 
used that $\norm{\varphi}_{L^\infty(\T)}\lesssim 1$ 
for such functions. In other words, the sequence 
$\left\{\tilde q_n^2\right\}_{n\in \N}$ is bounded in 
the reflexive Banach space
$$
L^p_{\tilde \omega,t}\bigl(H^{-1}_x\bigr)
=L^p\bigl(\tilde \Omega\times [0,T];H^{-1}(\T)\bigr).
$$
Hence, by a weak compactness argument 
and \eqref{eq:weak-conv-tilde}, 
we may assume that 
\begin{equation}\label{eq:Hneg-bound-oq3}
	\overline{q^2}\in L^p_{\tilde \omega,t}\bigl(H^{-1}_x\bigr)
	\quad \text{and} \quad
	\tilde q_n^2\tonweak \overline{q^2} 
	\quad \text{in $L^p_{\tilde \omega,t}
	\bigl(H^{-1}_x\bigr)$}, \quad 
	p\in \bigl[1,p_0/2\bigr],
\end{equation}
This implies the last part of \eqref{eq:Lr-HnegEst}.

Finally, we prove \eqref{eq:Kstar-weak-q2}. In view 
of Lemma \ref{thm:P-u2_bound-tilde}, it is enough 
to consider the $\overline{q^2}$ part of \eqref{eq:Kstar-weak-q2}.
Noting that
$$
\abs{K*\overline{q^2}(t,x)}
\leq \norm{K(x-\cdot)}_{H^1(\T)}
\norm{\overline{q^2}(t)}_{H^{-1}(\T)}
\lesssim \norm{\overline{q^2}(t)}_{H^{-1}(\T)},
$$
we arrive at
\begin{align*}
	\tilde{\Ex} \int_0^T 
	\norm{K*\overline{q^2}(t)}_{L^\infty(\T)}^p\, \d t
	\lesssim
	\tilde{\Ex} \int_0^T 
	\norm{\overline{q^2}(t)}_{H^{-1}(\T)}^p\, \d t
	\overset{\eqref{eq:Lr-HnegEst}}{\lesssim} 1,
\end{align*}
from which we infer that \eqref{eq:Kstar-weak-q2} holds.
\end{proof}

\begin{rem}
Notice how high integrability in $(\tilde \omega,t)$ 
is traded for low ``integrability" in $x$ in the second 
estimate in \eqref{eq:Lr-HnegEst}.
\end{rem}

\section{Existence of martingale solutions}
\label{sec:existence}

Recall that $\tilde X_n$ and $\tilde X$, cf.~\eqref{eq:tilde-Xn}, 
are collective symbols for all the Skorokhod--Jakubowski 
representations, which are built from a sequence 
$\left\{u_{\ep_n}\right\}_{n\in \N}$ of 
strong solutions to the viscous SPDE \eqref{eq:u_ch_ep} 
with initial data $u_{\ep_n}(0)=z_n$, cf.~\eqref{eq:u0-approx}, 
where the viscosity coefficients $\ep_n$ are positive 
numbers with $\ep_n \to 0$ as $n\to \infty$. 

We need filtrations linked to each $\tilde X_n$ 
and the a.s.~limit $\tilde X$. To this end, let 
us first introduce some notations. 
For $t \in [0,T]$, let $f \mapsto f|_{[0,t]}$ denote 
the restriction to the interval $[0,t]$ of a function $f$ 
defined on $[0,T]$. Moreover, we denote by $\Sigma(E)$ the smallest 
$\sigma$-algebra containing a collection $E$ 
of subsets of $\tilde{\Omega}$.  We specify 
$\bigl\{\tilde{\mathcal{F}}_t^n\bigr\}_{t\in [0,T]}$ 
and $\bigl\{\tilde{\mathcal{F}}_t\bigr\}_{t\in [0,T]}$ to be 
the $\tilde{\mathbb{P}}$--augmented canonical filtrations 
of the processes $\tilde X_n$ and $\tilde X$, 
respectively. More precisely, for $\tilde X$ 
the filtration and corresponding stochastic basis 
are defined as
\begin{align}\label{eq:sigma_algebra}
	&\tilde{\mathcal{F}}_t
	=\Sigma\bk{\Sigma\bigl(\tilde X|_{[0,t]}\bigr) 
	\bigcup 
	\bigl\{N \in \tilde{\mathcal{F}}: 
	\tilde{\mathbb{P}}(N) = 0\bigr\}}, 
	\quad t\in [0,T],
	\\ & \quad \text{and} \quad
	\tilde{\cS}=\bigl(\tilde{\Omega},\tilde{\mathcal{F}},
	\bigl\{\tilde{\mathcal{F}}_t\bigr\}_{t\in [0,T]},
	\tilde{\mathbb{P}}\bigr).
	\notag
\end{align}
For $\tilde X_n$ the filtration $\tilde{\mathcal{F}}_t^n$ 
and stochastic basis $\tilde{\cS}^n$ are defined similarly, 
with $\tilde X_n$ replacing $\tilde X$ and $\tilde{\mathcal{F}}_t^n$  
replacing $\mathcal{F}_t$. By construction, the processes 
$\tilde{X}_n$ and $\tilde{X}$ are adapted to 
their canonical filtrations.

By the equality of laws and the 
L\'{e}vy martingale characterization of 
a Wiener process, it is clear 
that $\tilde W_n$ is a Wiener processes with respect to its 
own canonical filtration. Furthermore, $\tilde W_n$ 
is a Wiener process relative to the filtration 
$\bigl\{\tilde{\mathcal{F}}_t^n\bigr\}$ defined in 
\eqref{eq:sigma_algebra}. To prove this, we must 
verify that $\tilde W_n(t)$ is $\tilde{\mathcal{F}}_t^n$--measurable 
and $\tilde W_n(t)-\tilde W_n(s)$ is independent of 
$\tilde{\mathcal{F}}_s^n$, for all $s< t$. 
However, these properties hold because $\tilde W_n$ and $W$ share 
the same law, and $W(t)$ is $\mathcal{F}_t$--measurable 
and $W(t)-W(s)$ is independent of $\mathcal{F}_s$, 
recalling that the unique $H^m$ solution of the viscous 
SPDE \eqref{eq:u_ch_ep}, by construction, 
depends measurably on the initial data 
and the Wiener process \cite{HKP-viscous}.

A standard argument reveals that 
the a.s.~limit $\tilde{W}$ of $\tilde W_n$, 
see \eqref{eq:weak-conv-tilde}, is a Wiener process relative to 
$\bigl\{\tilde{\mathcal{F}}_t\bigr\}_{t\in [0,T]}$
(see, e.g., \cite[Lemma 4.8]{Debussche:2016aa}). 

\begin{lem}[$\tilde W$ is a Wiener process]
\label{thm:W_BM}
The a.s.~representation $\tilde{W}$ 
from Proposition \ref{thm:skorohod} is a Wiener 
process defined on the stochastic basis 
$\tilde{\mathcal{S}}$, cf.~\eqref{eq:sigma_algebra}.
\end{lem}

\begin{proof}
By the equality of laws and L\'evy's characterisation 
theorem (see, e.g., \cite[Theorem IV.3.6]{Revuz:1999wi}), it 
remains to prove that $\tilde{W}$ is a 
$\tilde{\mathcal{F}}_t$ martingale.

Let $\gamma:\mathcal{X}|_{[0,s]}\to [0,1]$ 
be a continuous function, where $\mathcal{X}$ 
is the (countable product) path space defined 
by \eqref{eq:pathspaces}, \eqref{eq:pathspaces-joint} 
and by $\mathcal{X}|_{[0,s]}$ we understand the same space 
but with $[0,T]$ replaced by $[0,s]$. Clearly, $\mathcal{X}|_{[0,s]}$ 
is quasi-Polish (with the product topology), 
and the restriction operator $\mathcal{R}_s:\mathcal{X}\to\mathcal{X}|_{[0,s]}$ 
is continuous (as each single component is trivially continuous). 
Hence, $X_n|_{[0,s]}=\mathcal{R}_s \circ X_n$ is 
$\mathcal{F}_s\, /\otimes_{l\in \N}\mathcal{B}_{\mathcal{X}_{l}}
\big|_{[0,s]}$ measurable, where the countable 
vector $X_n$ is defined in \eqref{eq:Xn-def} and 
$\mathcal{B}_{\mathcal{X}_{l}}\big|_{[0,s]} $ 
denotes the Borel $\sigma$-algebra of 
the ``restricted" space $\mathcal{X}_l|_{[0,s]}$.

Now, by the equality of laws and the $\{\mathcal{F}_t\}$-martingale 
property of the original Wiener process $W$, 
for any $0\leq s<t\le T$ and for any $n \in \mathbb{N}$,
\begin{align*}
	& \tilde{\Ex}\left[\gamma\bigl(\tilde{X}_n|_{[0,s]}\bigr)
	\bigl(\tilde{W}_n(t)-\tilde{W}_n(s)\bigr)\right]
	=\Ex\left[\gamma\bigl(X_n|_{[0,s]}\bigr)
	\bigl(W(t)-W(s)\bigr)\right]=0,
\end{align*}
where $X_n$ is defined in \eqref{eq:Xn-def}. 
The lemma follows if we can pass to limit $n\to\infty$ in the 
left-hand side of the above identity. 
By \eqref{eq:weak-conv-tilde}, $\tilde{W}_n\to \tilde{W}$ 
in $C([0,T])$, $\tilde{\mathbb{P}}$--a.s. 
Moreover, by the equality of laws,
$\tilde \Ex \norm{\tilde{W}_n}_{C([0,T])}^p
= \Ex \norm{W}_{C([0,T])}^p\leq C(T,p)$, 
for any finite $p$, where the last estimate comes 
from the BDG inequality. Hence, by Vitali's convergence theorem, 
$\tilde{\Ex}\left[\gamma\bigl(\tilde{X}|_{[0,s]}\bigr)
\bigl(\tilde{W}(t)-\tilde{W}(s)\bigr)\right]=0$.
\end{proof}

The process $u_{\ep_n}$ satsfies 
the viscous SPDE \eqref{eq:u_ch_ep} with initial 
data $u_{\ep_n}(0)=z_n$, cf.~\eqref{eq:u0-approx}.
The next result shows that the Skorokhod--Jakubowski representation 
$\tilde u_n$ satisfies the same SPDE on the new 
probability space. There exist several approaches to proving this, see 
for example \cite{Bensoussan:1995aa,Brzezniak:2011aa,Ondrejat:2010aa}.
Here we are going to rely on a simple but general 
method discovered by Brze\'zniak and 
Ondrej\'at \cite{Brzezniak:2011aa,Ondrejat:2010aa}, 
and then used in several other works analysing different SPDEs, see for 
example \cite{Hofmanova:2013aa,Debussche:2016aa,Breit:2016aa}. 
To describe the idea, consider the following functional, defined 
for $(u,v,z)\in \mathcal{X}_u\times \mathcal{X}_{q^2}
\times \mathcal{X}_{u_0}$, cf.~\eqref{eq:pathspaces}, and $t\in [0,T]$:
\begin{align}\label{eq:mart_def1}
	&M_n[u,v,z](t)=\int_\T \varphi\, u(t)\,\d x
	-\int_\T \varphi \, z\,\d x 
	-\ep_n \int_0^t \int_\T \pd_{xx}^2\varphi \, u\,\d x\,\d s
	\\ \notag & \quad
	-\int_0^t \int_\T \pd_x \varphi
	\left(\frac12 \,u^2+P[u,v]\right)\,\d x\,\d s
	-\frac12 \int_0^t \int_\T 
	\pd_x\bk{\pd_x \bk{\sigma_{\ep_n}
	\, \varphi}\sigma_{\ep_n}}u \,\d x \,\d s,
\end{align}
viewing the test function $\varphi \in C^\infty(\T)$ 
as fixed. Here, we have augmented our usual 
notation, cf.~\eqref{eq:u_ch_ep}, to accommodate 
the weak limit of $\tilde q_n^2$, by setting 
$P[u,v]=K*\bk{u^2+\frac12 v}$.	
The $x$-weak formulation of the SPDE for 
$u_{\ep_n}$, cf.~\eqref{eq:u_ep_weakeq}, reads
$$
D_n=M_n\bigl[u_{\ep_n},q_{\ep_n}^2,z_n\bigr](t)
-\int_0^t \int_\T \pd_x\bk{\varphi \sigma_{\ep_n}}
u_{\ep_n} \,\d x\,\d W=0,
\quad q_{\ep_n}=\pd_x u_{\ep_n}.
$$ 
Replacing $u_{\ep_n}$, $W$ by $\tilde u_n$, $\tilde W_n$, respectively, 
we denote the corresponding quantity by $\tilde D_n$. 
The aim is to show that $D_n=0$ implies $\tilde D_n=0$. 
By the equality of laws, the real-valued stochastic process 
$\tilde D_n$ is a martingale (starting at $0$), and if 
one establishes that the quadratic variation of $\tilde D_n$ 
is zero, then $\tilde D_n$ is zero. Since 
$\tilde D_n$ is of the form $\tilde{D}_n^{(1)}
-\tilde{D}_n^{(2)}$, this boils down to 
computing the quadratic variation 
$\bigl \langle \tilde D_n\bigr \rangle$ as 
$\bigl \langle \tilde D_n^{(1)}\bigr \rangle 
-2 \bigl \langle \tilde D_n^{(1)}, D_n^{(2)}\bigr
\rangle +\bigl \langle \tilde D_n^{(2)}\bigr \rangle$, 
where $\bigl \langle \tilde D_n^{(2)}\bigr \rangle(t)
= \int_0^t \abs{\int_\T \pd_x\bk{\varphi \sigma_{\ep_n}}
\tilde u_n \,\d x}^2\,\d s$, and the first (quadratic variation) 
and second (co-variation) terms can be computed via the 
equality of laws and properties of the corresponding 
terms in $\bigl \langle D_n\bigr \rangle=0$, see 
\eqref{eq:ondrejat_arg} below.

\begin{lem}[$\tilde u_n$ solves SPDE]
\label{thm:nth_eq}
Let $\tilde{u}_n$, $\tilde{q}_n$, $\tilde W_n$, 
$\tilde{u}_{0,n}$ be the Skorokhod--Jakubowski 
representations from Proposition \ref{thm:skorohod}. 
Then, for any $\varphi \in C^\infty(\T)$ and $t\in [0,T]$,
\begin{align}
	& \int_\T \varphi\, \tilde{u}_n(t) \,\d x
	-\int_\T \varphi \, \tilde{u}_{0,n} \,\d x
	\notag
	\\ & \quad 
	=\int_0^t\int_\T \pd_x \varphi
	\left(\frac12 \tilde{u}_n^2+\tilde{P}_n
	\right)\,\d x\, \d s
	+\ep_n\int_0^t\int_\T \pd_{xx}^2 \varphi 
	\, \tilde{u}_n\,\d x\, \d s 
	\label{eq:nth_equation_twiddle}
	\\ & \quad\quad
	+\frac12\int_0^t\int_\T\pd_x \bk{\pd_x\bk{\sigma_{\ep_n}\, \varphi}
	\sigma_{\ep_n}}\, \tilde{u}_n \,\d x \,\d s
	+\int_0^t\int_\T \pd_x\bk{\varphi\, \sigma_{\ep_n}}
	\, \tilde{u}_n\,\d x\, \d \tilde{W}_n,
	\notag
\end{align}
$\tilde{\mathbb{P}}$--a.s., where $\tilde{P}_{n}=P[\tilde{u}_n]=
K*\left(\tilde{u}^2_n+\frac{1}{2}\tilde{q}_n^2\right)$.
\end{lem}

\begin{proof} 
We follow, e.g., \cite{Breit:2016aa}. For notational 
brevity, herein we use $\LL X \RR(t) $ to denote 
the quadratic variation $\LL X,X \RR(t)$ of a process $X$, 
whilst retaining $\LL X,Y \RR(t)$ for the co-varation 
between two processes $X,Y$. 

\medskip
\noindent \textit{1. Set-up and conclusion.}
\medskip

Given \eqref{eq:mart_def1}, let us also introduce 
the $n$-independent functionals
\begin{align}\label{eq:functionals_NR}
	R[u](t)=\int_0^t \abs{\int_\T 
	\pd_x\bk{\varphi \sigma} u \,\d x}^2\,\d s,
	\quad N[u](t)=-\int_0^t\int_\T 
	\pd_x\bk{\varphi \sigma}u\,\d x\,\d s.
\end{align}

The proof hinges on showing 
that $\tilde M_n=M_n\bigl[\tilde{u}_n,
\tilde{q}_n^2,\tilde{u}_{0,n}\bigr]$ 
is an $\bigl\{\tilde{\mathcal{F}}^n_t\bigr\}$-martingale 
with quadratic variation and covariation 
(with $\tilde W_n$) given by
\begin{equation}\label{eq:qvar_estb_pre}
\begin{aligned}
	\bigl \LL \tilde M_n \bigr\RR&= \int_0^t \abs{\int_\T 
	\pd_x\bk{\varphi \sigma_{\ep_n}} \tilde{u}_n \,\d x}^2\,\d s
	=:\tilde R_n,\\
	\bigl\LL \tilde M_n,\tilde{W}_n\bigr\RR&=-\int_0^t\int_\T 
	\pd_x\bk{\varphi \sigma_{\ep_n}}\tilde{u}_n\,\d x\,\d s
	=:\tilde N_n.
\end{aligned}
\end{equation}
These identities imply that  $\tilde D_n(t)=\tilde M_n(t)
-\int_0^t \int_\T \pd_x\bk{\varphi \sigma_{\ep_n}}
\tilde{u}_n \,\d x \,\d \tilde{W}_n$ has 
vanishing quadratic variation: 
\begin{align}
	& \bigl \langle \tilde D_n\bigr \rangle(t)
	=\left\LL \tilde M_n\right\RR(t)
	-2\int_0^t\int_\T \pd_x\bk{\varphi \sigma_{\ep_n}}\tilde{u}_n \,\d x \,\d 
	\left\LL  \tilde M_n,\tilde{W}_n \right\RR 
	\notag \\ &\qquad\qquad\qquad 
	+ \left\LL\int_0^\cdot \int_\T \pd_x \bk{\varphi \sigma_{\ep_n}}
	\tilde{u}_n\,\d x \,\d \tilde{W}_n\right\RR(t)=0,
	\label{eq:ondrejat_arg}
\end{align}
and $\bigl \langle \tilde D_n\bigr \rangle =0$ implies 
$\tilde D_n=0$, which is the sought-after 
equation \eqref{eq:nth_equation_twiddle}.

\medskip
\noindent \textit{2. Martingale properties and verification 
of \eqref{eq:qvar_estb_pre}.}
\medskip

We establish \eqref{eq:qvar_estb_pre} 
by verifying the martingale property of 
the three processes $\tilde M_n$, 
$\tilde M_n^2-\tilde R_n$, and
$\tilde M_n\tilde{W}_n-\tilde N_n$. 
However, first we must check that 
\begin{align*}
	\mathcal{X}_u\times \mathcal{X}_{q^2}\times \mathcal{X}_{u_0}
	\ni (u,v,z) \mapsto M_n[u,v, z](t)\in \R, \quad t\in [0,T],
\end{align*}
is measurable map. We will do this by proving continuity of $M_n$. 
Given this continuity of $M_n$ on a finite sub-collection 
$\mathcal{X}_u\times \mathcal{X}_{q^2}\times \mathcal{X}_{u_0}$ 
of the factors of the full Cartesian product space 
$\mathcal{X}$, cf.~\eqref{eq:pathspaces} and 
\eqref{eq:pathspaces-joint}, clearly $M_n$ may be 
seen as a continuous function on the full Cartesian 
product $\mathcal{X}$, and then Remark \ref{rem:cont-map-measurable} 
supplies the desired measurability. 
We prove the measurability of $R$ and $N$ in the same way.

Continuity will follow from the estimates already established. 
Fix  $u_i\in \mathcal{X}_u $, $v_i \in \mathcal{X}_{q^2}$ 
and $z_i \in \mathcal{X}_{u_0}$, see \eqref{eq:pathspaces}, 
for $i=1,2$. From \eqref{eq:mart_def1}, and by 
repeated applications of H\"older's inequality, we obtain
\begin{align*}
	& \abs{M_n[u_1,v_1,z_1]-M_n[u_2,v_2,z_2]}
	\\& \qquad 
	\lesssim  \norm{\varphi}_{L^2(\T)}
	\norm{u_1-u_2}_{C([0,T];L^2(\T))}
	+\norm{\varphi}_{L^2(\T)}\norm{z_1-z_2}_{L^2(\T)} 
	\\ & \qquad\quad
	+ \ep_n \norm{\pd_{xx}^2 \varphi}_{L^2(\T)}
	\norm{u_1-u_2}_{C([0,T];L^2(\T))}
	\\ & \qquad\quad
	+\norm{\pd_x \varphi}_{L^\infty(\T)}
	\norm{u_1+u_2}_{L^2([0,T] \times \T)} 
	\norm{u_1-u_2}_{L^2([0,T] \times \T)} 
	\\ & \qquad\quad
	+\norm{\varphi}_{L^\infty(\T)}
	\norm{\pd_x K}_{L^1(\T)}
	\norm{u_1+u_2}_{L^2([0,T] \times \T)} 
	\norm{u_1-u_2}_{L^2([0,T] \times \T)}  
	\\ & \qquad\quad 
	+\abs{\int_0^T \int_\T 
	\bigl((\tau \varphi)*\pd_x K\bigr)
	\bk{v_1-v_2}\,\d x\,\d t}
	\\ & \qquad\quad
	+\norm{\pd_x \bk{\pd_x
	\bk{\sigma_{\ep_n} \varphi}\sigma_{\ep_n}}}_{L^\infty(\T)} 
	\norm{u_1-u_2}_{L^1([0,T] \times \T)},
\end{align*}
writing ``$a^2-b^2=(a+b)(a-b)$" twice. 
Since $\mathcal{X}_{q^2} = L^r([0,T] \times \T)-w$ 
for some fixed $1 \le r < 3/2$, we have used a 
standard property of convolution to write
$$
\int_\T \varphi \,\pd_x K *(v_1-v_2) \,\d x 
=-\int_\T \bigl((\tau \varphi)*\pd_x K\bigr)
(v_1-v_2) \,\d x,
$$
noting that $(\tau\varphi)*\pd_x K\in 
L^{r'}([0,T]\times \T)$, $1/r + 1/r' = 1$ (so $r'>3$). 
Here the map $\tau$ is defined by $(\tau \varphi)(-x)=\varphi(x)$.
This shows that $M_n$ is continuous on 
$\mathcal{X}_u \times \mathcal{X}_{q^2}\times \mathcal{X}_{u_0}$, 
and thereby measurable (according to Remark 
\ref{rem:cont-map-measurable}). 
Similar arguments can now be made for 
$N[u]$ and $R[u]$. As $u \mapsto N[u]$ is linear, 
continuity follows from the bound
$$
\abs{N[u]} \le \norm{\pd_x\bk{\varphi \sigma}}_{L^\infty(\T)} 
\norm{u}_{L^1([0,T]\times \T)}.
$$
For the continuity of $R$,
\begin{gather*}
	\abs{R[u_1]-R[u_2]}
	\le \int_0^t \bk{\int_\T \pd_x 
	\bk{\varphi\sigma} \bk{u_1-u_2}\,\d x
	\int_\T \pd_x \bk{\varphi \sigma} 
	\bk{u_1+u_2}\,\d x}\,\d s
	\\ \le \norm{\pd_x\bk{\varphi \sigma}}_{L^\infty(\T)}^2
	\norm{u_1-u_2}_{L^2([0,T]; L^1(\T))}
	\norm{u_1 + u_2}_{L^2([0,T]; L^1(\T))}
	\\ \lesssim_{\varphi,\sigma,T} 
	\norm{u_1 + u_2}_{L^2([0,T]; L^1(\T))}
	\norm{u_1 - u_2}_{C([0,T]; L^2(\T))}.
\end{gather*}

Finally, we verify the announced martingale properties.  
For any c\`adl\`ag process $X$ on $[0,T]$ 
and $s,t\in [0,T]$ with $s<t$, denote by $\Delta_{s,t} X$ 
the difference $X(t)-X(s)$. Let $\gamma:\mathcal{X}|_{[0,s]}\to [0,1]$ 
be an arbitrary continuous function, 
where $\mathcal{X}$ is the path space defined by 
\eqref{eq:pathspaces}, \eqref{eq:pathspaces-joint}.
By the equality of laws in Proposition \ref{thm:skorohod} 
and the martingale property of the original processes
$M_n:=M_n\bigl[u_{\ep_n},q^2_{\ep_n},z_{0,n}\bigr]$, 
$M_n^2-R[u_n]$, and $M_n W-N[u_n]$, we obtain
\begin{equation} \label{eq:martingale_limitn}
	\begin{aligned}
		&\tilde{\Ex}\left[ \gamma\bigl(\tilde{X}_n|_{[0,s]}\bigr)
		\Delta_{s,t} \tilde M_n\right]=0,
		\quad 
		\tilde{\Ex}\left[ \gamma\bigl(\tilde{X}_n|_{[0,s]}\bigr)
		\left(\Delta_{s,t} \tilde M_n^2
		-\Delta_{s,t} \tilde R_n\right)\right]=0,
		\\ & \quad \text{and} \quad 
		\tilde{\Ex}\left[ \gamma\bigl(\tilde{X}_n|_{[0,s]}\bigr)
		\left(\Delta_{s,t} \bigl(\tilde M_n \tilde W_n\bigr)
		-\Delta_{s,t} \tilde N_n\right)\right]=0,
	\end{aligned}
\end{equation}
which proves that 
$\tilde M_n$, $\tilde M_n^2-\tilde R_n$, 
and $\tilde M_n\tilde{W}_n - \tilde N_n$ 
are $\bigl\{\tilde F_t^n \bigr\}$--martingales.
\end{proof}

Arguing as above, we prove next 
that the a.s.~limit $\tilde u$ from 
Proposition \ref{thm:skorohod} satisfies an SPDE 
on the new probability space 
that resembles the stochastic CH equation \eqref{eq:u_ch}, 
except that the nonlinear term $\tilde q^2$ is 
replaced by the weak limit $\overline{q^2}$. Once 
we make the identification $\overline{q^2}=\tilde q^2$, 
which is equivalent to the strong $L^2_{\tilde \omega,t,x}$ convergence 
of $\tilde q_n$ towards $\tilde q$ \cite[Lemma 3.34]{Novotny-book:2004}, 
the proof of Theorem \ref{thm:main} is concluded. 
But being rather long and technical, the identification step 
is postponed to Section \ref{sec:wk_limits_section}, 
which constitutes a central part of the paper.

\begin{prop}[limit $\tilde u$ solves SPDE]\label{thm:existence_H1}
Suppose the assumptions of Theorem \ref{thm:main} hold.
Let $\tilde{u}$, $\tilde q$, $\overline{q^2}$, $\tilde W$, 
$\tilde{u}_0$ be the Skorokhod--Jakubowski representations from 
Proposition \ref{thm:skorohod}, see 
also Remark \ref{rem:weak-notation}, and let $\tilde \cS$ 
be the stochastic basis defined in \eqref{eq:sigma_algebra}. 
Suppose further that the following identification holds:
\begin{equation}\label{eq:bigassumption1}
	\overline{q^2}=\tilde{q}^2, 
	\quad 
	\text{$\tilde{\mathbb{P}}\otimes \d t 
	\otimes \d x$--a.e.~in $\tilde \Omega 
	\times [0,T]\times \T$}. 
\end{equation}
Then $\bigl(\tilde \cS,\tilde u,\tilde W \bigr)$ 
is a weak martingale solution of 
\begin{equation*}
	\begin{aligned}
		0 &=\d \tilde u + \bigl[\tilde u\, \pd_x \tilde u
		+\pd_x \tilde P \bigr] \, \d t 
		-\frac12\sigma \pd_x \bk{\sigma \pd_x \tilde u}\,\d t
		+\sigma \pd_x \tilde u \, \d \tilde W,\\
		\tilde P & = K*\bk{\tilde u^2+\frac12\tilde q^2},
		\quad \tilde{u}(0) = \tilde{u}_0\sim u_0,
	\end{aligned}
\end{equation*}
in the sense that $\tilde \cS$ satisfies (a), $\tilde W$ 
satisfies (b), and $\tilde u$ satisfies (c), (d) 
of Definition \ref{def:dissp_sol}. Besides, the 
$x$-weak form \eqref{eq:u_weakeq} holds with $u$, $P$ 
replaced by $\tilde u$, $\tilde P$.
\end{prop}

\begin{proof}
We continue to use the functionals $M_n$, $N$, and $R$ defined 
in \eqref{eq:mart_def1} and \eqref{eq:functionals_NR}. 
In addition, we need the $n$-independent functional
\begin{equation*}
	M[u,v,z](t)=M_n[u,v,z](t) +\ep_n \int_0^t \int_\T 
	\pd_{xx}^2\varphi u \,\d x\,\d s.
\end{equation*}
To simplify the notation, set 
$\tilde{M}=M\bigl[\tilde{u},\tilde{q}^2,\tilde{u}_0\bigr]$, 
$\tilde R=R[\tilde{u}]$, and $\tilde N=N[\tilde{u}]$. 
Similarly, we continue to use 
$\tilde M_n=M_n\bigl[\tilde{u}_n,\tilde{q}^2_n,
\tilde{u}_{0,n}\bigr]$, and  $\tilde R_n$, 
$\tilde N_n$ of \eqref{eq:qvar_estb_pre}.

\medskip
\noindent \textit{1. Set-up and conclusion.}
\medskip

The underlying idea of the proof 
is the same as before. Here we want to verify that   
$\tilde M$, $\tilde M^2-\tilde R$, 
and $\tilde{M} \tilde{W} - \tilde N$ 
are all $\bigl\{\tilde{\mathcal{F}}_t\bigr\}$--martingales.
The limit statements corresponding 
to \eqref{eq:qvar_estb_pre} take the form
\begin{align*}
	\bigl\LL \tilde M\bigr\RR =\tilde R, \quad 
	\bigl\LL \tilde M,\tilde{W}\bigr\RR=\tilde N.
\end{align*}
As before \eqref{eq:ondrejat_arg}, these identities 
imply that $\tilde D=\tilde M
-\int_0^\cdot \int_\T \pd_x\bk{\varphi \sigma}
\tilde u \,\d x \,\d \tilde{W}$ is 
a martingale (starting at $0$) with vanishing 
quadratic variation, and $\tilde D =0$ 
is the desired equation \eqref{eq:u_weakeq}, 
replacing $u,P$ by $\tilde u,\tilde P$. 
Because of Section \ref{sec:equality-of-laws}, 
the remaining properties of 
$\bigl(\tilde \cS,\tilde u,\tilde W \bigr)$ 
are evident.

The martingale properties follow by sending 
$n\to \infty$ in \eqref{eq:martingale_limitn}, 
relying on the a.s.~convergences \eqref{eq:weak-conv-tilde} 
and the moment estimates in Lemma \ref{thm:q_bounds_skorohod}.
Eventually, we arrive at the required martingale equalities
\begin{align}
		&\tilde{\Ex}\left[\gamma\bigl(\tilde{X}|_{[0,s]}\bigr)
		\Delta_{s,t} \tilde M\right]=0,
		\label{eq:martingale_limit1} 
		\\ & 
		\tilde{\Ex}\left[\gamma\bigl(\tilde{X}|_{[0,s]}\bigr)
		\left(\Delta_{s,t} \tilde M^2
		-\Delta_{s,t} \tilde R\right)\right]=0,
		\label{eq:martingale_limit2}
		\\ &  
		\tilde{\Ex}\left[ \gamma\bigl(\tilde{X}|_{[0,s]}\bigr)
		\left(\Delta_{s,t} \bigl(\tilde M \tilde W\bigr)
		-\Delta_{s,t} \tilde N\right)\right]=0,
		\label{eq:martingale_limit3}
\end{align}
where $\tilde X$ is defined in \eqref{eq:tilde-Xn}, 
see also \eqref{eq:sigma_algebra}, and 
$\gamma:\mathcal{X}|_{[0,s]}\to [0,1]$ 
is an arbitrary continuous function.

\medskip
\noindent\textit{2. Passing to the limit in \eqref{eq:martingale_limitn} 
to obtain \eqref{eq:martingale_limit1}.}
\medskip

It remains to justify the 
passage to the limit in each equation 
of \eqref{eq:martingale_limitn}. Since $\gamma$ 
is bounded and continuous and $\tilde X_n \to \tilde X$ 
a.s.~(Proposition \ref{thm:skorohod}), 
it follows that 
\begin{equation}\label{eq:gamma-conv}
	\gamma\bigl(\tilde{X}_n|_{[0,s]}\bigr)
	\ton \gamma\bigl(\tilde{X}|_{[0,s]}\bigr)
	\quad \text{in $L^p(\tilde \Omega)$, 
	for any finite $p$.}
\end{equation}
We continue with the claim that, for any $t\in [0,T]$,
\begin{equation}\label{eq:M_converge}
	\tilde M_n(t)\ton M(t), 
	\quad \text{$\tilde{\mathbb{P}}$--a.s.}
\end{equation}
We verify \eqref{eq:M_converge} 
by proving the term-by-term convergence 
of $\tilde M_n$ to $\tilde M$. 
From \eqref{eq:weak-conv-tilde}, we have
$\tilde{u}_{0,n}\to \tilde{u}_{0}$ in 
$H^1(\T)$ and $\tilde{u}_n \to \tilde{u}$ 
in $C([0,T];L^2(\T))$, $\tilde{\mathbb{P}}$--almost surely. 
This (and $\ep_n\to 0$) implies 
\begin{align*} 
	& \abs{\int_\T \varphi\bk{\tilde{u}_{0,n}
	-\tilde{u}_0}\,\d x}  
	\ton 0,\quad \text{$\tilde{\mathbb{P}}$--a.s.},
	\\ & 
	\sup_{t \in [0,T]}\abs{\int_\T \varphi 
	\bk{\tilde{u}_n-\tilde{u}}(t)\,\d x}\ton 0,
	\quad \text{$\tilde{\mathbb{P}}$--a.s.},
	\\ &
	\abs{ \eps_n\int_0^t \int_\T \pd_{xx}^2 \varphi
	\, \tilde{u}_n\,\d x \,\d s} \ton 0,
	\quad \text{$\tilde{\mathbb{P}}$--a.s.}
	\\ & \abs{\int_0^t\int_\T \pd_x\varphi
	\bk{\frac{\tilde{u}^2_n}{2}
	-\frac{\tilde{u}^2}{2}}\,\d x \,\d s}
	\ton 0, \quad \text{$\tilde{\mathbb{P}}$--a.s.}
\end{align*}

Using $P\bigl[\tilde{u}_n,\tilde{q}_n^2\bigr]=K*\left(\tilde{u}^2_n
+\frac{1}{2}\tilde{q}_n^2\right)$, 
$\tilde u_n^2\to \tilde u^2$ in $C([0,T];L^1(\T))$ a.s., 
$\tilde q_n^2 \weak \overline{q^2}$ 
in $L^r([0,T]\times \T)$ a.s., and the weak limit 
identification \eqref{eq:bigassumption1},
\begin{align*}
	&\abs{\int_0^t \int_\T \pd_x \varphi 
	\Bigl(P\bigl[\tilde{u}_n,\tilde{q}_n^2\bigr] 
	-P\bigl[\tilde{u},\tilde{q}^2\bigr]\Bigr)\,\d x \,\d s}	\\ & \quad \leq 
	\norm{\pd_x \varphi}_{L^\infty(\T)}\norm{K}_{L^1(\T)}
	\norm{\tilde{u}^2_n-\tilde{u}^2}_{L^1([0,T]\times \T)} 
	\\ & \quad\quad
	+\abs{\int_0^t \int_\T \bk{\, \, \int_\T 
	\pd_x \varphi(x)K(x-y) \,\d x}(y)\,
	\bk{\tilde{q}_n^2(s,y)-\tilde{q}^2(s,y)}\,\d y\,\d s}\ton 0,
\end{align*}
exploiting $(s,y)\mapsto \int_\T 
\varphi(x)\pd_xK(x-y) \,\d x  \in L^{r'}([0,T]\times \T)$ 
(recall $r<3/2$ and therefore $r'>3$). 

Finally, using again $\tilde{u}_n \to \tilde{u}$ in 
$C([0,T];L^2(\T))$ a.s., and the a.e.~convergence 
$\pd_x\bk{\pd_x\bk{\sigma_{\ep_n}\varphi}\sigma_{\ep_n}}
\to \pd_x\bk{\pd_x\bk{\sigma \varphi}\sigma}$, 
\begin{align*}
	\abs{\int_0^t\int_\T\pd_x \bk{\pd_x\bk{\sigma_{\ep_n} \varphi}
	\sigma_{\ep_n}}\tilde{u}_n- \pd_x \bk{\pd_x\bk{\sigma \varphi}
	\sigma}\tilde{u} \,\d x \,\d s}
	\ton 0, \quad \text{$\tilde{\mathbb{P}}$--a.s.},
\end{align*}
which concludes the proof of \eqref{eq:M_converge}. 
By Lemma \ref{thm:q_bounds_skorohod}, 
\begin{equation}\label{eq:tMn-moments}
	\Ex \abs{\tilde M_n(t)}^{p_0}\lesssim_\varphi 1 
	\quad \text{(with $p_0>4$)}.	
\end{equation}
Hence, by Vitali's convergence theorem, 
\begin{equation}\label{eq:tMn-conv-Lp}
	\tilde M_n(t) \ton \tilde M(t)
	\quad \text{in $L^p(\tilde \Omega)$, 
	for any $p\in [1,p_0)$, $t\in [0,T]$}.
\end{equation}
In view of \eqref{eq:martingale_limitn}, \eqref{eq:gamma-conv}, 
and \eqref{eq:tMn-conv-Lp}, we see that 
\eqref{eq:martingale_limit1} holds.

\medskip
\noindent\textit{3. Passing to the limit in \eqref{eq:martingale_limitn} 
to obtain \eqref{eq:martingale_limit2}.}
\medskip

Recalling \eqref{eq:martingale_limitn}, we consider 
the convergences of $\Delta_{s,t}\tilde M_n^2$ 
and $\Delta_{s,t} \tilde R_n$ separately.  
Using \eqref{eq:tMn-conv-Lp} again ($p_0>4$), we have 
\begin{equation}\label{eq:tMn2-conv-Lp}
	\tilde M_n^2(t) \ton \tilde M^2(t)
	\quad \text{in $L^p(\tilde \Omega)$, 
	for any $p\in \bigl[1,p_0/2\bigr)$, 
	$t\in [0,T]$}.
\end{equation}
The convergence 
\begin{align}\label{eq:square_convergence}
	\abs{\int_0^t \abs{\int_\T 
	\pd_x \bk{\varphi\sigma_{\ep_n}}
	\tilde{u}_n\,\d x}^2
	-\abs{\int_\T 
	\pd_x \bk{\varphi\sigma}
	\tilde{u}\,\d x }^2 \d s}
	\ton 0, \quad \text{$\tilde{\mathbb{P}}$--a.s.},
\end{align}
follows by applying the algebraic identity $a^2-b^2=(a+b)(a-b)$, 
and using the a.s.~convergence $\tilde{u}_n\to \tilde{u}$ 
in $C([0,T];L^2(\T))$ and the a.e.~convergence 
$\pd_x\bk{\varphi \sigma_{\ep_n}} \to \pd_x \bk{\varphi \sigma}$. 
Clearly, \eqref{eq:square_convergence} 
implies the $\tilde{\mathbb{P}}$--a.s.~convergence 
$\tilde R_n(t)\to \tilde R(t)$, for $t\in [0,T]$.
By Lemma \ref{thm:q_bounds_skorohod}, 
$\Ex \abs{\tilde R_n(t)}^{p_0/2}\lesssim_\varphi 1$, 
and therefore, by Vitali's convergence theorem, 
\begin{equation}\label{eq:tRn-conv-Lp}
	\tilde R_n(t) \ton \tilde R(t)
	\quad \text{in $L^p(\tilde \Omega)$, 
	for any $p\in \bigl[1,p_0/2\bigr)$, $t\in [0,T]$}.
\end{equation}
Combining \eqref{eq:martingale_limitn} with 
\eqref{eq:gamma-conv}, \eqref{eq:tMn2-conv-Lp}, 
and \eqref{eq:tRn-conv-Lp}, the 
claim \eqref{eq:martingale_limit2} follows.

\medskip
\noindent\textit{4. Passing to the limit in \eqref{eq:martingale_limitn} 
to obtain \eqref{eq:martingale_limit3}.}
\medskip

From the a.s.~convergence $\tilde{W}_n 
\to \tilde{W}$ in $C([0,T])$, cf.~\eqref{eq:weak-conv-tilde}, 
along with \eqref{eq:M_converge},
$$
\tilde M_n(t)\, \tilde{W}_n(t)
\ton \tilde M(t)\, \tilde{W}(t), 
\quad \text{$\tilde{\mathbb{P}}$--a.s., $t\in [0,T]$}.
$$
By the Cauchy--Schwarz inequality,
\begin{align*}
	\tilde{\Ex}\abs{\tilde M_n \tilde{W}_n}^{p_0/2} 
	\leq \bk{\tilde{\Ex}\abs{\tilde M_n}^{p_0}}^{1/2}
	\bk{\tilde{\Ex} \abs{\tilde{W}_n}^{p_0}}^{1/2} \lesssim 1,
\end{align*}
where we have used \eqref{eq:tMn-moments} and 
the BDG martingale inequality to bound the $p_0$
moment of $\tilde W_n$. Thus, again by 
Vitali's convergence theorem, 
\begin{equation}\label{eq:tMn-tWm-conv-Lp}
	\tilde M_n(t)\, \tilde W_n(t) 
	\ton \tilde M(t)\, \tilde M(t)
	\quad \text{in $L^p(\tilde \Omega)$, 
	$p\in \bigl[1,p_0/2\bigr)$, $t\in [0,T]$}.
\end{equation}
The a.s.~convergence of $\tilde N_n$ 
to $\tilde N$, cf.~\eqref{eq:functionals_NR}, 
follows from $\tilde{u}_n\to \tilde{u}$ in $C([0,T];L^2(\T))$. 
By Lemma \ref{thm:q_bounds_skorohod}, 
$\Ex \abs{\tilde N_n(t)}^{p_0}\lesssim_\varphi 1$, 
and thus Vitali's convergence theorem yields
$$
\tilde N_n(t) \ton \tilde N(t)
\quad \text{in $L^p(\tilde \Omega)$, 
for any $p\in [1,p_0)$, $t\in [0,T]$}.
$$
Combining this, \eqref{eq:tMn-tWm-conv-Lp} and 
\eqref{eq:gamma-conv} with \eqref{eq:martingale_limitn}, 
the final identity \eqref{eq:martingale_limit3} emerges.
\end{proof}

\section{Identification of a weak limit}
\label{sec:wk_limits_section}

In this final section we prove the crucial assumption 
\eqref{eq:bigassumption1} of Proposition 
\ref{thm:existence_H1}, thereby concluding 
the proof of our main result (Theorem \ref{thm:main}).  

\begin{thm}[identification of weak limit]
\label{thm:products_convergences}
Suppose the assumptions of Theorem \ref{thm:main} hold.
Let $\tilde q$ and $\overline{q^2}$ be the Skorokhod--Jakubowski 
representations from Proposition \ref{thm:skorohod}, 
recalling that notationally we drop the 
tilde under the overline in $\overline{q^2}$ (see Remark 
~\ref{rem:weak-notation}). Then the weak limit identification 
\eqref{eq:bigassumption1} holds.
\end{thm}

For a high-level description of the proof, which is 
long and technical, we refer to Section \ref{sec:intro}. 
The proof depends on deriving It\^{o} 
differential inequalities for the differences 
$\overline{q_\pm^2}-\tilde{q}_\pm^2\ge 0$ 
(via numerous steps of truncation and regularisation). 
We record these inequalities over several results 
(see Lemmas \ref{thm:limiteq_1} -- \ref{thm:limiteq_2b}). 
With regards to the subscripts $\pm$ on $q_\pm$, it 
is in fact expedient to carry out this procedure, 
as in the deterministic setting 
\cite{Coclite:2005tq,Coclite:2015aa,Xin:2000qf}, 
for the positive and negative parts separately. 
It is a characteristic feature of dissipative 
solutions that $q_+$ does not 
blow up in $L^\infty$, but $q_-$ does. 
We refer to Section \ref{sec:intro} for a
discussion of the many differences 
between the deterministic and stochastic cases.

\subsection{Energy inequalities and a right-continuity property}
The differential inequalities mentioned above will serve  
to propagate strong compactness, assumed initially at $t=0$, 
via a (yet to be established) strong temporal 
continuity property at $t=0$. 
The existence of this strong initial trace 
is the content of Lemma \ref{thm:temp_cont_0} below, 
which encodes the dissipative 
nature of the considered solution class. 
However, first we need to transfer the 
energy balance \eqref{eq:energybalance} 
to the new probability space, expressed in terms 
of the Skorokhod--Jakubowski 
representations from Proposition \ref{thm:skorohod}.

\begin{lem}[energy inequality]
\label{lem:energy-of-weak-limit}
Let $\tilde u_n$, $\tilde W_n$, $\tilde u_{0,n}$ 
be respectively the Skorokhod--Jakubowski representations of 
$u_{\ep_n}$, $W$, $z_n$, where $u_{\ep_n}$ 
is the strong solution to the viscous 
SPDE \eqref{eq:u_ch_ep} with noise $W$ and 
initial data $u_{\ep_n}(0)=z_n$, cf.~\eqref{eq:u0-approx}.
The energy inequality \eqref{eq:energybalance} 
holds with $u_\ep$, $W$, $\ep$ replaced 
by $\tilde u_n$, $\tilde W_n$, $\ep_n$.

Let $\tilde{u}$, $\tilde q$, $\overline{q^2}$, $\tilde W$, 
$\tilde{u}_0$ be the a.s.~limits from Proposition \ref{thm:skorohod}, 
see also Remark \ref{rem:weak-notation}, and let $\tilde \cS$ 
be the stochastic basis defined in \eqref{eq:sigma_algebra}. Then
\begin{equation}\label{eq:energybalance-tu}
	\begin{aligned}
		& \frac{\d}{\d t}\int_\T \tilde u^2+\overline{q^2}\,\d x
 		\leq \int_\T \frac14 \pd_{xx}^2 \sigma^2 \tilde u^2
 		+\bk{\abs{\pd_x \sigma}^2
 		-\frac14 \pd_{xx}^2 \sigma^2}
 		\,\overline{q^2} \,\d x 
 		\\ &\qquad \quad 
		+\int_\T \pd_x\sigma \bk{\tilde u^2-\overline{q^2}}
 		\, \d x \, \dot{\tilde{W}}, 
 		\quad \text{in $\mathcal{D}'([0,T))$, 
 		$\tilde{\mathbb{P}}$--a.s.},
 		\\ & 
 		\int_\T \left(\tilde u^2+\overline{q^2}\right)(0)\,\d x
 		=\int_\T \tilde u_0^2+\abs{\pd_x \tilde u_0}^2\,\d x.
	\end{aligned}
\end{equation}
\end{lem}

\begin{rem}We emphasise that \eqref{eq:energybalance-tu} 
holds in the sense of distributions 
on the half-open interval $[0,T)$, 
$\tilde{\mathbb{P}}$-a.s., whilst 
$\int_\T \tilde u^2+\overline{q^2}\,\d x$ is 
understood to take the value 
$\int_\T \tilde u_0^2+\abs{\pd_x \tilde u_0}^2
\,\d x$ at $t = 0$. This means that for every 
non-negative $\psi \in C^\infty([0,T))$,
\begin{align*}
	- \int_0^T &\pd_t \psi \, \int_\T 
	\tilde u^2+\overline{q^2}
	\,\d x\,\d t
	- \psi(0) \int_\T \tilde u_0^2+\abs{\pd_x \tilde u_0}^2\,\d x 
	\\ & 
	\le \int_0^T\psi  \int_\T \frac14 
	\pd_{xx}^2 \sigma^2 \tilde u^2
 	+\bk{\abs{\pd_x \sigma}^2
 	-\frac14 \pd_{xx}^2 \sigma^2}
 	\,\overline{q^2} \,\d x \,\d t 
 	\\ &\qquad \quad
	+\int_0^T \psi \int_\T \pd_x\sigma 
	\bk{\tilde u^2-\overline{q^2}}
 	\, \d x \, \d \tilde{W}, 
 	\quad \text{
 	$\tilde{\mathbb{P}}$--a.s.}
\end{align*}
\end{rem}

\begin{proof}
Recall the properties of $\tilde u_n$ stated 
in Lemmas \ref{thm:smoothness_n} and 
\ref{thm:q_bounds_skorohod}. In particular, 
$\tilde u_n$ lies in  the intersection
$L^2([0,T];H^m(\T)) \cap C([0,T];H^1(\T))$ a.s., for 
any $m\in \N$. According to Lemma \ref{eq:nth_equation_twiddle}, 
$\tilde u_n$ satisfies the SPDE \eqref{eq:u_ch_ep} 
with $u_\ep$, $W$, $\ep$ replaced by $\tilde u_n$, 
$\tilde W_n$, $\ep_n$, respectively. 
If we differentiate this equation with respect to 
$x$, cf.~Lemma \ref{thm:lim_iden_simple}, then 
$\tilde q_n=\pd_x \tilde u_n$ satisfies the SPDE
\eqref{eq:intro-qest-tmp} with $q_\ep$, $W$, $\ep$ 
replaced by $\tilde q_n$, $\tilde W_n$, $\eps_n$. 
Consequently, we may apply the corresponding versions of 
\eqref{eq:SPDE-S(qe)} and \eqref{eq:SPDE-S(ue)} with $S(v)=v^2/2$. 
Adding the resulting equations yields the 
total energy equation \eqref{eq:intro-visc-energy} 
with $u_\ep$, $W$, $\ep$ replaced by 
$\tilde u_n$, $\tilde W_n$, $\ep_n$. Integrating 
this equation in $x$, dropping the dissipation term, 
and expressing the temporal differential 
as a time-derivative in $\mathcal{D}'([0,T{)})$, we arrive at
\begin{equation}\label{eq:energybalance-tun}
	\begin{aligned}
		&\frac{\d}{\d t}\int_\T \tilde u_n^2 + 
		\tilde q_n^2\,\d x
 		\leq \int_\T \frac14 \pd_{xx}^2 
		\sigma_{\ep_n}^2  \tilde u_n^2
 		+\bk{\abs{\pd_x \sigma_{\ep_n}}^2
 		-\frac14 \pd_{xx}^2 \sigma_{\ep_n}^2}
 		\, \tilde q_n^2 \,\d x 
 		\\ &\qquad \qquad 
		+\int_\T \pd_x \sigma_{\ep_n}
		\bk{\tilde u_n^2-\tilde q_n^2}
 		\,\d x \, \dot{\tilde{W}}_n, 
 		\quad  \text{in $\mathcal{D}'([0,T{)})$, 
 		\, $\tilde{\mathbb{P}}$--a.s.},
 	\end{aligned}
\end{equation}
where $\int_\T \left(\tilde u_n^2
+\abs{\tilde q_n}^2\right)(0)\,\d x
=\int_\T \tilde u_{n,0}^2+\abs{\pd_x \tilde u_{n,0}}^2\,\d x$ 
and $\dot{\tilde{W}}_n= \frac{\d}{\d t}\tilde{W}_n$.

Equipped with the a.s.~convergences in \eqref{eq:weak-conv-tilde}, 
in particular $\tilde u_n\ton \tilde u$ in $C_tL^2_x$ a.s.~and 
$\tilde q_n^2\ton \overline{q^2}$ in $L^r\bigl(L^r_w\bigr)$ a.s., 
recalling that we write $\overline{q^2}$ instead of 
$\overline{\tilde q^2}$, we can send $n\to \infty$ 
in \eqref{eq:energybalance-tun} 
to arrive at \eqref{eq:energybalance-tu}. We refer to 
Lemma \ref{thm:limiteq_1} for a detailed 
convergence proof of an inequality that is more 
general than \eqref{eq:energybalance-tun}. 

During the derivation of \eqref{eq:energybalance-tu}, one issue 
was swept under the rug. Indeed, a priori, it is 
not clear that the process $\tilde M(t)=\int_0^t\int_\T 
\pd_x\sigma \overline{q^2}\, \d x \, \d \tilde{W}$ 
is a square-integrable martingale. 
The matter in question is that the limit 
$\overline{q^2}$ belongs to $L^r_{t,x}$ with 
merely $r<3/2$, cf.~\eqref{eq:weak-conv-tilde}; note 
carefully that we do not have this issue with the 
related process $\int_0^t\int_\T \pd_x\sigma\,
\tilde q^2\, \d x \, \d \tilde{W}$, 
where $\tilde q = \pd_x \tilde u$ is the a.s.~limit of 
$\tilde q_n=\pd_x \tilde u_n$, as $\tilde u$ 
satisfies  \eqref{eq:tu-Lp-Linfty-H1}. 
Fortunately, according to \eqref{eq:Lr-HnegEst} 
of Lemma \ref{thm:P-u2_weak}, we may assume 
that $\overline{q^2}\in L^2_{\tilde \omega,t}\bigl(H^{-1}_x\bigr)$ 
and whence $\tilde M$ be interpreted as a square-integrable 
martingale, recalling that $\sigma \in W^{2,\infty}(\T)$:
\begin{equation}\label{eq:mart-prop-part1}
	\begin{aligned}
		\tilde \Ex \int_0^T \abs{\int_\T 
		\pd_x\sigma\, \overline{q^2}\, \d x}^2\, \d t
		\leq \norm{\pd_x \sigma}_{H^1(\T)}^2 \tilde \Ex \int_0^T 
		\norm{\overline{q^2}(t)}_{H^{-1}(\T)}^2
		\, \d t <\infty.
	\end{aligned}
\end{equation}
\end{proof}

The pathwise inequality \eqref{eq:energybalance-tu},
the convergence $\tilde u_n\to \tilde u$ in 
$C\bigl([0,T];H^1(\T)-w\bigr)$ a.s.~(see
proof of Lemma \ref{thm:u_in_Linfty}), and the
strong $H^1$ convergence of $\tilde u_{0,n}$ towards $\tilde u_0$,
imply the strong right-continuity at $t=0$ in $H^1$.
We have the following result:

\begin{lem}[one-sided temporal continuity at $t=0$]
\label{thm:temp_cont_0}
Let $\tilde{u}$, $\tilde q$, $\tilde u_0$, and
$\overline{q^2}$ be the Skorokhod--Jakubowski representations
from Proposition \ref{thm:skorohod}.
Then
\begin{equation}\label{eq:right-cont1}
        \lim_{t \downarrow 0}
        \norm{\tilde{u}(t)-\tilde{u}_0}_{H^1(\T)}=0,
        \quad \text{$\tilde{\mathbb{P}}$--a.s.}
\end{equation}
Moreover, for the nonlinearities $S(v)=S_\ell(v_{\pm})$ defined
by \eqref{eq:entropies_S_ell},
\begin{equation}\label{eq:right-cont2}
        \lim_{t\downarrow 0}
        \norm{S\bigl(\tilde{q}(t)\bigr)
        -S\bigl(\pd_x \tilde u_0\bigr)}_{L^1(\T)}=0,
        \quad \text{$\tilde{\mathbb{P}}$--a.s.}
\end{equation}
\end{lem}

\begin{rem}
In view of Lemma \ref{thm:u_in_Linfty}
and Vitali's convergence theorem, \eqref{eq:right-cont1} and
\eqref{eq:right-cont2} imply 
\begin{align*}
        \lim_{t\downarrow 0}
        \tilde \Ex \norm{\tilde{u}(t)
        -\tilde{u}_0}_{H^1(\T)}^2=0,
        \quad
        \lim_{t\downarrow 0} \tilde \Ex
        \norm{S\bigl(\tilde{q}(t)\bigr)
        -S\bigl(\pd_x \tilde u_0\bigr)}_{L^1(\T)} = 0.
\end{align*}
\end{rem}

\begin{proof}
We divide the proof into two steps.

\medskip
\noindent\textit{1. One-sided temporal continuity in
$H^1(\T)$, \eqref{eq:right-cont1}}.
\medskip

In the process of proving Lemma \ref{thm:u_in_Linfty},
we demonstrated that
$$
\text{$\tilde u_n\ton \tilde u$ in
$C\bigl([0,T];H^1(\T)-w\bigr)$, a.s.}
$$
Accordingly, employing the weak lower semicontinuity of $v
\mapsto \norm{v}_{H^1(\T)}^2$,
\begin{equation}\label{eq:right-cont-new1}
        \norm{\tilde u(t)}_{H^1(\T)}^2
        \leq \liminf_{n \to \infty}    
        \norm{\tilde{u}_n(t)}_{H^1(\T)}^2
        \leq \limsup_{n \to \infty}    
        \norm{\tilde{u}_n(t)}_{H^1(\T)}^2, \quad t>0.
\end{equation}
Define
\begin{align*}
	& I_n(t) = \int_0^t \int_\T \frac14 \pd_x^2 \sigma^2 \tilde u_n^2
	+\bk{\abs{\pd_x \sigma}^2-\frac14 \pd_x^2 \sigma^2}
	\,\tilde q_n^2 \,\d x \,ds
	\\ &\qquad  \qquad \qquad
	+\int_0^t\int_\T \pd_x\sigma
	\bk{\tilde u_n^2-\tilde q^2}\, \d x \, \d \tilde{W}_n,
	\\ &
	I(t) = \int_0^t \int_\T \frac14 \pd_x^2 \sigma^2 \tilde u^2
	+\bk{\abs{\pd_x \sigma}^2-\frac14 \pd_x^2 \sigma^2}
	\,\overline{q^2} \,\d x \,ds
	\\ &\qquad  \qquad \qquad
	+\int_0^t\int_\T \pd_x\sigma
	\bk{\tilde u^2-\overline{q^2}}\, \d x \, \d \tilde{W}.
\end{align*}
Arguing as in the proof of Lemma \ref{lem:energy-of-weak-limit},
$I_n\ton I$ a.s.~(with $t>0$ fixed).
By a standard deterministic argument, see for
example \cite[page 653]{Evans:2010gf},
we can turn the (pathwise) distributional
inequality \eqref{eq:energybalance-tun}
into the pointwise inequality
\begin{equation}\label{eq:right-cont-new2}
        \norm{\tilde{u}_n(t)}_{H^1(\T)}^2
        \leq \norm{\tilde{u}_{0,n}}_{H^1(\T)}^2+I_n(t), \quad t>0.
\end{equation}
Importantly, as $n\to \infty$, the right-hand side
of \eqref{eq:right-cont-new2} converges
almost surely to $\norm{\tilde{u}_{0}}_{H^1(\T)}^2+I(t)$.
In view of \eqref{eq:right-cont-new1}
and \eqref{eq:right-cont-new2}, we conclude that
\begin{equation}\label{eq:right-cont-new3}
        \norm{\tilde u(t)}_{H^1(\T)}^2
        \leq \norm{\tilde{u}_{0}}_{H^1(\T)}^2+I(t),
        \quad \text{a.s., for all $t>0$}.
\end{equation}

Since $\tilde u\in C\bigl([0,T];H^1(\T)-w\bigr)$ a.s., it is evident that
\begin{equation}\label{eq:weak-timecont-as-tmp}
        \tilde u(t) \weak \tilde u_0
        \quad
        \text{in $H^1(\T)$ as $t\downarrow 0$}.
\end{equation}
Whence, again by the weak lower semicontinuity of
$\norm{\cdot}_{H^1(\T)}^2$,
\begin{align*}
	\norm{\tilde u_0}_{H^1(\T)}^2
	& \leq \liminf_{t \downarrow  0}
	\norm{\tilde{u}(t)}_{H^1(\T)}^2
	\leq \limsup_{t \downarrow  0}
	\norm{\tilde{u}(t)}_{H^1(\T)}^2
	\\ & \notag
	\overset{\eqref{eq:right-cont-new3}}{\leq}
	\limsup_{t \downarrow  0}
	\left(\norm{\tilde{u}_{0}}_{H^1(\T)}^2+I(t)\right)
	=\norm{\tilde{u}_{0}}_{H^1(\T)}^2,
\end{align*}
where we have used that $I(t)\to 0$ as $t\downarrow 0$, a.s. 
Therefore, a.s.,
\begin{equation}\label{eq:limit_of_norms-as}
        \lim_{t\downarrow 0}\norm{\tilde{u}(t)}^2_{H^1(\T)}
        = \norm{\tilde{u}_0}_{H^1(\T)}^2.
\end{equation}

Combining \eqref{eq:weak-timecont-as-tmp} and
\eqref{eq:limit_of_norms-as} (``weak convergence plus
convergence of norms imply strong convergence"), we
attain \eqref{eq:right-cont1}.

\medskip
\noindent\textit{2. One-sided temporal continuity
for nonlinearities, \eqref{eq:right-cont2}.}
\medskip

Fix $a,b\in \R$ with $a<b$. Assume that $b>0$,
otherwise there would be nothing to prove.
It is easy to verify that
\begin{align*}
        \abs{S_\ell(b_+)-S_\ell(a_+)}
        & =\int_{a\vee 0}^b S'_\ell(v) \,\d v
        \overset{\eqref{eq:entrop1}}{\leq}
        \int_{a\vee 0}^b v \,\d v
        \\ & \leq b^2 - \bk{a \vee 0}^2
        \leq \bk{\abs{b}+\abs{a}}\abs{b-a}.
\end{align*}
By symmetry, the inequality holds for $b<a$,
and a similar calculation establishes the inequality
for $S(v_-)$.

Fix any $\tilde \omega \in \tilde \Omega$ for
which \eqref{eq:right-cont1} holds. For $S(v)=S_\ell(v_{\pm})$
defined by \eqref{eq:entropies_S_ell}, we
then proceed as follows:
\begin{align*}
	& \norm{S\bigl(\tilde{q}(t)\bigr)
	-S\bigl(\pd_x \tilde u_0\bigr)}_{L^1(\T)}
	\\ & \qquad \leq
	\int_\T \bigl(\abs{\tilde{q}(t,x)}+\abs{\pd_x \tilde u_0(x)}\bigr)
	\abs{\tilde{q}(t,x)-\pd_x \tilde u_0(x)}\,\d x
        \\ & \qquad \leq
        2\norm{\tilde{q}}_{L^\infty([0,T];L^2(\T))}
        \norm{\tilde{q}(t)-\pd_x \tilde u_0}_{L^2(\T)}
        \\ & \qquad
        \overset{\eqref{eq:tu-Lp-Linfty-H1}}{\lesssim_{\tilde \omega}}
        \norm{\tilde{q}(t)-\pd_x \tilde u_0}_{L^2(\T)}
        \overset{\eqref{eq:right-cont1}}{\longrightarrow} 0
        \quad \text{as $t\downarrow 0$.}
\end{align*}
This concludes the proof of \eqref{eq:right-cont2}.
\end{proof}

Once we have made the identification
\eqref{eq:bigassumption1}, the
inequality \eqref{eq:energybalance-tu} becomes
\begin{equation}\label{eq:energybalance-tu-new}
        \begin{aligned}
        		& \frac{\d}{\d t}\int_\T \tilde u^2+\tilde q^2\,\d x
		\leq \int_\T \frac14 \pd_x^2 \sigma^2 \tilde u^2
		+\bk{\abs{\pd_x \sigma}^2
		-\frac14 \pd_x^2 \sigma^2}
		\,\tilde q^2 \,\d x
		\\ &\qquad \quad
		+\int_\T \pd_x\sigma \bk{\tilde u^2-\tilde q^2}
		\, \d x \, \dot{\tilde{W}},
		\quad \text{in $\mathcal{D}'([0,T))$,
		$\tilde{\mathbb{P}}$--a.s.},
		\\ &
		\int_\T \left(\tilde u^2+\tilde q^2\right)(0)\,\d x
		=\int_\T \tilde u_0^2+\abs{\pd_x \tilde u_0}^2\,\d x,
        \end{aligned}
\end{equation}
where $\tilde q=\pd_x \tilde u$ and so
$\int_\T \tilde u^2+\tilde q^2\,\d x=\norm{\tilde u}_{H^1(\T)}^2$.
By modifying the proof of Lemma \ref{thm:temp_cont_0}, we can use
\eqref{eq:energybalance-tu-new} to establish
the validity of the claim \eqref{eq:energybalance-sch}
in Theorem \ref{thm:main}, and also that the limit $\tilde u$ satisfies
part (f) of Definition \ref{def:dissp_sol}.
This is the content of the next lemma.

\begin{lem}[energy inequality and one-sided temporal continuity]
\label{lem:energy-rightcont}
Suppose \eqref{eq:energybalance-tu-new} holds.
Then the total energy inequality \eqref{eq:energybalance-sch}
holds a.s., for a.e.~$s\in [0,T)$ and every $t$ with $s< t\leq T$.
Specifically, it holds for $s=0$ and any $t\in (0,T]$, with
$$
\int_\T \left(\tilde u^2 + \abs{\pd_x \tilde u}^2\right)(0)
\,\d x=\int_\T \tilde u_0^2+\abs{\pd_x \tilde u_0}^2\,\d x.
$$ 
Consequently, a.s., for a.e.~$t_0\in [0,T)$,
\begin{equation}\label{eq:right-cont1-t0}
        \lim_{t \downarrow t_0}
        \norm{\tilde{u}(t)-\tilde{u}(t_0)}_{H^1(\T)}=0,
\end{equation}
where the case $t_0=0$, for which $\tilde{u}(0)=\tilde{u}_0$,
is covered by Lemma \ref{thm:temp_cont_0}.
\end{lem}

\begin{rem}\label{rem:almost-sure}
When we assert that a property
is true ``a.s.~for a.e.~$t\in [0,T]$", it means
that for almost every $\tilde \omega\in \tilde \Omega$, under the probability
measure $\tilde{\mathbb{P}}$, there exists a Lebesgue negligible subset
$N=N(\tilde \omega)\subset [0,T]$ such that the
stated property holds true for every $t\in [0,T]\setminus N(\tilde \omega)$.
Consider the inequality \eqref{eq:energybalance-sch}, which
can be abstractly represented as $\mathcal{I}(\tilde \omega,t)\leq 0$ for
some function $\mathcal{I}$ on $\tilde \Omega\times [0,T]$ (with $s$ fixed).
For a.e.~$\tilde \omega$, there exists a negligible
set $N(\tilde \omega)\subset [0,T]$, such that the inequality
$\mathcal{I}(\tilde \omega,t)\leq 0$ remains valid for all
$t\in [0,T]\setminus N(\tilde \omega)$.
Now note that our function $\mathcal{I}$ is integrable on the product
space $\tilde \Omega\times [0,T]$. Considering this, we can
employ the Tonelli theorem to conclude
that the inequality $\mathcal{I}(\tilde \omega,t)\leq 0$ is indeed
valid for a.e.~$(\tilde \omega,t)\in
\tilde \Omega\times [0,T]$, i.e., $\mathcal{I}$ is nonpositive
on the product space $\tilde \Omega\times [0,T]$.
\end{rem}

\begin{proof}
We employ a standard deterministic argument, see
for example \cite[page 653]{Evans:2010gf}. Consider test functions
$0 \leq \beta_\delta \in W^{1,\infty}([0, T])$ with
$\delta>0$ taking values in a sequence converging to zero.
For given $s$ and $t$ with $0 \leq s < t \leq T$, consider $\delta > 0$ such
that $s + \delta < t - \delta$. For such $\delta$, let $\beta_\delta$ be
the continuous piecewise linear function that equals $1$
on $[s + \delta, t -\delta]$, $0$ on $[0, s]$ and $[t, T]$, and
is linear on $[s, s + \delta]$ and $[t-\delta, t]$. Then 
$\beta_\delta(t')\to \mathds{1}_{[s,t]}(t')$ for a.e.~$t'\in [0,T]$.
For $t'\in [0,T]$, define
\begin{equation*}
        \begin{split}
                & I_\delta(t')
                =\int_\T \left(\frac14 \pd_x^2 \sigma^2 \tilde u^2
                 +\bk{\abs{\pd_x \sigma}^2-\frac14 \pd_x^2 \sigma^2}
                \,\tilde q^2\right)(t')
                \, \beta_\delta(t') \,\d x,
                 \\ &
                 \Psi_\delta(t') =\int_\T \pd_x\sigma
                 \bk{\tilde u^2-\tilde q^2}(t')
                \, \beta_\delta(t')\, \d x.
        \end{split}
\end{equation*}
By using $\beta_\delta$ as the test function
in \eqref{eq:energybalance-tu-new},
we obtain the following result (a.s.):
\begin{equation*}
        \begin{split}
                &\frac{1}{\delta}\int_{t-\delta}^t
                \norm{\tilde u(t')}_{H^1(\T)}^2\, dt'
                -\frac{1}{\delta}\int_s^{s+\delta}
                \norm{\tilde u(t')}_{H^1(\T)}^2\, dt'
                \\ & \qquad
                \leq \int_0^T I_\delta(t')\,dt'
                +\int_0^T \Psi_\delta(t')\,d\tilde W(t').
        \end{split}
\end{equation*}
We apply Lebesgue's differentiation theorem to
send $\delta$ to zero. As a result, we obtain the following
inequality for all Lebesgue points $0\leq s<t\leq T$ of
the function $t'\mapsto \norm{\tilde u(\tilde \omega,t')}_{H^1(\T)}^2$,
which is integrable on $[0,T]$ for a.e.~$\tilde \omega$:
\begin{equation}\label{eq:right-cont-tmp1-Lebesgue}
        \norm{\tilde u(\tilde \omega,t)}_{H^1(\T)}^2
        -\norm{\tilde u(\tilde \omega,s)}_{H^1(\T)}^2
        \leq \int_s^t I(t')\,dt'
        +\int_s^t \Psi(t') \,d\tilde W(t').
\end{equation}
Here, $I$ and $\Psi$ are defined in the same
manner as $I_\delta$ and $\Psi_\delta$, respectively, but
with the substitution of $\beta_\delta$ by $1$.
To be more precise, for each fixed $\tilde{\omega}$
from a set $F$ of full $\tilde{\mathbb{P}}$--measure, there exists a
subset $N(\tilde{\omega}) \subset [0, T]$ of zero Lebesgue measure
such that \eqref{eq:right-cont-tmp1-Lebesgue} holds for every
$t \in [0,T] \setminus N(\tilde{\omega})$.

The only distinction from the
deterministic argument is the necessity to pass to
the ($\delta\to 0$) limit in the stochastic integrals
$\int_0^T \Psi_\delta \,d\tilde{W}$, where
we clearly have $\abs{\Psi_\delta-\mathds{1}_{[s,t]}\Psi}\to 0$
a.e.~in $\tilde \Omega \times [0,T]$.
Furthermore, leveraging Lemma \ref{thm:u_in_Linfty}, it follows
that $\abs{\Psi_\delta-\mathds{1}_{[s,t]}\Psi}^2
\leq 4\abs{\Psi}^2\in L^1(\tilde \Omega\times [0,T])$.
Thus, by the Lebesgue dominated convergence
theorem, $\Psi_\delta\to \mathds{1}_{[s,t]}\Psi$
in $L^2(\tilde \Omega\times [0,T])$.
Hence, by the BDG inequality, we
conclude that $\int_0^\cdot \Psi_\delta(t')\,d\tilde{W}(t')
\to \int_0^\cdot \mathds{1}_{[s,t]}(t')\Psi(t')\,d\tilde{W}(t')$
in $L^2\bigl(\tilde \Omega; C([0,T])\bigr)$.
Passing to a subsequence, this convergence
holds a.s. in $C([0,T])$.

Next, we note that \eqref{eq:right-cont-tmp1-Lebesgue} is valid
for all values of $t$, not exclusively limited to the Lebesgue points.
To see this, fix an arbitrary $t > s$, with
$s\in [0,T) \setminus N(\tilde{\omega})$, $\tilde \omega \in F$
($s$ is a Lebesgue point of
$\norm{\tilde u(\tilde \omega,\cdot)}_{H^1(\T)}^2$).
Let $t_\ell>s$, $t_\ell\in [0,T) \setminus N(\tilde{\omega})$, be
a sequence of (Lebesgue) points converging to $t$
as $\ell \to \infty$. In \eqref{eq:right-cont-tmp1-Lebesgue}
we replace $t$ by $t_\ell$. Recalling that
$\tilde u\in C\bigl([0,T];H^1(\T)-w\bigr)$ a.s.,
see Lemma \ref{thm:u_in_Linfty}, which implies
that $\tilde u$ is a.s.~weakly lower semicontinuous
in $H^1(\T)$, it then follows that
\begin{align*}
        \norm{\tilde u(\tilde \omega,t)}_{H^1(\T)}^2
        &\leq \liminf_{\ell\to \infty}
        \norm{\tilde u(\tilde \omega,t_\ell)}_{H^1(\T)}^2
        \\ &
        \overset{\eqref{eq:right-cont-tmp1-Lebesgue}}{\leq}
        \norm{\tilde u(\tilde \omega,s)}_{H^1(\T)}^2
        +\int_s^t I(t')\,dt'
        +\int_s^t \Psi(t') \,d\tilde W(t').
\end{align*}
Summarising, the inequality \eqref{eq:right-cont-tmp1-Lebesgue}
holds for $\tilde{\mathbb{P}}$--a.e.~$\tilde \omega$
(i.e., for any $\tilde \omega \in F$ with
$\tilde{\mathbb{P}}(F)=1$), for
any time $t\in (0,T]$ and for any Lebesgue point
$s$ with $0\leq s<t\leq T$ (i.e., 
$s\in [0,T)\setminus N(\tilde \omega)$,
$\abs{N(\tilde \omega)}=0$).
This proves the first part of the lemma.

The right-continuity of $\tilde u$ in $H^1(\T)$
at a Lebesgue point $s=t_0$ can be inferred
from \eqref{eq:right-cont-tmp1-Lebesgue}.
More precisely, by the a.s.~weak lower semicontinuity
of $\tilde u$ and \eqref{eq:right-cont-tmp1-Lebesgue},
\begin{align*}
        \norm{\tilde u(\tilde \omega,t_0)}_{H^1(\T)}^2
        \leq \liminf_{t\downarrow t_0}
        \norm{\tilde u(\tilde \omega,t)}_{H^1(\T)}^2
        & \leq \limsup_{t\downarrow t_0}
        \norm{\tilde u(\tilde \omega,t)}_{H^1(\T)}^2
        \\ & \leq \norm{\tilde u(\tilde \omega,t_0)}_{H^1(\T)}^2,
\end{align*}
so that $\lim_{t\downarrow t_0}
\norm{\tilde u(\tilde \omega,t)}_{H^1(\T)}^2
=\norm{\tilde u(\tilde \omega,t_0)}_{H^1(\T)}^2$,
for any $\tilde \omega \in F$, $\tilde{\mathbb{P}}(F)=1$, and
$t_0\in [0,T)\setminus N(\tilde \omega)$, $\abs{N(\tilde \omega)}=0$. 
As a result, we can employ a similar reasoning
as in the proof of Lemma \ref{thm:temp_cont_0}
to conclude that the right-continuity
claim \eqref{eq:right-cont1-t0} holds.

Finally, utilizing the strong initial trace result \eqref{eq:right-cont1}, we
can conclude that $s = 0$ is a Lebesgue point
of $\norm{\tilde{u}(\cdot)}_{H^1(\mathbb{T})}^2$.
\end{proof}

\subsection{Equation for the weak limits $\overline{S(q)}$}
We will need to know that products like 
$S'(\tilde q_n)\, \tilde P_n$ converge weakly. 
Since $S'(\tilde q_n)$ converges weakly, it is crucial 
that $\tilde P_n$ converges strongly to  $\tilde P$. 
To this end, we will make essential use of the 
space $L^r\bigl(L^r_w\bigr)$. 
First, by \eqref{eq:Hneg-conv},
\begin{equation}\label{eq:Hneg-bound-qn2-oq2-as}
	\norm{\tilde q_n^2
	-\overline{q^2}}_{L^r_t(H^{-1}_x)}^{r}
	=\int_0^T \norm{\tilde q_n^2(t)
	-\overline{q^2}(t)}_{H^{-1}(\T)}^{r}\, \d t
	\ton 0, \quad \text{a.s.}
\end{equation} 
By \eqref{eq:Hneg-bound-qn2} 
and \eqref{eq:Hneg-bound-oq2},
\begin{equation}\label{eq:Hneg-bound-qn2-oq2}
	\tilde \Ex \int_0^T 
	\norm{\tilde q_n^2(t)
	-\overline{q^2}(t)}_{H^{-1}(\T)}^{p}\, \d t\lesssim 1, 
	\quad  p\in \bigl[1,p_0/2\bigr].
\end{equation}
For any $\bar p>1$ with $r\bar p\in \bigl[1,p_0/2\bigr]$ (recall 
that $r<3/2$ and $p_0>4$), we use H\"older's inequality 
and \eqref{eq:Hneg-bound-qn2-oq2} to deduce that 
$\tilde \Ex \norm{\tilde q_n^2
-\overline{q^2}}_{L^r_t(H^{-1}_x)}^{r\bar p} \lesssim_T 1$. 
Therefore, by \eqref{eq:Hneg-bound-qn2-oq2-as} and 
Vitali's convergence theorem,
\begin{equation}\label{eq:Hneg-conv2}
	\tilde q_n^2\ton \overline{q^2}
	\quad \text{in 
	$L^r_{\tilde \omega,t}\bigl(H^{-1}_x\bigr)$}.
\end{equation}

Given \eqref{eq:Hneg-conv2}, passing to a 
subsequence if necessary, we may assume that
\begin{equation}\label{eq:Hneg-conv3}
	\tilde q_n^2(\tilde\omega,t)\ton 
	\overline{q^2}(\tilde \omega,t)
	\quad \text{in $H^{-1}(\T)$, 
	for a.e.~$(\tilde \omega,t)\in 
	\tilde \Omega\times [0,T]$}.
\end{equation}

By Lebesgue interpolation between 
the convergence in $L^r_{\tilde \omega,t}$, see \eqref{eq:Hneg-conv2}, 
and the uniform boundedness in $L^{p_0/2}_{\tilde \omega,t}$, 
see \eqref{eq:Hneg-bound-qn2-oq2}, we can 
improve \eqref{eq:Hneg-conv2} to 
\begin{equation}\label{eq:Hneg-conv4}
	\tilde q_n^2\ton \overline{q^2}
	\quad \text{in 
	$L^p_{\tilde \omega,t}\bigl(H^{-1}_x\bigr)$},
	\quad p\in \bigl[1,p_0/2\bigr).
\end{equation}

We can now prove the following result:

\begin{lem}[strong convergence of $\tilde P_n$]
\label{lem:strong-conv-tPn}
Let $\tilde{u}_n$, $\tilde u$, $\tilde{q}_n$, $\overline{q^2}$ 
be the Skorokhod--Jakubowski representations 
from Proposition \ref{thm:skorohod}. Setting
\begin{equation}\label{eq:tPn-tP-def}
	\begin{split}
		& \tilde{P}_n=K*\bk{\tilde u_n^2+\frac12 \tilde q_n^2}, 
		\quad n\in \N, 
		\\ & 
		\tilde{P}=K*\bk{\tilde{u}^2+\frac12 \overline{q^2}},
	\end{split}
\end{equation}	
the following strong convergence holds:
\begin{equation}\label{eq:tPn-strong-conv}
	\tilde P_n \ton \tilde P \quad 
	\text{in $L^r([0,T]\times \T)$, 
	$\tilde{\mathbb{P}}$--a.s.},
\end{equation}
where $r\in [1,3/2)$ is fixed 
in \eqref{eq:pathspaces}. In addition, for 
any $p\in \bigl[1,p_0/2\bigr)$,
\begin{equation}\label{eq:tPn-strong-conv2}
	\tilde P_n \ton \tilde P \quad 
	\text{in $L^p(\tilde\Omega \times [0,T]\times \T)$},
\end{equation}
where $p_0>4$ is specified 
in Theorem \ref{thm:bounds1}.
\end{lem}

\begin{proof}
For any $(t,x)$,
\begin{align*}
	\abs{K*\tilde q_n^2-K*\overline{q^2}}(t,x)
	& = \abs{\, \int_\T K(x-y)\left(\tilde q_n^2(t,y)
	-\overline{q^2}(t,y)\right) \, \d y}
	\\ & 
	\leq \norm{K(x-\cdot)}_{H^1(\T)}
	\norm{\tilde q_n^2(t)-\overline{q^2}(t)}_{H^{-1}(\T)},
\end{align*}
where $\norm{K(x-\cdot)}_{H^1(\T)}\lesssim 1$ for all $x$. 
Raising this to the $r$th power and then 
integrating in $t$ and $x$, we arrive at
\begin{align*}
	& \int_0^T\int_{\T} \abs{\, \int_\T 
	K(x-y)\left(\tilde q_n^2(t,y)
	-\overline{q^2}(t,y)\right)\, \d y}^r \, \d x\, \d t
	\\ & \qquad \lesssim
	\norm{\tilde q_n^2-\overline{q^2}}_{L^r([0,T];H^{-1}(\T))}^r
	\ton 0, \quad \text{$\tilde{\mathbb{P}}$--a.s.},
\end{align*}
using \eqref{eq:Hneg-bound-qn2-oq2-as}. 
By \eqref{eq:weak-conv-tilde}, 
$\tilde u_n^2 \ton \tilde u^2$ 
in $C_tL^1_x$ a.s.~and so 
\begin{align*}
	& \int_0^T\int_{\T} \abs{\, \int_\T 
	K(x-y)\left(\tilde u_n^2(t,y)
	-\tilde u^2(t,y)\right)\, \d y}^r \, \d x\, \d t
	\\ & \qquad 
	\lesssim 
	\norm{\tilde u_n^2-\tilde u^2}_{L^\infty([0,T];L^1(\T))}^r
	\ton 0, \quad \text{$\tilde{\mathbb{P}}$--a.s.}
\end{align*}
Hence, \eqref{eq:tPn-strong-conv} follows.

Let us now turn to the proof of \eqref{eq:tPn-strong-conv2}. 
From the previous calculations, 
\begin{align}
	\abs{\tilde P_n(\tilde \omega,t,x)
	-\tilde P(\tilde \omega,t,x)}
	& \lesssim \norm{\tilde u_n^2(\tilde \omega,t)
	-\tilde u^2(\tilde \omega,t)}_{L^1(\T)}
	\notag \\ & \qquad 
	+\norm{\tilde q_n^2(\tilde \omega,t)
	-\overline{q^2}(\tilde \omega,t)}_{H^{-1}(\T)}
	\ton 0,
	\label{eq:tP-ae-conv}
\end{align}
for a.e.~$(\tilde \omega,t,x)$, using \eqref{eq:Hneg-conv3} 
and also that $\tilde u_n^2(\tilde \omega,t) \ton \tilde u^2(\tilde \omega,t)$ 
in $L^1_x$, uniformly in $t\in [0,T]$, $\tilde\prob$--a.e.~in $\tilde \omega$, 
cf.~\eqref{eq:weak-conv-tilde}. 
By Lemma \ref{thm:P-u2_bound-tilde} and \eqref{eq:Kstar-weak-q2},
$$
\tilde \Ex \int_0^T \int_\T
\abs{\tilde P_n-\tilde P}^{p_0/2}\, \d x\, \d t
\lesssim 1.
$$
Combining the a.e.~convergence \eqref{eq:tP-ae-conv} with this 
$n$-uniform bound in $L^{p_0/2}_{\tilde \omega,t,x}$, the
Vitali convergence theorem gives \eqref{eq:tPn-strong-conv2}.
\end{proof}

The remaining part of this section is devoted to 
the study of the defect measure $\mathbb{D}$ defined 
in \eqref{eq:intro-defect}, which will be done by 
analysing the related defects $\overline{S(q)}-S(\tilde q)$, 
for an appropriate class of nonlinearities $S$ (for 
reasons outlined in Section \ref{sec:intro}). 
We compute $\overline{S(q)}$ and $S(\tilde q)$ 
in this section and Section \ref{subsec:renorm-tq}, 
before we put the results together 
in Section \ref{subsec:defect} to 
conclude that $\mathbb{D}=0$.

Lemma \ref{thm:limiteq_1} below shows that the a.s.~weak limit 
$\overline{S(q)}$ of $S(\tilde q_n)$, see \eqref{eq:weak-conv-tilde} 
and \eqref{eq:SPDE-S(qe)}, satisfies 
the following pathwise inequality 
in $\mathcal{D}'([0,T)\times \T)$:
\begin{align}
	\pd_t \overline{S(q)}
	&+\pd_x \left[\tilde u\, \overline{S(q)}
	+\frac14 \pd_x \sigma^2 
	\left(3\overline{S(q)}-2\overline{S'(q)q}\right) \right]
	\notag \\ &
	-\pd_{xx}^2\left[\frac12\sigma^2\, \overline{S(q)}\right]
	+\Biggl[ \overline{S'(q)} 
	\bk{P\bigl[\tilde u,\overline{q^2}\,\bigr]-\tilde u^2}
	-\left(\overline{S(q)q}-\frac12 \overline{S'(q)q^2}\right)
	\notag \\ & \qquad 
	-\frac14 \pd_{xx}^2\sigma^2 
	\left(\overline{S(q)}-\overline{S'(q)q}\right)
	-\frac12 \abs{\pd_x \sigma}^2
	\, \overline{S''(q) \,q^2}\Biggr]
	\notag \\ &
	+\left[\pd_x \left(\sigma \, \overline{S(q)}\right)
	-\pd_x\sigma\left(\overline{S(q)}
	-\overline{S'(q)q}\right)\right]\, \dot{\tilde{W}}\leq 0,
	\label{eq:SPDE-S(qe)-limit}
\end{align}
along with the initial data 
$\overline{S(q)}(0)=S(\pd_x \tilde u_0)$. 
Regrettably, we cannot establish 
\eqref{eq:SPDE-S(qe)-limit} along the lines 
of Proposition \ref{thm:existence_H1}. 
The obstacle is that passing to the limit 
in some terms is hampered by the lack of 
strong temporal compactness. Instead we will 
furnish a ``direct" weak convergence proof, relying on 
\cite[Lemma 2.1]{Debussche:2011aa} to establish 
the convergence of stochastic integrals of 
processes like $\int_{\T} S'(\tilde q_n)\tilde q_n\,\d x$. 
A priori, these processes only converge weakly in $L^{2r}_t$. 
However, we have devoted much effort to showing 
that, e.g., $S'(\tilde q_n)\tilde q_n$ converges 
a.s.~in the strong-weak space $L^{2r}\bigl(L^{2r}_w\bigr)$, 
cf.~\eqref{eq:weak-conv-tilde}. This implies that 
$\int_{\T} S'(\tilde q_n)\tilde q_n\,\d x$ 
converges strongly in $L^2_t$, which in turn allows 
for the application of \cite[Lemma 2.1]{Debussche:2011aa}.

\begin{lem}[characterisation of weak limit]\label{thm:limiteq_1}
Denote by $S=S(v)$ any of the functions $S_\ell(v_\pm)$, defined 
by \eqref{eq:entropies_S_ell}, or $\frac12 v^2,\frac12 v_\pm^2$.  
Let $\overline{S(q)}$, $\overline{S'(q)}$, $\overline{S(q)q}$, 
$\overline{S'(q)q}$, $\overline{S'(q)q^2}$ and 
$\overline{S''(q)q^2}$ be the Skorokhod--Jakubowski representations 
from Proposition \ref{thm:skorohod}, see also 
Remark \ref{rem:weak-notation}, and let $\tilde{P}$ 
be defined by \eqref{eq:tPn-tP-def}. Then the 
inequality \eqref{eq:SPDE-S(qe)-limit} holds weakly 
in $(t,x)$, almost surely, that is, for any $0\leq \varphi 
\in C^\infty_c([0,T)\times \T)$,
\begin{equation}\label{eq:S_equation_n_limit}
	\begin{aligned}
		&\int_0^T \int_\T \overline{S(q)}\, \pd_t \varphi \,\d x \,\d t
		+\int_\T  S(\pd_x\tilde{u}_0)\, \varphi(0,x)  \,\d x
		\\ & \quad 
		+ \int_0^T \int_\T\left[\tilde u\, \overline{S(q)}
		+\frac14 \pd_x \sigma^2\, \overline{H^{(1)}(q)} \right] 
		\pd_x \varphi \,\d x \,\d t
		\\ & \quad  
		+\int_0^T \int_\T \frac12 \sigma^2\, \overline{S(q)}\,
		\pd_{xx}^2 \varphi \,\d x \,\d t
		\\ & \quad
		-\int_0^T \int_\T\Biggl[ \overline{S'(q)} 
		\bk{\tilde P-\tilde u^2}
		-\overline{H^{(2)}(q)}
		\\ & \quad\qquad \qquad \qquad
		-\frac14 \pd_{xx}^2\sigma^2\, \overline{H^{(3)}(q)}
		-\frac12 \abs{\pd_x \sigma}^2
		\, \overline{S''(q) \,q^2}\Biggr] \varphi \,\d x \,\d t
		\\ & \quad	
		+\int_0^T \int_\T \sigma \, \overline{S(q)}\, \pd_x \varphi
		+\pd_x\sigma\, \overline{H^{(3)}(q)}\,\varphi
		\,\d x\, \d \tilde W\geq 0, 
		\quad \text{$\tilde{\mathbb{P}}$--a.s.},
	\end{aligned}
\end{equation}
where we have introduced the functions
\begin{equation}\label{eq:H-def}
	\begin{split}
		& H^{(1)}(v)=3S(v)-2S'(v)v, 
		\quad 
		H^{(2)}(v)=S(v)v-\tfrac12 S'(v)v^2,
		\\ &
		H^{(3)}(v)=S(v)-S'(v)v. 
	\end{split}
\end{equation}
By the linearity of weak limits, we have
$\overline{H^{(1)}(q)}= 3\overline{S(q)}-2\overline{S'(q)q}$, 
$\overline{H^{(2)}(q)}=\overline{S(q)q}
-\frac12 \overline{S'(q)q^2}$, and 
$\overline{H^{(3)}(q)}=\overline{S(q)}
-\overline{S'(q)q}$.
\end{lem}

\begin{proof}
Referring to Proposition \ref{thm:skorohod}, 
$\bigl(\tilde u_n,\tilde q_n=\pd_x \tilde u_n, 
\tilde W_n,\tilde u_{0,n}\bigr)$ are the 
Skorokhod--Jakubowski representations of 
$\bigl(u_{\ep_n},q_{\ep_n}=\pd_x u_{\ep_n}, W, z_n\bigr)$, 
respectively, where $u_{\ep_n}$ is the strong solution to the viscous 
SPDE \eqref{eq:u_ch_ep} with noise $W$ and 
initial function $u_{\ep_n}(0)=z_n$, cf.~\eqref{eq:u0-approx}. 
As in the proof of Lemma \ref{lem:energy-of-weak-limit},  
$S(\tilde q_n)$ satisfies \eqref{eq:SPDE-S(qe)} 
with $u_\ep$, $q_\ep$, $W$, $\ep$ replaced by $\tilde u_n$, 
$\tilde q_n$, $\tilde W_n$, $\ep_n$, respectively, 
where $\ep_n \to 0$ as $n\to \infty$. 

Fix a non-negative test function $\varphi \in C^\infty_c([0,T) \times \T)$. 
Multiply \eqref{eq:SPDE-S(qe)} by $\varphi$, integrate over $(t,x)$, 
and then do integration-by-parts in time, keeping in mind that 
$S(\tilde q_n(0))=S(\pd_x \tilde u_{0,n})$. Dropping 
the dissipation term and employing the notation 
\eqref{eq:H-def}, the end result is
\begin{equation}\label{eq:S_equation_n_weak-I}
	\tilde I_n(\tilde \omega)+\tilde{M}_n(\tilde\omega)
	\ge 0, \quad \text{for 
	$\tilde{\mathbb{P}}$--a.e.~$\tilde \omega
	\in \tilde \Omega$},
\end{equation}
where $\tilde{M}_n(\tilde \omega)
=\tilde{\mathcal{M}}_n(\tilde\omega,T)$ with
\begin{equation}\label{eq:S_equation_n_weak-II}
	\tilde{\mathcal{M}}_n(t)
	=\int_0^t \int_\T \sigma_{\ep_n} \, S(\tilde q_n)\, \pd_x \varphi
	+\pd_x\sigma_{\ep_n}\, H^{(3)}(\tilde q_n)\, \varphi
	\,\d x\, \d \tilde W_n, 
	\quad t\in [0,T],
\end{equation}
and $\tilde I_n = \sum_{i=1}^{10} \tilde I_n^{(i)}$ with
\begin{align}
	& \tilde I_n^{(1)}=\int_0^T \int_\T S(\tilde q_n)
	\, \pd_t \varphi \,\d x \,\d t,
	\quad
	\tilde I_n^{(2)}=\int_\T  S(\pd_x\tilde u_{0,n})
	\, \varphi(0,x)  \,\d x,
	\notag	\\ & 
	\tilde I_n^{(3)}= \int_0^T \int_\T \tilde u_n\, S(\tilde q_n)\,
	\pd_x \varphi \, \d x \, \d t,
	\notag \\ &  
	\tilde I_n^{(4)}= \frac14\int_0^T
	\int_\T \pd_x \sigma_{\ep_n}^2\, 
	H^{(1)}(\tilde q_n)\,
	\pd_x \varphi \, \d x \, \d t,
	\notag \\ & 
	\tilde I_n^{(5)}= \int_0^T \int_\T
	\left(\frac12 \sigma_{\ep_n}^2 +\eps_n\right)
	\, S(\tilde q_n)\,\pd_{xx}^2\, \varphi \,\d x \,\d t,
	\label{eq:S_equation_n_weak-III} \\ & 
	\tilde I_n^{(6)}=-\int_0^T \int_\T S'(\tilde q_n)
	\, \tilde P_n\, \varphi \,\d x \,\d t,
	\quad 
	\tilde I_n^{(7)}=\int_0^T \int_\T S'(\tilde q_n) 
	\, \tilde u_n^2\, \varphi \,\d x \,\d t,
	\notag \\ & 
	\tilde I_n^{(8)}=\int_0^T \int_\T 
	\, H^{(2)}(\tilde q_n)\, \varphi \,\d x \,\d t,
	\quad
	\tilde I_n^{(9)}=\frac14\int_0^T \int_\T 
	\pd_{xx}^2\sigma_{\ep_n}^2\, H^{(3)}(\tilde q_n) 
	\, \varphi \,\d x \,\d t,
	\notag \\ & 
	\tilde I_n^{(10)}=\frac12\int_0^T \int_\T 
	\abs{\pd_x \sigma_{\ep_n}}^2 \, S''(\tilde q_n) \,\tilde q_n^2
	\, \varphi \,\d x \,\d t.
	\notag
\end{align}

We can also write the claim \eqref{eq:S_equation_n_limit}
of the lemma in the form
\begin{equation}\label{eq:S_equation_weak-IV}
	\tilde I(\tilde \omega)+\tilde{M}(\tilde\omega)
	\ge 0, \quad \text{for 
	$\tilde{\mathbb{P}}$--a.e.~$\tilde \omega\in \Omega$},
\end{equation}
where $\tilde{M}(\tilde\omega)=\tilde{\mathcal{M}}(T)$ and 
$\tilde I=\sum_{i=1}^{10} \tilde I^{(i)}$ are defined 
as in \eqref{eq:S_equation_n_weak-II} and 
\eqref{eq:S_equation_n_weak-III} via the corresponding 
limit terms identified in \eqref{eq:S_equation_n_limit}. 
Below we will prove that
\begin{equation}\label{eq:S_equation_weak-V}
	\begin{split}
		& \text{$\tilde I_n^{(i)} \ton I^{(i)}$   
		$\tilde{\mathbb{P}}$--a.s.~and 
		strongly in $L^2(\tilde \Omega)$},
		\quad \forall i\notin \{3,6,7\},
		\\ & 
		\text{$\tilde I_n^{(i)} \tonweak I^{(i)}$ 
		in $L^1(\tilde \Omega)$},
		\quad i\in \{3,6,7\}.
	\end{split}
\end{equation}
Moreover, we will prove that the stochastic integral term 
converges strongly in the sense that 
(for a non-relabelled subsequence)
\begin{equation}\label{eq:S_equation_weak-VI}
	\text{$\tilde{M}_n \ton 
	\tilde{M}$ in $L^2(\tilde \Omega)$}.
\end{equation}
Given \eqref{eq:S_equation_n_weak-I}, 
the convergences \eqref{eq:S_equation_weak-V} 
and \eqref{eq:S_equation_weak-VI} imply that
$$
\int_{\tilde \Omega} \mathds{1}_{A}(\tilde \omega)
\bigl( I(\tilde\omega)+M(\tilde\omega)\bigr)
\, \d \tilde{\mathbb{P}}(\tilde \omega)\geq 0,
$$
for any measurable set $A\in \tilde{\mathcal{F}}$, 
which is enough to conclude 
that \eqref{eq:S_equation_weak-IV} holds. 

\medskip

It remains to verify \eqref{eq:S_equation_weak-V} 
and \eqref{eq:S_equation_weak-VI}. 
Let us start with \eqref{eq:S_equation_weak-V}. We do this 
only for the most challenging choices of $S$, 
namely $S(v)= \frac12 v^2,\frac12 v_\pm^2$. 
The argument is the same for $S=S_\ell$ (in fact, it is 
simpler because $S_\ell(v)\lesssim_\ell \abs{v}$).

Regarding the convergences of $I_n^{(i)}$ 
for $i\neq \{3,6,7\}$, cf.~\eqref{eq:S_equation_weak-V}, 
they are all direct consequences of the a.s.~convergences 
in \eqref{eq:weak-conv-tilde}. We detail only the 
case $i=1$. By \eqref{eq:weak-conv-tilde} and 
$S(v)=\frac12 v^2,\frac12 v_\pm^2$, we have, in particular, that 
$S(\tilde q_n)\tonweak \overline{S(q)}$ 
in $L^r([0,T]\times \T)$ a.s.~and so  $\tilde I_n^{(1)}
\ton \tilde I^{(1)}$ a.s. By Lemma \ref{thm:q_bounds_skorohod}, 
we also have the $n$-independent bound
$$
\tilde \Ex\norm{\tilde I_n^{(1)}}_{L^p(\tilde \Omega)}^p
\lesssim_\varphi \tilde \Ex
\left(\int_0^T\int_\T \abs{\tilde q_n}^2
\, \d x\,\d t\right)^p 
\lesssim_T \tilde \Ex 
\norm{\tilde q_n}_{L^\infty([0,T];L^2(\T))}^{2p}
\lesssim 1,
$$ 
for $p\in \bigl[1,p_0/2\bigr)$. Thus, 
by Vitali's convergence theorem, $\tilde I_n^{(1)}
\ton \tilde I^{(1)}$ in $L^2(\tilde \Omega)$. 

Let us consider the exceptional term $I_n^{(3)}$. 
In view of \eqref{eq:weak-conv-tilde}, 
$S(\tilde q_n)\tonweak \overline{S(q)}$ 
in $L^{r}_{t,x}$ a.s., where $r\in [1,3/2)$ (and so $r'=\tfrac{r}{r-1}>3$). 
Given Lemma \ref{thm:q_bounds_skorohod}, we also have the 
bound $\tilde \Ex 
\norm{S(\tilde q_n)}_{L^r_{t,x}}^r\lesssim 1$. 
Hence, by a weak compactness argument, we 
may assume that $\overline{S(q)}\in 
L^r(\tilde\Omega\times[0,T]\times \T)$ and
\begin{equation}\label{eq:improved-S(tun)-conv}
	S(\tilde q_n)\tonweak \overline{S(q)} \quad 
	\text{in $L^r(\tilde\Omega\times[0,T]\times \T)$}. 
\end{equation}
By \eqref{eq:weak-conv-tilde}, $\tilde u_n\ton \tilde u$  
in $C_tL^2_x$ a.s. In view of 
Lemma \ref{thm:P-u2_bound-tilde} and 
Vitali's convergence theorem, 
we thus obtain $\tilde u_n \ton \tilde u$ in 
$L^2(\tilde \Omega\times [0,T]\times \T)$. 
Apart from that, Lemma \ref{thm:P-u2_bound-tilde} delivers 
the bounds $\tilde \Ex \norm{\bigl(\tilde u_n,
\tilde u\bigr)}_{L^{p_0}_{t,x}}^{p_0}
\lesssim 1$, where $p_0>4$ and we may assume $3<r'<p_0$. 
Accordingly, we gather that 
\begin{equation}\label{eq:improved-tun-conv}
	\tilde u_n \ton \tilde u \quad 
	\text{in $L^{r'}(\tilde \Omega\times [0,T]\times \T)$}.
\end{equation}
Given \eqref{eq:improved-S(tun)-conv} 
and \eqref{eq:improved-tun-conv}, by the weak convergence 
of products of strongly and weakly converging sequences,
$$
\tilde u_n\, S(\tilde q_n) \tonweak \tilde u \, \overline{S(q)}
\quad \text{in $L^1(\tilde \Omega\times [0,T]\times \T)$},
$$
which proves \eqref{eq:S_equation_weak-V} for $i=3$. 
Next, consider the term $I_n^{(7)}$. We need to verify 
the weak $L^1_{\tilde \omega,t,x}$ convergence 
of the product $S'(\tilde q_n)\tilde u_n^2$. 
From \eqref{eq:weak-conv-tilde}, $S'(\tilde q_n)\tonweak 
\overline{S'(q)}$ in $L^{2r}_{t,x}$ a.s., 
but Lemma \ref{thm:q_bounds_skorohod} 
also supplies the bound $\tilde \Ex 
\norm{S'(\tilde q_n)}_{L^{2r}_{t,x}}^{2r}\lesssim 1$. Thus, 
by a weak compactness argument, we may assume 
$\overline{S'(q)} \in 
L^{2r}(\tilde \Omega\times [0,T]\times \T)$ and
\begin{equation}\label{eq:Sprime-conv}
	S'(\tilde q_n)\tonweak \overline{S'(q)}
	\quad \text{in $L^{2r}(\tilde \Omega\times [0,T]\times \T)$}.
\end{equation}
As $2<2r<3$ (with $2r$ close to $3$) 
and so $3/2<(2r)'<2$, arguing as for 
\eqref{eq:improved-tun-conv}, we may assume that 
$\tilde u_n \ton \tilde u$ in 
$L^{2(2r)'}(\tilde \Omega\times [0,T]\times \T)$. 
As a result, writing $\tilde u_n^2-\tilde u^2
=\left(\tilde u_n-\tilde u\right) 
\left(\tilde u_n+\tilde u\right)$ 
and using the Cauchy--Schwarz inequality,
$$
\tilde u_n^2 \ton \tilde u^2 \quad 
\text{in $L^{(2r)'}(\tilde \Omega\times [0,T]\times \T)$}.
$$
Combining this with \eqref{eq:Sprime-conv}, we obtain
$S'(\tilde q_n)\, \tilde u_n^2\tonweak 
\overline{S'(q)}\, \tilde u^2$ in $L^1_{\tilde \omega,t,x}$, 
thereby establishing \eqref{eq:S_equation_weak-V} for $i=7$. 

Next, we turn to the weak $L^1_{\tilde \omega,t,x}$ convergence 
of the delicate product term $S'(\tilde q_n)\, \tilde P_n$. 
Fortunately, most of the ``heavy lifting" has already 
been done, since \eqref{eq:tPn-strong-conv2} implies 
that $\tilde P_n \ton \tilde P$ in 
$L^2(\tilde\Omega \times [0,T]\times \T)$. On the other hand, 
given \eqref{eq:Sprime-conv}, $S'(\tilde q_n)
\tonweak \overline{S'(q)}$ in $L^2(\tilde \Omega\times [0,T]\times \T)$, 
and thus $S'(\tilde q_n)\, \tilde P_n 
\tonweak  \overline{S'(q)}\, \tilde P$ 
in $L^1(\tilde \Omega\times [0,T]\times \T)$. This proves 
\eqref{eq:S_equation_weak-V} for $i=6$.

\medskip

Finally, we consider the stochastic integral term 
\eqref{eq:S_equation_n_weak-II}, which we write as
$$
\tilde{\mathcal{M}}_n(t)
=\int_0^t \tilde{\mathcal{J}}_n(s) \, \d \tilde{W}_n,
\quad 
\tilde{\mathcal{J}}_n=
\int_\T \sigma_{\ep_n} \, S(\tilde q_n)\, \pd_x \varphi
+\pd_x\sigma_{\ep_n}\, H^{(3)}(\tilde q_n)
\, \varphi\,\d x.
$$ 
We divide the argument into two cases, 
depending on the choice of $S$, namely 
$S_\ell(v_\pm)$, cf.~\eqref{eq:entropies_S_ell},
or $\frac12 v^2,\frac12 v_\pm^2$. 

Let us begin with the  case $S(v)=S_\ell(v_\pm)$. 
By \eqref{eq:weak-conv-tilde}, 
$S(\tilde q_n)\ton \overline{S(q)}$, 
$S'(\tilde q_n) \tilde q_n\ton \overline{S'(q)q}$, 
and thus $H^{(3)}(\tilde q_n)\ton \overline{H^{(3)}(q)}$
in the strong-weak space $L^{2r}\bigl(L^{2r}_w\bigr)$ 
a.s.~(with $2r>2$). Because of this and $\sigma \pd_x \varphi,
\pd_x\sigma \varphi \in L^\infty_{t,x}$,
\begin{align*}
	& \tilde{\mathcal{J}}_n\ton \int_\T \sigma \, \overline{S(q)}
	\, \pd_x \varphi
	+\pd_x\sigma\, \overline{H^{(3)}(q)}\,\varphi
	\,\d x =: \tilde{\mathcal{J}}
	\quad \text{in $L^{2r}([0,T])$, a.s.}
\end{align*}
This implies that 
$\tilde{\mathcal{J}}_n\to \tilde{\mathcal{J}}$ 
in $L^2([0,T])$, in probability.  
By \eqref{eq:weak-conv-tilde}, $\tilde W_n\to 
\tilde W$ in $C([0,T])$ a.s., 
and thus in probability. The assumptions of 
\cite[Lemma 2.1]{Debussche:2011aa} are therefore fulfilled, 
with the result that
\begin{equation}\label{eq:stoch-int-conv-main}
	\tilde{\mathcal{M}}_n \ton \tilde{\mathcal{M}} 
	\quad \text{in $L^2([0,T])$, in probability},
\end{equation}
where $\tilde{\mathcal{M}}(t)=
\int_0^t \tilde{\mathcal{J}}(s) \, \d \tilde{W}$. 
By passing to a subsequence if necessary, we may assume that this 
convergence holds $\tilde{\mathbb{P}}$--almost surely. 
Note also that the exceptional set does not 
depend on the particular test function $\varphi\in 
C^\infty_c([0,T)\times \T)$ (by the 
separability of $C^\infty_c$).

Next, suppose $S(v)=\frac12 v^2$, 
noting that $H^{(3)}(v)=S(v)-S(v)'v$ in this case 
becomes $-\frac12 v^2$ and so
\begin{align*}
	& \tilde{\mathcal{M}}_n(t)
	=\int_0^t \tilde{\mathcal{J}}_n(s) 
	\, \d \tilde{W}_n,
	\quad 
	\tilde{\mathcal{J}}_n=
	\int_\T \frac12 \bigl(\sigma_{\ep_n} \, \pd_x \varphi
	-\pd_x\sigma_{\ep_n}\,\varphi\bigr)\, \tilde q_n^2\,\d x,
	\\ & 
	\tilde{\mathcal{M}}(t)
	=\int_0^t \tilde{\mathcal{J}}(s) 
	\, \d \tilde{W},
	\quad 
	\tilde{\mathcal{J}}=
	\int_\T \frac12 \bigl(\sigma \, \pd_x \varphi
	-\pd_x\sigma\,\varphi\bigr)\, \overline{q^2}\,\d x.
\end{align*}
According to our previous considerations---leading 
up to \eqref{eq:Hneg-conv4}---we have the crucial
convergence $\tilde q_n^2\ton \overline{q^2}$ 
in $L^2_{\tilde \omega,t}\bigl(H^{-1}_x\bigr)$, 
which amounts to strong $L^2$ convergence in $\tilde \omega, t$, 
see also the proof of Lemma \ref{lem:energy-of-weak-limit}.
This implies that $\tilde{\mathcal{J}}_n\to \tilde{\mathcal{J}}$ 
in $L^2([0,T])$, in probability. As before, $\tilde W_n\to \tilde W$ 
in $C([0,T])$ a.s., and thus in probability. As a result, 
Lemma 2.1 of \cite{Debussche:2011aa} supplies 
\eqref{eq:stoch-int-conv-main}. By passing to 
a subsequence, we may assume that this 
convergence holds $\tilde{\mathbb{P}}$--almost surely. 

The cases $S(v)=\frac12 v_\pm^2$ can be viewed 
in the same way, noting that $H^{(3)}(v)=S(v_\pm)-S(v_\pm)'q
=-\frac12 v_\pm^2$ and that the pivotal convergence 
\eqref{eq:Hneg-conv4} still holds for $(\tilde q_n)_\pm^2$. Indeed, 
with the same proof, $(\tilde q_n)_\pm^2\ton \overline{q^2_\pm}$ 
in $L^2_{\tilde \omega,t}\bigl(H^{-1}_x\bigr)$.

Finally, let us establish the sought-after convergence claim 
\eqref{eq:S_equation_weak-VI}. Indeed, by the previous findings,
\begin{equation}\label{eq:as-conv-tmp1}
	\norm{\tilde{\mathcal{M}}_n
	-\tilde{\mathcal{M}}}_{L^2([0,T])}^2
	\ton 0 \quad \text{a.s.},
\end{equation}
and, for any $p\in \bigl[2,p_0/2\bigr]$ (recall $p_0>4$), 
\begin{equation}\label{eq:as-conv-tmp2}
	\tilde \Ex\norm{\tilde{\mathcal{M}}_n
	-\tilde{\mathcal{M}}}_{L^2([0,T])}^p
	\lesssim_T \tilde \Ex \sup_{t\in [0,T]} 
	\abs{\tilde{\mathcal{M}}_n(t)}^p
	+\tilde \Ex \sup_{t\in [0,T]}
	\abs{\tilde{\mathcal{M}}(t)}^p
	\lesssim_T 1.
\end{equation}

The last bound follows from the following
calculations:
\begin{equation}\label{eq:mart-prop-part2-tmp}
	\begin{split}
		& \tilde \Ex \sup_{t\in [0,T]} 
		\abs{\tilde{\mathcal{M}}_n(t)}^p
		\lesssim 
		\tilde \Ex \left[\left(\int_0^T 
		\abs{\tilde{\mathcal{J}}_n(t)}^2 
		\, \d t\right)^{p/2}\right]
		\\ & \quad 
		\lesssim_T \tilde \Ex \int_0^T 
		\abs{\int_\T \tilde q_n^2(t)\,\d x}^p \, \d t
		\lesssim_T \tilde \Ex 
		\norm{\tilde q_n}_{L^\infty([0,T];L^2(\T))}^{2p}
		\overset{\text{Lemma \ref{thm:q_bounds_skorohod}}}{\lesssim} 1,
	\end{split}
\end{equation}
where we have used the BDG and H\"older inequalities, 
and similarly
\begin{equation}\label{eq:mart-prop-part2}
	\begin{aligned}
		&\tilde \Ex \sup_{t\in [0,T]} 
		\abs{\tilde{\mathcal{M}}(t)}^p
		\lesssim 
		\tilde \Ex \left[\left(\int_0^T 
		\abs{\tilde{\mathcal{J}}(t)}^2 
		\, \d t\right)^{p/2}\right]
		\\ & \quad 
		\lesssim_T \tilde \Ex \int_0^T 
		\abs{\int_\T \overline{q^2}(t)\,\d x}^p \, \d t
		\leq \tilde \Ex \int_{0}^{T} 
		\norm{\overline{q^2}(t)}_{H^{-1}(\T)}^p\, \d t
		\overset{\eqref{eq:Lr-HnegEst}}{\lesssim} 1,
	\end{aligned}
\end{equation}

Given \eqref{eq:as-conv-tmp1} 
and \eqref{eq:as-conv-tmp2}, 
Vitali's convergence theorem returns
$$
\tilde \Ex \int_0^T
\abs{\tilde{\mathcal{M}}_n(t)
-\tilde{\mathcal{M}(t)}}^2 \d t \ton 0.
$$ 
Passing to a subsequence, we conclude that
\begin{align}\label{eq:D_n(t)_to_0}
D_n(t)=\tilde \Ex \abs{\tilde{\mathcal{M}}_n(t)
-\tilde{\mathcal{M}}(t)}^2
\ton 0, \quad 
\text{for a.e.~in $t\in [0,T]$}.
\end{align}

We may assume that $D_n(t)\to 0$ 
for \textit{all} $t\in [0,T]$, as the function $D_n(t)$ depends 
continuously on $t\in [0,T]$, uniformly in $n$. 
Let us explain why. Through some straightforward 
manipulations and by utilising \eqref{eq:mart-prop-part2-tmp} 
and \eqref{eq:mart-prop-part2}, 
$$
\abs{D_n(t_2) - D_n(t_1)}^2
\lesssim  \tilde \Ex  \abs{\tilde{\mathcal{M}}_n(t_2)
-\tilde{\mathcal{M}}_n(t_1)}^2 
+ \tilde \Ex \abs{\tilde{\mathcal{M}}(t_2)
-\tilde{\mathcal{M}}(t_1)}^2.
$$
We will estimate the terms on the right separately. 
Suppose $S(v)=\frac12v^2$. The other cases 
$S(v)=S_\ell(v),\frac12v_\pm^2$ can be treated similarly.  
For $0\leq t_1<t_2\leq T$, the 
It\^{o} isometry implies
\begin{align*}
	& \tilde \Ex \abs{\tilde{\mathcal{M}}_n(t_2)
	-\tilde{\mathcal{M}}_n(t_1)}^2
	=\tilde \Ex \int_{t_1}^{t_2} 
	\abs{\tilde{\mathcal{J}}_n(t)}^2\, \d t
	\lesssim_{\sigma,\varphi}\tilde \Ex \int_{t_1}^{t_2} 
	\abs{\int_\T \tilde q_n^2(t)\,\d x}^2\, \d t
	\\ & \qquad 
	\leq \abs{t_2-t_1}
	\tilde \Ex \norm{\tilde q_n}_{L^\infty([0,T];L^2(\T))}^4
	\overset{\text{Lemma \ref{thm:q_bounds_skorohod}}}{\lesssim} 
	\abs{t_2-t_1}.
\end{align*}
Similarly, for any $p$ such that 
$2p\in \bigl(1,p_0/2\bigr]$ ($p_0>4$) and $1/p+1/p'=1$,
\begin{align*}
	&\tilde \Ex \abs{\tilde{\mathcal{M}}(t_2)
	-\tilde{\mathcal{M}}(t_1)}^2
	\lesssim_{\sigma,\varphi}\tilde \Ex \int_{t_1}^{t_2} 
	\norm{\overline{q^2}(t)}_{H^{-1}(\T)}^2\, \d t
	\\ & \qquad 
	\leq \abs{t_2-t_1}^{\frac{1}{p'}}
	\left(\tilde \Ex\int_0^T 
	\norm{\overline{q^2}(t)}_{H^{-1}(\T)}^{2p}
	\, \d t\right)^{1/p}
	\overset{\eqref{eq:Lr-HnegEst}}{\lesssim}
	\abs{t_2-t_1}^{\frac{1}{p'}}.
\end{align*}
From the above estimations, we can infer that $\abs{D_n(t)} \lesssim 1$ 
uniformly across $n \in \mathbb{N}$ and $t \in [0,T]$. 
Using the Arzelà--Ascoli theorem, we deduce that $D_n(t) \to D(t)$ 
uniformly for $t \in [0,T]$ along a subsequence, where 
$D\in C([0,T])$. According to \eqref{eq:D_n(t)_to_0}, the  
entire sequence must converge to $D \equiv 0$. This 
implies that the random variable $\tilde{M}_n=\tilde{\mathcal{M}}_n(T)$ 
satisfies \eqref{eq:S_equation_weak-VI}.
\end{proof}

Lemma \ref{thm:limiteq_1} applies to the linearly growing 
approximations $S=S_\ell(v_\pm)$ of $\frac12 v^2_\pm$ 
as well as $S=\frac12 v_\pm^2$ (and $S=\frac12v^2$). 
As explained in the introduction, for 
the deterministic CH equation \cite{Xin:2000qf}, the 
analysis relies on the use of $S_\ell(v_+)$ 
and a one-sided gradient bound (Oleinik-type estimate) 
to control the error 
that arises when replacing $v_+^2$ by $S_\ell(v_+)$. 
As one-sided gradient bounds are not available 
to us, we will insist on applying Lemma \ref{thm:limiteq_1} 
with $S=\frac12 v^2_+$ and then use some different ideas 
to control the defect measure. Exploiting the identities
\begin{align*}
	& S(v)=\frac12 v_+^2, \quad S'(v)=v_+,
	\quad 
	S''(v)=\mathds{1}_{\{v>0\}},
	\quad 
	H^{(1)}(v)=-\frac12 v_+^2,
	\\ &
	H^{(2)}(v)=0,
	\quad
	H^{(3)}(v)=-\frac12 v_+^2,
	\quad 
	S''(v) \,v^2=v_+^2,
\end{align*}
and the linearity of the weak 
limit (i.e., $\overline{a}+\overline{b}=\overline{a+b}$), 
the inequality \eqref{eq:S_equation_n_limit} 
with $S(v)=\frac12 v_+^2$ simplifies into
\begin{equation}\label{eq:tqp2-limit}
	\begin{aligned}
		&\int_0^T \int_\T \frac12\overline{q_+^2}
		\, \pd_t \varphi \,\d x \,\d t
		+\int_\T \frac12\bk{\pd_x\tilde{u}_0}_+^2
		\, \varphi(0,x)  \,\d x
		\\ & \quad \quad
		+\int_0^T \int_\T\left(\tilde u
		-\frac14 \pd_x \sigma^2\right) 
		\frac12\overline{q_+^2}\, \pd_x \varphi \,\d x \,\d t
		+\int_0^T \int_\T \frac14 \sigma^2\, \overline{q_+^2}\,
		\pd_{xx}^2 \varphi \,\d x \,\d t
		\\ & \quad\quad
		-\int_0^T \int_\T\Biggl[
		\, \overline{q_+}\bk{\tilde P-\tilde u^2}
		+\left(\frac14 \pd_{xx}^2\sigma^2
		-\abs{\pd_x \sigma}^2\right)
		\, \frac12\overline{q_+^2}\Biggr] \varphi \,\d x \,\d t
		\\ & \quad	\quad
		+\int_0^T \int_\T \bigl(\sigma \, \pd_x \varphi
		-\pd_x\sigma\,\varphi\bigr) 
		\frac12\overline{q_+^2}\,\d x\, \d \tilde W\geq 0,
		\quad \text{$\tilde{\mathbb{P}}$--a.s.},
	\end{aligned}
\end{equation}
for all non-negative $\varphi \in C^\infty_c([0,T)\times \T)$. 

\subsection{Renormalised equation for the weak limit $\tilde{q}$}
\label{subsec:renorm-tq}
According to Proposition \ref{thm:existence_H1}, the 
a.s.~limit $\tilde u$ from Proposition \ref{thm:skorohod} satisfies
\begin{equation}\label{eq:tu-limit}
	\begin{split}
		& 0 =\d \tilde u+\bigl[\tilde u\, \pd_x \tilde u
		+\pd_x \tilde P \bigr] \, \d t 
		-\frac12\sigma \pd_x \bk{\sigma \pd_x \tilde u}\,\d t
		+\sigma \pd_x \tilde u \, \d \tilde W,
	\end{split}
\end{equation}
weakly in $x$, almost surely, where 
$-\pd_{xx}^2\tilde P+\tilde P
=\tilde u^2+\frac12\overline{q^2}$. 
By Lemma \ref{thm:lim_iden_simple}, 
the a.s.~limit $\tilde q$ of Proposition \ref{thm:skorohod} 
satisfies $\tilde q= \pd_x \tilde u$ weakly.  
Differentiating \eqref{eq:tu-limit} 
with respect to $x$, we thus obtain the SPDE
\begin{equation}\label{eq:tq-limit-SPDE}
	\begin{split}
		 0 & = \d \tilde q 
		 +\left(\pd_x \bk{\tilde u \tilde q}
		 -\frac12 \overline{q^2}
		 +\tilde P-\tilde u^2 \right) \,\d t
		 \\ & \qquad \qquad
		 -\frac12 \pd_x \bk{\sigma \pd_x\bk{\sigma \tilde q}} \,\d t
		 +\pd_x\bk{ \sigma \tilde q}\,\d \tilde W.
	\end{split}
\end{equation}

Consider a linearly growing 
$S\in W^{2,\infty}_{\operatorname{loc}}(\R)$ 
(of the type considered before). 
Note that, thanks to \eqref{eq:basic-identities},
\begin{align*}
	S'(\tilde q)\left(\pd_x \bk{\tilde u \tilde q}
	-\frac12 \overline{q^2}\right)
	= \pd_x \bk{\tilde u S(\tilde q)}
	-H^{(2)}(\tilde q)
	-\frac12 S'(\tilde q)
	\left(\overline{q^2}-\tilde q^2\right),
\end{align*}
where $H^{(2)}$ is defined in \eqref{eq:H-def}. 
Formally applying It\^{o}'s formula 
to \eqref{eq:tq-limit-SPDE} as 
in \eqref{eq:SPDE-S(qe)}, expressing the 
temporal differential as a time-derivative 
in $\mathcal{D}'([0,T))$, we obtain
\begin{equation}\label{eq:SPDE-S(tq)}
	\begin{split}
		0=\pd_t S(\tilde q)
		& +\pd_x \left[\tilde u\, S(\tilde q)
		+\frac14 \pd_x \sigma^2 
		\, H^{(1)}(\tilde q) \right] 
		-\pd_{xx}^2\left[\frac12\sigma^2S(\tilde q)\right]
		\\ &  
		+\Biggl[ S'(\tilde q) \bk{\tilde P-\tilde u^2}-H^{(2)}(\tilde q)
		-\frac12 S'(\tilde q)\left(\overline{q^2}-\tilde q^2\right)
		\\ & \qquad \qquad\quad
		-\frac14 \pd_{xx}^2\sigma^2 \, H^{(3)}(\tilde q)
		-\frac12 \abs{\pd_x \sigma}^2
		\, S''(\tilde q) \, \tilde q^2\Biggr] 
		\\ & 
		+\Bigl[\pd_x \bigl(\sigma \, S(\tilde q)\bigr)
		-\pd_x\sigma \, H^{(3)}(\tilde q)\Bigr]\, \dot{\tilde W}, 
		\quad \text{in $\mathcal{D}'([0,T)\times \T)$, a.s.},
	\end{split}	
\end{equation}
with initial data $S(\tilde q)(0)=S(\pd_x \tilde u_0)$, 
here relying crucially on Lemma \ref{thm:temp_cont_0} 
(strong right-continuity at $t=0$). 
Note carefully that $S$ is assumed linearly growing in order 
to make sense to the product $S'(\tilde q)\bigl(\overline{q^2}
-\tilde q^2\bigr)$, given the meagre 
integrability \eqref{eq:weak-conv-tilde}. 
This excludes the functions $S(v)=\frac12v^2,\frac12v_\pm^2$ 
allowed by Lemma \ref{thm:limiteq_1}.

The processes $\tilde u$ and $\tilde q$ 
appearing in \eqref{eq:SPDE-S(tq)} exhibit limited 
regularity. Specifically, $\tilde q$ does not belong to 
any spatial Sobolev space, as the second-order 
part of the SPDE \eqref{eq:tq-limit-SPDE} does not manifest 
``parabolic regularity". The rigorous derivation 
of \eqref{eq:SPDE-S(tq)} is therefore quite involved: it relies 
on the regularisation (by convolution) method and 
the real-valued It\^{o} formula, along with 
non-standard DiPerna--Lions estimates 
to control the regularisation error linked to the second 
order operator and the martingale part of 
the equation \eqref{eq:tq-limit-SPDE}, see 
Appendix \ref{sec:commutator-est} for details 
and Section \ref{sec:intro} for some relevant references.

\begin{lem}[renormalisation of limit SPDE]
\label{thm:limiteq_2}
Denote by $S(v)$ any one of the functions $S_\ell(v_\pm)$ 
defined by \eqref{eq:entropies_S_ell}. 
Let $\tilde u$, $\tilde q=\pd_x \tilde u$, 
cf.~Lemma \ref{thm:lim_iden_simple}, and $\overline{q^2}$ 
be the Skorokhod--Jakubowski representations 
from Proposition \ref{thm:skorohod}, see also 
Remark \ref{rem:weak-notation}, and 
$H^{(1)},H^{(2)},H^{(3)}$ be the functions 
defined in \eqref{eq:H-def} with $S(v)=S_\ell(v_\pm)$. 
The SPDE \eqref{eq:SPDE-S(tq)} holds weakly 
in $(t,x)$, almost surely, that is, 
$\tilde{\mathbb{P}}$--a.s.,
\begin{equation}\label{eq:S_equation_n_limit2}
	\begin{aligned}
		&\int_0^T \int_\T S(\tilde q)\, \pd_t \varphi \,\d x \,\d t
		+\int_\T S(\pd_x\tilde{u}_0)\, \varphi(0,x)  \,\d x
		\\ & \quad 
		+ \int_0^T \int_\T\left[\tilde u\, S(\tilde q)
		+\frac14 \pd_x \sigma^2 \, H^{(1)}(\tilde q)\right] 
		\pd_x \varphi \,\d x \,\d t
		\\ & \quad  
		+\int_0^T \int_\T \frac12 \sigma^2\, S(\tilde q)\,
		\pd_{xx}^2 \varphi \,\d x \,\d t
		\\ & \quad
		-\int_0^T \int_\T\Biggl[S'(\tilde q) 
		\bk{\tilde P-\tilde u^2}-\, H^{(2)}(\tilde q)
		-\frac12 S'(\tilde q)\left(\overline{q^2}-\tilde q^2\right)
		\\ & \qquad \qquad\qquad\quad
		-\frac14 \pd_{xx}^2\sigma^2 \, H^{(3)}(\tilde q)
		-\frac12\abs{\pd_x \sigma}^2 S''(\tilde{q})\, 
		\tilde q^2\Biggr] \varphi \,\d x \,\d t
		\\ & \quad	
		+\int_0^T \int_\T \sigma \, S(\tilde q)\, \pd_x \varphi
		+\pd_x\sigma \, H^{(3)}(\tilde q)\,\varphi
		\,\d x\, \d \tilde W = 0,
	\end{aligned}
\end{equation}
for all $\varphi \in C^\infty_c([0,T)\times \T)$, 
where $\tilde{P}$ is defined in \eqref{eq:tPn-tP-def}.
\end{lem}

\begin{proof}
Let $J_\delta$ be a standard Friedrichs mollifier on $\T$. 
For $f \in L^p(\T)$, write $f_\delta=J_\delta*f$. 
Mollifying the limit SPDE 
\eqref{eq:tu-limit} against $J_\delta$, we obtain
\begin{equation}\label{eq:mollifying_example}
	\begin{aligned}
		0 &=\d \tilde{u}_{\delta} 
		+\tilde u_{\delta} \tilde q_{\delta} \,\d t 
		+E^{(1)}_\delta\,\d t 
		+\pd_x K*\bk{\tilde{u}^2
		+\frac12 \overline{q^2}}*J_\delta \,\d t   
		\\ &\qquad \quad 
		-\frac12 \sigma  \pd_x \bk{\sigma \tilde{q}_{\delta}} \,\d t
		+E^{(3)}_\delta\,\d t +\left[\sigma \tilde{q}_{\delta}
		+E^{(2)}_\delta\right]\,\d \tilde{W},
	\end{aligned}
\end{equation}
where $E^{(1)}_\delta$, $E^{(2)}_\delta$, $E^{(3)}_\delta$ 
denote the following convolution error terms:
\begin{equation*}
	\begin{aligned}
		&E^{(1)}_\delta 
		=\bk{\tilde{u}\tilde{q}}*J_\delta 
		-\tilde{u}_\delta\tilde{q}_\delta,
		\quad 
		E^{(2)}_\delta 
		=\bigl( \sigma\tilde{q}\bigr)*J_\delta
		-\sigma\tilde{q}_{\delta},
		\\
		& E^{(3)}_\delta
		=-\frac12 \bigl(\sigma \pd_x\bk{\sigma \tilde{q}}\bigr)*J_\delta
		+\frac12 \sigma\,\pd_x\bk{\sigma\tilde{q}_{\delta}}.
	\end{aligned}
\end{equation*}

Next, differentiating \eqref{eq:mollifying_example}
with respect to $x$, we arrive at 
\begin{equation}\label{eq:mollifying_example-q}
	\begin{aligned}
		0 = \d \tilde{q}_\delta 
		& +\left[\pd_x \bk{\tilde{u}_\delta\tilde{q}_\delta} 
		+K*\bk{\tilde{u}^2
		+\frac12 \overline{{q}^2}}*J_\delta 
		-\bk{\tilde{u}^2 + \frac12\overline{{q}^2}}*J_\delta 
		\right]\,\d t 
		\\ & 
		-\frac12 \pd_x\bigl(\sigma \pd_x \bk{\sigma \tilde{q}_\delta}
		\bigr)\,\d t 
		+\pd_x \bk{\sigma \tilde{q}_\delta}\,\d \tilde{W}
		\\ &
		+\pd_x E^{(1)}_\delta\,\d t
		+\pd_x E^{(2)}_\delta \,\d \tilde{W} 
		+\pd_x E^{(3)}_\delta\,\d t.
	\end{aligned}
\end{equation}

Consider $S(v) = S_\ell(v_\pm)$ as 
in the lemma. Given \eqref{eq:mollifying_example-q}, 
applying the standard It\^o formula to $S(\tilde{q}_\delta)$,
as in \eqref{eq:SPDE-S(tq)} 
(cf.~also \eqref{eq:basic-identities}
and \eqref{eq:SPDE-S(qe)}), we obtain
\begin{equation}\label{eq:S_equation_limit}
	0 =\int_0^T\int_\T S(\tilde{q}_\delta)
	\pd_t \varphi \,\d x \,\d t
	+ \int_\T S(\tilde{q}_\delta(0))
	\varphi(0,x)\,\d x
	+\sum_{i=1}^6 I^{(i)}_{\delta}, 
\end{equation}
where
\begin{align*}
	I_\delta^{(1)} & 
	= \int_0^T \int_\T
	\bk{H^{(2)}(\tilde{q}_\delta)
	+\frac14 \,\pd_{xx}^2\sigma^2 H^{(3)}(\tilde{q}_\delta)}
	\varphi \,\d x\,\d t 
	\\ &\qquad 
	+ \int_0^T \int_\T 
	\bk{\tilde{u}_\delta S(\tilde{q}_\delta)
	+\frac14 \pd_x \sigma^2 H^{(1)}(\tilde{q}_\delta)}
	\pd_x \varphi \,\d x\,\d t
	\\ &\qquad 
	-\frac12 \int_0^T \int_\T  
	\sigma^2\, S(\tilde{q}_\delta)
	\pd_{xx}^2 \varphi \, \d x \,\d t,
	\\ 
	I_\delta^{(2)} & 
	=-\int_0^T \int_\T S'(\tilde{q}_\delta) 
	\overline{I}_\delta^{(2)}
	\varphi \,\d x\,\d t,
	\\
	I_\delta^{(3)} & 
	=\frac12\int_0^T\int_\T \abs{\pd_x \sigma}^2
	S''(\tilde{q}_\delta)\tilde{q}_\delta^2 
	\varphi  \,\d x\, \d t,
	\\
	I^{(4)}_\delta & 
	=\int_0^T \int_\T \sigma S(\tilde{q}_\delta)\pd_x\varphi 
	+\pd_x\sigma H^{(3)}(\tilde{q}_\delta) \varphi
	\, \d x \,\d \tilde{W},
	\\
	I_\delta^{(5)} &
	=\int_0^T \int_\T -\varphi 
	S'(\tilde{q}_\delta)\pd_x E^{(3)}_\delta 
	+\varphi S''(\tilde{q}_\delta)
	\overline{I}_\delta^{(5)}\,\d x\,\d t
	\\ &\qquad
	-\int_0^T \int_\T \varphi S'(\tilde{q}_\delta)
	\pd_x E^{(1)}_\delta \,\d x\,\d t,
	\\
	I^{(6)}_\delta &
	=-\int_0^T \int_\T \varphi S'(\tilde{q}_\delta)
	\pd_x E^{(2)}_\delta\,\d x \,\d \tilde{W},
\end{align*}
and
\begin{align*}
	&\overline{I}_\delta^{(2)}
	=K*\bk{\tilde{u}^2
	+\frac12 \overline{q^2}}*J_\delta
	-\bk{\tilde{u}^2
	+\frac12 \overline{q^2}}*J_\delta + \frac12 \tilde{q}_\delta^2,
	\\ & 
	\overline{I}_\delta^{(5)}
	=\bk{\pd_x E^{(2)}_\delta\pd_x \bk{\sigma \tilde{q}_\delta}
	+\frac12 \abs{\pd_xE^{(2)}_\delta}^2}.
\end{align*}
In deriving $\overline{I}_\delta^{(2)}$, we used the 
fact that $K$ is the Green's function of $1 - \pd_{xx}^2$ on $\T$.

Denote by $I^{(i)}$ the expression corresponding 
to formally taking $\delta$ to zero 
in $I^{(i)}_\delta$, $i=1,\ldots,6$, 
and the same for $\overline{I}^{(2)}$, 
$\overline{I}^{(5)}$.

Recall that
$$
\text{$\tilde q\in L^{p_0}_{\tilde \omega}L^\infty_tL^2_x
\cap L^{2r}_{\tilde \omega,t,x}$},
\quad 
\text{$\overline{q^2}\in L^r_{\tilde \omega,t,x}$},
\quad 
\text{$\tilde u\in L^{p_0}_{\tilde \omega}L^\infty_{t,x}$},
$$
where $r\in [1,3/2)$ and $p_0>4$, 
see Lemmas \ref{thm:u_in_Linfty}, \ref{thm:P-u2_bound-tilde}, 
and \ref{thm:P-u2_weak}. By a standard property of mollifiers, 
$\tilde{q}_\delta \todelta \tilde{q}$ 
a.e.~in $\tilde \Omega\times [0,T]\times \T$. 
Denote by $f(v)$ any one of the nonlinear functions $S(v)$, $S'(v)$, 
$S''(v) v^2$, $H^{(1)}(v)$, $H^{(2)}(v)$, $H^{(3)}(v)$, 
where we recall that $S(v)=S_\ell(v_\pm)$, 
cf.~\eqref{eq:entropies_S_ell}, 
and $H^{(1)},H^{(2)},H^{(3)}$ are defined 
in \eqref{eq:H-def} with $S(v)=S_\ell(v_\pm)$. 
By the continuity of $f(v)$, 
\begin{equation}\label{eq:f-ae-conv}
	f(\tilde{q}_\delta) \todelta f(\tilde{q})
	\quad  \text{a.e.~in $\tilde \Omega
	\times [0,T]\times \T$}.	
\end{equation}

We have the bound
$$
\norm{f(\tilde{q}_\delta)}_{L^\infty_{t,x}} 
\lesssim_\ell 1, 
\quad 
f(v) = S'(v),\, S''(v)v^2, 
$$
and thus, by 
\eqref{eq:f-ae-conv} and Vitali's convergence theorem,
\begin{equation}\label{eq:strongconvg_delta0}
	f(\tilde{q}_\delta) \todelta f(\tilde{q}) 
	\quad \text{in $L^p_{t,x}$, a.s.},
\end{equation}
for any $1\le p < \infty$.

Similarly, by the bounds 
$$
\norm{f(\tilde{q}_\delta)}_{L^\infty_tL^2_x},\, 
\norm{f(\tilde{q}_\delta)}_{L^{2r}_{t,x}}
\lesssim_{\ell,\tilde \omega} 1, 
\quad 
f(v)=S(v),\, H^{(1)}(v),\, H^{(3)}(v), 
$$
we have the a.s.~convergences 
\begin{equation}\label{eq:strongconvg_delta1}
	f(\tilde{q}_\delta) \todelta f(\tilde{q})
	\quad \text{in $L^{p_1}_tL^{p_2}_x$ 
	and in $L^{p}_{t,x}$}, 
\end{equation}
for any $1\le p_1<\infty$, $1 \le p_2 < 2$, 
and $1\leq p< 2r$. In addition, since $f(\tilde{q}_\delta)$ 
is bounded in $L^{2r}_{\tilde \omega,t,x}$ (and $2r>2$),
\begin{equation}\label{eq:strongconvg_delta1-new}
	f(\tilde{q}_\delta) \todelta f(\tilde{q})
	\quad \text{in $L^2_{\tilde \omega,t,x}$}. 
\end{equation}

Next, from the bound
$$
\norm{f(\tilde{q}_\delta)}_{L^r_{t,x}}
\lesssim_{\ell,\tilde \omega} 1, 
\quad 
f(v)=v^2,\, H^{(2)}(v),
$$
we obtain the convergence
\begin{equation}\label{eq:strongconvg_delta2}
	f(\tilde{q}_\delta) \todelta f(\tilde{q})
	\quad \text{in $L^{p}_{t,x}$, a.s.},
\end{equation}
for any $1\leq p <r$.

Since $\overline{q^2}\in L^{r}_{t,x}$ a.s., we 
have the convergence
\begin{equation}\label{eq:strongconvg_delta3}
	\overline{q^2}*J_\delta \todelta \overline{q^2}
	\quad\text{in $L^{p}_{t,x}$, a.s.} 
\end{equation}

Finally, since $\tilde{u}\in L^\infty_{t,x}$, a.s., 
\begin{equation}\label{eq:strongconvg_delta4}
	\tilde{u}_\delta \todelta \tilde{u},
	\quad \tilde{u}^2*J_\delta \todelta \tilde{u}^2 
	\quad \text{in $L^p_{t,x}$, a.s., for any 
	$p\in [1,\infty)$.}
\end{equation}

\smallskip
\noindent \textit{1. The first two terms in \eqref{eq:S_equation_limit}.}
\smallskip

By \eqref{eq:strongconvg_delta1} and $\pd_t \varphi 
\in L^\infty_{t,x}$,
\begin{align*}
	\int_0^T \int_\T 
	S(\tilde{q}_\delta) \pd_t \varphi
	\,\d x \,\d t
	\todelta  
	\int_0^T \int_\T 
	S(\tilde{q})\pd_t \varphi 
	\,\d x \,\d t, 
	\quad \text{a.s.}
\end{align*}

Regarding the initial term,
$$
\int_\T S(\tilde{q}_\delta(0))\varphi(0,x)\,\d x
=\int_\T S(\pd_x \tilde{u}_0)\varphi(0,x)\,\d x
+A+B_\delta,
$$
where
\begin{align*}
	& A=\int_\T \bigl(S(\tilde{q}(0))
	-S(\pd_x \tilde{u}_0)\varphi(0,x)\,\d x,
	\\ & 
	B_\delta =\int_\T \bigl(S(\tilde{q}_\delta(0))
	-S(\tilde{q}(0))\bigr)\varphi(0,x)\,\d x.
\end{align*}
By Lemma \ref{thm:temp_cont_0} (strong 
initial trace of $\tilde q=\pd_x \tilde u$ 
in $L^2$), we have that $A=0$. 
As in \eqref{eq:strongconvg_delta1}, 
$S(\tilde{q}_\delta(0))\todelta S(\tilde{q}(0))$ 
in $L^1_x$ a.s., and hence we easily deduce that 
$B_\delta \todelta 0$; accordingly,
$$
\int_\T S(\tilde{q}_\delta(0))\varphi(0,x)\,\d x
\todelta\int_\T S(\pd_x \tilde{u}_0)
\varphi(0,x)\,\d x, 
\quad \text{a.s.}
$$

\smallskip
\noindent \textit{2. The term $I^{(1)}_\delta$.}
\smallskip

Since $\varphi, \pd_x \varphi, \pd_{xx}^2 \varphi, 
\sigma^2, \pd_x \sigma^2, \pd_{xx}^2 \sigma^2 
\in L^\infty_{t,x}$, the convergences 
\eqref{eq:strongconvg_delta1} and 
\eqref{eq:strongconvg_delta2} ensure that
$I^{(1)}_\delta  \todelta I^{(1)}$ a.s.

\smallskip
\noindent \textit{3. The term $I^{(2)}_\delta$.}
\smallskip

In view of Young's convolution inequality 
and the convergences \eqref{eq:strongconvg_delta2}, 
\eqref{eq:strongconvg_delta3}, and 
\eqref{eq:strongconvg_delta4}, we obtain 
$\overline{I}_\delta^{(2)} \todelta \overline{I}^{(2)}$ 
in $L^p_{t,x}$, a.s., for any $p\in [1,r)$, where
$$
\overline{I}^{(2)}
=K*\bk{\tilde{u}^2+\frac12 \overline{q^2}}
-\tilde{u}^2- \frac12\bk{\overline{q^2}-\tilde{q}^2}.
$$
Moreover, by \eqref{eq:strongconvg_delta0}, 
$S'(\tilde{q}_\delta) \todelta S'(\tilde{q})$ 
in $L^{p'}_{t,x}$, a.s., where $\frac{1}{p}+\frac{1}{p'}=1$. 
Consequently, as $\varphi\in L^\infty_{t,x}$, 
it follows that $I^{(2)}_\delta \todelta I^{(2)}$, a.s.

\smallskip
\noindent \textit{4. The term $I^{(3)}_\delta$.}
\smallskip

Using \eqref{eq:strongconvg_delta0}, recalling that 
$\frac12 \abs{\pd_x \sigma}^2 \varphi  \in L^\infty_{t,x}$,
$$
I^{(3)}_\delta \todelta 
\frac12 \int_0^T \int_\T  
\abs{\pd_x \sigma}^2 S''(\tilde{q}) 
\tilde{q}^2\varphi\,\d x \, \d s.
$$

\smallskip
\noindent \textit{5. The term $I^{(4)}_\delta$.}
\smallskip
By the It\^o isometry and the 
Cauchy--Schwarz inequality,
\begin{multline*}
	\tilde{\Ex}\abs{I^{(4)}_\delta-I^{(4)}}^2
	\leq \tilde{\Ex}\int_0^T \int_\T
	\abs{\sigma\pd_x \varphi}^2 
	\abs{S(\tilde{q}_\delta)-S(\tilde{q})}^2 \,\d x\,\d t
	\\ +\tilde{\Ex} \int_0^T \int_\T 
 	\abs{\varphi\pd_x\sigma}^2 
 	\abs{H^{(3)}(\tilde{q}_\delta)
 	-H^{(3)}(\tilde{q})}^2 \, \d x \,\d t.
\end{multline*}
Both these integrals tend to nought by 
\eqref{eq:strongconvg_delta1-new}. 
Therefore, along a subsequence 
$\delta=\delta_j \downarrow 0$ as $j\to\infty$, 
$I^{(4)}_\delta \toj  I^{(4)}$, a.s.

\smallskip
\noindent \textit{6. The terms $I^{(5)}_\delta$ 
and $I^{(6)}_\delta$.}
\smallskip

Since $\varphi\in L^\infty_{t,x}$, 
$\abs{S'(\cdot)}\lesssim_\ell 1$, and 
$\abs{S''(\cdot)}\leq 1$, 
Lemma \ref{thm:commutator1} and 
Proposition \ref{thm:commutator2} allow us 
to directly conclude that $I^{(5)}_\delta 
\todelta I^{(5)}$, a.s., along a 
subsequence $\delta=\delta_j\downarrow 0$ 
as $j\to \infty$. Similarly, again invoking the It\^o isometry 
and the Cauchy--Schwarz inequality,
\begin{align*}
	\tilde{\Ex} \abs{I^{(6)}_\delta}^2 
	\leq \tilde{\Ex} \int_0^T\int_\T
	\abs{\varphi S'(\tilde{q}_\delta)}^2
	\abs{\pd_x E^{(2)}_\delta}^2 
	\,\d x \,\d t \todelta 0,
\end{align*}
by Lemma \ref{thm:commutator1}.
So along a subsequence $\delta=\delta_j  \downarrow 0$ 
as $j\to \infty$, $I^{(6)}_\delta \to 0$, a.s.

This concludes the proof of the lemma.
\end{proof}

In the remaining part of this section, we will make 
hefty use of the formulas gathered in the next remark.

\begin{rem}\label{rem:pos-neg-parts}
Recall the formulas \eqref{eq:entropies_S_ell}--\eqref{eq:entrop2} 
involving $S_\ell(v)$ and $S_\ell(v_\pm)$. 
The following identities are straightforward to verify:
\begin{align*}
	& S_\ell(v_+) 
	=\frac12 v_+^2
	-\frac1{6\ell} \bk{v-\ell}^3\one{\ell<v<2\ell} 
	-\frac16\bk{3v^2-9\ell v+7\ell^2}\one{v \geq 2\ell},
	\\
	& S_\ell(v_+)'
	= v_+-\frac1{2\ell}\bk{v-\ell}^2\one{\ell<v<2\ell} 
	+\frac12\bk{3\ell-2v}\one{v\geq 2\ell},
	\\ 
	& S_\ell(v_+)''
	=\one{0<v<2\ell}-\frac1\ell\bk{v-\ell}
	\one{\ell<v<2 \ell},
	\\
	& S_\ell(v_+)-S_\ell(v_+)'v
	=-\frac12 v_+^2
	+\frac1{6\ell} \bk{2v^3-3 \ell v^2+\ell^3}
	\one{\ell<v< 2\ell}
	\\ & \qquad \qquad\qquad \qquad\qquad
	+\frac16\bk{3v^2-7\ell^2}\one{v \geq 2 \ell},
	\\ 
	& 3S_\ell(v_+)-2S_\ell(v_+)' v
	=-\frac12 v^2_+
	+\frac1{2\ell}\bk{v^3-\ell v^2-\ell^2v+\ell^3}
	\one{\ell<v<2 \ell}
	\\ & \qquad \qquad\qquad \qquad\qquad
	+\frac12 \bk{v^2+3\ell v-7\ell^2}\one{v \ge 2 \ell},
	\\
	& S_\ell(v_+)v - \frac12S_\ell(v_+)'v^2 
	= \frac1{12\ell}\bk{v^4-3\ell^2 v^2+2\ell^3 v}
	\one{\ell<v<2\ell}
	\\ & \qquad \qquad\qquad \qquad\qquad
	+\frac1{12}\bk{9\ell v^2-14\ell^2v}
	\one{v \geq 2\ell},
	\\ 
	& \frac12 S_\ell(v_+)''v^2
	=\frac12 v^2_+- \frac1{2\ell} v^2\bk{v-\ell}
	\one{\ell<v<2\ell}-\frac12 v^2\one{v\geq 2\ell},
\end{align*}
and
\begin{align*}
	& S_\ell(v_-)
	=\frac12 v_-^2
	+\frac1{6\ell}\bk{v+\ell}^3
	\one{-2\ell < v < -\ell} 
	-\frac16\bk{3v^2+9\ell v+7\ell^2}
	\one{v\leq-2\ell},
	\\
	& S_\ell(v_-)'
	= v_-	+\frac1{2\ell}\bk{v+\ell}^2 
	\one{-2\ell<v<-\ell}
	-\frac12 \bk{3\ell+2v}\one{v\leq-2\ell},
	\\ &
	S_\ell(v_-)''
	=\one{-2\ell<v<0}+\frac1\ell\bk{v+\ell}
	\one{-2\ell<v<-\ell},
	\\ & 
	S_\ell(v_-)-S_\ell(v_-)' v 
	=-\frac12 v_-^2
	+\frac1{6\ell} \bk{-2v^3-3\ell v^2+\ell^3}
	\one{-2\ell<v<-\ell}
	\\ & \qquad \qquad\qquad \qquad\qquad 
	+\frac16\bk{3v^2-7\ell^2}
	\one{v \leq -2 \ell},
	\\ & 3S_\ell(v_-)-2S_\ell(v_-)'v 
	=-\frac12 v^2_-
	+\frac1{2\ell}\bk{-v^3-\ell v^2+\ell^2 v+\ell^3}
	\one{-2\ell<v<-\ell}
	\\ & \qquad\qquad\qquad\qquad\qquad 
	+\frac12 \bk{v^2-3\ell v-7\ell^2}
	\one{v \leq -2 \ell},
	\\ & 
	S_\ell(v_-)v-\frac12S_\ell(v_-)'v^2
	=-\frac1{12\ell}\bk{v^4-3\ell^2 v^2-2\ell^3 v}
	\one{-2\ell<v<-\ell} 
	\\ &\qquad\qquad\qquad\qquad\qquad
	-\frac1{12}\bk{9\ell v^2+14\ell^2 v}
	\one{v\leq-2\ell},
	\\ & 
	\frac12 S_\ell(v_-)'' v^2
	=\frac12 v^2_- +\frac1{2\ell} v^2
	\bk{v+\ell}\one{-2\ell<v<-\ell}
	-\frac12v^2\one{v\leq-2\ell}.
\end{align*}
\end{rem}

Lemma \ref{thm:limiteq_2} applies to the linearly growing 
approximations $S_\ell(v_\pm)$ of $v^2_\pm$, but not 
the functions $v^2_\pm$ themselves. However, by exploiting 
some structural property of the SPDE \eqref{eq:SPDE-S(tq)}, 
we will be able to write an SPDE---up to an inequality---for the 
positive part $\tilde q_+^2$. 
Together with \eqref{eq:tqp2-limit}, this 
observation makes it possible to control 
the positive part $\frac12\bk{\overline{q_+^2}-
\tilde q_+^2}$ of the defect measure \eqref{eq:intro-defect}, 
without counting on a one-sided gradient estimate 
(available in the deterministic 
case \cite{Xin:2000qf} but not here).

\begin{lem}[characterisation of $\tilde q_+^2$]
\label{thm:limiteq_2_positive_ell_infty}
Let $\tilde u$, $\tilde q=\pd_x \tilde u$, 
cf.~Lemma \ref{thm:lim_iden_simple}, and $\overline{q^2}$ 
be the Skorokhod--Jakubowski representations 
from Proposition \ref{thm:skorohod}, see also 
Remark \ref{rem:weak-notation}. Then, for any 
nonnegative $\varphi \in C^\infty_c([0,T)\times \T)$,
\begin{equation}\label{eq:S_equation_n_limit3}
	\begin{aligned}
		&\int_0^T \int_\T \frac12\tilde q_+^2
		\, \pd_t \varphi \,\d x \,\d t
		+\int_\T \frac12\bk{\pd_x\tilde{u}_0}_+^2
		\, \varphi(0,x)  \,\d x
		\\ & \quad
		+ \int_0^T \int_\T\left(\tilde u
		-\frac14 \pd_x \sigma^2 \right) 
		\frac12\tilde q_+^2 \, \pd_x \varphi \,\d x \,\d t
		-\int_0^T \int_\T \frac12 \sigma^2\, 
		\frac12\tilde q_+^2
		\,\pd_{xx}^2 \varphi \,\d x \,\d t
		\\ & \quad
		-\int_0^T \int_\T\Biggl[
		\tilde q_+\bk{\tilde P-\tilde u^2}
		+\left(\frac14 \pd_{xx}^2\sigma^2
		-\abs{\pd_x \sigma}^2\right)
		\, \frac12\tilde q_+^2\Biggr] \varphi \,\d x \,\d t
		\\ & \quad
		+\int_0^T \int_\T \bigl(\sigma \, \pd_x \varphi
		-\pd_x\sigma\,\varphi\bigr)\frac12\tilde q_+^2
		\,\d x\, \d \tilde W\leq 0, 
		\quad \text{$\tilde{\mathbb{P}}$--a.s.},
	\end{aligned}
\end{equation}
where $\tilde{P}$ is defined in \eqref{eq:tPn-tP-def}.
\end{lem}

\begin{proof}
Denote the left-hand side of \eqref{eq:S_equation_n_limit3} 
by $I+M$, where $M$ is the stochastic integral term. 
We will demonstrate that
\begin{equation}\label{eq:qp2-claim}
	\int_{\tilde \Omega} \mathds{1}_{A}(\tilde \omega)
	\bigl(I(\tilde\omega)+M(\tilde\omega)\bigr)
	\, \d \tilde{\mathbb{P}}(\tilde \omega)\leq 0,
\end{equation}
for any measurable set $A\in \tilde{\mathcal{F}}$. 

Given \eqref{eq:S_equation_n_limit2} with 
$S(v)=S_\ell(v_+)$, observe that
\begin{equation}\label{eq:temporarytag}
	H^{(2)}(\tilde q)=S(\tilde q)\tilde q
	-\frac12 S'(\tilde q)\tilde q^2\geq 0, 
	\qquad
	S'(\tilde{q})\bk{\overline{{q}^2}-\tilde{q}^2}\ge 0,
\end{equation}
using \eqref{eq:H-def}, 
Remark \ref{rem:pos-neg-parts} and the weak 
convergence $\tilde q_n^2\tonweak \overline{q^2}$ 
in $L^r_{t,x}$ a.s., cf.~\eqref{eq:weak-conv-tilde} 
(the weak convergence implies that 
$\tilde{q}^2\leq \overline{{q}^2}$). 
In addition,
\begin{equation}\label{eq:temporarytag2}
	\begin{split}
		S(\tilde q) 
		& = \frac12\tilde q_+^2+e^{(1)}_\ell(t,x),
		\\ 
		H^{(1)}(\tilde q)=3S(\tilde q)-2S'(\tilde q)\tilde q
		& =-\frac12 \tilde q_+^2+ e^{(2)}_\ell(t,x),
		\\ 
		S'(\tilde q)
		& = \tilde q_+ +e^{(3)}_\ell(t,x),
		\\ 
		H^{(3)}(\tilde q)=S(\tilde q)-S'(\tilde q)\tilde q
		& =-\frac12 \tilde q_+^2+e^{(4)}_\ell(t,x),
		\\ 
		\frac12 S''(\tilde q) \tilde{q}^2
		& = \frac12 \tilde{q}^2_+ +  e^{(5)}_\ell(t,x),
	\end{split}
\end{equation}
where
\begin{align*}
	& e^{(1)}_\ell 
	=-\frac1{6\ell}\bk{\tilde{q}-\ell}^3
	\one{\ell<\tilde{q}<2\ell}
	-\frac16\bk{3\tilde{q}^2 
	-9\ell \tilde{q}+7\ell^2}
	\one{\tilde{q} \ge 2\ell},
	\\ &
	e^{(2)}_\ell 
	=\frac1{2\ell}\bk{\tilde{q}^3-\ell\tilde{q}^2 
 	-\ell^2\tilde{q}+\ell^3}\one{\ell<\tilde{q}<2 \ell}
 	+\frac12 \bk{\tilde{q}^2 
 	+3\ell\tilde{q}-7\ell^2}
 	\one{\tilde{q}\geq 2\ell},
 	\\ & 
 	e^{(3)}_\ell =-\frac1{2\ell}\bk{\tilde{q}-\ell}^2 
 	\one{\ell<\tilde{q}<2\ell}
 	+\frac12 \bk{3\ell-2\tilde{q}}
 	\one{\tilde{q} \ge 2 \ell},
 	\\ & 
 	e^{(4)}_\ell =\frac1{6\ell}\bk{2\tilde{q}^3 
	-3\ell\tilde{q}^2+\ell^3}\one{\ell<\tilde{q}<2\ell}
	+\frac16\bk{3\tilde{q}^2
	-7\ell^2} \one{\ell\ge 2\ell},
	\\ & 
	e^{(5)}_\ell=-\frac1{2\ell}\tilde{q}^2
	\bk{\tilde{q}-\ell}
	\one{\ell<\tilde{q}<2\ell}
	-\frac12\tilde{q}^2\one{\tilde{q}\geq 2\ell}.
\end{align*}
By \eqref{eq:tu-Lp-Linfty-H1}, 
$\tilde{q} \in L^{p_0}\bigl(\tilde{\Omega};
L^\infty([0,T];L^2(\T))\bigr)$, where $p_0>4$ is 
fixed in Theorem \ref{thm:main}. As a result, 
the error terms converge to zero in the sense that
\begin{equation}\label{eq:eiell-conv-Ii}
	\begin{split}
		& \abs{e^{(i)}_\ell}\lesssim 
		\mathds{1}_{\{\tilde q>\ell\}} \tilde q^2
		\toell 0 \quad 
		\text{a.e.~in $(\tilde \omega,t,x)$, 
		\, $i=1,2,4,5$},
		\\ &
		\abs{e^{(3)}_\ell}\lesssim 
		\mathds{1}_{\{\tilde q>\ell\}} \tilde q
		\toell 0 \quad
		\text{a.e.~in $(\tilde \omega,t,x)$}.
	\end{split}
\end{equation}

Inserting the inequalities \eqref{eq:temporarytag} 
and the identities \eqref{eq:temporarytag2} into 
\eqref{eq:S_equation_n_limit2}, with $S(v)=S_\ell(v_+)$, 
we arrive at
\begin{align*}
	&\int_{\tilde \Omega} \mathds{1}_{A}(\tilde \omega)
	\bigl(I(\tilde\omega)+M(\tilde\omega)\bigr)
	\, \d \tilde{\mathbb{P}}(\tilde \omega)
	\\ & \quad 
	\leq
	C_1\,\tilde \Ex\int_0^T \int_\T 
	\abs{e^{(1)}_\ell}+\abs{e^{(1)}_\ell(0)}
	+\abs{e^{(2)}_\ell}+\abs{e^{(4)}_\ell}
	+ \abs{e^{(5)}_\ell}\, \d x \, \d t
	\\ & \quad \qquad
	+C_2\,\tilde \Ex\int_0^T \int_\T 
	\abs{\tilde u}\abs{e^{(1)}_\ell}\, \d x \, \d t
	+C_3\, \tilde \Ex\int_0^T \int_\T 
	\abs{\tilde P-\tilde u^2}\abs{e^{(3)}_\ell}\, \d x \, \d t
	\\ & \quad\qquad 
	+\tilde \Ex\abs{\int_0^T \int_\T \sigma \, e^{(1)}_\ell
	\, \pd_x \varphi+\pd_x\sigma \, e^{(4)}_\ell \varphi
	\,\d x\, \d \tilde W}
	=:\sum_{i=1}^4 R_{i,\ell},
\end{align*}
for some $\ell$-independent constants
$C_1=C_1(\sigma,\varphi,\pd_x\varphi,\pd_{xx}^2 \varphi)$, 
$C_2=C_2(\pd_x\varphi)$, and $C_3=C_2(\varphi)$. 
Here, $e^{(1)}_\ell(0)$ refers to
$$
-\frac1{6\ell} \bk{\pd_x\tilde{u}_0-\ell}^3
\one{\ell<\pd_x\tilde{u}_0<2\ell} 
-\frac16\bk{3\bk{\pd_x\tilde{u}_0}^2 
-9\ell \pd_x\tilde{u}_0+
7 \ell^2}\one{\pd_x\tilde{u}_0 \ge 2\ell}.
$$

Notice that $\abs{e^{(i)}_\ell}\lesssim \tilde q^2
\in L^1_{\tilde \omega,t,x}$, $i=1,2,4, 5$, and 
$\abs{e^{(1)}_\ell(0)}\lesssim (\pd_x \tilde u_0)^2\in 
L^1_{\tilde \omega,t,x}$, cf.~\eqref{eq:weak-conv-tilde}. 
Therefore, by \eqref{eq:eiell-conv-Ii}, and 
the Lebesgue dominated convergence theorem, 
$R_{1,\ell}\toell 0$. The same argument applies to 
$R_{2,\ell}$, $R_{3,\ell}$, as
$$
\abs{\tilde u}\abs{e^{(1)}_\ell}
\lesssim \abs{\tilde u}\tilde q^2, 
\quad 
\abs{\tilde P-\tilde u^2}\abs{e^{(3)}_\ell}
\lesssim 
\abs{\tilde P}\tilde q^2
+\abs{\tilde u}^2\tilde q^2, 
$$
and, by Lemma \ref{thm:P-u2_bound-tilde} and 
\eqref{eq:tu-Lp-Linfty-H1},
\begin{align*}
	& \tilde \Ex \int_0^T \int_\T 
	\abs{\tilde u}^p\tilde q^2\, \d x\, \d t
	\lesssim_T 
	\left(\tilde \Ex \norm{\tilde u}_{L^\infty_{t,x}}^{2p}
	\right)^{\frac12}
	\left(\tilde \Ex \norm{\tilde q}_{L^\infty_tL^2_x}^4
	\right)^{\frac12}\lesssim 1, 
	\quad p\in \bigl[1,p_0/2\bigr],
	\\ & 
	\tilde \Ex \int_0^T \int_\T 
	\abs{\tilde P}\tilde q^2\, \d x\, \d t
	\lesssim
	\left(\tilde \Ex \norm{\tilde P}_{L^\infty_{t,x}}^2
	\right)^{\frac12}
	\left(\tilde \Ex \norm{\tilde q}_{L^\infty_tL^2_x}^4
	\right)^{\frac12}\lesssim 1.
\end{align*} 

Finally, let us consider the stochastic integral term. 
By the Cauchy--Schwarz inequality and the It\^{o} isometry,
\begin{align*}
	\abs{R_{4,\ell}}^2
	& \leq \tilde \Ex 
	\int_0^T \abs{\int_\T \sigma \, e^{(1)}_\ell
	\, \pd_x \varphi+\pd_x\sigma \, e^{(4)}_\ell \varphi
	\,\d x}^2 \d t
	\\ & \lesssim_{\sigma,\varphi}
	\tilde \Ex\int_0^T \abs{\, \int_\T 
	\mathds{1}_{\{\tilde q_+>\ell\}} 
	\tilde q_+^2\,\d x}^2 \d t.
\end{align*}
By \eqref{eq:tu-Lp-Linfty-H1}, 
$\tilde \Ex \int_0^T \int_{\T} \tilde{q}^2\, \d x\, \d t <\infty$ 
and so $\tilde{q}\in L^2(\T)$ for a.e.~$(\tilde \omega,t)$. 
Thus
$$
\abs{\, \int_\T \mathds{1}_{\{\tilde q_+>\ell\}} 
\tilde q_+^2\,\d x}^2 \toell 0,
\quad \text{for a.e.~$(\tilde \omega,t)$}.
$$
Besides, $\abs{\int_\T \mathds{1}_{\{\tilde q_+>\ell\}} 
\tilde q_+^2\,\d x}^2\leq 
\abs{\int_\T \tilde q_+^2\,\d x}^2$ and
\begin{equation}\label{eq:qpluss-square-int}
	\begin{aligned}
		& \tilde \Ex 
		\int_0^T \abs{\, \int_{\T} \tilde q_+^2\, \d x}^2\, \d t
		\lesssim_T 
		\tilde \Ex \norm{\tilde q}_{L^\infty([0,T];L^2(\T))}^4
		\overset{\eqref{eq:tu-Lp-Linfty-H1}}{\lesssim} 1.
	\end{aligned}
\end{equation}
Therefore, by Lebesgue's dominated convergence theorem, 
$\abs{R_{4,\ell}}^2\toell 0$. This concludes the proof 
of \eqref{eq:qp2-claim}.
\end{proof}

\subsection{Controlling the defect measure}
\label{subsec:defect}
We define the positive part of the
defect measure \eqref{eq:intro-defect} by
\begin{equation}\label{eq:defect-pospart}
        \mathbb{D}^{+}=
        \frac12\bk{\overline{q_+^2}-\tilde q_+^2}\ge 0.
\end{equation}
One can construe
$\mathbb{D}^{+} = \mathbb{D}^{+}(\tilde \omega,t,x)$
(and similar objects introduced later on) as a random
variable that assumes values in a path space of functions depending
on $t \in [0,T]$ and $x \in \T$. 
Alternatively, $\mathbb{D}^{+}$ can be conceived of  
as a stochastic process $(\tilde \omega,t) \mapsto
\mathbb{D}^{+}(\tilde\omega,t,\cdot)$ that takes
values in some functional space (over $x$).
If $\mathbb{D}^{+}$ is conceptualised as a random variable
in the Lebesgue space $L^r([0,T] \times \T)$ (recall that
$1\leq r<3/2$), then the pointwise value $\mathbb{D}^{+}(t)$ is only
determinable modulo a set of times with zero measure in $[0,T]$.
Consequently, discerning $\mathbb{D}^{+}$ as a traditional
stochastic process proves to be a complex endeavour.
Therefore, in the forthcoming discussion, we consider
$\mathbb{D}^{+}$ (and other similar objects) as a
random variable within the space $L^r([0,T] \times \T)$.
In view of our previous results, $\mathbb{D}^{+}$ obeys an SPDE inequality.
This inequality is interpreted almost surely in the distributional sense
on the domain $[0,T)\times \T$, with the inclusion
of the zero function as initial data in the distributional formulation.

To provide more clarity, by directly subtracting \eqref{eq:tqp2-limit}
and \eqref{eq:S_equation_n_limit3},
\begin{equation}\label{eq:Sdifference_plus}
        \begin{aligned}
                \pd_t \mathbb{D}^{+} & + \pd_x\left[\left(\tilde u
                -\frac14 \pd_x \sigma^2\right)\mathbb{D}^{+}\right]
                -\pd_{xx}^2\left[\frac12 \sigma^2 \mathbb{D}^{+}\right]
                \\ &
                +  \left(\overline{q_+}-\tilde q_+\right)
                \bk{\tilde P-\tilde u^2}
                + \left(\frac14 \pd_{xx}\sigma^2
                -\abs{\pd_x \sigma}^2\right)\mathbb{D}^{+}
                \\ &
                +\left[\pd_x \left(\sigma \, \mathbb{D}^{+}\right)
                +\pd_x\sigma\, \mathbb{D}^{+} \right]
                \, \dot{\tilde{W}} \leq 0
                \quad \text{in $\mathcal{D}'([0,T)\times \T)$,
                 \, $\tilde{\mathbb{P}}$--a.s.},
        \end{aligned}
\end{equation}
with zero initial data (in the sense of distributions).
This formulation is weak in $(t,x)$. 
Employing a reasoning approach akin to the one
used in the proof of Lemma \ref{lem:energy-rightcont}, we
can transform this into a formulation that
is pointwise (a.e.) in $(\tilde \omega, t)$ and integrated in $x$.

\begin{lem}[positive part of defect measure]
\label{thm:limiteq_2b}
Let $\mathbb{D}^{+}$ be defined by \eqref{eq:defect-pospart}
and $\tilde{P}$ by \eqref{eq:tPn-tP-def}. 
Then, for a.e.~$(\tilde \omega,t)\in \tilde \Omega \times [0,T]$,
\begin{equation}\label{eq:Sdifference_plus_ptws}
        \begin{aligned}
                & \int_\T \mathbb{D}^{+}(t) \,\d x
                +\int_0^t \int_\T
                 \bk{\overline{q_+}-\tilde{q}_+}
                 \bk{\tilde{P}-\tilde{u}^2}\, \d x\, \d s
                 \\ & \quad
                 +\int_0^t \int_\T\left(\frac14 \pd_{xx}\sigma^2
                -\abs{\pd_x \sigma}^2\right)\mathbb{D}^{+}
                \, \d x\,\d s
                 + \int_0^t\int_\T \pd_x\sigma\,
                 \mathbb{D}^{+} \,\d x \, \d \tilde{W} \leq 0.
        \end{aligned}
\end{equation}
The stochastic integral is a square-integrable martingale.
\end{lem}

\begin{proof}
Using the test function $\varphi(t,x)=\psi(t)\phi(x)$
in \eqref{eq:Sdifference_plus}, with $0\leq \psi\in C^\infty_c([0,T))$
arbitrary  and $\phi\equiv 1$, we obtain
\begin{equation}\label{eq:Sdifference_plus-integrated}
        \begin{aligned}
                &\frac{\d}{\d t}\int_\T \mathbb{D}^{+}\, \d x
                +\int_\T \left(\overline{q_+}-\tilde q_+\right)
                \bk{\tilde P-\tilde u^2} \,\d x
                \\ & \qquad 
                +\int_\T \left(\frac14 \pd_{xx}\sigma^2
                -\abs{\pd_x \sigma}^2\right)\mathbb{D}^{+}\, \d x
                + \int_\T \pd_x\sigma\,
                \mathbb{D}^{+} \,\d x \, \dot{\tilde{W}}
                \leq 0,
        \end{aligned}
\end{equation}
which holds in $\mathcal{D}'([0,T))$, a.s.,
with zero initial data (in the sense of distributions).

Following the proof of Lemma \ref{lem:energy-rightcont}, for a given
Lebesgue point $t$ of the integrable function $s\mapsto
\int_\T \mathbb{D}^{+}(s) \,\d x$ (with $\tilde \omega$ 
fixed from a set $F$ of full $\tilde{\mathbb{P}}$--measure),
consider $\delta$ such that $0<t-\delta<T$.
For such $\delta$, let $\beta_\delta$ be
the continuous piecewise linear function that equals $1$
on $[0,t -\delta]$, $0$ on $[t, T]$, and is linear on $[t-\delta,t]$.
Then  $\beta_\delta(s)\to  \mathds{1}_{[0,t]}(s)$ for a.e.~$s\in [0,T]$.
Using $\beta_\delta$ as test function in
\eqref{eq:Sdifference_plus-integrated} gives
\begin{align*}
        & \frac{1}{\delta}\int_{t-\delta}^t
        \left(\,\, \int_\T \mathbb{D}^{+}(s) \,\d x \right) \, ds
        +\int_0^T \int_\T \bk{\overline{q_+}-\tilde{q}_+}
         \bk{\tilde{P}-\tilde{u}^2}(s)
        \, \beta_\delta(s)\, \d x\, \d s
         \\ & \qquad
         +\int_0^T \int_\T\left(\frac14 \pd_{xx}\sigma^2
        -\abs{\pd_x \sigma}^2\right)\mathbb{D}^{+}(s)
        \, \beta_\delta(s) \, \d x\,\d s
         \\ & \qquad \qquad
        + \int_0^T\int_\T \pd_x\sigma\,
         \mathbb{D}^{+}(s)\, \beta_\delta(s)
        \,\d x \, \d \tilde{W}(s) \leq 0.
\end{align*}
The stochastic integral is a square-integrable
martingale on $[0,T]$, which follows from calculations
like \eqref{eq:mart-prop-part1}, \eqref{eq:mart-prop-part2},
and \eqref{eq:qpluss-square-int}.

By adhering to the proof of Lemma \ref{lem:energy-rightcont}
and considering Remark \ref{rem:almost-sure}, we
can take the limit as $\delta\to 0$ in this
inequality, leading us to \eqref{eq:Sdifference_plus_ptws}.
\end{proof}

Define
\begin{equation}\label{eq:defect-negpart}
	\mathbb{D}^{-}_{\ell}
	=\overline{S_\ell(q_-)}-S_\ell(\tilde q_-)\ge 0,
	\quad \ell \in \N,
\end{equation}
so that $\mathbb{D}^{-}_{\ell}$ approximates 
the negative part $\mathbb{D}^{-}
=\frac12\bk{ \overline{q_-^2}- \tilde q_-^2}$ of the defect 
measure \eqref{eq:intro-defect}. 
We first make explicit the approximation error by the 
following result:

\begin{lem}\label{thm:negD_ell_error}
Let $r' = r/(r-1)$ be the H\"older conjugate of $r$, 
recalling that $r< 3/2$. Then
\begin{align*}
	\abs{\tilde{\Ex} \int_0^T \int_\T 
	\mathbb{D}_\ell^- - \mathbb{D}^- \,\d x\,\d t} 
	\lesssim \ell^{-2(r-1)}.
\end{align*}
\end{lem}

\begin{proof}
Using the weak convergences \eqref{eq:Hneg-bound-oq2} 
and \eqref{eq:improved-S(tun)-conv} of 
$\tilde{q}_n^2 \weak \overline{q^2}$ in $L^{r}_{\tilde \omega, t ,x}$ 
and $S_\ell(\tilde{q}_n) \weak \overline{S_\ell(q)}$  
in $L^{2r}_{\tilde \omega, t ,x}$, 
\begin{equation*}
	\begin{aligned}
		& \tilde{\Ex} \int_0^T \int_\T 
		\mathbb{D}_\ell^- - \mathbb{D}^-  \,\d x\,\d t 
		\\ & \qquad 
		=\lim_{n \to \infty}\tilde{\Ex} \int_0^t \int_\T
		S_\ell\bigl((\tilde{q}_n)_-\bigr) - (\tilde{q}_n)_-^2
		+ S_\ell(\tilde{q}_-) - \tilde{q}_-^2\,\d x\,\d t.
	\end{aligned}
\end{equation*}
Remark \ref{rem:pos-neg-parts} implies that 
\begin{align*}
	\abs{S_\ell\bigl(v_-\bigr) - v_-^2 }
	& \lesssim \frac1\ell \abs{v + \ell}^3 \one{-2\ell \le v \le - \ell} 
	\\ & \qquad 
	+ v^2 \one{ v \le - 2\ell}
	\lesssim 
	v^2 \one{\abs{v} \ge \ell} 
	\leq \ell^{-2 (r - 1)} v^{2r} \one{\abs{v} \ge \ell}.
\end{align*}
Therefore, by Lemma \ref{thm:q_bounds_skorohod},
\begin{align*}
	& \abs{\tilde{\Ex} \int_0^T\int_\T
	S_\ell\bigl((\tilde{q}_n)_-\bigr) - (\tilde{q}_n)_-^2 \,\d x \,\d t}
	\\ & \qquad 
	\leq  \ell^{-2(r-1)}\tilde{\Ex} \int_0^T\int_\T
	 \abs{\tilde{q}_n}^{2r} \one{\abs{\tilde{q}_n} \ge \ell}
	 \,\d x \,\d t\lesssim  \ell^{-2(r - 1)}.
\end{align*}
A similar bound for $\tilde{q}$ in place of $\tilde{q}_n$ 
can be derived by invoking Lemma \ref{thm:u_in_Linfty}.
\end{proof}

Now, we introduce the functions
\begin{equation}\label{eq:H-ell-neg}
	\begin{split}
		& H^{(1)}_{\ell,-}(v)
		=3S_\ell(v_-)-2S_\ell(v_-)'v, 
		\quad
		H^{(2)}_{\ell,-}(v)
		=S_\ell(v_-)v-\tfrac12 S_\ell(v_-)'v^2,
		\\ &
		H^{(3)}_{\ell,-}(v)= S_\ell(v_-)-S_\ell(v_-)' v.
	\end{split}
\end{equation}
Subtracting \eqref{eq:S_equation_n_limit2} 
from \eqref{eq:S_equation_n_limit} yields
\begin{equation}\label{eq:Sdifference_minus-diff-form}
	\begin{aligned}
		& \pd_t \mathbb{D}^{-}_{\ell}
		+\pd_x\left[ \tilde u\, \mathbb{D}^{-}_{\ell}
		+\frac14 \pd_x \sigma^2 
		\left(\, \overline{H^{(1)}_{\ell,-}(q)}
		-H^{(1)}_{\ell,-}(\tilde q)\right)\right]
		-\pd_{xx}^2\left[\frac12 \sigma^2
		\,\mathbb{D}^{-}_{\ell}\right]
		\\ & \qquad
		+ \left(\,\overline{S_\ell(q_-)'}
		-S_\ell(\tilde q_-)'\right)\bk{\tilde P-\tilde u^2}
		\\ & \qquad 
		-\left(\, \overline{H^{(2)}_{\ell,-}(q)}
		-H^{(2)}_{\ell,-}(\tilde q)\right)
		+\frac12 S_\ell(\tilde q_-)'
		\left(\, \overline{q^2}-\tilde q^2\right)
		\\ & \qquad 
		-\frac14\pd_{xx}^2\sigma^2
		\left(\, \overline{H^{(3)}_{\ell,-}(q)}
		-H^{(3)}_{\ell,-}(\tilde q)\right)
		-\frac12 \abs{\pd_x \sigma}^2
		\left(\, \overline{S_\ell(q_-)''\,q^2}
		-S_\ell(\tilde{q}_-)''\tilde q^2\right)
		\\ & \qquad 
		+\left[\pd_x\left(\sigma \, \mathbb{D}^{-}_{\ell}\right)
		-\pd_x\sigma \left(\, \overline{H^{(3)}_{\ell,-}(q)}
		-H^{(3)}_{\ell,-}(\tilde q)\right) \right]\, \dot{\tilde{W}}
		\\ & \qquad \qquad 
		\leq 0 
		\quad \text{in $\mathcal{D}'([0,T)\times \T)$, 
		\, $\tilde{\mathbb{P}}$--a.s.},
	\end{aligned}
\end{equation}	
with zero initial data: $\mathbb{D}^{-}_{\ell}(0)=0$. 

Combining \eqref{eq:Sdifference_minus-diff-form} 
with the formulas in Remark \ref{rem:pos-neg-parts}, we obtain 
the following bound for the negative part $\mathbb{D}^-_\ell$ 
of the defect measure:

\begin{lem}[negative part of defect measure]\label{thm:limiteq_2c}
Let $\mathbb{D}^{-}_{\ell}$ be defined by 
\eqref{eq:defect-negpart} {and} $\mathbb{D}^{+}$ 
by \eqref{eq:defect-pospart}. 
Let $\bigl\{\tilde{u}_n\bigr\}_{n \ge 1}$, 
$\bigl\{\tilde{P}_n\bigr\}_{n \ge 1}$, and $\tilde{P}$ be 
as in Proposition \ref{lem:strong-conv-tPn}. 
For any { $n_0 \in \N$ and $L > 0$, 
define the measurable set 
\begin{align}\label{eq:A_L_set_defin}
	A^{n_0}_L = \left\{
	\tilde{\omega} \in \tilde{\Omega}
	:\norm{\tilde{P}_{n_0}
	-\tilde{u}_{n_0}^2}_{L^\infty([0,T]\times \T)} \le L\right\},
\end{align}}
which satisfies $\tilde{\mathbb{P}}\bigl(A_L^{n_0}\bigr) \to 1$ 
as $L \to \infty$, uniformly in $n_0$. 

For a.e.~$t\in [0, T]$ 
and sufficiently large $\ell$ (depending on $L$),
\begin{equation}\label{eq:Sdifference_minus} 
	\begin{split}
		&\int_\T \mathbb{D}^{-}_{\ell}(t)\, \d x
		+\int_0^{{t}}\int_\T 
		\left(\, \overline{q_-}-\tilde q_-\right)
		\bk{\tilde P-\tilde u^2} \, \d x\, \d {s}
		\\ & \quad
		+\int_0^{t}\int_\T \bk{\frac14\pd_{xx}^2\sigma
		-\abs{\pd_x \sigma}^2} \mathbb{D}_\ell^-
		-\frac{3\ell}{2} \, \bigl(\mathbb{D}^{-}_{\ell}
		+\mathbb{D}^{+}\bigr)\, \d x\, \d {s}
		\\ &\quad
		+\int_0^t \int_\T\bk{\overline{S_\ell(q_-)' - q_-}
		-\bk{S_\ell(\tilde{q}_-)' - \tilde{q}_-}}
		\bk{\tilde{P} - \tilde{u}^2
		-\tilde{P}_{n_0}+\tilde{u}_{n_0}^2} \,\d x\,\d s
		\\ & \quad	
		+\mathcal{M}^-_\ell(t) 
		\leq 0, \quad \text{a.s.~on $A_L^{n_0}$},
	\end{split}
\end{equation}
where $\mathcal{M}^-_\ell(t)$ is a
square-integrable martingale, with
$\tilde \Ex \abs{\mathcal{M}^-_\ell(T)}^2 \lesssim_\ell 1$.  
\end{lem}

\begin{proof}
Using the test function $\varphi(t,x)=\psi(t)\phi(x)$ 
with $0\leq \psi\in C^\infty_c([0,T))$ and $\phi\equiv 1$ in 
\eqref{eq:Sdifference_minus-diff-form}, we obtain
\begin{equation}\label{eq:Sdifference_minus-tmp1}
	\frac{\d}{\d t}\int_\T \mathbb{D}^{-}_{\ell}\, \d x
	+\sum_{i=1}^4 \int_\T I_\ell^{(i)}\, \d x
	+\int_\T I_\ell^{(5)}\, \d x \, \dot{\tilde{W}} \leq 0,
\end{equation}
which holds in $\mathcal{D}'([0,T))$, $\tilde{\mathbb{P}}$--a.s., 
with zero initial data, where
\begin{equation*}
	\begin{split}
		& I_\ell^{(1)}=\left(\,\overline{S_\ell(q_-)'}
		-S_\ell(\tilde q_-)'\right)\bk{\tilde P-\tilde u^2},
		\\ & 
		I_\ell^{(2)}=\frac12 S_\ell(\tilde q_-)'
		\left(\, \overline{q^2}-\tilde q^2\right)
		-\left(\, \overline{H^{(2)}_{\ell,-}(q)}
		-H^{(2)}_{\ell,-}(\tilde q)\right),
		\\ & 
		I_\ell^{(3)}=-\frac14 \pd_{xx}^2\sigma^2
		\left(\, \overline{H^{(3)}_{\ell,-}(q)}
		-H^{(3)}_{\ell,-}(\tilde q)\right),
		\\ &  
		I_\ell^{(4)}=-\frac12 \abs{\pd_x \sigma}^2
		\left(\, \overline{S_\ell(q_-)'' \,q^2}
		-S''(\tilde{q})\tilde q^2\right),
		\\ &  
		I_\ell^{(5)}=-\pd_x\sigma 
		\left(\, \overline{H^{(3)}_{\ell,-}(q)}
		-H^{(3)}_{\ell,-}(\tilde q)\right),
	\end{split}
\end{equation*}
and $H^{(1)}_{\ell,-}$, $H^{(2)}_{\ell,-}$, 
$H^{(3)}_{\ell,-}$ are defined in \eqref{eq:H-ell-neg}.

\medskip
\noindent \textit{1. The term $I^{(1)}_\ell$.}
\medskip

In view of Remark \ref{rem:pos-neg-parts}, 
\begin{equation*}
\begin{aligned}
	I_\ell^{(1)}  
	& =\left( \, \overline{q_-}-\tilde q_-\right)
	\bk{\tilde P-\tilde u^2}
	+e^{(1)}_\ell\bk{\tilde P-\tilde u^2}\\
	& =  \left( \, \overline{q_-}-\tilde q_-\right)
	\bk{\tilde P-\tilde u^2}
	\\ & \qquad 
	+e^{(1)}_\ell\bk{\tilde P_{n_0}-\tilde u_{n_0}^2} 
	+ e^{(1)}_\ell\bk{\tilde{P} - \tilde{u}^2 
	- \tilde P_{n_0} +\tilde u_{n_0}^2},
\end{aligned}
\end{equation*}
where
\begin{align}
	& e^{(1)}_\ell
	= \overline{S_\ell(q_-)' - q_-}
	-\bk{S_\ell(\tilde{q}_-)' -\tilde{q}_-}
	\label{eq:e1_ell_def}
	\\ & \,\,\,\,\,\,\,\,\,
	= \overline{f_1(q)\one{-2\ell<q<-\ell} 
	+g_1(q)\one{q\leq -2\ell}} 
	\notag
	\\ & \qquad\qquad
	-\bk{f_1(\tilde{q})\one{-2\ell<\tilde{q}<-\ell}
	+g_1(\tilde{q})\one{\tilde{q}\leq -2\ell}},
	\notag
	\\ 
	& f_1(v)=\frac1{2\ell}\bk{v+\ell}^2,
	\quad 
	g_1(v)=-\frac12 \bk{3\ell+2v},
	\notag
\end{align}
Note the real-valued mapping 
$r_1(v)=f_1(v) \one{-2\ell<v<-\ell}
+g_1(v)\one{v\leq -2\ell}$ is convex:
\begin{align*}
	& r_1'(v)= \frac1{4\ell}\bk{v+\ell}\one{-2\ell<v<-\ell}
	-\one{v\leq -2\ell},
	\quad 
	r_{{1}}''(v)=\frac1{4\ell}\one{-2\ell<v<-\ell}\geq 0.
\end{align*}
Because of the convexity, it follows that $e^{(1)}_\ell \geq 0$ 
\cite[Corollary 3.33]{Novotny-book:2004} and thus
\begin{align*}
	I^{(1)}_\ell 
	& \geq \bk{\overline{q_-}-\tilde{q}_-}
	\bk{\tilde{P}-\tilde{u}^2} 
	\\ & \quad 
	-e^{(1)}_\ell 
	\norm{\tilde{P}_{n_0}
	-\tilde{u}^2_{n_0}}_{L^\infty([0,T] \times \T)}
	+ e^{(1)}_\ell \bk{\tilde{P} - \tilde{u}^2 
	- \tilde P_{n_0} +\tilde u_{n_0}^2}.
\end{align*}

\medskip
\noindent \textit{2. The term $I^{(2)}_\ell$.}
\medskip

Recalling the definition \eqref{eq:defect-pospart} 
of $\mathbb{D}^+$, we next manipulate 
$I_\ell^{(2)}$ into the form ``$C_\ell\bk{\mathbb{D}^+ 
+\mathbb{D}^-_\ell} + \text{error}$''. 
From Remark \ref{rem:pos-neg-parts},
\begin{equation}\label{eq:oq2-q2-formula}
	\frac12\bk{\overline{q^2_-}-\tilde q^2_-}
	=\overline{S_\ell(q_-)}-S_\ell(\tilde q_-)
	+e^{(2,1)}_\ell,
\end{equation}
where
\begin{align*}
	& e^{(2,1)}_\ell
	=\overline{f_{2,1}(q)\one{-2\ell<q<-\ell} 
	+g_{2,1}(q)\one{q \leq -2\ell}}
	\\ &\qquad\qquad 
	-\bk{f_{2,1}(q)\one{-2\ell<q<-\ell}
	+g_{2,1}(q)\one{q\leq -2\ell}},
	\\ \intertext {and}
	& f_{2,1}(v)=-\frac1{6\ell} \bk{v+\ell}^3, 
	\quad 
	g_{2,1}(v)=\frac16\bk{3v^2+9\ell v+7\ell^2},
\end{align*}
recalling that we drop the tilde atop 
a variable sitting under an overline (see 
Remark \ref{rem:weak-notation}). 
Given the identity \eqref{eq:oq2-q2-formula}, writing 
\begin{equation*}
	\overline{q^2}-\tilde q^2=
	\overline{q^2_-}-\tilde q^2_-
	+\overline{q^2_+}-\tilde q^2_+,
\end{equation*}
it follows that
\begin{align*}
	\frac12 S_\ell(\tilde q_-)'
	\left(\, \overline{q^2}-\tilde q^2\right)
	& =\frac12 S_\ell(\tilde q_-)'
	\left(\, \overline{q^2_+}-\tilde q^2_+\right)
	\\ & \qquad 
	+S_\ell(\tilde q_-)'\left(\overline{S_\ell(q_-)}
	-S_\ell(\tilde q_-)\right)
	+S_\ell(\tilde q_-)' e^{(2,1)}_\ell.
\end{align*}

Regarding $e^{(2,1)}_\ell$, 
observe that $r_{2,1}(v)=f_{2,1}(v)\one{-2\ell<v<-\ell} 
+g_{2,1}(v)\one{v\leq -2\ell}$ 
is non-negative and convex. Indeed,  by construction, 
$r(v)$ and $r'(v)$ are continuous functions, recalling 
that $S_\ell(v_\pm) \in W^{3,\infty}(\R)$, and so
\begin{align*}
	& r_{2,1}'(v)=-\frac1{2\ell} \bk{v+\ell}^2\one{-2\ell<v<-\ell}
	+\frac12\bk{2v+3\ell}\one{v \leq -2\ell},
	\\ 
	& r_{2,1}''(v)=-\frac{1}{\ell}\bk{v+\ell}\one{-2\ell<v<-\ell}
	+\one{v \leq -2\ell} \geq 0.
\end{align*}

Making use of $S_\ell(\tilde{q}_-)' \ge - 3\ell/2$ and the 
positivity (negativity) of $\overline{f(q)}-f(\tilde q)$ 
for any convex (concave) $f$ 
\cite[Corollary 3.33]{Novotny-book:2004} 
\begin{align*}
	I_\ell^{(2)} & \geq 
	-\frac{3\ell}4\left(\, \overline{q^2_+}-\tilde q^2_+\right) 
	-\frac{3\ell}2\left(\overline{S_\ell(q_-)}
	-S_\ell(\tilde q_-)\right)
	\\ &\qquad\quad 
	- \frac{3\ell}2 e_\ell^{(2,1)} 
	-\left(\, \overline{H^{(2)}_{\ell,-}(q)}
	-H^{(2)}_{\ell,-}(\tilde q)\right)
	\\ &
	=-\frac{3\ell}2\left(\, \mathbb{D}^+
	+\mathbb{D}^-_\ell \right)+e^{(2)}_\ell,
\end{align*}
where, recalling that the expression 
$H^{(2)}_{\ell,-}(v)=S_\ell(v_-)v
-\frac12 S_\ell(v_-)' v^2$, 
see \eqref{eq:H-ell-neg}, takes the explicit 
form calculated in Remark \ref{rem:pos-neg-parts},
\begin{align*}
	& e^{(2)}_\ell 
	=\overline{f_2(q)\one{-2\ell<q<-\ell} 
	+g_2(q) \one{q\leq -2 \ell}}
	\\ &\qquad\qquad
	-\bk{f_2(\tilde q)\one{-2\ell<\tilde q<-\ell} 
	-g_2(\tilde q)\one{\tilde{q} \leq -2\ell}}, 
	\\
	& f_2(v)=\frac1{12\ell}\bk{v^4+3\ell v^3+6\ell^2 v^2 
	+7\ell^3 v+3 \ell^4},
	\quad
	g_2(v)=-\frac1{12}\bk{13\ell^2 v+21\ell^3}.
\end{align*}

\medskip
\noindent \textit{3. The terms $I^{(3)}_\ell$ and $I^{(4)}_\ell$.}
\medskip

Similarly, using \eqref{eq:H-ell-neg} and 
Remark \ref{rem:pos-neg-parts}, we obtain
\begin{align*}
	H^{(3)}_-(v) 
	&=\left(H^{(3)}_-(v)+S_\ell(v_-)\right)
	-S_\ell(v_-)
	\\ &
	=-S_\ell(v_-)+f_3(v) \one{-2\ell<v<-\ell}
	+g_3(v) \one{v \leq -2\ell},
\end{align*}
where
\begin{align*}
	f_3(v)=-\frac1{6\ell} v^3+\frac12\ell v
	+\frac13\ell^2, 
	\quad  
	g_3(v)=-\frac32\ell v-\frac73 \ell^2.
\end{align*}
Furthermore,
\begin{align*}
	\frac12 S_\ell(v_-)'' v_-^2 
	& =\left(\frac12 S_\ell(v_-)'' v_-^2
	-S_\ell(v_-)\right)+S_\ell(v_-) 
	\\ & 
	=S_\ell(v_-)
	+f_4(v)\one{-2 \ell < v < - \ell}
	+g_4(v) \one{ v \leq -2\ell},
\end{align*}
where
\begin{align*}
	f_4(v)=\frac1{3\ell}v^3-\frac12\ell v-\frac16 \ell^2, 
	\quad  
	g_4(v)=\frac32\ell v+\frac76 \ell^2.
\end{align*}
Therefore, if we set 
\begin{align*}
	& e^{(3)}_\ell 
	=\overline{f_3(q)\one{-2\ell<q<-\ell}
	+g_3(q)\one{q\leq -2\ell}}
	\\ & \qquad\qquad 
	-f_3(\tilde{q})\one{-2\ell<\tilde{q}<-\ell} 
 	+g_3(\tilde{q})\one{\tilde{q}\leq -2\ell},
 	\\
 	& e^{(4)}_\ell=\overline{f_4(q)\one{-2\ell<q<-\ell}
 	+g_4(q) \one{q\leq-2\ell}}
 	\\ & \qquad\qquad 
 	-f_4(\tilde{q})\one{-2\ell<\tilde{q}<-\ell}
 	+g_4(\tilde{q}) \one{\tilde{q} \leq-2\ell},
\end{align*}
we get
\begin{align*}
	&I^{(3)}_\ell
	=\frac14 \pd_{xx}^2 \sigma^2 \mathbb{D}^-_\ell
	-\frac14 \pd_{xx}^2 \sigma^2 e^{(3)}_\ell,
	\\ &
	I^{(4)}_\ell
	=-\abs{\pd_{x} \sigma}^2 \mathbb{D}^-_\ell 
	-\abs{\pd_{x} \sigma}^2e^{(4)}_\ell.
\end{align*}

\medskip
\noindent \textit{4. The term $I^{(5)}_\ell$.}
\medskip

By Lemmas \ref{thm:q_bounds_skorohod} 
and \ref{thm:u_in_Linfty}, recalling 
that $\abs{S_\ell(v_-)-S_\ell(v_-)'v}
\lesssim_\ell \abs{v}$, cf.~\eqref{eq:entrop2}, we may assume 
that $\overline{H^{(3)}_{\ell,-}(q)}$, 
$H^{(3)}_{\ell,-}(\tilde q)$, cf.~\eqref{eq:H-ell-neg}, 
and thus $I^{(5)}_\ell$ belong to 
$L^{2r}_{\tilde\omega,t,x}$ (with $2r>2$), 
for each fixed $\ell$. In particular, this implies that 
\begin{align*}
	\mathcal{M}^-_\ell(t)
	=\int_0^t \int_\T I^{(5)}_\ell\,\d x \,\d \tilde W
\end{align*}
is a square-integrable martingale on $[0,T]$.
 
\medskip
\noindent \textit{5. The inequality \eqref{eq:Sdifference_minus}.}
\medskip

Introduce the ``total error" function 
\begin{align*}
	h_\ell(v) & = 
	\Biggl(-f_1(v) \norm{\tilde{P}_{n_0}
	-\tilde{u}^2_{n_0}}_{L^\infty([0,T] \times \T)} 
	\\ & \qquad \qquad
	+f_2(v) 
	-\frac14 \pd_{xx}^2 \sigma^2 f_3(v) 
	-\abs{\pd_{x}\sigma}^2f_4(v)\Biggr)\one{- 2\ell<v<-\ell}  
	\\ & \qquad
	+ \Biggl(-g_1(v)\norm{\tilde{P}_{n_0}
	-\tilde{u}^2_{n_0}}_{L^\infty([0,T] \times \T)}
	\\ & \qquad \qquad \qquad
	+g_2(v)
	-\frac14 \pd_{xx}^2 \sigma^2 g_3(v) 
	-\abs{\pd_{x} \sigma}^2 g_4(v) \Biggr) \one{v\leq-2\ell}.
\end{align*}
Gathering the findings of the first {three} steps, 
we deduce that
\begin{equation}\label{eq:foregoing0}
	\begin{split}
		&I^{(1)}_\ell+I^{(2)}_\ell
		+I^{(3)}_\ell+I^{(4)}_\ell 
		\\ & \quad
		\geq \bk{\overline{q_-}-\tilde{q}_-}\bk{\tilde{P}-\tilde{u}^2}
		-\frac{3\ell}2\bk{\mathbb{D}^++\mathbb{D}^-_\ell}
		+\left(\frac14 \pd_{xx}^2 \sigma^2
		-\abs{\pd_{x} \sigma}^2 \right)\mathbb{D}^-_\ell
		\\ & \quad \qquad
		+\overline{h_\ell(q)}-h_\ell(\tilde{q})
		+e^{(1)}_\ell \bk{\tilde{P} - \tilde{u}^2
		-\tilde{P}_{n_0} + \tilde{u}_{n_0}^2},
	\end{split}	
\end{equation}
where the overlines denote the weak limits 
in $n$ only ($n_0$ is kept fixed).

Recall the definition of $A_L^{n_0}$ 
in \eqref{eq:A_L_set_defin}. We claim that $h_\ell(v)$ 
is convex {on $A_L^{n_0}$}, at 
least for a sufficiently large $\ell=\ell(L)$.
To see this, we can compute the second derivative 
of $h_\ell$ directly on each of the two 
subsets $\{-2\ell<v<-\ell\}$ and $\{v\leq-2\ell\}$, 
thanks to the continuity of $h_{{\ell}}$ and 
$h'_{{\ell}}$ that follows from 
the continuity of $f_i$, $f_i'$, $g_i$, $g_i'$ ($i = 1, \ldots 4$), 
and then add the results. Indeed,
\begin{align*}
	& f_1''(v)=\frac1\ell, 
	\quad 
	f_2''(v)=\frac1\ell v^2+\frac32 v+\ell,
	\quad
	f_3''(v)=-\frac1\ell v,
	\\ &
	f_4''(v) = \frac2\ell v,
	\quad 
	g_1''(v), \, g_2''(v), 
	\, g_3''(v), \, g_4''(v) \equiv 0,
\end{align*}
and so on $A_L^{n_0}$, for any $v \in \R$ 
and a.e.~$x \in \T$,
$$
h''_{{\ell}}(v) \ge 
\bk{ \frac1\ell v^2  + \frac32 v + \ell 
-\frac1\ell\left\{L -\frac1{4}\pd_{xx}^2\sigma^2 v 
+2\abs{\pd_x \sigma}^2 v\right\}} 
\one{- 2\ell < v < - \ell},
$$
which is non-negative for sufficiently 
large $\ell$ because the term in braces can be made small 
relative to the terms outside the braces.

The convexity of $h_\ell(v)$ implies that 
on $A_L^{n_0}$,
\begin{equation}\label{eq:h_convex}
	\overline{h_\ell(q)}- h_\ell(\tilde{q}) 
	\ge 0 
	\quad \text{a.e.~in $[0,T] \times \T$.}
\end{equation}
Using \eqref{eq:h_convex}, which holds for a sufficiently large 
$\ell=\ell(L)$, \eqref{eq:foregoing0} becomes
\begin{equation}\label{eq:foregoing}
	\begin{split}
		& I^{(1)}_\ell+I^{(2)}_\ell+I^{(3)}_\ell+I^{(4)}_\ell
		\\ & \quad 
		\geq \bk{\overline{q_-}-\tilde{q}_-}\bk{\tilde{P}-\tilde{u}^2}
		-\frac{3\ell}2\bk{\mathbb{D}^++\mathbb{D}^-_\ell}
		\\ & \quad\qquad
		+\left(\frac14 \pd_{xx}^2 \sigma^2 
		-\abs{\pd_{x} \sigma}^2 \right)\mathbb{D}^-_\ell
		+e^{(1)}_\ell \bk{\tilde{P} - \tilde{u}^2 
		-\tilde{P}_{n_0} + \tilde{u}_{n_0}^2}.
	\end{split}
\end{equation}

Now we multiply \eqref{eq:Sdifference_minus-tmp1} 
by $\mathds{1}_{A_L^{n_0}}$ and insert \eqref{eq:foregoing}, 
arriving at
\begin{equation}\label{eq:Sdifference_minus-tmp3}
	\begin{split}
		&\frac{\d}{\d t}\int_\T \mathbb{D}^{-}_{\ell}\, \d x
		+\int_\T \left(\, \overline{q_-}-\tilde q_-\right)
		\bk{\tilde P-\tilde u^2} \, \d x
		\\ & \quad
		+\int_\T \bk{ \frac14 \pd_{xx}^2 \sigma
		-\abs{\pd_x \sigma}^2} \mathbb{D}_\ell^-
		-\frac{3\ell}2 \bk{\mathbb{D}^{+}
		+\mathbb{D}^{-}_{\ell}}\, \d x
		\\ &\quad 
		+\int_\T  e^{(1)}_\ell \bk{\tilde{P}-\tilde{u}^2
		-\tilde{P}_{n_0} + \tilde{u}_{n_0}^2}\,\d x
		\\ & \quad
		+\int_\T I^{(5)}_\ell\, \d x \, \dot{\tilde{W}} \leq 0,
		\quad 
		\text{in $\mathcal{D}'([0,T))$, 
		 a.s.~on $A_L^{n_0}$},
	\end{split}
\end{equation}
with zero initial data (in the distributional sense). 
Arguing as in the proofs of Lemmas \ref{thm:temp_cont_0}, 
\ref{eq:energybalance-tu-new} and \ref{thm:limiteq_2b}, 
we can turn \eqref{eq:Sdifference_minus-tmp3} 
into the inequality \eqref{eq:Sdifference_minus} 
that holds a.e.~in $A_L^{n_0} \times [0,T]$.

\begin{rem}\label{rem:delicate}
Note carefully that the fifth step is rather delicate, 
relying on having precise control of the error terms 
leading up to the convexity of the total error function  
$h_\ell(v)$, and thus \eqref{eq:h_convex}. Along  
the way, we exploit some crucial ``coercivity" induced 
by the specific error term $e^{(2)}_\ell$ linked to 
the difference $\overline{H^{(2)}_{\ell,-}(q)}
-H^{(2)}_{\ell,-}(\tilde q)$, recalling that 
$H^{(2)}_{\ell,-}(v)=S_\ell(v_-)v-\tfrac12 S_\ell(v_-)'v^2$. 
It may be instructive to keep in mind 
that $S(v)v-\tfrac12 S'(v)v^2\equiv 0$ if 
$S(v)=\frac12 v_\pm^2$ or $v^2$.
\end{rem}

\medskip
\noindent \textit{6. Properties of the set $A_L^{n_0}$.}
\medskip

The set $A_L^{n_0} \subset \tilde{\Omega}$ is measurable as 
$\normb{\tilde{P}_{n_0} - \tilde{u}^2_{n_0}}_{L^\infty([0,T] \times \T)}$ 
is a random variable (see Lemma \ref{thm:P-u2_bound-tilde}).
Denote by $\bk{A_L^{n_0}}^c$ the complement  
$\tilde{\Omega} \backslash A_L^{n_0}$. 
It further follows from Markov's inequality 
applied to the bound of Lemma \ref{thm:P-u2_bound-tilde} that 
\begin{align}\label{eq:A_Lc_bound}
	\tilde{\mathbb{P}}\bigl(\bk{A_L^{n_0}}^c\bigr) 
	\leq \frac1L 
	\tilde{\Ex} \norm{\tilde{P}_{n_0}
	-\tilde{u}^2_{n_0}}_{L^\infty([0,T]\times \T)} 
	\lesssim \frac1L.
\end{align}
\end{proof}

We can now identify the weak limit 
$\overline{q^2}$ with $\tilde{q}^2$, thereby 
concluding the proof of Theorem \ref{thm:products_convergences}.

\begin{proof}[Proof of Theorem \ref{thm:products_convergences}]
We shall be adding \eqref{eq:Sdifference_plus_ptws} 
and \eqref{eq:Sdifference_minus}. The purpose of doing so 
is that using
$$
\overline{q_+}+\overline{q_-}=\tilde{q}
=\tilde q_++\tilde q_-\Longrightarrow
\overline{q_+}-\tilde{q}_+ +\overline{q_-}-\tilde q_-=0,
$$ 
the term involving $\bigl(\tilde{P}-\tilde{u}^2\bigr)$ 
disappears, allowing us to conclude via 
taking an expectation and applying Gronwall's inequality, 
as we will demonstrate next.

We observe first that the 
inequality \eqref{eq:Sdifference_plus_ptws} 
holds a.s.~on $A_L^{n_0}$, where $A_L^{n_0}$ is 
defined in \eqref{eq:A_L_set_defin}. 
We now multiply each of 
\eqref{eq:Sdifference_plus_ptws} 
and \eqref{eq:Sdifference_minus} 
by $\mathds{1}_{A_L^{n_0}}$, 
add these two equations together 
and then take an expectation to obtain, 
for all sufficiently large $\ell$ (with $L, n_0$ fixed) 
and a.e.~$t \in [0,T]$,
\begin{equation}\label{eq:defect-final-ineq}
	\begin{aligned}
		& \tilde{\Ex} \int_\T \mathds{1}_{A_L^{n_0}}
		\bigl(\mathbb{D}^{+}+\mathbb{D}^{-}_{\ell}\bigr)({t}) \,\d x
		\\ & \quad 
		+\tilde{\Ex} \int_0^t \int_\T {f_\ell(x)} \mathds{1}_{A_L^{n_0}}
		\bigl(\mathbb{D}^{+} +\mathbb{D}^{-}_{\ell}\bigr)({s})
		\, \d x\,\d s
		\\ & \quad\quad 
		\leq -\Ex \int_0^t \int_{\T}\mathds{1}_{A_L^{n_0}}  
		e^{(1)}_\ell \bk{\tilde{P} - \tilde{u}^2 
		-\tilde{P}_{n_0}+\tilde{u}_{n_0}^2} 
		\,\d x\,\d s,
	\end{aligned}
\end{equation}
where 
\begin{align*}
	f_\ell(x)=\frac14  \pd_{xx}^2\sigma^2
	-\abs{\pd_x \sigma}^2-\frac{3\ell}2, 
	\qquad \norm{f_\ell}_{L^\infty(\T)}\leq C \ell ,
\end{align*}
and, for brevity, we have retained the notation 
\eqref{eq:e1_ell_def} for $e^{(1)}_\ell$.

Applying Gronwall's inequality 
to \eqref{eq:defect-final-ineq}, we get 
for a.e. $t \in [0,T]$ that
\begin{align*}
	\tilde{\Ex}&\bk{\mathds{1}_{A_L^{n_0}} 
	\norm{\mathbb{D}^+(t)+\mathbb{D}^-_\ell(t)}_{L^1(\T)}} 
	\\ & \qquad 
	\leq  e^{C t\ell}\, \tilde{\Ex} \int_0^t \int_{\T}
	\mathds{1}_{A_L^{n_0}} 
	\,\, e^{(1)}_\ell \bk{\tilde{P}-\tilde{u}^2 
	-\tilde{P}_{n_0} + \tilde{u}_{n_0}^2}\,\d x\,\d s.
\end{align*}
Integrating the above over $[0,T]$, 
\begin{align*}
	\int_0^T \tilde{\Ex}&\bk{\mathds{1}_{A_L^{n_0}} 
	\norm{\mathbb{D}^+
	+\mathbb{D}^-_\ell}_{L^1(\T)}} \,\d t 
	\\ & \le T e^{C T\ell} 
	\,\tilde{\Ex} \int_0^T \int_{\T} 
	\abs{e^{(1)}_\ell} \abs{\tilde{P} - \tilde{u}^2 
	-\tilde{P}_{n_0} + \tilde{u}_{n_0}^2}\,\d x \,\d t
	=: Te^{C T\ell} \mathcal{J}_{n_0}.
\end{align*}
Adding $\Ex \int_0^T \mathds{1}_{\bk{A_L^{n_0}}^c}
\norm{\mathbb{D}^+ + \mathbb{D}_\ell^-}_{L^1(\T)} \,\d t$ 
to both sides, 
\begin{equation}\label{eq:defect(l_n)-zero}
	\begin{aligned}
		\tilde{\Ex} &\int_0^T 
		\norm{\mathbb{D}^+
		+\mathbb{D}^-_\ell}_{L^1(\T)}\,\d t 
		\\ &
		\le  \Ex \int_0^T \mathds{1}_{\bk{A_L^{n_0}}^c}
		\norm{\mathbb{D}^+ + \mathbb{D}_\ell^-}_{L^1(\T)} \,\d t 
		+Te^{C T \ell}\mathcal{J}_{n_0}
		\\ & \lesssim_T \bk{\tilde{\Ex}\int_0^T \norm{\mathbb{D}^+
		+\mathbb{D}^-_\ell}_{L^1(\T)}^r\, \d t}^{1/r} 
		\bk{\tilde{\mathbb{P}}\big(\bk{A_L^{n_0}}^c\big)}^{1/r'}
		+e^{C T\ell}\mathcal{J}_{ n_0}
		\\ & 
		\lesssim_T 
		\norm{\mathbb{D}^++\mathbb{D}^-_\ell}_{L^r_{\tilde \omega,t,x}}
		\bk{\frac1L}^{1/r'}+e^{C T \ell}\mathcal{J}_{ n_0},
	\end{aligned}
\end{equation}
where the final inequality follows from  \eqref{eq:A_Lc_bound}. 
The implicit constant in the final $\lesssim$ 
is independent of $L$ and $n_0$.  Also note that
$\norm{\mathbb{D}^++\mathbb{D}^-_\ell}_{L^r_{\tilde \omega,t,x}}
\lesssim 1$, uniformly in $\ell$ by Lemma 
\ref{thm:u_in_Linfty} and the 
first inequality of \eqref{eq:Lr-HnegEst}.  
On the other hand, as we shall presently argue, 
$\mathcal{J}_{n_0} \to 0$ as $n_0 \uparrow \infty$, uniformly in $\ell$. 
The convergence of $\mathcal{J}_{n_0}$ is a consequence 
of two facts. First, by the second bound of \eqref{eq:tu-Lp-Linfty-H1-new}  
and \eqref{eq:Sprime-conv}, we have  
$$
\norm{e^{(1)}_\ell }_{L^{2r}_{\tilde\omega,t,x}} 
= \norm{\overline{S'_\ell(q_-)  - q_- } 	
- \bk{S'_\ell(\tilde{q}) -  \tilde{q}_-}}_{L^{2r}_{\tilde\omega,t,x}} 
\lesssim 1.
$$
Second, as $n_0\uparrow \infty$, we have the strong 
convergences \eqref{eq:tPn-strong-conv2} 
of $\tilde{P}_{n_0} \to \tilde{P}$ in $L^p_{\tilde\omega,t,x}$ 
and \eqref{eq:improved-tun-conv} of $\tilde{u}_{n_0}^2 \to \tilde{u}^2$  
in $L^p_{\tilde\omega,t,x}$, for any $p < p_0/2$. 
Since $p_0 > 4$, and $p := 2r/(2r - 1) < 2$ 
for $r$ close to $3/2$, this implies that
\begin{align*}
	\mathcal{J}_{n_0} & = \tilde{\Ex} \int_0^T \int_{\T} 
	\abs{e^{(1)}_\ell} \abs{\tilde{P} - \tilde{u}^2 
	-\tilde{P}_{n_0} + \tilde{u}_{n_0}^2}\,\d x \,\d t
	\\ & 
	\leq \norm{e^{(1)}_\ell}_{L^{2r}_{\tilde\omega,t,x}} 
	\norm{\tilde{P}-\tilde{u}^2
	-\tilde{P}_{n_0}+\tilde{u}_{n_0}^2}_{L^{p}_{\tilde\omega,t,x}}
	\xrightarrow{n_0 \uparrow \infty} 0,
\end{align*}
nullifying the harmful exponential factor in  
\eqref{eq:defect(l_n)-zero}, and yielding
\begin{align}\label{eq:defect(l)-zero}
	\tilde{\Ex} &\int_0^T \int_\T \mathbb{D}^+
	+\mathbb{D}^-_\ell\,\d x \,\d t 
	\lesssim L^{-1/r'},
\end{align}
for any sufficiently large $\ell$ (with $L$ fixed).

Finally, upon taking the limits 
$\ell \to \infty$ first---making use of 
Lemma \ref{thm:negD_ell_error}---and 
$L\to \infty$ second in \eqref{eq:defect(l)-zero}, 
we work out that
$$
\overline{q_+^2}=\tilde{q}_+^2, \quad 
\overline{q_-^2}=\tilde{q}_-^2
\quad \text{a.e.~in $\tilde \Omega \times 
\bigl[0,T\bigr]\times \T$},
$$
which concludes the proof 
of Theorem \ref{thm:products_convergences}.
\end{proof}

\begin{rem}
Let us refine Remark \ref{rem:delicate} further by exploring 
the treatment of the residual ``bad" error term $e^{C T \ell}\mathcal{J}_{ n_0}$ 
in \eqref{eq:defect(l_n)-zero}. This specific term does not lend itself 
to a resolution through the delicate balance of convexity and coercivity 
discussed in Remark \ref{rem:delicate}. Rather, its successful 
management primarily depends on the strong 
convergence \eqref{eq:tPn-strong-conv2}. 
This strong convergence, in turn, stems from 
the employment of the quasi-Polish strong-weak 
space $L^r\bigl(L^r_w\bigr)$ featured among the 
path spaces \eqref{eq:pathspaces}.
\end{rem}

\section{Acknowledgements}

This research was partially supported 
by the Research Council of Norway via the Toppforsk 
project \textit{Waves and Nonlinear Phenomena} (250070),
the project {\it IMod — Partial differential equations, statistics 
and data: An interdisciplinary approach to data-based 
modelling} (325114), and the project \textit{INICE} (301538).

The current address of Luca Galimberti is the 
Department of Mathematics, King's College London, 
WC2R 2LS, London, UK.

We are appreciative of the reviewer’s 
meticulous examination of this paper. This individual’s 
thoughtful and nuanced inquiries, including making 
us aware of the reference \cite{CL2022}, have 
contributed to the improvement of the paper.

\appendix

\section{Formal derivation of stochastic CH equation}
\label{sec:derivation-SCH}
Let us outline a formal derivation 
of the stochastic CH equation \eqref{eq:u_ch}.
Denote by $M_m$ be the multiplication operator 
$M_mf = mf$, and by $D$ the spatial derivative operator. 
As is well-known, the deterministic CH equation (for 
the momentum variable $m$) takes the form
\begin{equation}\label{eq:bi-Hamiltonian}
	0 = \pd_t m +M_m D 
	\frac{\delta \tilde{h} [m]}{\delta m}
	+D M_m \frac{\delta \tilde{h}[m]}{\delta m},
\end{equation}
where $\tilde{h}[m]=\int_\T \frac{1}{2} 
m(t,x)\, (K*m)(t,x)\, \d x$ is a nonlocal 
Hamiltonian based on the kernel $K$ 
defined in \eqref{eq:P-def}. 
Setting $u := K*m$, one can formally 
convert the bi-Hamiltonian equation 
\eqref{eq:bi-Hamiltonian} into the deterministic 
CH equation (for $u$), i.e., \eqref{eq:SCH-tmp} 
with $\sigma\equiv 0$. The stochastic CH equation 
is obtained by considering a stochastic perturbation of the 
temporally-integrated Hamiltonian:
$$
H[m]:=\int_\T \left(\, \, \int_0^t 
\frac{1}{2}m(s,x) (K*m)(t,x) \,\d s
+\int_0^t \bigl(m(s,x) \sigma(x) \bigr)
\circ \d{W}(s) \right)\,\d x.
$$
We recover the deterministic Hamiltonian
$\tilde h$ by taking $\sigma\equiv 0$ and 
computing $\d H/\d t$. The first variation of $H[m]$ is
$\frac{\delta H[m]}{\delta m} 
= u+\sigma(x) \,\dot{W}$.
Note that this expression is highly irregular 
in time $t$ (of class $C^{-1/2-0}$), but 
at the formal level, compared 
with \eqref{eq:bi-Hamiltonian}, 
the analogous stochastic CH equation becomes
$$
0 = \d m + M_{m} D \bigl(u \, \d t+ \sigma(x) 
\,\d{W}\bigr) + D M_{m}\bigl(u\, \d t
+\sigma(x)\, \d{W}\bigr),
$$
where the multiplication operator $M$ 
here uses the Stratonovich product $\circ$; 
written out more explicitly, we obtain
\begin{equation}\label{eq:m_Stoch-CH}
	0 = \d m + \bigl(m\, \pd_x u
	+\pd_x (mu) \bigr)\, \d t 
	+ m\, \pd_x \sigma(x) \circ \d W 
	+ \pd_x \bigl(m \sigma(x) \bigr)\circ \d W.
\end{equation}
Recalling that $u=K*m$, i.e., $m=(1-\pd_{xx}^2)u$, 
we can formally expand \eqref{eq:m_Stoch-CH} 
into \eqref{eq:SCH-tmp}, or \eqref{eq:u_ch} thanks to 
the Stratonovich--It\^{o} conversion 
formula. In this paper we use \eqref{eq:u_ch} as 
the operational form of the stochastic CH equation.

\section{Primer on quasi-Polish spaces}\label{sec:qpolish}

We detail here some definitions and results 
that have been applied repeatedly in our proofs. 
Ready references for some---but not all---of the material 
collected here are \cite{Jakubowski:1997aa} 
and \cite{Brzezniak:2013ab,Brzezniak:2016wz,Ondrejat:2010aa}. 

\subsection{Examples of quasi-Polish spaces}\label{sec:qPexamples}
In this subsection we give the definition 
and some examples of quasi-Polish spaces.  

First, given two measurable spaces 
$(\mathfrak{X}_i,\mathcal{M}_i),i=1,2$, 
by the statement ``$A$ is $\mathcal{M}_1/\mathcal{M}_2$", 
we mean that $A:(\mathfrak{X}_1,\mathcal{M}_1)
\to(\mathfrak{X}_2,\mathcal{M}_2)$ is measurable. 
Let $\mathcal{A}$ be a collection of subsets, or 
a collection of maps. On a few occasions, see for example 
\eqref{eq:sigma_algebra}, by $\Sigma(\mathcal{A})$ 
we mean the $\sigma$-algebra generated by $\mathcal{A}$. 

\begin{defin}[quasi-Polish space]\label{def:quasipolish}
A topological space $(Z,\tau)$ is \textit{quasi-Polish} 
if there is a sequence $\{f_n\}_{n\in\N}$ 
of continuous functions $f_n:Z\to [-1,1]$ 
separating points of $Z$. 
\end{defin}

Quasi-Polish spaces are Hausdorff. 
Below we exhibit some examples of commonly 
encountered quasi-Polish spaces, remarking 
specifically on the existence of a sequence of continuous 
functions to $[-1,1]$ that separates points. 
By considering the continuous injection 
$(f_1,f_2,\ldots): Z \hookrightarrow [-1,1]^\mathbb{N}$ 
of $Z$ into the Polish space $[-1,1]^\mathbb{N}$, 
one can recover many of the properties of 
Polish spaces for compact subsets of a 
quasi-Polish space $Z$, see \cite[Section 2]{Jakubowski:1997aa}. 
The topology induced by $\{f_n\}_{n\in \mathbb{N}}$, 
sometimes referred to as $\tau_f$, coincides 
with the topology of $Z$ on $\tau$-compact subsets. 
This is the cardinal property that allows
theorems on Polish spaces 
(such as the Skorokhod representation theorem) 
to carry over to quasi-Polish spaces.

\medskip 

\noindent \textit{Examples.}

\begin{enumerate}

\item If $(Z,\norm{\cdot})$ is 
a separable normed space (with dual $Z'$), 
then it is quasi-Polish. Indeed, let 
$\left\{\phi_n\right\}_{n\in\N}\subset Z'$ be such that
\begin{equation*}
	\norm{z}=\sup_n \, \langle \phi_n,z\rangle,
	\quad z\in Z. 
\end{equation*}
Define $f_n(z)=\frac2\pi\arctan\bigl(\langle 
\phi_n,z\rangle\bigr)$, $n\in\N$. 
Given $z_1\neq z_2$, choose an integer $m$ such that 
$\langle \phi_m,z_1-z_2\rangle > \frac12\norm{z_1-z_2}>0$. 
Whereupon $\langle\phi_m,z_1\rangle>\langle\phi_m,z_2\rangle$ 
and hence $f_m(z_1)>f_m(z_2)$. This also shows that 
$Z-w$ ($Z$ endowed with the weak topology $\tau_w$)  
is quasi-Polish. Therefore, the spaces $L^p([0,T]\times \T)-w$, 
$1\leq p<\infty$, are quasi-Polish; they are used 
in \eqref{eq:pathspaces} with $p=r$ and $p=2r$, $r\in [1,3/2)$.

\medskip

\item \label{example:C_tH_x}Let $(Z,\langle\cdot,\cdot\rangle)$ be a separable 
Hilbert space. Equipping $C([0,T];Z-w)$ with the 
locally convex topology generated by the seminorms $\norm{h}_\phi 
:=\sup_{0\leq t \leq T}\abs{\langle h(t),\phi \rangle}$ 
for $\phi\in Z$, the space $C([0,T];Z-w)$ 
becomes quasi-Polish \cite[Remark 4.1]{Brzezniak:2016wz}. 
An example is $C([0,T];H^1(\T)-w)$, which is used 
on a few occasions.

\medskip

\item If $(Z,\norm{\cdot})$ is a separable Banach space, 
then its dual $Z'$ endowed with the weak-star 
topology $\tau_\star$ is quasi-Polish. To see this, take an arbitrary 
countable dense subset $D\subset Z, D=\{d_1,d_2,\dots\}$. 
Given $\phi_1,\phi_2\in Z', \phi_1\neq \phi_2$, there 
must exist $d_n\in D$ such that $\phi_1(d_n)\neq \phi_2(d_n)$, 
because, if this were not the case, then 
$\phi_1(d)=\phi_2(d)$, $d\in D$, and thus $\phi_1\equiv \phi_2$. 
Define $f_n:Z'\to\R$, $\phi\mapsto\phi(d_n)$, $n\in \N$. 
Then $f_n$ is $\tau_\star$-continuous, and we 
conclude that $(Z',\tau_\star)$ is quasi-Polish, 
with a separating sequence provided by $\{f_n\}_{n\in\N}$. 
An example is the space of Radon measure 
(with a separable pre-dual), equipped 
with the weak-star topology.
\end{enumerate}
\medskip

For a quasi-Polish space $(Z,\tau)$, the 
point-separating sequence does not always 
characterise the topology $\tau$, because 
if $Z$ has the topology $\tau_f$ induced by functions $\{f_n\}$ 
that separate points, then changing 
the topology of $Z$ to the discrete topology, 
$Z$ remains quasi-Polish under the maps $\{f_n\}$.
In general, we have $\tau_f\subset \tau$. 
However, for the quasi-Polish spaces used in this paper, 
we will always know that $\tau_f=\tau$. 

We recall the following result 
\cite[Proposition 1.1.1]{BanachI:2016}, 
also known as ``the linear characterisation 
of the Borel $\sigma$-algebra", which 
we will need below.

\begin{lem}\label{lem:qpolish-borel-sigma-algebra}
Let $(Z,\norm{\cdot})$ be a separable normed space. 
Let $\{\phi_n\}_{n\in\N}\subset Z'$ 
be a norming sequence in the sense that
\begin{equation*}
	\norm{z}=\sup_{n\in \N}\,
	\bigl\langle \phi_n,z\bigr\rangle,
	\quad z\in Z. 
\end{equation*}
Denote by $\mathcal{B}(Z)$ the 
Borel $\sigma$-algebra on $Z$. 
Then $\mathcal{B}(Z)=\Sigma\bigl(\phi_n, n\in\N\bigr)$.
\end{lem}

In Section \ref{sec:Jak-Skor}, we made essential 
use of the (topological) space 
\begin{equation}\label{eq:LptLpx-weak}
	L^{p_1}\bigl(L^{p_2}_w\bigr)
	=L^{p_1}\bigl([0,T];L^{p_2}(\T)-w\bigr), 
	\quad p_1,p_2\in (1,\infty).
\end{equation}
To define $L^{p_1}\bigl(L^{p_2}_w\bigr)$, consider 
the classical Bochner space of equivalence 
classes of measurable functions $z:[0,T]\to L^{p_2}(\T)$ 
for which $\norm{z(\cdot)}_{L^{p_2}(\T)}\in L^{p_1}([0,T])$, 
denoted by $L^{p_1}(L^{p_2})=L^{p_1}([0,T];L^{p_2}(\T))$. 
Equipping this space with the locally 
convex topology generated by the seminorms
\begin{equation}\label{eq:semi-norm}
	L^{p_1}(L^{p_2})\ni z\mapsto \left(\int_0^T \abs{\int_\T 
	\phi(x) z(t,x)\,\d x}^{p_1} \d t\right)^{\frac{1}{p_1}},
	\quad \phi \in L^{p_2'}(\T),
\end{equation} 
where $\frac{1}{p_2}+\frac{1}{p_2'}=1$, we denote the 
resulting topological space by $L^{p_1}\bigl(L^{p_2}_w\bigr)$, 
see \eqref{eq:LptLpx-weak}.  

For notational simplicity in what follows, 
set $Z=L^{p_1}(L^{p_2})$ and denote by $\tau_N$ the 
new topology \eqref{eq:semi-norm}. 
We will then write $L^{p_1}\bigl(L^{p_2}_w\bigr)
=(Z,\tau_N)$. We claim that $(Z,\tau_N)$ 
is a quasi-Polish space. Indeed, we first notice that, given 
an arbitrary net $\left\{z_\alpha\right\}_\alpha\subset 
(Z,\tau_N)$ converging to $z\in (Z,\tau_N)$, 
then $z_\alpha\to z$ with respect to the 
standard weak topology $\tau_w$ 
of $Z$. In other words, $(Z,\tau_N)$ 
embeds continuously into $(Z,\tau_w)$. 
Trivially, $\tau_N$ is weaker than the 
norm topology of $Z$, called $\tau_s$. 
Consequently, we have $\tau_w\subset \tau_N \subset \tau_s$. 
By the separability of $Z$ (endowed with $\tau_s$) and 
the previous discussion of this appendix, we know 
that $(Z,\tau_w)$ is quasi-Polish. 
As a result, any separating sequence of continuous 
functions for this space will 
also be a separating sequence for $(Z,\tau_N)$, thereby 
turning it into a quasi-Polish space. Finally, from the 
inclusions $\tau_w\subset \tau_N \subset \tau_s$, we also 
obtain that $\mathcal{B}_{\tau_N}=\mathcal{B}_{\tau_s}$, because 
in a separable Banach space the Borel $\sigma$-algebra 
$\mathcal{B}_{\tau_w}$ generated by the weak topology 
coincides with the strong Borel $\sigma$-algebra 
$\mathcal{B}_{\tau_s}$, 
cf.~Lemma \ref{lem:qpolish-borel-sigma-algebra}, 
and since $\tau_N$ is an intermediate topology, this must 
hold for $\tau_N$ as well.

\medskip

In Section \ref{sec:equality-of-laws}, we 
used a quasi-Polish analogue of the 
Kuratowski--Lusin--Souslin (KLS) theorem, taken 
from \cite[Corollary A.2]{Ondrejat:2010aa} and 
\cite[Proposition C.2]{Brzezniak:2013ab}. This result 
is used repeatedly in Section \ref{sec:equality-of-laws}. 

\begin{lem}\label{thm:LS4qP}
Let $Z$ be a quasi-Polish space and let $Y$ 
be a Polish space for which exists a continuous 
injection $b:Y\to Z$. For any Borel set $B \subset Y$, 
the set $b[B]$ is Borel in $Z$.
\end{lem}

The proof is a direct application of the KLS theorem 
after the injection $Z\hookrightarrow [-1,1]^\mathbb{N}$, 
which puts us in the Polish space setting.

\medskip

New quasi-Polish spaces can be constructed by forming 
Cartesian products of countable collections of 
them (see, e.g., \cite{Brzezniak:2011aa} and 
next subsection). This fact is heavily 
used in Section \ref{sec:Jak-Skor}. 
In this paper, we avoided using 
\textit{intersections} of quasi-Polish spaces 
in our application of the Skorokhod--Jakubowski theorem 
\cite{Jakubowski:1997aa} (see Theorem \ref{thm:Jakubowski}). 
Let us consider a Skorokhod--Jakubowski representation 
$\left\{\tilde{u}_n\right\}$ of a sequence 
$\left\{u_n\right\}$, and suppose we need 
to know the a.s.~convergence of $\left\{\tilde{u}_n\right\}$ 
in two different spaces $Z_1$ and $Z_2$. It is then 
natural to use a space $Z$ for which
\begin{itemize}
	\item[(i)] $Z$ is quasi-Polish,
	
	\item[(ii)] compact subsets of $Z$ can be identified, 
	in order to verify tightness of the laws of $\left\{\tilde u_n\right\}$,

	\item[(iii)] $Z$ respects the topologies of $Z_1$ and $Z_2$, in 
	the sense that a.s.~convergence in $Z$ 
	implies a.s.~convergence in $Z_1$ 
	and $Z_2$ separately.
\end{itemize}
These three requirements are in tension.  
As the topology chosen with which to equip $Z$ 
is strengthened, (iii) is more easily satisfied, 
whereas (i) and (ii) are less easily so. 
For the intersection $Z=Z_1\cap Z_2$, endowed with 
the upper bound topology, (i) and (ii) 
are fulfilled as soon as $Z_1$ and $Z_2$ are 
quasi-Polish, since $Z$ embeds continuously 
in $Z_1$ and $Z_2$ by construction. However, 
additional arguments are required to find 
compact subsets of $Z$ to satisfy (ii) 
(see, e.g., \cite{Brzezniak:2013aa}); 
the reason is that there is no general way 
to construct compact subsets of $Z$ using compact 
subsets of $Z_1$ and $Z_2$. 
On the other hand, if one considers the 
Cartesian product with the product topology, 
the three requirements above are 
automatically satisfied. In particular, 
Tychonoff's theorem allows us to readily 
construct compact subsets of $Z$.
    
\subsection{Products of quasi-Polish spaces}

In Section \ref{sec:Jak-Skor}, we worked systematically 
with random variables defined on countable  
products of quasi-Polish spaces.

\begin{lem}\label{thm:a_0}
Let $\{Z_i\}_{i \in \mathbb{N}}$ be a countable collection 
of quasi-Polish spaces. Then $\mathfrak{X}
=\prod_{i \in \mathbb{N}} Z_i$, endowed with 
the product topology, is a quasi-Polish space. 
\end{lem}

\begin{proof}
This is immediate on invoking the definition of 
a quasi-Polish---that there is a countable, 
point-separating collection of maps 
$g_n:\mathfrak{X} \to [-1,1]$. Let $\pi_i:\mathfrak{X} \to Z_i$ 
be the $i$th canonical projection. 
Since there is a collection $f_{i,n} : Z_i \to [-1,1]$ 
for the $i$th factor space in $\mathfrak{X}$, the 
maps $\{f_{i,n}\circ \pi_i\}_{(i,n) \in \mathbb{N}^2}$ 
can be reordered to give $\{g_n\}_{n\in \mathbb{N}}$ on 
the product space $\mathfrak{X}$.\footnote{Banakh, Bogachev, 
and Kolesnikov \cite{BBK2005} derives the 
stronger conclusion of the weak Skorokhod property 
instead of the weak sequential Skorokhod property 
derived here, under the stronger assumption of the existence 
of a fundamental sequence of compact sets. 
We do not require this assumption, and our result applies to 
arbitrary countable collections of quasi-Polish spaces.}
\end{proof}

In what follows, we will continue to focus on 
products of quasi-Polish spaces and the measures 
that can be defined on them, starting with some 
subtle issues arising from the general 
non-coincidence of the Borel $\sigma$-algebra 
$\mathcal{B}\left(\prod_i Z_i\right)$ 
and the product Borel $\sigma$-algebra 
$\bigotimes_i \mathcal{B}(Z_i)$. To take an example, in 
Section \ref{sec:equality-of-laws}, we 
implicitly identified $\tilde{u}_0 \in H^1(\T)$ 
with $\tilde{u}(0) \in L^2(\T)$. By the equality of laws, 
the probability law of $\tilde{u}(0)$ is supported on $Z=H^1(\T)$. 
In order to identify $\tilde{u}_0 \in Z$ 
with $\tilde u(0) \in Z$,  we need to ensure 
that the joint law of $\bigl(\tilde{u}_0,\tilde{u}(0)\bigr)$ 
is supported on the diagonal 
$\Delta_{Z \times Z}=\{(z,z)\in Z\times Z:z\in Z\}$. 
For arbitrary topological spaces $Z$, 
this is not always possible, for the 
surprising reason that the diagonal 
$\Delta_{Z \times Z}$, whilst certainly in the 
Borel $\sigma$-algebra $\mathcal{B}(Z\times Z)$, 
is not necessarily in the product Borel 
$\sigma$-algebra $\mathcal{B}(Z)\otimes 
\mathcal{B}(Z)$, for large enough topologies on $Z$ 
(known as \textit{Nedoma's pathology} 
\cite[Chapter 15.9]{Schilling-book:2021}). 
However, this is no impediment in 
quasi-Polish spaces.

\begin{lem}\label{thm:a_1}
Let $Z$ be a quasi-Polish space. 
Then the diagonal $\Delta_{Z\times Z}$ belongs to 
$\mathcal{B}(Z)\times \mathcal{B}(Z)$, i.e., 
the diagonal is measurable.
\end{lem}

\begin{proof}
Let $\{f_n\}_{n \in \N}$ be the point-separating 
sequence of continuous maps, $f_n:Z \to [-1,1]$. 
Define the following class of subsets of $Z$:
\begin{equation*}
	\mathfrak{C} = \left\{ f_n^{-1}([q,r]):
	q,r\in \mathbb{Q},\,\, 0\leq q < r\leq 1,
	\,\, n\in\N\right \}.
\end{equation*}
The collection  $\mathfrak{C}$ is countable and 
in $\mathcal{B}(Z)$, because $f_n^{-1}([q,r])$ 
is closed. Let $x,y\in Z, x\neq y$. 
Choose $m\in\N$ such that $f_m(y)<f_m(x)$, and two rational 
numbers $0<q<r\leq 1$  such that $f_m(x)\in [q,r]$ 
and $f_m(y)<q$, i.e., $x\in f_m^{-1}([q,r])$ 
and $y\notin f_m^{-1}([q,r])$. As a result,
$\mathfrak{C}$ separates point of $Z$ 
and, by a theorem of Draveck\'{y} \cite[Theorem 1]{Drav1975}, 
the diagonal $\Delta_{Z \times Z}$ is measurable.
\end{proof}

\begin{lem}\label{thm:a_2}
Consider a quasi-Polish space $Z$ with a point-separating 
sequence of continuous maps $\{f_n : Z \to [-1,1]\}_{n\in\N}$, 
and denote by $\mathcal{B}_f$ the $\sigma$-algebra generated by 
$\{f_n\}_{n \in \mathbb{N}}$. Let $\mu:\mathcal{B}_f\to[0,1]$ 
be a tight probability measure. Define the $\sigma$-algebra
$$
\mathcal{B}_\ast(Z):=\left\{  V\subset Z: 
V\cap K \in \mathcal{B}(Z), 
\, \forall K\subset Z \, \text{compact}\right\},
$$
where $\mathcal{B}(Z)$ is the Borel $\sigma$-algebra of $Z$. 
Then there exists a unique Radon extension 
$\lambda :\mathcal{B}_\ast(Z)\to [0,1]$ of $\mu$.
\end{lem} 

\begin{proof}
Recall that the $\sigma$-algebra generated 
by the compact sets of $Z$ belongs to $\mathcal{B}_f$ 
\cite{Jakubowski:1997aa}. 
Then, for any $n\in\N$, there exists a compact 
$K_n\in\mathcal{B}_f$ such that $\mu(K_n^c)<n^{-1}$. 
Setting $K=\cup_n K_n \in \mathcal{B}_f$, it follows that 
$$
\mu(K^c)=\mu\left(\cap_n K_n^c\right) 
\leq \mu\left(K_n^c\right)<n^{-1},
$$
i.e., $\mu(K^c)=0$. Thus, the support 
of $\mu$ belongs to $K$. By \cite[Section 3]{Brzezniak:2016wz}, 
there exists a unique Radon extension 
$\lambda:\mathcal{B}_\ast(Z)\to [0,1]$ of $\mu$.
\end{proof}

\begin{rem}
Since for each $V \in \mathcal{B}(Z)$ 
and compact $K \subset Z$, $V \cap K \in \mathcal{B}(Z)$, 
it follows from the definition of $\mathcal{B}_\ast(Z)$ 
that $\mathcal{B}(Z)\subset \mathcal{B}_\ast(Z)$.
\end{rem}

Suppose $A_i$, $i\in \N$, lives on a quasi-Polish space 
and is measurable. The next lemma shows 
that $A=\{A_i\}_{i\in \N}$ is measurable with respect 
to the product of the individual Borel $\sigma$-algebras.  
    
\begin{lem}\label{thm:a_3}
Let $\{Z_i\}_{i \in \mathbb{N}}$ be a countable
collection of quasi-Polish spaces, and denote by 
$\mathcal{B}_i=\mathcal{B}(Z_i)$ the Borel 
$\sigma$-algebra on $Z_i$. 
Define $\X=\prod_{i\in \mathbb{N}} Z_i$, 
endowed with the product topology. 
Let $\pi_i:\X\to Z_i$ be the $i$th canonical projection. 
Consider a probability space $(\Omega,\mathcal{F},\mathbb{P})$ 
and  random variables $A_i:(\Omega,\mathcal{F})\to(Z_i,\mathcal{B}_i)$, 
i.e., for each $i$, $A_i$ is $\mathcal{F}/\mathcal{B}_{i}$-measurable.  
Finally, consider the unique map $A:\Omega\to \X$ characterised by
$\pi_i(A(\omega))=A_i(\omega)$ $\forall$ $i$ and $\omega$. 
Then $A$ is $\mathcal{F}/\bigotimes_{i \in \mathbb{N}}
\mathcal{B}_{i}$-measurable.
\end{lem} 

\begin{proof}
By countability, $\bigotimes_{i\in\N} \mathcal{B}_{i}$ 
is generated by the family $\left\{\prod_{i\in\N}E_i:
E_i \in \mathcal{B}_{i} \right\}$. It is therefore 
enough to check measurability for these sets only.  
Measurability here is evident, because
$$
\left \{A\in \prod_{i\in\N}E_i \right\}
=\bigcap_{i\in\N}\left\{A_i\in E_i\right\}
\in\mathcal{F}.
$$
\end{proof}

By the previous lemma, $A$ is generally 
only $\mathcal{F}/\bigotimes_{i\in\N} 
\mathcal{B}_{i}$-measurable, and is hence precluded from 
being a random variable with respect to the natural $\sigma$-algebra 
on $\mathfrak{X}$, namely the Borel 
$\sigma$-algebra $\mathcal{B}(\mathfrak{X})$ (in 
which case the term ``random mapping" is used), 
because generally for quasi-Polish spaces
we only have $\bigotimes_{i\in\N} \mathcal{B}_{i}
\subset \mathcal{B}(\mathfrak{X})$, where 
$\mathcal{B}(\mathfrak{X})$ is the Borel $\sigma$-algebra 
on $\mathfrak{X}$ with the product topology. 
Fortunately, in applications with random mappings 
whose laws are tight, this is not a major problem, 
for the reason conferred about in 
Remark \ref{rem:cont-map-measurable} below. 

We conclude this section by clarifying the 
relationship between the measures defined 
via restrictions and extensions on the hierarchy 
of $\sigma$-algebras introduced so far. 
Given a random variable $A$ on $(\Omega,\mathbb{P})$, let us
denote by $\mathbb{P}_A$ its law $\mathbb{P}\circ A^{-1}$.
    
\begin{lem}\label{thm:a_4}
Let $\{Z_i\}_{i \in \mathbb{N}}$, $\mathfrak{X}$, 
$\mathcal{B}_{i}$, $\mathcal{B}(\mathfrak{X})$, $\pi_i$ 
be defined as in Lemma \ref{thm:a_3}. 
For each $i \in \mathbb{N}$, let 
$\left\{f_{i,n}:Z_i \to [-1,1]\right\}_{n \in \mathbb{N}}$ 
be the point-separating sequence of continuous 
maps linked to $Z_i$. For $i,n\in \N$, define
$h_{i,n}=f_{i,n}\circ \pi_i$ and denote by
$\mathcal{B}_f$ the $\sigma$-algebra generated by 
$\left\{h_{i,n}\right\}_{(i,n)\in \mathbb{N}^2}$. 
Finally, set 
$$
\mathcal{B}_\ast(\mathfrak{X})
=\bigl\{V\subset \mathfrak{X}:V\cap K \in 
\mathcal{B}(\mathfrak{X}), 
\, \forall K\subset \mathfrak{X}
\, \text{compact}\bigr\}.
$$
For each $i\in\mathbb{N}$, let 
$A_{i,\nu}:\bigl(\Omega,\mathcal{F},\mathbb{P}\bigr)
\to (Z_i,\mathcal{B}_i)$ be a family of random variables, 
indexed over $\nu\in (0,1)$, with a corresponding 
tight family of laws $\left\{\mu_{i,\nu}\right\}_{0<\nu<1}$. 
Let $A_\nu:\Omega \to \mathfrak{X}$ be uniquely 
characterised by $\pi_i(A_\nu(\omega))=A_{i,\nu}(\omega)$ for all 
$i$ and $\omega$, and denote by $\eta_\nu$ the law of 
$A_\nu$ restricted to $\mathcal{B}_f$. Finally, denote 
by $\mathcal{K}$ the family of compact subsets of $\X$.
\begin{itemize}
	\item[(i)] We have the inclusions
	\begin{equation*}
		\Sigma(\mathcal{K}) \subset 
		\mathcal{B}_f\subset 
		\bigotimes_{i\in \mathbb{N}} \mathcal{B}_i 
		\subset \mathcal{B}(\mathfrak{X}) 
		\subset \mathcal{B}_\ast(\X).
	\end{equation*}

	\item[(ii)] For each $\nu$, the law $\eta_\nu$ of $A_\nu$ 
	can be uniquely extended to $\mathcal{B}_\ast(\X)$ 
	as a Radon probability measure $\lambda_\nu$. The family 
	$\left\{\lambda_\nu\right\}_{0<\nu<1}$ is tight.

	\item[(iii)] The restriction of $\lambda_\nu$ to 
	$\bigotimes_{i\in \mathbb{N}}\mathcal{B}_i$ 
	is $\mathbb{P}_{A_\nu}$.
\end{itemize}
\end{lem}
        
\begin{rem}
Part (iii) of Lemma \ref{thm:a_4} can be summed 
up in the assertion that the diagram below commutes:
$$
\begin{tikzcd}
\mathbb{P}_{A_\nu} \arrow[r, "{\rm id }"] 
\arrow[d, "|_{\mathcal{B}_f}"] 
& {\lambda_\nu}\Big|_{\bigotimes_i \mathcal{B}_i} 
\\ \eta_\nu \arrow[r, "\mathrm{ext}"] & 
\lambda_\nu \arrow[u, "|_{\bigotimes_i \mathcal{B}_i}"]                      
\end{tikzcd}
$$
Here, $\mathrm{ext}$ denotes the 
extension to $\mathcal{B}_\ast(\X)$.
\end{rem}

\begin{proof}
We divide the proof into three natural steps.

\medskip

\noindent \textit{Claim (i).} Given any Borel 
set $B \subset [-1,1]$,
\begin{align*}
	h_{i,n}^{-1}(B)
	& =\left\{h_{i,n}\in B\right\}
	=\left\{\pi_i\in f_{i,n}^{-1}(B)\right\} 
	\\ & 
	= Z_1\times\cdots \times Z_{i-1}\times f_{i,n}^{-1}(B)
	\times Z_{i+1}\times \cdots
	\in \bigotimes_{i\in\N} \mathcal{B}_{i}.
\end{align*}
By construction, we infer 
$\mathcal{B}_f\subset \bigotimes_{i\in\N} \mathcal{B}_{i} 
\subset \mathcal{B}(\X)$. The inclusion 
$\mathcal{B}(\X) \subset \mathcal{B}_\ast(\X)$ 
is justified in the proof of Lemma \ref{thm:a_2}. 
The final inclusion $\Sigma(\mathcal{K})\subset \mathcal{B}_f$ 
is recorded in \cite[Section 3]{Brzezniak:2016wz}; it follows 
from the fact that the topology of $Z$ and the topology 
induced by the separating sequence 
$\left\{h_{i,n}\right\}_{(i,n)\in \mathbb{N}^2}$ 
coincide on compact subsets.
   
\medskip

\noindent \textit{Claim (ii).} For each fixed $i$, by 
tightness of the laws $\left\{\mu_{i,\nu}\right\}_{\nu}$ 
of $\left\{ A_{i,\nu}\right\}_{\nu}$, 
for each $\eps>0$, there exists a compact 
set $Z_{i,\eps}\subset Z_i$ such that
$$
\sup_\nu \mathbb{P}_{A_{i,\nu}}
\left(Z_{i,\eps}^c\right)< 2^{-i}\eps.
$$
Set $Z_\eps=\prod_{i\in\N}Z_{i,\varepsilon}$,
which is a compact subset of $\X$ by the Tychonoff theorem. 
Moreover, $Z_\varepsilon$ belongs to $\mathcal{B}_f$. 
By the inclusion
$$
\bk{\prod_{i\in\N}Z_{i,\varepsilon}}^c 
\subset \bigcup_i \X_1\times\cdots
\times\X_{i-1}\times 
Z_{i,\varepsilon}^c\times\X_{i+1}\times\cdots
$$
and the sub-additivity of measures, we 
deduce that, uniformly in $\nu$,
\begin{align*}
	\mathbb{P}_{A_\nu}\left(Z_\varepsilon^c\right) 
	& \leq \sum_{i=1}^\infty 
	\mathbb{P}_{A_\nu}\left(\X_1\times\cdots
	\times\X_{i-1}\times Z_{i,\varepsilon}^c
	\times\X_{i+1}\times\cdots\right)
	\\ & = \sum_{i=1}^\infty 
	\mathbb{P}_{A_{i,\nu}}\left(Z_{i,\varepsilon}^c\right)
	<\varepsilon,
\end{align*}
where $\mathbb{P}_{A_\nu}$ is the law 
of $A_\nu$ on $\bigotimes_{i\in\N} \mathcal{B}_{i}$. 
It follows that  $\left\{\mathbb{P}_{A_\nu}\right\}_{0<\nu<1}$ 
is tight on $\bigotimes_{i\in\N} \mathcal{B}_{i}$ 
and, \textit{a fortiori}, 
$\left\{\mathbb{P}_{A_\nu}\big|_{\mathcal{B}_f}
\right\}_{0<\nu<1}$ is tight on $\mathcal{B}_f$.

\medskip

\noindent \textit{Claim (iii).} Set $\Bar{\lambda}_\nu=\lambda_\nu
\big|_{\bigotimes_{i\in\N}\mathcal{B}_{i}}$ and let 
$F\in\bigotimes_{i\in\N}\mathcal{B}_{i}$ be arbitrary. 
By definition, for any compact $K\in\mathcal{K}$, 
$$
\Bar{\lambda}_\nu(K) = \lambda_\nu(K)
=\eta_\nu (K)=\mathbb{P}_{A_\nu}(K).
$$
In particular, if $K\subset F$, then 
$\mathbb{P}_{A_\nu}(F) \geq \mathbb{P}_{A_\nu}(K)$ and 
$\mathbb{P}_{A_\nu}(F)\geq \Bar{\lambda}_\nu(K)
=\lambda_\nu(K)$. Therefore,
$$
\mathbb{P}_{A_\nu}(F)\geq 
\sup_{K\in\mathcal{K},K\subset F} 
\Bar{\lambda}_\nu(K)
=\sup_{K\in\mathcal{K},K\subset F} 
\lambda_\nu(K). 
$$
Since $\lambda_\nu$ is Radon on $\mathcal{B}_\ast(\X)$,
$$
\mathbb{P}_{A_\nu}(F)
\geq \lambda_\nu(F)=\bar{\lambda}_\nu(F), 
\quad F\in \bigotimes_{i\in\N}\mathcal{B}_{i}.
$$
Using the arbitrariness of $F$ by considering 
$\mathfrak{X}\backslash F$ in place of $F$, this 
majorisation implies $\mathbb{P}_{A_\nu}
=\Bar{\lambda}_\nu=\lambda_\nu
\big|_{\bigotimes_{i\in\N}\mathcal{B}_{i}}$.
Nevertheless, observe that this does 
not imply that $A_\nu$ becomes either 
$\mathcal{F}/\mathcal{B}_\ast(\X)$ measurable 
or $\mathcal{F}/\mathcal{B}(\mathfrak{X})$ measurable.
\end{proof}

The next remark is important and 
used extensively throughout the paper.  

\begin{rem}\label{rem:cont-map-measurable}
Even though $A_\nu$ is not 
$\mathcal{F}/\mathcal{B}(\mathfrak{X})$ measurable 
in general, we still have the following 
crucial fact: because 
$\left\{\prob_{A_\nu}\right\}_{0<\nu<1}$ is tight, 
as soon as we assume that the original probability space
$\bigl(\Omega,\mathcal{F},\mathbb{P}\bigr)$ 
is complete, it follows that 
$$
f\circ A_\nu \quad \text{is} 
\quad \text{$\mathcal{F}/\mathcal{B}(\R^k)$ 
measurable},
$$
for any continuous function $f$ 
from $\mathfrak{X}$ to $\R^k$, see 
\cite[page 170]{Jakubowski:1997aa} 
for further details. 
\end{rem}

\subsection{The Skorokhod--Jakubowski theorem}\label{sec:SJThm_appendix}

We recall the following result due 
to Jakubowski \cite[Theorem 2]{Jakubowski:1997aa}.

\begin{thm}[Jakubowski]\label{thm:Jakubowski}
Let $Z$ be a quasi-Polish space. Consider 
a sequence $Y_j:\bigl(\Omega,\mathcal{F},\mathbb{P}\bigr) 
\to \bigl(Z,\mathcal{B}_Z\bigr)$ of random mappings 
with a tight sequence of laws $\mu_j$, $j\in\N$. 
Then there exist a subsequence $\left\{Y_{j_k}\right\}_{k\in \N}$ 
and $Z$-valued random variables $V_0,V_1,V_2,\dots$, 
defined on $\bigl([0,1],\mathcal{B}([0,1]),
\mathrm{Leb}\bigr)$, where $\mathrm{Leb}$ is 
the Lebesgue measure, such that
$$
Y_{j_k}\sim V_k,\quad k\in \N,
\qquad
V_k(\xi)\tok V_0(\xi)\quad 
\text{for a.e.~$\xi\in [0,1]$.}
$$
\end{thm}

Recently, the Jakubowski theorem was used by many authors 
to prove existence of solutions to various classes 
of SPDEs, see Section \ref{sec:intro} for a 
few references. Here we only recall the first 
works \cite{Ondrejat:2010aa,Brzezniak:2013ab}.

The following simple but useful lemma is
deployed in the proof of Theorem \ref{thm:skorohod}.  

\begin{lem}[a.s.~representations of 
nonlinear compositions]\label{lem:SJ-nonlinear}
Let $Z,W$ be quasi-Polish spaces, and 
suppose $F: Z \to W$ is a Borel function. 
Consider a sequence $\left\{Y_j\right\}_{j=1}^\infty$ 
of $Z$-valued random variables on 
$\bigl(\Omega,\mathcal{F},\mathbb{P}\bigr)$. 
Denote by $\bigl(V_k,\widetilde{F}_k\bigr)$ 
the a.s. representations of $\bigl(Y_j,F(Y_j)\bigr)$, 
see Theorem \ref{thm:Jakubowski}. Then 
$$
\widetilde{F}_k=F(V_{n_k}), 
\quad \text{a.s.}, \quad 
k\in \N.
$$
\end{lem}

\begin{proof}
We divide the proof into two steps.

\medskip

\noindent \textit{Step 1}. Let $(Z,\mathcal{B}_Z)$ 
and $(W,\mathcal{B}_W)$ be quasi-Polish 
spaces with $\sigma$-algebras $\mathcal{B}_Z$ and 
$\mathcal{B}_W$. Consider a mapping $F:Z\to W$ that 
is $\mathcal{B}_Z / \mathcal{B}_W$ measurable. 
Define the mapping $H:Z\times W \to W\times W$ by 
$$
(z,w)\mapsto 
H(z,w)=\bigl(H_1(z,w),H_2(z,w)\bigr)=(F(z),w).
$$
Then $H$ is $\mathcal{B}_Z\otimes\mathcal{B}_W/
\mathcal{B}_W\otimes\mathcal{B}_W$ measurable. The 
validity of this claim comes from the 
$\mathcal{B}_Z\otimes\mathcal{B}_W/\mathcal{B}_W$ measurability 
of the coordinate mappings $(z,w)\stackrel{H_1}{\mapsto} F(z)$ and 
$(z,w)\stackrel{H_2}{\mapsto} w$, 
see, e.g., ~\cite[Lemma 1.9]{Kallenberg-book:2021}.

\medskip

\noindent \textit{Step 2}. Consider three random mappings
$U:\big(\Omega,\mathcal{F},\mathbb{P}\bigr)
\to (Z,\mathcal{B}_Z)$, 
$U':\bigl(\Omega',\mathcal{F}',\mathbb{P}'\bigr)
\to (Z,\mathcal{B}_Z)$, 
and $V:\bigl(\Omega',\mathcal{F}',\mathbb{P}')
\to (W,\mathcal{B}_W)$. Suppose $F:Z\to W$ 
is a (deterministic) mapping that is 
$\mathcal{B}_Z / \mathcal{B}_W$ measurable. 
If $\bigl(U,F(U)\bigr)\sim (U',V)$, then 
$V=F(U')$, a.s. It remains to prove 
this assertion, which implies the claim of the lemma. 

By Step 1 and the measurability of 
compositions of measurable mappings, we conclude that
$H(U,F(U))$ is $\mathcal{F}/\mathcal{B}_W\otimes\mathcal{B}_W$
and $H(U',V)$ is $\mathcal{F}'/\mathcal{B}_W\otimes\mathcal{B}_W$. 
Moreover, we have $H(U,F(U))\sim H(U',V)$. 
Since the diagonal $\Delta_{W\times W}$ belongs 
to $\mathcal{B}_W\otimes\mathcal{B}_W$, 
cf.~Lemma \ref{thm:a_1}, we obtain
$$
\mathbb{P}\bigl(\left\{\omega: H(U,F(U))
\in \Delta_{W\times W}\right\}\bigr) 
=\mathbb{P}'\bigl(\left\{\omega': H(U',V)
\in \Delta_{W\times W}\right\}\bigr).
$$
Trivially, $H(U,F(U))\equiv \bigl(F(U),F(U)\bigr)
\in\Delta_{W\times W}$, and whence
$$
1=\mathbb{P}'\bigl(\left\{\omega': H(U',V)
\in \Delta_{W\times W}\right\}\bigr)
=\mathbb{P}'\bigl(\left\{\omega':
\bigl(F(U'),V\bigr)\in 
\Delta_{W\times W}\right\}\bigr).
$$
This shows that $V=F(U')$, thereby 
ending the proof of the lemma. 
\end{proof}

\begin{rem}\label{rem:compact_relcompact}
Theorem \ref{thm:Jakubowski} applies to
tight sequences of probability measures, where tightness
implies that the global behaviour of the measures 
``concentrates" on a compact set. Since we are 
not generally working in a metric space setting, 
to prove that a subset $K$ is compact, one would a priori 
be required to use nets rather than sequences. 
However, an essential property of quasi-Polish spaces is 
that one can restrict considerations to sequences; as 
a matter of fact, a subset $K$ of a quasi-Polish space 
is compact if and only if it is sequentially 
compact \cite{Jakubowski:1997aa}.
\end{rem}

The coincidence of compactness and sequential compactness 
is not necessarily inherited by the relativised notions of 
relative compactness and relative sequential compactness. 
Using sequences is advantageous when assessing 
the precompactness of subsets, as per the application 
of the Skorokhod--Jakubowski theorem. 
In these situations, we want to know that the 
closure of relatively sequentially compact 
subsets is at least sequentially compact. 
Let us delve into supplementary structures that quasi-Polish 
spaces must possess, in order to ensure the 
coincidence of relative compactness and 
relative sequential compactness. 

We recall first that a subset $A$ of a 
Hausdorff space $(\X,\tau)$ is 
\begin{itemize}
	\item[a.] \textit{relatively  compact} if the closure 
	of $A$ is compact in $\X$;
	
	\item[b.]\textit{relatively countably compact} 
	if each sequence in $A$ has a cluster point in $\X$;
	
	\item[c.] \textit{relatively sequentially compact} if each 
	sequence in $A$ has a convergent sub-sequence with limit in $\X$.
\end{itemize}

The requisite additional structure is the following: 

\begin{defin}[{\cite[page 30]{Flo1980}}]\label{def:angelic}
A topological Hausdorff space $(\X,\tau)$ is \textit{angelic} 
if for every relatively countably compact 
set $A\subset\X$ the following holds:
\begin{enumerate}
	\item[(i)] $A$ is relatively compact;
	
	\item[(ii)] For each $x\in \overline{A}$ there 
	is a sequence in $A$ which converges to $x$.
\end{enumerate}
\end{defin}

Angelic spaces have several remarkable properties:

\begin{lem}[{\cite[Lemma 3.1, Theorem 3.3]{Flo1980}}]
In any angelic space, 
\begin{enumerate}
	\item[(i)] compact, countably compact and 
        sequentially compact subsets coincide;

        \item[(ii)] relatively compact, relatively 
        countably compact and relatively sequentially 
        compact subsets coincide.
\end{enumerate}
\end{lem}

Which spaces are angelic? A theorem of 
Eberlein and \v{S}mulian provides us 
with necessary conditions:

\begin{lem}[{\cite[Theorem 3.10]{Flo1980}}]
Let $(\X,\tau)$ be a locally convex metrizable space. 
Let $\tau_w$ denote its weak topology. 
Then $(\X,\tau_w)$ is angelic. Moreover, 
if ${\tau}_r$ is a regular topology 
finer than $\tau_w$, then $(\X,{\tau}_r)$ is angelic.
\end{lem}

The regularity of a topology plays a pivotal role in 
affirming the angelic nature of a space. For 
relatively sequentially compact subsets $A$ of 
regular spaces, the closure $\overline{A}$ 
maintains sequential compactness.

\medskip

As we conclude this section, we will give 
three pertinent examples. 
\begin{itemize}
	\item If $\X$ is a normed space, then $(\X,\tau_w)$ is angelic.

	\item $\X=L^{p_1}\bigl([0,T];L^{p_2}(\T)-w\bigr)$, 
	with $1<p_1,p_2<\infty$: Recall the topologies 
	$\tau_s$ and \eqref{eq:semi-norm},  denoted by $\tau_N$, on 
	the Bochner space $Z=L^{p_1}([0,T];L^{p_2}(\T))$. 
	Clearly, $(Z,\tau_s)$ is normed; therefore, 
	when endowed with $\tau_w$, it is angelic. 
	Moreover, $\tau_N$ is (completely) regular, 
	because it originates from a family of semi-norms, 
	i.e., it is locally convex. Since $\tau_w\subset \tau_N$, 
	$(Z,\tau_N)$ is angelic.
	
	\item $\X=C([0,T];Z-w)$ for a separable 
        Hilbert space $Z$, see Example \eqref{example:C_tH_x} 
        of Section \ref{sec:qPexamples}: Let $(\varphi_n)_n\subset Z$ 
        be such that $\norm{h}_Z=\sup_n\abs{\langle\varphi_n,h\rangle_Z}$ 
	for all $h\in Z$. Define the following semi-norms on $C([0,T];Z-w)$:
	$$
	\norm{f}_{\varphi_n} := \sup_{0\le t \le T}
	\abs{\left\langle\varphi_n,f(t)\right\rangle_Z},
	\quad f \in C([0,T];Z-w), \quad n\in \N.
	$$
	The locally convex topology $\tau_0$ generated by 
	the seminorms $\norm{\cdot}_{\varphi_n}$, $n\in \N$, is 
	(completely) regular. The topology $\tau_0$ is 
	metrizable \cite[Remark 4.2]{Brzezniak:2016wz}, 
	and thus $\big(C([0,T];Z-w),\tau_0\big)$ 
	is a locally convex metrizable space. 
	Denote by $\tau_w$ its weak topology. Then 
	$\big(C([0,T];Z-w,\tau_w\big)$ and 
	$\big(C([0,T];Z-w),\tau_0\big)$ are angelic, 
	because trivially $\tau_0\supset \tau_w$. 
	Furthermore, denote by $\tau$ the regular 
	locally convex topology generated by the 
	semi-norms $\norm{\cdot}_\phi$ in 
	Example \eqref{example:C_tH_x}. Since $\tau\supset\tau_0$, 
	the quasi-Polish space $\bigl(C([0,T];Z-w),\tau\bigr)$ is angelic.   
\end{itemize}
   
\section{Regularisation errors}
\label{sec:commutator-est}

In Section \ref{sec:wk_limits_section}, we derived 
the SPDE satisfied by $S(\tilde q)$, where $\tilde q$ 
solves the second-order transport-type SPDE \eqref{eq:tq-limit-SPDE} 
and $S:\R\to \R$ is a nonlinear function. 
This renormalisation step involved regularising the process 
$\tilde q$ by a spatial mollifier $J_\delta$, which
generates several error terms. Below we reproduce some
convergence results---but not their proofs---for
controlling these error terms. Analysing one of 
the (noise-related) terms requires a non-standard 
commutator estimate that goes beyond 
the DiPerna--Lions folklore lemma (see Proposition 
\ref{thm:commutator2} below). 
Similar estimates have been 
used recently in \cite{Punshon-Smith:2018aa} 
and \cite{Holden:2020aa}.

\begin{lem}[first order commutator errors]
\label{thm:commutator1}
Consider 
$$
w\in L^4\bigl(\Omega;L^\infty([0,T];H^1(\T))\bigr),
$$ 
and suppose $\sigma \in W^{1,\infty}(\T)$. Let $J_\delta$ 
be a standard Friedrichs mollifier in $x$, and set 
$w_\delta=w*J_\delta$. 
Define the error processes
\begin{equation*}
	\begin{aligned}
		& E^{(1)}_\delta=E^{(1)}_\delta(w,v)
		=\bigl(w \,\pd_x w \bigr)*J_\delta
		-w_\delta \,\pd_x w_\delta,
		\\
		& E^{(2)}_\delta= E^{(2)}_\delta(w)
		=\bigl( \sigma \,\pd_x w\bigr)
		*J_\delta -\sigma \,\pd_x w_\delta,
		\\ 
		& E^{(3)}_\delta=E^{(3)}_\delta(w) 
		=-\frac12 \bigl(\sigma\,
		\pd_x\bk{\sigma\,\pd_x w} \bigr)*J_\delta  
		+\frac12 \sigma\,\pd_x\bk{\sigma\,\pd_x w_\delta} .
	\end{aligned}
\end{equation*}
The following convergences hold:
\begin{equation}\label{eq:commutator_converge1}
	\begin{split}
		&\Ex \norm{\pd_x E^{(1)}_\delta}_{L^1([0,T]\times \T)}
		\todelta 0,
		\quad \Ex \norm{E^{(2)}_\delta}_{L^2([0,T];H^1(\T))}^2
		\todelta 0,
		\\ & \Ex \norm{E^{(3)}_\delta}_{L^2( [0,T] \times \T)}^2
		\todelta 0.
	\end{split}
\end{equation}
\end{lem}

The first and second parts 
of \eqref{eq:commutator_converge1} come from 
\cite[Lemma 2.3]{Lions:NSI} and 
\cite[Lemma II.1]{DiPerna:1989aa}, respectively. 
For the final part, see 
\cite[Lemma 7.1]{HKP-viscous}.

To handle a regularisation error 
linked to the stochastic part of  
\eqref{eq:tq-limit-SPDE}, we need the 
following proposition, which is a consequence of
a second order commutator estimate found 
in \cite[Lemma 7.3]{HKP-viscous}:

\begin{prop}[It\^{o}--Stratonovich related error term
{\cite[Proposition 7.4]{HKP-viscous}}]
\label{thm:commutator2}
Let $S \in \dot{W}^{2,\infty}(\R)$ be 
such that $S'(r)=O(r)$ and 
$\sup_{r\in \R} {\abs{S''(r)}} <\infty$. 
 Let $w$, $w_\delta$, $E^{(2)}_\delta$, 
and $E^{(3)}_\delta$ be defined as in 
Lemma \ref{thm:commutator1}. For each
$\varphi \in C^\infty([0,T]\times \T)$, 
the following convergence holds:
\begin{equation*}
	\begin{aligned}
		\Ex \int_0^T\biggl| \, \int_\T &
		-\varphi S'(\pd_x w_\delta) \pd_x E^{(3)}_\delta 
		\\ & +\varphi  S''(\pd_x w_\delta) \bk{\frac12
		\,\abs{\pd_x E^{(2)}_\delta}^2
		+\pd_x \bk{\sigma \,\pd_x w_\delta}
		\,\pd_x E^{(2)}_\delta}
		\,\d x\biggr| \,\d t\todelta 0.
	\end{aligned}
\end{equation*}
\end{prop}

\section{Temporal continuity in $H^1$ for viscous equation}

We consider the viscous equation \eqref{eq:u_ch_ep}. 
Since $\ep > 0$ is fixed in this section, 
we suppress the $\ep$-subscript. 
In \cite[Proposition 7.8]{HKP-viscous}, the 
authors demonstrated that
$$
\lim_{t \to t_0} \Ex \norm{u(t) - u(t_0)}_{H^1(\T)}^2 = 0,
$$
for every $t_0$. Furthermore, they posited that $u$ 
belongs to the space $L^{p_0}_\omega C_tH^1_x$, 
with $p_0>4$ defined in \eqref{eq:u0-cond-Hmp0}, implying 
that $u$ is almost surely continuous on the interval $[0,T]$ 
with values in $H^1(\T)$. However, they did not 
provide a comprehensive proof to support this latter assertion. 
The aim of this appendix is to establish this 
temporal continuity assertion, which is utilised in this paper.

\medskip

\begin{lem}[temporal continuity in $H^1$ for viscous equation]
\label{thm:cauchy_CtH1_2}
Consider a solution $u$ to \eqref{eq:u_ch_ep} with initial 
condition $u_0$, as guaranteed by Theorem \ref{thm:bounds1}. 
Specifically, $u$ satisfies 
$u \in L^{p_0}\bigl(\Omega;L^\infty([0,T];H^1(\T))\bigr)
\cap L^2\bigl(\Omega; L^2([0,T];H^2(\T))\bigr)$ 
and, $\mathbb{P}$-a.s., $u \in C([0,T];H^1(\T)-w)$. 
Then we have the inclusion $u \in 
L^{\bar{p}}\bigl(\Omega;C([0,T];H^1(\T))\bigr)$, 
for any $\bar{p}<p_0$. 
\end{lem}

\begin{proof}
Let $\{J_\delta\}_{\delta > 0}$ be a spatial 
Friedrichs mollifier. We continue to employ the 
notation $f_\delta = f * J_\delta$. Using the a.s.~inclusion 
$u \in C_t(H_x^1-w)$, we find that the quantities $u_\delta$ and 
$q_\delta=\pd_x u_\delta$ exhibit time-continuity, pointwise in $x$. 
It then follows quite straightforwardly that $u_\delta$ belongs 
to the space $C_t H^1_x$ a.s., for each fixed $\delta>0$. 

For given $\delta, \eta > 0$, applying Itô's formula 
to the mollified version of the SPDE \eqref{eq:u_ch_ep} and 
its $x$-derivative, we obtain
\begin{equation}\label{eq:cauchy_interim1}
	\begin{aligned}
		\frac12 \norm{u_\delta - u_\eta}_{H^1_x}^2(t)
		& = \frac12  \norm{u_{0,\delta} - u_{0,\eta}}_{H^1_x}^2
		\\ &\qquad \,\, 
		+ \int_0^t \sum_{i=1}^8 I_i \,\d s
		+\int_0^t \bk{M_1 + M_2 }\,\d W, 
	\end{aligned}
\end{equation}
where
\begin{align*}
	I_1 & = \ep \int_{\T}\bk{u_\delta - u_\eta}\, 
	\pd_{xx}^2 \bk{u_\delta - u_\eta}\,\d x, 
	\\
	I_2 &=  - \int_{\T} \bk{u_\delta - u_\eta}\, 
	\bk{\bk{u\,\pd_x u + \pd_x P}_\delta 
	- \bk{u \,\pd_x u + \pd_x P}_\eta} \,\d x,
	\\
	I_3 & = \frac12 \int_{\T} \bk{u_\delta - u_\eta}
	\bk{\bk{\sigma \,\pd_x \bk{\sigma \,\pd_x u}}_\delta 
	-\bk{\sigma \,\pd_x \bk{\sigma \,\pd_x u}}_\eta }\,\d x,
	\\
	I_4 & = \frac12 \int_{\T} \abs{ \bk{\sigma \pd_x u}_\delta
	 -\bk{\sigma \pd_x u}_\eta }^2\,\d x, 
	 \\
	 I_5 & = \ep \int_{\T}\pd_x \bk{u_\delta - u_\eta}\, 
	 \pd_{xxx}^3 \bk{u_\delta - u_\eta}\,\d x, 
	 \\ 
	 I_6 & = -\int_{\T} \pd_x \bk{u_\delta - u_\eta}\, 
	 \pd_x \bk{\bk{u\,\pd_x u + \pd_x P}_\delta 
	-\bk{u \,\pd_x u + \pd_x P}_\eta} \,\d x,
	\\
	I_7 & = \frac12 \int_{\T} \pd_x\bk{u_\delta-u_\eta}\,\pd_x
	\bk{\bk{\sigma \,\pd_x \bk{\sigma \,\pd_x u}}_\delta
	-\bk{\sigma \,\pd_x \bk{\sigma \,\pd_x u}}_\eta }\,\d x,
	\\
	I_8 & = \frac12 \int_{\T} \abs{\pd_x\bk{\sigma \pd_x u}_\delta
	-\pd_x \bk{\sigma \pd_x u}_\eta }^2\,\d x, 
	\\ 
	M_1 &=  \int_{\T} \bk{u_\delta-u_\eta} 
	\bk{\bk{\sigma \,\pd_x u}_\delta
	-\bk{\sigma \,\pd_x u}_\eta}\,\d x,
	\\ M_2 & =  \int_{\T} \pd_x \bk{u_\delta-u_\eta} 
	\,\pd_x \bk{\bk{\sigma \,\pd_x u}_\delta
	-\bk{\sigma \,\pd_x u}_\eta}\,\d x. 
\end{align*}

Applying integration-by-parts, we can ascertain that 
$I_1, I_5 \le 0$, a.s.

We shall use repeatedly the fact that 
$u \in L^2_{\omega,t}H^2_x$, and thus, as $\delta \to 0$, \\
$\bigl(u_\delta,\, \pd_x u_\delta,\, \pd_{xx}^2 u_\delta\bigr) 
\to \bigl(u, \,\pd_x u,\,\pd_{xx}^2 u\bigr)$ in $L^2_{\omega,t,x}$. 
In addition, using that $\sigma \in W^{2,\infty}$, 
$\bigl(u_\delta, \bk{\sigma\, \pd_x u}_\delta, 
\pd_x\bk{\sigma\,\pd_x u}_\delta\bigr) 
\to \bigl(u,\sigma\,\pd_x u, \pd_x\bk{\sigma\,\pd_x u}\bigr)$ 
in $L^2_{\omega, t, x}$. 
The convergences just mentioned 
immediately imply that 
$\Ex \int_0^T I_4 + I_8 \,\d t
\xrightarrow{\delta, \eta \downarrow 0} 0$.

For $I_3$, we use the 
Cauchy--Schwarz inequality to get
\begin{equation}\label{eq:I3-tmp}
	\int_0^T \abs{I_3}\,\d t 
	\lesssim_\sigma \norm{u_\delta - u_\eta}_{L^2_{t,x}} 
	\norm{f_\delta - f_\eta}_{L^2_{t,x}},
\end{equation}
where $f=\sigma \,\pd_x \bk{\sigma \,\pd_x u}\in L^2_{\omega,t,x}$. 
On the right-hand side, we take an expectation.
Subsequently, we utilize the Cauchy--Schwarz inequality. 
Each of the factors of the term on the right-hand 
side tend to $0$ in $L^2(\Omega)$ as 
$\delta, \eta \downarrow 0$.  After an integration 
by parts, the integral $\int_0^T \abs{I_7} \, \d t$ can be bounded 
in a similar manner as stated in \eqref{eq:I3-tmp}, but with the first 
factor on the right-hand side replaced by 
$\norm{u_\delta-u_\eta}_{L^2_t H^2_x}$. 
Consequently, the integral can be treated 
using the same approach.
 
 For the integrals $I_2$ and $I_6$, we again  
apply the Cauchy--Schwarz inequality and 
use the fact that $u \in L^2_{\omega}L^2_tH^2_x 
\cap L^2_{\omega} L^\infty_t H^1_x$. 
To be more precise, adding $I_2$ and $I_6$,  and using 
the property that $K*f = f-\pd_{xx}^2 K* f$, we find
\begin{align*}
	\abs{I_2 + I_6} 
	& =\bigg| \int_{\T}\bk{u_\delta - u_\eta} 
	\bk{\bk{u\,\pd_x u}_\delta - \bk{u\,\pd_x u}_\eta} \,\d x
	\\ &\quad\quad 
	+  \int_{\T}\pd_x \bk{u_\delta - u_\eta} 
	\pd_x \bk{\bk{u\,\pd_x u}_\delta - \bk{u\,\pd_x u}_\eta} \,\d x
	\\ &\quad\quad 
	+ \int_{\T} \pd_x \bk{u_\delta - u_\eta} 
	\bk{\bk{u^2}_\delta - \bk{u^2}_\eta 
	+\frac12 \bk{q^2}_\delta - \frac12 \bk{q^2}_\eta}\,\d x \bigg|
	\\ & 
	\le \norm{u_\delta - u_\eta}_{H^2_x } 
	\norm{\bk{u\,\pd_x u}_\delta - \bk{u\,\pd_x u}_\eta}_{L^2_x }
	\\ & \quad\quad 
	+\norm{\pd_x\bk{u_\delta - u_\eta}}_{L^\infty_x }
	\norm{\bk{u^2}_\delta - \bk{u^2}_\eta 
	+\frac12 \bk{q^2}_\delta-\frac12 \bk{q^2}_\eta}_{L^1_x },
\end{align*}
and therefore
\begin{align*}
	& \Ex \int_0^T \abs{I_2 + I_6} \,\d t 
	\\ & \quad
	\le \Ex \bk{\norm{u_\delta - u_\eta}_{L^2_tH^2_x} 
	\norm{\bk{u\,\pd_x u}_\delta-\bk{u\,\pd_x u}_\eta}_{L^\infty_t L^2_x }}
	\\ &\quad\qquad
	+\Ex \Biggl(\norm{\pd_x\bk{u_\delta - u_\eta}}_{L^2_t L^\infty_x }
	\\ & \quad\qquad\qquad\qquad
	\times \norm{\bk{u^2}_\delta - \bk{u^2}_\eta 
	+\frac12 \bk{q^2}_\delta-\frac12 \bk{q^2}_\eta}_{L^2_t L^1_x }\Biggr)
	\\ & \quad
	\le \bk{\Ex\norm{u_\delta - u_\eta}_{L^2_tH^2_x}^2}^{1/2}
	\bk{\Ex\norm{\bk{u\,\pd_x u}_\delta
	-\bk{u\,\pd_x u}_\eta}_{L^\infty_t L^2_x }^2}^{1/2}
	\\ & \quad\qquad
	+ \bk{\Ex\norm{\pd_x\bk{u_\delta-u_\eta}}_{L^2_t L^\infty_x }^2}^{1/2}
	\\ & \quad\qquad \qquad \qquad
	\times\bk{\Ex \norm{\bk{u^2}_\delta-\bk{u^2}_\eta
	+\frac12 \bk{q^2}_\delta
	-\frac12 \bk{q^2}_\eta}_{L^2_t L^1_x }^2}^{1/2}.
\end{align*}
By the Lebesgue dominated convergence theorem,
the inclusions $u \in L^2_{\omega,t}H^2_x$ 
and $u^2, q^2 \in L^2_\omega L^\infty_t L^1_x 
\subset  L^2_{\omega,t} L^1_x$ imply that one factor 
in each summand above tends to zero as $\eta, \delta \downarrow 0$. 
At the same time, the remaining factor in each summand is bounded  
because of the inclusions $u \in L^{p_0}_{\omega} L^\infty_{t,x}$ 
(with $p_0>4$), $\pd_x u \in L^{p_0}_{\omega} L^\infty_t L^2_x$ (so that 
$u\,\pd_x u \in L^2_{\omega} L^\infty_t L^2_x$), 
and $u \in L^2_{\omega,t}H^2_x 
\subset L^2_{\omega,t}W^{1,\infty}_x$. 

Next, by the BDG inequality, 
\begin{align*}
	&\Ex \sup_{t \in [0,T]} \abs{\int_0^t M_1 \,\d W}
	\le \Ex \bk{\int_0^T \abs{M_1}^2\,\d s }^{1/2}
	\\ &\qquad 
	\lesssim_\sigma \Ex \bk{\norm{u_\delta-u_\eta}_{L^\infty_tL^2_x} 
	\norm{f_\delta-f_\eta}_{L^2_tH^1_x)}} 
	\\ & \qquad 
	\le \bk{\Ex \norm{u_\delta-u_\eta}_{L^\infty_tL^2_x }^2}^{1/2}
	\bk{\Ex \norm{f_\delta-f_\eta}_{L^2_tL^2_x}^2}^{1/2},
\end{align*}
where $f=\sigma \,\pd_x u\in L^2_{\omega,t}H^1_x$. 
Similarly, with $g=\pd_x \bk{\sigma \,\pd_x u} 
\in L^2_{\omega,t}L^2_x$,
\begin{align*}
	&\Ex \sup_{t \in [0,T]} \abs{\int_0^t M_2 \,\d W}
	\lesssim_\sigma 
	\bk{\Ex \norm{u_\delta-u_\eta}_{L^\infty_t H^1_x}^2}^{1/2}
	\bk{\Ex \norm{g_\delta - g_\eta}_{L^2_tL^2_x}^2}^{1/2}.
\end{align*}
One factor of each term on the right 
in the two inequalities above tend to nought 
whilst the other remains bounded. 

Consolidating our findings, we execute an integration 
of \eqref{eq:cauchy_interim1} over the interval $s \in [0,t]$. 
This is succeeded by taking the supremum over $t \in [0,T]$ and 
subsequently computing the expectation. 
With these steps, we arrive at 
\begin{align*}
	\Ex \norm{u_\delta - u_\eta}_{C_t H^1_x}^2 
	& \le\Ex \norm{u_{0,\delta}-u_{0,\delta}}_{H^1_x}^2 
	+\Ex\int_0^T 
	\abs{\, \sum_{i\notin \{1,5\}}I_i}\,\d t
	\\ & \qquad
	+\Ex  \sup_{t \in [0,T]}
	\abs{\int_0^t \left(M_1 + M_2\right)\,\d W}
 	\xrightarrow{\delta, \eta \downarrow 0} 0.
\end{align*}
This implies that  $\left\{u_\delta\right\}$ is a 
Cauchy sequence in $L^2_\omega C_t H^1_x$. 
The limit $U$ of $\left\{u_\delta\right\}$ in 
$L^2_\omega C_t H^1_x$ and $u$ must 
coincide $(\omega,t, x)$-a.e. Indeed, 
since $U \in L^2_\omega C_t H^1_x$, it is evident that 
$U \in L^2_\omega L^2_t H^1_x$. Furthermore, due to the 
fact that $p_0 > 2$, we can conclude that $u$ also 
belongs to $L^2_\omega L^2_t H^1_x$. Therefore, we can establish 
that $u = U$ for a.e.~$(\omega, t, x)$, indicating that they belong 
to the same equivalence class.
\end{proof}

\end{document}